\documentclass[reqno,11pt]{article}

\usepackage{amsmath, amsthm, amsfonts, amssymb,mathrsfs,color}
\usepackage{upgreek}
\usepackage{authblk}   
\usepackage{pgfplots}
\usepackage{CJK}
\usepackage{capt-of}
\usepackage{charter}  
\usepackage{stix}  
\usepackage
[
a4paper,
hmargin=1.37cm,
vmargin=1.63cm
]{geometry}  

\usepackage{setspace}
\usepackage{tensor}
\usepackage{todonotes}
\usepackage{tikz}
\usetikzlibrary{decorations.pathmorphing}
\usepackage{hyperref}
\hypersetup{
	colorlinks,
	citecolor=blue,
	filecolor=black,
	linkcolor=blue,
	urlcolor=blue
}

\usepackage{url}

\usepackage{soul}
\usepackage{enumitem}

\usepackage[normalem]{ulem}
\usepackage[numbers,sort&compress]{natbib}

\usepackage{indentfirst}

\def\R{\mathbb R}
\def\Z{\mathbb Z}
\def\T{\mathbb T}
\def\N{\mathbb N}
\def\E{\mathbb E}
\def\K{\mathbb K}

\def\rr{\mathbf{r}}
\def\p{\mathbb P}

\def\F{\mathcal F}
\def\SS{\mathbf S}
\def\D{\mathcal D}
\def\L{\mathcal L}
\def\LL{\mathscr L}

\def\Q{\mathcal Q}
\def\QQ{\mathbf Q}
\def\ZZ{\mathbf Z}
\def\hh{\widehat}
\def\d{{\,\rm{d}}}
\def\nn{\nabla}
\def\I{{\mathbf I}}
\def\Wlip{W^{1,\infty}}
\def\WLip{{\mathbb W}^{1,\infty}}
\def\WP{{\mathbb W}^{p,\infty}}
\def\Wp{W^{p,\infty}}
\def\OP{{\rm OP}}
\def\B{\mathcal B}
\def\A{\mathcal A}
\def\H{\mathbb H}
\def\G{\mathcal G}

\def\M{\mathbf M}
\def\bfH{\mathbf H}
\def\J{\mathcal J}
\def\vrho{\varrho}
\def\V{\mathcal V}

\def\pas{{\mathbb P}\text{-}{\rm a.s.}}
\def\Div{{\rm div}\,}
\def\Hdiv{H_{\rm div}}

\numberwithin{equation}{section}

\newtheorem{Theorem}{Theorem}[section]
\newtheorem{Proposition}{Proposition}[section]
\newtheorem{Lemma}{Lemma}[section]
\newtheorem{Definition}{Definition}[section]

\newtheorem{Hypothesis}{Assumption}[section]

\newtheorem{Problem}{Problem}

\theoremstyle{definition}
\newtheorem{Example}{Example}[section]

\newtheorem{Remark}{Remark}[section]

\newcommand{\IP}[1]{\left\langle#1\right\rangle }       
\newcommand{\bIP}[1]{\big\langle#1\big\rangle }
\newcommand{\BIP}[1]{\Big\langle#1\Big\rangle }
\newcommand{\Abs}[1]{\left|#1\right|}   
\newcommand{\norm}[1]{\left\lVert#1\right\rVert}

\begin{document}
	
	\allowdisplaybreaks

	\title{\textbf{\textsc{Stochastic Euler Equations with Pseudo-differential Noise: Continuous and Discontinuous Perturbations in Compressible and Incompressible Flows}}}

	\author{
		\textbf{Kenneth. H. Karlsen}\thanks{K.H.K. is funded by the Research Council of Norway under project 351123 (NASTRAN).}
	}
	
	\affil{{\small Department of Mathematics, University of Oslo, NO-0316 Oslo, Norway.
			\href{mailto:kennethk@math.uio.no}{kennethk@math.uio.no}}}
	
	\author{
		\textbf{Hao Tang} 
		\thanks{H.T.  is supported by the National Natural Science Foundation of China under projects  24AAA00530 and 12531007.}
	}
	
	\affil{{\small 
			Center for Applied Mathematics, Tianjin University, Tianjin 300072, China.   \href{mailto:haotang@tju.edu.cn}{haotang@tju.edu.cn}
		}
	}
	
	\author{
		\textbf{Feng-Yu Wang}\thanks{F.-Y.W. is supported by the National Natural Science Foundation of China under project 12531007.}
	}

	\affil{{\small Center for Applied Mathematics, Tianjin University,Tianjin 300072, China.  
			\href{mailto:wangfy@tju.edu.cn}{wangfy@tju.edu.cn}}}

	\date{\today}

	\maketitle
	
	\begin{abstract}
		
		We study stochastic Euler equations in both compressible and incompressible regimes, on the whole space and on the torus, driven by genuinely mixed multiplicative noise: continuous Stratonovich/It\^o components and a discontinuous Marcus component. The Stratonovich and Marcus noise amplitudes are (nonlocal) pseudo-differential operators that include the classical transport operator as a special case. Within this setting, we develop a local-in-time theory of classical solutions for both regimes, establishing existence, uniqueness, and a blow-up criterion. The presence of discontinuous pseudo-differential Marcus noise necessitates new analytical tools, which we develop to control the delicate interaction between jump discontinuities and nonlocal operators.
		
		For the compressible barotropic equations, we formulate a
		transformation-and-extension principle that recovers the classical
		Makino variable and extends its associated symmetric quasilinear
		formulation. In the transformed variables, compatible nonlinear
		Sobolev composition estimates are developed to close the high-order
		stochastic energy bounds. The resulting framework accommodates a
		broad class of physically relevant equations of state, including
		piecewise-defined $\gamma$-laws, Chaplygin-type laws, and the pressure
		law for white dwarf stars. To the best of our knowledge, many of these
		concrete equations of state have remained unexplored in the
		multidimensional stochastic compressible setting, even under purely
		It\^o-type forcing.

		For the incompressible damped case, we identify a hierarchy of damping--noise regimes that successively guarantee global-in-time existence, uniform-in-time bounds, and decay property. To study the long-time statistical behavior, we establish a novel abstract existence criterion for invariant probability measures tailored to Markov semigroups satisfying a \emph{restricted Feller property under mismatched metrics}. By explicitly circumventing the requirement of Feller continuity within a single topology, this framework provides a robust extension of the classical Krylov--Bogoliubov theory. Utilizing this criterion, we construct invariant probability measures for a broad class of singular stochastic evolution systems in Hilbert spaces.
		As a principal application, for genuinely mixed multiplicative noise and for $d\ge 2$, we provide the existence of invariant measures for the stochastic damped Euler equations on $\mathbb{T}^d$ under a moderate damping--noise regime, and uniqueness of invariant measures under a strong damping--noise regime on both $\mathbb{T}^d$ and $\mathbb{R}^d$. 
	\end{abstract}

	\noindent\textbf{2020 AMS subject Classification:} Primary: 60H15, 35R60; Secondary: 35Q31, 35S10, 37L40.
	
	\noindent\textbf{Keywords:} Compressible and incompressible Euler equations; Pseudo-differential noise; Continuous and discontinuous perturbations; Local and global solutions; Generalized Krylov--Bogoliubov argument; Invariant probability measure.

	\newpage

	\tableofcontents
	
	\newpage
	
	\section{Introduction: Overview of Main Results}\label{Section : Introduction-Overview}
	
	The purpose of this paper is to study the dynamics of compressible and incompressible ideal fluids subject to random perturbations. The fluid occupies either the whole space $\R^d$ or the torus $\T^d\triangleq (\R/2\pi\mathbb Z)^d$, collectively denoted by $\K^d$, with spatial dimension $d \ge 2$. We write ${\rm div} = {\rm div}_x$ and $\nabla=\nabla_x$ for the divergence and gradient operators on $\K^d$, respectively.
	
	In the deterministic case, the Cauchy problem for the barotropic Euler equations describing the motion of a fluid without vacuum on $\mathbb{K}^d$ reads
	\begin{equation}\label{Euler (rho u)}
		\left\{
		\begin{aligned}
			&\partial_t \rho + \Div (\rho u) = 0, \quad \rho > 0,\\
			&\partial_t u + (u \cdot \nabla)u + \frac{1}{\rho}\nabla P(\rho) = 0,\\
			&\rho(0) = \rho_0, \ u(0) = u_0,
		\end{aligned}
		\right.
	\end{equation}
	where $u = u(t, x): [0, \infty) \times \mathbb{K}^d \to \mathbb{R}^d$ and $\rho = \rho(t, x): [0, \infty) \times \mathbb{K}^d \to (0, \infty)$ denote the unknown velocity field and density, respectively. The condition $\rho > 0$ excludes vacuum states, and the barotropic case corresponds to a pressure function $P: [0, \infty) \to \mathbb{R}$ depending solely on the density $\rho$.
	
	Similarly, the Cauchy problem for the deterministic incompressible Euler equations on $\K^d$ is
	\begin{equation}\label{Euler (u P)}
		\left\{\begin{aligned}
			&\partial_t u + (u \cdot\nabla)u + \nabla P= 0,\\
			&\Div u= 0,\\
			&u(0) = u_0.
		\end{aligned}
		\right.
	\end{equation}
	In \eqref{Euler (u P)}, $u = u(t, x): [0, \infty) \times \mathbb{K}^d \to \mathbb{R}^d$ and $P=P(t, x): [0, \infty) \times \mathbb{K}^d \to \mathbb{R}$ represent the unknown velocity field and pressure, respectively.
	
	We now introduce random perturbations into these classical models. We fix a complete, right-continuous filtered probability space $(\Omega,\mathcal{F},\{\mathcal{F}_t\}_{t \ge 0}, \mathbb{P})$, on which we let $W=W(t)$ and $\widetilde{W}=\widetilde{W}(t)$ be independent standard 1-D $\{\mathcal{F}_t\}$-adapted Brownian motions, and $L=L(t)$ be a \textit{pure-jump} 1-D $\{\mathcal{F}_t\}$-adapted L\'{e}vy process, independent of both $W$ and $\widetilde{W}$. The random perturbation acting on the velocity equation combines a continuous Stratonovich component ``$\circ\,\mathrm{d} W$'' and an It\^{o} component ``$\mathrm{d} \widetilde{W}$'', together with a \textit{discontinuous} Marcus component ``$\diamond \mathrm{d} L$''. This leads to the following stochastic version of the barotropic Euler system on $\mathbb{K}^d$:
	\begin{equation}\label{SEuler-(rho u)}
		\left\{
		\begin{aligned}
			&\mathrm{d}\rho + \Div (\rho u)\d t = 0,\quad \rho(t) > 0,
			\\ &\mathrm{d} u + \left[(u \cdot \nabla)u + \frac{\nabla P(\rho)}{\rho}\right]\d t = \mathcal{Q}_1 u \circ \mathrm{d} W + \mathcal{Q}_2 u \diamond \mathrm{d} L + h(t,\rho, u)\mathrm{d} \widetilde{W}, \\
			&\rho(0) = \rho_0, \ u(0) = u_0,
		\end{aligned}
		\right.
	\end{equation}
	and, in the incompressible setting, to the stochastic incompressible Euler system:
	\begin{equation}\label{SEuler-(u P)}
		\left\{
		\begin{aligned}
			&\mathrm{d}u + [(u \cdot \nabla)u]\d t + \nabla \mathrm{d}P = \mathcal{Q}_1 u \circ \mathrm{d} W + \mathcal{Q}_2 u \diamond \mathrm{d} L + \widetilde{h}(t, u)\mathrm{d} \widetilde{W},\\
			&\Div u = 0,\\
			&u(0) = u_0.
		\end{aligned}
		\right.
	\end{equation}
	In \eqref{SEuler-(rho u)} and \eqref{SEuler-(u P)},
	the noise-amplitude operators $\Q_1$ and $\Q_2$ are (matrix-valued) pseudo-differential operators satisfying suitable conditions that include classical differential operators as special cases and may be highly nonlocal. 
	The It\^{o}-type forcing coefficients $h$ and $\widetilde h$ are nonlinear functions satisfying conditions specified below. In the long-time analysis we will   restrict to the time-homogeneous case $\widetilde h(t,u)\equiv \widetilde h(u)$; see \eqref{SEuler-(u P) damp} below. 
	
	Here we note that, without the \textit{a priori} assumption that the noise coefficients are divergence-free, the pressure in \eqref{SEuler-(u P)} must absorb the martingale components of the stochastic forcing to enforce the incompressibility constraint $\operatorname{div} u = 0$. Consequently, the pressure inherently acts as a semi-martingale in time rather than an absolutely continuous function, which implies we must use the differential form $\nabla \mathrm{d}P$ instead of $\nabla P \,\mathrm{d}t$.

	We summarize the main contributions of this work as follows. Detailed discussion of the main challenges, novel approaches, and comparisons with existing literature appear in Section \ref{Section : novel-introduction}.
	
	$\bullet$  
	We develop a local-in-time theory of classical solutions (establishing existence, uniqueness, and a blow-up criterion; see Theorems \ref{Thm-(rho u)} and \ref{Thm-(u P)-local}) for both stochastic compressible and incompressible regimes on $\mathbb{K}^d$, driven by genuinely mixed noise comprising continuous Stratonovich and It\^o components alongside a discontinuous Marcus component. The analysis addresses the coupling between the nonlinear hyperbolic structure of the Euler equations and nonlocal pseudo-differential noise, interpreted in both the Stratonovich and Marcus senses. In particular, as pseudo-differential Marcus noise has not been previously treated, we develop  analytical tools to handle the delicate interplay between jump discontinuities and nonlocal operators (see Theorem \ref{Thm-cancel} and Section \ref{Section : Appl-cancel}).
	
	$\bullet$  In the compressible regime, we formulate an abstract
	transformation-and-extension principle; see \ref{Hypo-Pressure}.
	It combines sound-speed symmetrization with a global extension
	criterion for the transformed sound-speed coefficient and with the
	nonlinear Sobolev composition and local Lipschitz estimates required
	for the stochastic high-regularity analysis. This framework allows
	us to treat \eqref{SEuler-(rho u)} for a broad class of physically admissible pressure
	laws $P=P(\rho)$, including two structural regimes beyond the
	classical polytropic Makino setting: laws whose sound speed becomes
	singular as vacuum is approached and laws exhibiting genuinely
	non-power-law sound-speed degeneracy near vacuum. In addition to the
	classical polytropic $\gamma$-law, this class includes piecewise-defined
	$\gamma$-laws, (piecewise-defined) Chaplygin-type laws, the pressure
	law for white dwarf stars, and other astrophysically motivated
	examples; see see Examples \ref{Example-Chaplygin}--\ref{Example-vacuum limit}. To the best of our knowledge, even
	in the purely It\^o setting, the multidimensional stochastic
	compressible Euler equations have not been analyzed under such
	general equations of state. Consequently, our framework provides
	the first investigation of these models in the stochastic context.

	$\bullet$  For the damped incompressible case (i.e., augmenting system \eqref{SEuler-(u P)} with a linear damping term $\Upgamma u\d t$), we identify a hierarchy of interaction conditions between the damping and the noise terms $\mathcal{Q}_1 u\, \circ\, {\rm d} W$, $\mathcal{Q}_2 u \diamond {\rm d} L$, and $\widetilde{h}(u)\d \widetilde{W}$  (see \ref{Hypo-DNI-1}--\ref{Hypo-DNI-3}) that effectively suppresses potential finite-time blow-up, leading to global-in-time existence, uniform-in-time bounds, and decay estimates.
	To capture the corresponding statistical behavior, we confront a fundamental topological obstruction: the absence of viscous smoothing prevents us from establishing the combination of compactness and strong-topology Feller continuity required by the classical argument. We bypass this limitation by developing a robust extension of the classical Krylov--Bogoliubov criterion (Theorem \ref{Thm-generalized KB}) anchored in the concept of a \emph{restricted Feller property under mismatched metrics} (Definition \ref{Locally Feller definition}).  This criterion enables us to establish the existence of invariant probability measures for a broad class of singular stochastic evolution systems in Hilbert spaces.
	When specialized to the stochastic damped Euler equations under fully mixed multiplicative noise, across arbitrary spatial dimensions $d \ge 2$ (Theorem \ref{Thm-(u P)-long-time}), we provide the existence and uniqueness of invariant measures under different damping--noise regimes. To the best of our knowledge,  these results have not been previously established in the literature.

	\subsection{Noise Structure}\label{Section : Noise Structure}
	
	The first key contribution of this work is the systematic treatment of both continuous and discontinuous stochastic effects, incorporating local (transport-type) and nonlocal (pseudo-differential) perturbations. 
	
	More precisely, in \eqref{SEuler-(rho u)} and \eqref{SEuler-(u P)}, we consider common noise terms of the form
	$$
	\mathcal{Q}_1 u\, \circ\, {\rm d}W
	+
	\mathcal{Q}_2 u \diamond {\rm d}L+h(t,\rho,u)\,{\rm d}\widetilde{W}\quad  \text{and}\quad  \mathcal{Q}_1 u\, \circ\, {\rm d}W
	+
	\mathcal{Q}_2 u \diamond {\rm d}L+ \widetilde{h}(t,u)\,{\rm d}\widetilde{W},
	$$
	respectively. 
	Here, the standard 1-D Brownian motions $W=W(t)$ and $\widetilde{W}=\widetilde{W}(t)$, along with the \textit{pure-jump} 1-D L\'{e}vy process $L=L(t)$, are assumed to be mutually independent. Furthermore, the noise amplitude operators $\mathcal{Q}_j$ ($j=1,2$) may be local transport-type differential operators or nonlocal pseudo-differential operators.
	
	Here the pure-jump L\'{e}vy process $L$ is represented in the standard form \cite{Applebaum-2009-Book}
	\begin{equation}\label{eq:Levy-Itocomp}
		L(t)=\int_0^t \int_{|l|\le 1} l
		\, \tilde{\eta}({\rm d}l,{\rm d}t')
		+\int_0^t \int_{|l|>1} l
		\, \eta({\rm d}l,{\rm d}t'),
	\end{equation}
	where $\eta$ is a Poisson random measure counting jumps of size $l\in\R$, $\tilde{\eta} = \eta - \nu \otimes {\rm d}t$ is its compensated version, and $\nu(\cdot)$ denotes the L\'{e}vy measure.
	
	\subsubsection{Continuous and Discontinuous Noise}\label{Section:introduce noise}
	
	Recall that for a semimartingale $g=g(t)$, the following identity holds:
	\begin{equation}\label{Stratonovich to Ito}
		g\,\circ \d W = 
		g  \d W+\frac{1}{2} \d \left \langle g , W\right \rangle,\quad
		\left \langle \cdot ,\cdot \right \rangle \
		\text{denotes the quadratic variation}.
	\end{equation}
	Hence, the Stratonovich term $\mathcal{Q}_1 u\, \circ\, {\rm d}W$ can be rewritten in 
	the sense of It\^{o} as
	\begin{equation}\label{eq:Strat-Ito}
		\mathcal{Q}_1 u\, \circ\, {\rm d}W
		=
		\tfrac{1}{2}\mathcal{Q}_1^2 u\,{\rm d}t+\mathcal{Q}_1 u\,{\rm d}W ,
	\end{equation}
	so that the velocity component of the 
	stochastic Euler system \eqref{SEuler-(rho u)}
	can be expressed as
	$$
	{\rm d}u 
	+\Big[(u\cdot\nabla)u + \frac{\nabla P(\rho)}{\rho} 
	-\tfrac{1}{2}\mathcal{Q}_1^2 u \Big]{\rm d}t
	=\mathcal{Q}_1 u\,{\rm d}W 
	+\mathcal{Q}_2 u \diamond {\rm d}L 
	+h(t,\rho,u)\,{\rm d}\widetilde{W}.
	$$
	
	The noise-amplitude operators $\mathcal{Q}_j$, $j=1,2$, may be local differential operators; for example,
	\begin{equation}\label{eq:local-Q}
		\mathcal{Q}_j u=\sum_{i=1}^d
		f_i\partial_{x_i} u,
	\end{equation}
	where $f_i=f_i(\cdot)$ ($i=1,2,\cdots,d$) are smooth functions.  Note that, although the operator $\tfrac{1}{2}\mathcal{Q}_1^2$ may be of parabolic type, the overall equation remains hyperbolic. In recent years, different classes of stochastic partial differential equations (SPDEs) driven by transport-type noise of the form $\sum_i f_i\partial_{x_i} u\, \circ\, {\rm d}W$ have attracted considerable attention, with substantial progress made particularly for incompressible fluid dynamics; see, e.g., \cite{Flandoli-Luo-2020-AoP,
		Flandoli-Luo-2021-PTRF,Crisan-Flandoli-Holm-2019-JNLS,
		Flandoli-Lisei-2004-SAA,Mikulevicius-Rozovskii-2005-AoP,Alonso-Rohde-Tang-2021-JNLS}. For shallow water wave  models, see \cite{Holden-Karlsen-Pang-2021-JDE,
		Holden-Karlsen-Pang-2023-DCDS,
		Albeverio-etal-2021-JDE,
		Galimberti-etal-2024-JDE}, and for the compressible fluid equations, see \cite{Breit-etal-2022-SIMA,Boadi-Breit-Moyo-2025-arXiv}. 
	Nonlocal cases are discussed in Section \ref{Section : PD Noise Amplitudes}.
	
	The modern study of stochastic transport equations is largely motivated by Kraichnan's model of turbulent advection \cite{Kraichnan-1968-PhyFluid,Kraichnan-1994-PRL}, in which the velocity field is modeled as a Gaussian random field. Subsequent work (see, e.g., \cite{Birnir-2013-Book}) has argued that realistic turbulent forcing is better captured by non-Gaussian noise, necessitating a framework capable of handling discontinuous stochastic signals. The Marcus integral thus emerges as a natural tool for describing stochastic transport under jump noise.  
	
	Moreover, the Marcus framework extends Stratonovich calculus to discontinuous signals by lifting each jump of the driving L\'{e}vy process to a deterministic flow along an associated vector field, thereby preserving geometric consistency and conservation properties. To make the discussion precise, let $b:\mathbb{H}\to \mathbb{H}$ be a sufficiently regular vector field on a separable Hilbert space $\mathbb{H}$, and let $\wp(r,l,f)$ solve the infinite-dimensional ODE
	\begin{equation}\label{Marcus flow H}
		\frac{\mathrm{d}}{\mathrm{d}r}\wp(r,l,f)
		= l \cdot b(\wp(r,l,f)), 
		\qquad \wp(0)
		= f \in \mathbb{H}, \quad r \in [0,1],
	\end{equation}
	which, in our context, corresponds to a local or nonlocal PDE. For an $\mathbb{H}$-valued process $X(t)$ driven by a pure-jump L\'{e}vy process $L(t)$ (see \eqref{eq:Levy-Itocomp}), the Marcus integral is defined as 
	\begin{align}
		\int_0^t b(X(t')) \diamond \mathrm{d}L(t')
		\triangleq \
		& \int_0^t \int_{|l|\le 1}
		\big\{ \wp(1,l,X(t'-)) - X(t'-)\big\}\,
		\widetilde{\eta}(\mathrm{d}l,\mathrm{d}t') \notag\\ 
		&  +\int_0^t \int_{|l|\le 1}
		\big\{ \wp(1,l,X(t')) - X(t')
		- l \cdot b(X(t'))\big\}\,
		\nu(\mathrm{d}l)\,\mathrm{d}t' \notag\\ 
		&  +\int_0^t \int_{|l|>1}
		\big\{ \wp(1,l,X(t'-)) - X(t'-) \big\}\,
		\eta(\mathrm{d}l,\mathrm{d}t').\label{Marcus integral define}
	\end{align}
	Each jump of $L$ induces a deterministic flow along $b$, ensuring compatibility with the chain rule and with the geometric structure of the system. In this sense, the Marcus calculus provides the appropriate generalization of Stratonovich integration to discontinuous noise signals. 
	
	In the same spirit, for the Cauchy problems \eqref{SEuler-(rho u)} and \eqref{SEuler-(u P)} in Sobolev spaces, the underlying Marcus differential $\mathcal{Q}_2 u \diamond \mathrm{d}L$ can be defined analogously by formally setting $b(\bullet) = \mathcal{Q}_2 \bullet$. However, we emphasize that, in this case, \eqref{Marcus flow H} becomes a singular ODE, since the map $\mathcal{Q}_2$ cannot preserve Sobolev regularity if $\mathcal{Q}_2$ is of positive order. More details on this difficulty will be provided in Section \ref{Section : novel-introduction}.
	
	Although the theory of Marcus stochastic differential equations is well developed \cite{Applebaum-2009-Book}, the corresponding theory for SPDEs remains at an early stage. Only a handful of works \cite{Brzezniak-Manna-2019-CMP,Brzezniak-Manna-Panda-2019-JDE,Chen-Duan-Gao-2024-SPA,Brzezniak-etal-2025-arXiv,Luo-Teng-2025-arXiv} address SPDEs driven by Marcus-type noise. In particular, only \cite{Chen-Duan-Gao-2024-SPA,Brzezniak-etal-2025-arXiv,Luo-Teng-2025-arXiv} consider classical differential operators in the noise term. To the best of our knowledge, no  results exist for 3-D Euler systems with Marcus noise, even in the incompressible case. The present work  helps close this gap by establishing a local-in-time theory for both compressible and incompressible Euler equations with discontinuous Marcus-type noise of the form $\mathcal{Q}_2 u \diamond {\rm d}L$, where $\mathcal{Q}_2$ may, for example, be a local (unbounded) operator as in \eqref{eq:local-Q}. The nonlocal case of $\Q_i$ ($i=1,2$) is discussed below in Section \ref{Section : PD Noise Amplitudes}.
	
	Summarizing, $\mathcal{Q}_1 u\, \circ\, {\rm d} W
	+\mathcal{Q}_2 u \diamond {\rm d} L $ in \eqref{SEuler-(rho u)} and \eqref{SEuler-(u P)} are understood in the Stratonovich sense for the continuous part $\mathcal{Q}_1u\,\circ\, {\rm d}W$—interpreted in the It\^{o} form through \eqref{eq:Strat-Ito}—and in the \emph{Marcus} sense for the jump part $\mathcal{Q}_2 u \diamond {\rm d}L$—interpreted via \eqref{Marcus integral define} with $b(\bullet)=\mathcal{Q}_2 \bullet$. Both interpretations \textit{obey the classical chain rule} for smooth functionals, a property essential for preserving the conservative structure of the underlying deterministic system. The Stratonovich formulation arises as the limit of smooth approximations of Brownian motion (the Wong--Zakai principle) and is therefore the appropriate choice for continuous stochastic signals. However, when the driving process exhibits jumps, as in the case of L\'{e}vy processes $L$, the Stratonovich integral no longer satisfies the chain rule. This shortcoming is remedied by the Marcus interpretation, first introduced by Marcus \cite{Marcus-1980/81-Stoch} and subsequently developed by Applebaum, Kunita, Kurtz, Pardoux, and Protter, with a comprehensive account provided in \cite{Applebaum-2009-Book}.
	
	\subsubsection{Pseudo-differential Noise Amplitudes}\label{Section : PD Noise Amplitudes}
	
	The operators $\mathcal{Q}_j$ in 
	\eqref{eq:local-Q} are local in $x$, but 
	several studies of turbulent flow models have 
	highlighted the importance of nonlocal transport 
	effects \cite{Majda-Gershgorin-2013-PTRSL,
		Hamba-2022-JFM}. 
	Motivated by these observations, we allow both 
	the continuous and discontinuous noise terms, 
	$\mathcal{Q}_1 u\, \circ\, {\rm d}W$ and 
	$\mathcal{Q}_2 u \diamond {\rm d}L$, 
	to involve nonlocal (pseudo-differential) operators 
	generalizing the local (directional) derivatives 
	appearing in classical transport 
	noise (see \eqref{eq:local-Q}). 
	Let $\I$ denote the identity operator, and let 
	$f_i=f_i(\cdot)$ ($i=1,2,\cdots,d$) be smooth functions. 
	A representative example is the 
	Bessel-type transport operator
	\begin{equation}\label{eq:Bessel-Q}
		\mathcal{Q}_j u
		=\sum_{i=1}^d
		f_i\partial_{x_i}
		(\I-\Delta)^{\alpha}u,
		\qquad \alpha\in\R,
	\end{equation}
	which provides a pseudo-differential extension of 
	local derivatives, changing the differential order 
	of $\mathcal{Q}_j$ from $1$ (in \eqref{eq:local-Q}) 
	to $1+2\alpha$ (in \eqref{eq:Bessel-Q}) and 
	coupling local and global spatial scales. 
	For constant coefficients $f_i$, this $\mathcal{Q}_j$ 
	is skew-adjoint (so $\langle \mathcal{Q}_j u, u\rangle_{L^2}=0$) 
	while for variable coefficients this symmetry is preserved only up to 
	lower-order terms. Our framework allows 
	for operators that are almost skew-adjoint, 
	differing from a skew-adjoint operator by a zero-order 
	pseudo-differential operator, which can be 
	handled perturbatively in the Sobolev estimates.
	SPDEs involving (continuous) pseudo-differential noise have only recently begun to be studied; see \cite{Tang-Wang-2022-arXiv,Tang-2023-JFA,Alonso-Pang-Tang-2026-JLMS}.

	Another interesting example arises 
	from \emph{fractional homogeneous operators}
	$$
	\mathcal{Q}_j u
	= \sum_{i=1}^d c_i 
	\mathcal{R}_i(-\Delta)^{\varsigma/2} u,
	\qquad \varsigma\ge0
	$$
	representing nonlocal directional derivatives 
	of order $\varsigma$, built from the fractional 
	Laplacian $(-\Delta)^{\varsigma/2}$, the Riesz 
	transforms $\mathcal{R}_i$, and  coefficients $c_i\in\R$. 
	The symbol of $\mathcal{Q}_j$ is purely imaginary and 
	odd, so $\mathcal{Q}_j$ is skew-adjoint 
	on $L^2(\mathbb{K}^d)$, ensuring energy conservation (when $\K=\T$ we need a zero-average condition, see \eqref{Pi=Pi-d-0} below). 
	Although the symbol is not smooth at $0$ for 
	non-integer $\varsigma$, such operators are 
	homogeneous Fourier multipliers and can be 
	incorporated within our pseudo-differential 
	framework with some modifications of the 
	arguments (see Section \ref{Section:Mikhlin}). 
	In contrast, the fractional Laplacian 
	$(-\Delta)^{\varsigma/2}$ 
	is self-adjoint and dissipative, thereby 
	breaking the conservative balance.
	
	\subsection{General Pressure in the Compressible Barotropic Case}
	
	The second key novelty of this work is the identification of an abstract transformation  framework that generalizes the classical Makino transform \cite{Makino-1986-chapter}. This allows us to study \eqref{SEuler-(rho u)} under a broad class of physically admissible pressure laws \(P=P(\rho)\), far beyond the standard polytropic setting.
	
	Indeed, most existing works on stochastic barotropic fluids are restricted to the polytropic equation of state
	\[
	P(\rho)=a\rho^\gamma,\qquad a>0,\ \gamma\ge1,
	\]
	(cf. the \(\gamma\)-law in Example~\ref{Example-gamma}; see, for instance,
	\cite{Breit-Feireisl-Hofmanova-2017-CMP,Breit-Feireisl-Hofmanova-2018-CPDE,Breit-Hofmanova-2016-IUMJ,Breit-Mensah-2019-NA,Hofmanova-Koley-Sarkar-2022-CPDE}),
	which in particular includes the isothermal case \(\gamma=1\) and the monatomic gas case \(\gamma=\tfrac53\).
	A common feature of these works is the use of the Makino transform, which reformulates the equations in terms of a suitable density-dependent variable. However, the Makino transform \cite{Makino-1986-chapter} fundamentally relies on the algebraic structure of the polytropic pressure law and the corresponding power-law behavior near vacuum. Consequently, it does not apply to more general equations of state, such as Chaplygin-type laws, for which the sound speed becomes singular as vacuum is approached, nor does it capture genuinely non-power-law vacuum degeneracies.
	
	Our abstract framework overcomes these limitations. It contains the classical polytropic \(\gamma\)-law as a special case, while at the same time covering several important pressure laws that have not previously been treated in the stochastic literature. For instance, our theory includes the Chaplygin gas—arising in cosmological and dark-energy models \cite{Bento-Bertolami-Sen-2002-PRD,Ferreira-Avelino-2018-PRD}—as well as piecewise-defined pressure laws combining polytropic and Chaplygin-type behavior (see Example~\ref{Example-Mixed}). Such models describe transitions between distinct thermodynamic regimes, in which the constitutive response of the fluid changes with the density. Our framework also applies to physically relevant equations of state from astrophysics, such as models for white dwarf stars (cf. Example~\ref{Example-white dwarf star}; see, e.g., \cite{Strauss-Wu-2020-Nonlinearity,Chen-etal-2024-CMP}), where the pressure may deviate substantially from polytropic behavior and may, for example, be given by
	\[
	P(\rho) = c_1 \int_0^{c_2 \rho^{1/3}} 
	\frac{y^4}{\sqrt{c_3 + y^2}}\,\mathrm{d}y, 
	\qquad \rho>0,
	\]
	with \(c_1,c_2,c_3>0\). We also refer to Example~\ref{Example-vacuum limit} for an  astrophysically motivated 
	pressure law whose sound speed exhibits genuinely non-power-law vacuum degeneracy.
	
	Notably, even in the purely It\^o-forced case (\(\mathcal{Q}_1=\mathcal{Q}_2=0\)), these general pressure laws have not been investigated in the stochastic setting, and, as far as we know, even the corresponding local-in-time existence theory was open prior to the present work.

	\subsection{Long-time Behavior in the Incompressible Damped Case}\label{Section : Introduction: Long-time behavior}
	
	Shifting focus to the incompressible regime, we investigate the long-time behavior of solutions under purely zeroth-order dissipation. Consider the general stochastic damped Euler equations on $\K^d$:
	\begin{equation}\label{SEuler-damp-general}
		\partial_t u + (u \cdot\nabla)u +\Upgamma u+ \nabla P= \phi,\quad \Div u= 0,\quad \Upgamma\ge0,
	\end{equation}
	where $u$ and $P$ are defined as in \eqref{SEuler-(u P)}, and $\phi$ represents a random forcing structured to ensure that \eqref{SEuler-damp-general} generates a well-defined Markov transition semigroup. For $\Upgamma > 0$, the Ekman damping term $\Upgamma u$ constitutes the most rudimentary form of dissipation---it acts uniformly across all wavenumbers and provides strictly zeroth-order energy depletion devoid of any parabolic regularization. This essential lack of smoothing renders the ergodic analysis of \eqref{SEuler-damp-general} especially  challenging, leaving fundamental questions unresolved. Among these is a notable question proposed by Armen Shirikyan in \cite{Shirikyan-2018-problem}:
	
	\begin{Problem}\label{Problem-AS}
		In the 2-D periodic setting ($\K=\T$ and $d=2$), prove that the stochastic damped Euler system \eqref{SEuler-damp-general} admits a unique invariant probability measure, and that the law of any solution converges to this measure as $t \to +\infty$.
	\end{Problem}
	
	To bypass the analytical difficulties caused by this lack of dissipation, most of the existing literature introduces (fractional) diffusion, which transforms the system into a Navier--Stokes-type equation:
	\begin{equation}\label{SNS-general}
		\partial_t u + (u \cdot\nabla)u +(-\Delta)^{\ell/2}u+ \nabla P= \phi,\quad \Div u= 0,\quad \ell\in(0,2].
	\end{equation}
	
	In this setting on $\mathbb{T}^2$, the existence and uniqueness of invariant probability measures for full viscosity ($\ell = 2$) driven by random kick-type noise were famously established in \cite{Kuksin-Shirikyan-2000-CMP} (see also \cite{Kuksin-Shirikyan-2012-book}). For models driven by continuous additive noise, seminal contributions include \cite{E-Mattingly-Sinai-2001-CMP,Hairer-Mattingly-2006-Annals} for the standard viscous case ($\ell= 2$), and \cite{Constantin-Glatt-Holtz-Vicol-2014-CMP} for the fractional regime ($\ell> 0$). When $\phi$ is a regular bounded function of time satisfying some decomposability and observability hypotheses, we refer to \cite{Kuksin-Nersesyan-Shirikyan-2020-GAFA}. 
	Furthermore, the mathematically delicate framework of stationary, non-$\delta$-correlated noise has been systematically resolved in recent work \cite{Kuksin-Shirikyan-2025-GAFA}.  In the more complex 3-D setting, the existence of invariant measures via Markov selections for \eqref{SNS-general} with additive noise was studied in \cite{DaPrato-Debussche-2003-JMPA,Flandoli-Romito-2008-PTRF}. Non-uniqueness in ergodicity is established in \cite{Hofmanova-Zhu-Zhu-2025-ARMA}, with related models discussed in \cite{Hofmanova-Zhu-Zhu-2023-JFA}.  
	
	Returning to the purely damped, inviscid regime \eqref{SEuler-damp-general}, the lack of viscosity makes the ergodic analysis highly non-trivial. For instance, when explaining their preference for analyzing the 2-D stochastic Euler equations with fractional diffusion rather than simple damping, the authors of \cite{Constantin-Glatt-Holtz-Vicol-2014-CMP} noted: \emph{``Unfortunately, the ergodic theory for the stochastically forced damped-driven Euler equations seems to be far from reach at the moment: this is in part due to the lack of compactness or continuous dependence in a suitable Polish space.''}
	
	Indeed, progress in the purely damped setting remains scarce. The authors in \cite{Bessaih-Ferrario-2020-CMP} studied the vorticity formulation of the 2-D stochastic Euler equations with additive noise, establishing the existence of an invariant probability measure defined on $\mathscr{B}\big((L^\infty, \mathscr{T}_{\rm bw*})\big)$---the Borel $\sigma$-algebra generated by the bounded weak-$\ast$ topology.  However,  it remains unclear whether such a measure is well-defined on the much larger Borel $\sigma$-algebra associated with the strong $L^\infty$-norm topology. Furthermore, questions regarding the uniqueness of the invariant measure and the long-time convergence of the distributions remain open.

	In this work, we study
	Shirikyan's problem \ref{Problem-AS} on $\K^d$, and we do so in a setting that allows for genuinely mixed multiplicative noise. Assuming for simplicity that the noise coefficient $\widetilde{h}$ in \eqref{SEuler-(u P)} is time-homogeneous, our main target model reads:
	\begin{equation}\label{SEuler-(u P) damp}
		\left\{
		\begin{aligned}
			&{\rm d}u + \big[(u \cdot \nabla)u +\Upgamma u\big] \d t  + \nabla \mathrm{d}P = \mathcal{Q}_1 u\, \circ\, {\rm d} W
			+\mathcal{Q}_2 u \diamond {\rm d} L		+ \widetilde{h}(u)\d \widetilde{W},\quad \Upgamma\ge0,\\
			&\Div u = 0,\\
			&u(0) = u_0.
		\end{aligned}
		\right.
	\end{equation}
	Through our analysis, we abstract the limitation arising from the absence of viscous dissipation and  propose a new concept of a \emph{restricted Feller property under mismatched metrics} (see Definition \ref{Locally Feller definition}).  Then, we 
	develop a robust extension of the classical Krylov--Bogoliubov criterion (see Theorem \ref{Thm-generalized KB} blow) anchored in such a weaker condition. 
	In addition to this novel criterion, we also  identify a clear hierarchy of damping--noise conditions  in any dimension $d\ge2$:
	\begin{itemize}[leftmargin=0.79cm]\setlength\itemsep{0.2em}
		\item weak damping--noise condition (\ref{Hypo-DNI-1}) $\Longrightarrow$ global-in-time solution on $\K^d$;
		\item  moderate damping--noise condition (\ref{Hypo-DNI-2}) $+$ new criterion (Theorem \ref{Thm-generalized KB}) 
		$\Longrightarrow$  existence of invariant measures on $\T^d$;
		\item strong damping--noise condition (\ref{Hypo-DNI-3}) $\Longrightarrow$ existence and uniqueness of the invariant measure on $\K^d$.
	\end{itemize}
	The last case provides a positive answer to a considerably strengthened version  of Shirikyan's problem \ref{Problem-AS} for genuinely mixed noise $\phi$: while the original question is posed on $\mathbb{T}^2$, our results apply to both $\mathbb{T}^d$ and $\mathbb{R}^d$ in all dimensions $d \ge 2$. 
	It is worthwhile emphasizing that our conclusions apply naturally to the viscous regime---specifically, to the stochastic damped Navier--Stokes equations with fractional or degenerate diffusion on $\R^d$. The reason is that the above framework does \textbf{not} rely on any regularizing effect from viscosity (see Remarks \ref{Remark: local-(u P) viscosity} and \ref{Remark: global-(u P) viscosity} below).

	\subsection{Outline of the paper}
	We conclude this section with a brief outline. 
	
	Section \ref{Section : novel-introduction} summarizes the main contributions, highlights the key difficulties and the novel approaches developed to address them, and situates our results within the related literature.
	
	Section \ref{Section:Notations} introduces the notation and function spaces used throughout the paper, recalls basic properties of pseudo-differential operators, and formulates the fundamental assumptions on the driving noise process $(W,\widetilde W, L)$.
	
	Section \ref{Section : Cancel} develops the analytic tools for handling the pseudo-differential noise amplitudes. We introduce the class of nearly skew-symmetric operators (Section \ref{Section : Nearly Skew-Symmetric}) and state the cancellation properties for such operators (Theorem \ref{Thm-cancel}, Section \ref{Section : Main Cancellation Results}), which
	serve as core analytic tools for the remainder of the paper. We provide detailed proofs of these estimates (Section \ref{Section : Proof of Thm-cancel}) and subsequently apply them to Stratonovich (Section \ref{Section : Q1n}) and Marcus (Section \ref{Section : Q2n}) type noise, clarifying how the structure of the noise can be regularized while preserving the cancellation properties uniformly.

	Section \ref{Section : Compressible case} concerns the stochastic compressible Euler equations. We present the precise assumptions and the target model (Section \ref{Section : Compressible-Assumptions}), discuss a broad class of admissible pressure laws (Section \ref{Section : Pressure Examples}), and state the local-in-time theory (Theorem \ref{Thm-(rho u)}, Section \ref{Section : Result Thm-compressible}). The proof (Section \ref{Section : Proof of Compressible}) proceeds in two main steps: an unconstrained approximation scheme with uniform estimates (Section \ref{Section : Unconstrained Problem}), followed by the recovery of the positivity of the density (Section \ref{Section : Constrained Problem}).
	
	Section \ref{Section : ASSES} addresses the existence of invariant probability measures for Markov semigroups with restricted Feller property under mismatched metrics. We review the classical Krylov--Bogoliubov approach (Section \ref{Section : Mismatched topo}) and state a novel criterion permitting mismatched topologies (Theorem \ref{Thm-generalized KB}). We provide the proof (Section \ref{Section : gKB-proof}) and apply this framework to an abstract singular stochastic evolution system encompassing a broad class of fluid models (Section \ref{Section : SSES}).
	
	Section \ref{Section : Incompressible case} treats the incompressible damped Euler model. We specify the standing assumptions (Section \ref{Section : Incompressible-Assumptions}), state the local existence and long-time behavior of solutions (Theorems \ref{Thm-(u P)-local} and \ref{Thm-(u P)-long-time}, Section \ref{Section : Result Thm-DNI}), and provide the proofs (Sections \ref{Section : SEuler-local-proof} and \ref{Section : SEuler-global-proof}).
	
	Section \ref{Section : Further Discussions} first presents extensions concerning homogeneous Mikhlin-type noise amplitudes, results to a large class of related fluid models, and sequences of noise. Finally, we discuss an open problem regarding the global regularity for the 2-D  stochastic Euler system (Section \ref{Section : Further-2D-problem}).
	
	Appendix \ref{Section : Appendix} collects auxiliary results on stochastic analysis (It\^o's formula, the Skorokhod topology, weak convergence), deterministic estimates for mollifiers and pseudo-differential operators.

	\section{Challenges, Novel Ideas, and Related Work}\label{Section : novel-introduction}
	
	We recall that the equations are driven by mixed Stratonovich-It\^o-Marcus noise as described in Section \ref{Section : Noise Structure}. 
	With the continuous and discontinuous noise terms interpreted as in Section \ref{Section : Noise Structure}, we establish local-in-time theories for the stochastic compressible and incompressible damped Euler systems \eqref{SEuler-(rho u)} and \eqref{SEuler-(u P) damp}, respectively, including existence, uniqueness, and a blow-up criterion for classical solutions. Moreover, for the stochastic incompressible damped case, we quantify distinct long-time regimes induced by the interaction between damping $\Upgamma u\d t$ and the noise terms $\mathcal{Q}_1 u\, \circ\, {\rm d} W$, $\mathcal{Q}_2 u \diamond {\rm d} L$, and $\widetilde h( u)\d \widetilde{W}$.
	These main contributions have been summarized in Section \ref{Section : Introduction-Overview}. 
	
	In this section we highlight several key difficulties and the novel approaches developed to address them, and we provide a comparison between our work and the related literature.

	$\bullet$ \textbf{\textit{Singularities.}} 
	First, the presence of pseudo-differential operators of non-negative order in the stochastic terms poses a major obstacle to closing the $H^s$-energy estimates. Applying It\^{o}'s formula to \eqref{SEuler-(u P)} generates singular quantities such as
	\begin{equation}\label{Q singular terms}
		\langle \mathcal{Q}_i u, u \rangle_{H^s}, \quad 
		\langle \mathcal{Q}_i^2 u, u \rangle_{H^s}, \quad 
		\text{and} \quad 
		\langle \mathcal{Q}_i u, \mathcal{Q}_i u \rangle_{H^s}.
	\end{equation}
	Controlling the terms in \eqref{Q singular terms} purely by $\|u\|_{H^s}^2$ is far from straightforward, especially when $\Q_i$ is of positive order.
	
	A second essential obstacle stems from the discontinuous Marcus noise component $\mathcal{Q}_2 u \diamond {\rm d}L$. In the same spirit as \eqref{Marcus flow H}, with slight abuse of notation, we reuse the notation 
	$\wp$ for the flow generated by $\Q_2$: 
	\begin{equation}\label{Marcus flow Q2 introduction}
		\wp(r)\triangleq\wp(r,l,f),\quad 
		\frac{{\rm d}}{{\rm d}r}\wp(r)= l\cdot \Q_2\wp(r),\quad r\in[0,1],  \quad \wp(0) = f.
	\end{equation}
	Crucially, for a pseudo-differential operator $\mathcal{Q}_2$ of non-negative order, this flow is \textbf{not} invariant in $H^s(\K^d;\R^d)$. Consequently, the corresponding stochastic integral $\int_0^t \mathcal{Q}_2u(t')\diamond{\rm d}L(t')$ fails to preserve Sobolev regularity, introducing further singularity into the system.
	
	To rigorously handle \eqref{Q singular terms} and \eqref{Marcus flow Q2 introduction} within $H^s(\K^d;\R^d)$, we establish robust cancellation estimates for nearly skew-adjoint pseudo-differential operators $\mathcal Q$ (see Theorem \ref{Thm-cancel} for the precise and general statement):
	\begin{equation}\label{Cancel-introduction}
		\Abs{\big \langle \mathcal Q f, f \big \rangle _{H^{s}}} \lesssim \|f\|^2_{H^s},\qquad
		\Big|\left \langle \mathcal Q^{2}f, f \right \rangle _{H^{s}} + \left \langle \mathcal Q f, \mathcal Q f \right \rangle _{H^{s}} \Big| \lesssim \|f\|^2_{H^s}.
	\end{equation}
	Crucially, our estimates reveal that the combination $\big\langle \mathcal Q^{2}f, f \big \rangle _{H^{s}} + \big \langle \mathcal Q f, \mathcal Q f \big \rangle _{H^{s}} $--which exactly captures the energy contribution from the It\^o correction and the quadratic variation--exhibits a hidden cancellation and is bounded by $\|f\|^2_{H^s}$. 
	
	A priori, these estimates require $f\in H^\sigma(\K^d;\R^d)$ for some $\sigma>s$ to be well-defined. To overcome this, we implement a systematic renormalization procedure governed by a parameter $n$ (see Lemma \ref{Lemma:Qn}), ensuring that these cancellation properties hold \textit{uniformly} in $n$ for $f \in H^s(\K^d;\R^d)$. This uniform control is the core analytic mechanism that enables us to tame the pseudo-differential singularities via a limiting process.

	$\bullet$  \textbf{\textit{General pressure.}} 
	In the standard deterministic theory, the loss of derivatives in terms such as 
	\begin{equation}
		\bIP{\rho \, \Div u,\rho}_{H^s} \quad \text{and}\quad \IP{\frac{1}{\rho} \nabla P(\rho), u}_{H^s},
		\label{cannot close Hs (rho u)}
	\end{equation}
	is traditionally handled by a symmetrization procedure; see \cite{Majda-1984-Book} and \cite[(2.13)]{Taylor-2011-PDEbook3}. To see why this fails near vacuum, we first recall that, in the isentropic setting, the enthalpy $\hbar(\rho)$ and the sound speed $c_{{\rm sound}}(\rho)$ satisfy
	\begin{equation}
		\hbar'(\rho)=P'(\rho)/\rho,\quad c_{{\rm sound}}(\rho)=\sqrt{P'(\rho)}.\label{enthalpy+sound}
	\end{equation}
	Let ${\rm diag}(\cdot,\dots,\cdot)$ denote a block-diagonal matrix, and let $\bullet^T$ denote the transpose of $\bullet$. 
	Roughly speaking, if one retains the original variables $(\rho, u)^T$ and multiplies the equations by a symmetrizer matrix 
	$$A_0(\rho) = \text{diag}\big(c_{{\rm sound}}^2(\rho)/\rho, \, \rho I_{d\times d}\big)$$ 
	to symmetrize the system, the corresponding $L^2$ energy functional used for \textit{a priori} estimates is 
	$$\frac12 \int_{\K^d} (\rho, u) A_0(\rho) (\rho, u)^T \d x =  \frac12\int_{\K^d}\left(\rho c_{{\rm sound}}^2(\rho)+\rho |u|^2\right)\d x.$$ 
	A fundamental limitation is that as $\rho \to 0$, the sound speed $c_{{\rm sound}}(\rho) \to 0$ (for typical physical gases), causing the matrix $A_0(\rho)$ to degenerate. Moreover, the energy norm \textbf{loses its coercivity}, and the classical energy estimates completely break down at the vacuum boundary. 
	
	Alternatively, if one chooses the enthalpy-based variables $(\hbar, u)^T$, the symmetrizer becomes 
	$$\widetilde{A}_0(\rho) = \text{diag}\big(1/c_{{\rm sound}}^2(\rho), \, I_{d\times d}\big),$$ 
	and the associated $L^2$ energy functional reduces to
	$$\frac12 \int_{\K^d} (\hbar, u) \widetilde{A}_0(\rho) (\hbar, u)^T \d x = \frac12 \int_{\K^d} \left(\frac{\hbar^2}{c_{\rm sound}^2(\rho)}+|u|^2\right)\d x.$$
	In this case, the term $1/c_{\rm sound}^2(\rho)$ blows up as $\rho \to 0$. The symmetrizer becomes singular, meaning the energy functional \textbf{loses its boundedness} and becomes ill-defined near vacuum.
	
	To overcome this structural difficulty, one seeks a reformulation that remains regular up to the vacuum boundary. Makino achieved this for the polytropic \(\gamma\)-law by introducing a specific density-dependent variable \cite{Makino-1986-chapter}. However, his construction relies fundamentally on the algebraic structure of the polytropic pressure law and therefore does not extend to broader physical models like Chaplygin gases or non-power-law degeneracies.

	In the stochastic case,  even though the present work is restricted to the non-vacuum regime, using the classical symmetrization method may require further structure conditions on the noise coefficients. Indeed,  the terms in \eqref{cannot close Hs (rho u)} must be controlled in expectation. Since expectations of products generally cannot be decoupled, the nonlinear pressure $P(\cdot)$ introduces additional difficulties. 
	Moreover,  a transformation that extends continuously to the vacuum serves as an essential foundation for subsequent investigations of vacuum problems. 
	
	Thus, one seeks a transformation that unifies the main advantages of classical symmetrization and the Makino framework: a symmetrizable structure that closes high-order energy estimates without derivative loss, while remaining regular and continuously extendable up to vacuum. It should not rely on an algebraic pressure law, but also apply to non-polytropic equations of state.
	To this end, we introduce a carefully chosen nonlinear mapping \(\rr:(0,\infty)\to\mathbb{R}\), which is smooth and strictly increasing, and define the new state variable
	$
	\varrho(t,x)=\rr(\rho(t,x)).
	$
	The transformation is characterized by the structural identity (see \ref{Structural Identity P' r'} for the precise assumptions)
	\begin{equation*} 
		P'(\rho) = (\rho\,\rr'(\rho))^2.
	\end{equation*}
	This change of variables reformulates the original system as a quasi-linear abstract stochastic evolution equation for \(X=(\varrho,u)\) taking value in \(H^s(\mathbb{K}^d;\mathbb{R})\times H^s(\mathbb{K}^d;\mathbb{R}^d)\):
	\begin{equation*}
		{\rm d}X + F(X)\,{\rm d}t
		= \QQ_1 X\, \circ\, {\rm d}W
		+\QQ_2 X \diamond {\rm d}L
		+\ZZ(t,X)\,{\rm d}\widetilde{W}.
	\end{equation*}
	Here, \(F\) collects the deterministic transport and pressure contributions $\Theta(\varrho)
	=
	\rr'\bigl(\rr^{-1}(\varrho)\bigr)
	\rr^{-1}(\varrho)
	$, \(\QQ_{i} \triangleq 
	{\rm diag}(0,\Q_i)\ (i=1,2)\), and \(\ZZ\) represents the It\^o forcing term; see \eqref{Target problem X=(q u)} below. This formulation reveals the quasi-linear stochastic structure of the compressible Euler system and allows us to treat, within a unified framework, power-type and logarithmic transforms associated with polytropic and Chaplygin gases, as well as piecewise-defined pressure laws; see Lemma~\ref{Lemma-F}.

	Crucially, to solve the transformed system, we construct an extension
	of $\Theta$, equipped with tailored nonlinear Sobolev composition
	estimates (see \ref{Regularity-Sound} in \ref{Hypo-Pressure}), specifically designed
	to close the delicate high-order stochastic energy estimates.
	This extension permits systematic regularization and cut-off
	approximations, yielding \emph{a priori} estimates that are uniform
	with respect to the approximation parameters. Once we verify
	\emph{a posteriori} that the transformed solution remains strictly
	within the range of $\rr$, we recover the strictly positive physical
	density via $\rho=\rr^{-1}(\varrho)$.

	For more details on the transformation $\rr$ and related comparisons between $\rr$ and the classical enthalpy-based symmetrization, we refer the readers to Remark~\ref{Remark-Hypo-Pressure} below. 
	
	$\bullet$  \textbf{\textit{Long-time behavior in the incompressible damped case.}} 
	For the stochastic damped incompressible Euler equations \eqref{SEuler-(u P) damp} on $\K^d$ with $d\ge3$, the primary difficulty lies in establishing global regularity of solutions.  Under suitable growth conditions on the It\^o-forcing $\widetilde h(u) \d \widetilde W$, we first identify a blow-up criterion: a classical solution blows up if and only if $\|u\|_{\Wlip}$ diverges (see \eqref{Blow-up criterion u}). Guided by this, we introduce a customized class of Lyapunov-type functionals to control $\|u\|_{\Wlip}$. This yields a hierarchy of damping-noise conditions (\ref{Hypo-DNI-1}--\ref{Hypo-DNI-3}) that guarantee, respectively: global-in-time existence, uniform-in-time bounds, and decay estimates. 
	
	While uniform-in-time bounds offer a natural starting point for constructing invariant probability measures, one immediately confronts a fundamental topological obstruction: the simultaneous lack of compactness and continuous dependence within any single Polish space suitable for the dynamics (see Section~\ref{Section : Introduction: Long-time behavior}). For the stochastic 2-D Euler equations, Constantin et al.~\cite{Constantin-Glatt-Holtz-Vicol-2014-CMP} circumvented this obstacle by introducing fractional diffusion, which restores spatial smoothing and thereby yields both compactness and stability. In sharp contrast, zeroth-order Ekman damping is purely dissipative yet strictly non-smoothing; although it continuously drains energy, it provides no parabolic regularization. Consequently, the inherent regularity loss in \eqref{SEuler-(u P) damp} dictates that continuous dependence on initial data can be established only in a strictly weaker topology, and crucially, only up to stopping times controlled by higher-order norms. Generally, for nonlinear SPDE, these stopping times can degenerate arbitrarily to zero for unbounded initial data. Hence, suitable estimates on stability can be guaranteed only if the initial values are confined to a bounded subset in the strong topology. 
	
	As a result, it seems impossible to achieve uniform-in-time bounds, compactness, and Feller continuity simultaneously within a single natural phase space. Instead, we demonstrate that the associated Markov transition semigroup $\{\mathscr{P}_t\}_{t\ge0}$ satisfies what we term the \textit{restricted Feller property under mismatched metrics}. Specifically, let $H^{s,\theta}$ (with $\theta<s$) denote the space $H^s$ endowed with the weaker $H^\theta$-topology. The semigroup acts on test functions $\phi \in C_B(H^{s,\theta})$ such that, for every $t\ge0$ and any strongly closed ball $B_R \subset (H^s, \|\cdot\|_{H^s})$, 
	\[
	\text{the restricted mapping } x \mapsto \mathscr{P}_t\phi(x)\big|_{B_R} \text{ is continuous with respect to the } H^\theta\text{-topology}.
	\]
	This setup exposes two essential structural difficulties:\\
	\textit{\underline{1. Mismatched metric}:} The dynamics naturally reside in $H^s$, yet compactness and continuous dependence are  accessible only via the weaker $H^\theta$-metric. Notably, the space $H^{s,\theta}$ is topologically incomplete.\\
	\textit{\underline{2. Loss of  Feller continuity}:} Because the Feller property is strictly localized to bounded subsets in the strong norm, there is no guarantee that $\mathscr{P}_t\phi$ globally belongs to $C_B(H^{s,\theta})$.

	Addressing the damped stochastic Euler dynamics across all spatial dimensions therefore necessitates a notable  departure from classical ergodic frameworks. To overcome these coupled  obstructions, we formulate a criterion for the existence of invariant probability measures (see Theorem~\ref{Thm-generalized KB} and Section~\ref{Section : gKB-proof}) tailored to general Markov processes whose state space carries a weaker, possibly incomplete topology, and whose propagator satisfies only the restricted Feller property under mismatched metrics. By explicitly bypassing the single-topology constraint and the Feller assumption inherent in the classical Krylov--Bogoliubov approach, this generalized framework allows us to establish the existence of invariant probability measures for a broad class of singular stochastic evolution systems, including the stochastic damped Euler equations \eqref{SEuler-(u P) damp} in all dimensions $d\ge2$.
	
	$\bullet$ \textbf{\textit{Significance of the Results.}}
	To situate our contributions within the existing literature, we briefly review related work on SPDEs of fluid type. While a comprehensive survey of stochastic fluid dynamics lies beyond the scope of this paper, we highlight only representative contributions. For stochastic incompressible fluid models driven by classical It\^o or continuous transport-type noise, the theory has been extensively developed, yielding a solid understanding of existence, uniqueness, stability, and long-time behavior (see, e.g., \cite{Miao-Rohde-Tang-2024-SPDE,Brzezniak-Dhariwal-2020-JMFM,Brzezniak-Motyl-2019-SIAM,Li-Liu-Tang-2021-SPA,Hofmanova-Zhu-Zhu-2024-JEMS,Flandoli-Luo-2020-AoP,Flandoli-Romito-2008-PTRF,Tang-2023-JFA,Holden-Karlsen-Pang-2021-JDE,GlattHoltz-Vicol-2014-AoP,Hairer-Mattingly-2006-Annals,Hofmanova-Zhu-Zhu-2025-ARMA,E-Mattingly-Sinai-2001-CMP,Kuksin-Shirikyan-2025-GAFA}).
	
	Despite these advances, the mathematical theory for SPDEs driven by \textit{Marcus noise} remains in its early stages, with prior contributions largely restricted to highly simplified settings (where $\Q_2$ is either the identity mapping or the classical transport operator as in \eqref{eq:local-Q}), see \cite{Brzezniak-Manna-2019-CMP,Brzezniak-Manna-Panda-2019-JDE,Chen-Duan-Gao-2024-SPA,Brzezniak-etal-2025-arXiv,Luo-Teng-2025-arXiv}. To the best of our knowledge, a well-posedness theory for the 3-D Euler system driven by Marcus noise has remained open, let alone formulations accommodating the rich complexity of mixed-type driving forces. 
	Our results therefore bridge this gap by incorporating a genuinely mixed multiplicative noise structure—combining Stratonovich, It\^o, and Marcus interpretations—into the stochastic Euler dynamics.
	
	Furthermore, recognizing that nonlocal phenomena are frequently observed in fluid dynamics, particularly in turbulence models (see, e.g., \cite{Majda-Gershgorin-2013-PTRSL,Hamba-2022-JFM}), we allow both $\mathcal{Q}_1$ and $\mathcal{Q}_2$ to be pseudo-differential operators. This significantly generalizes classical transport noise, capturing the intricate nature of nonlocal random interactions and providing a versatile framework for modeling sophisticated turbulent effects. Given that SPDEs featuring continuous pseudo-differential noise have only recently been explored \cite{Tang-Wang-2022-arXiv,Tang-2023-JFA,Alonso-Pang-Tang-2026-JLMS}, the present work serves as the first  integration of pseudo-differential operators within a Marcus noise framework, expanding the analytical toolkit available for such singular systems.

	In stark contrast to the incompressible case, the stochastic compressible Euler equations have received considerably less attention, with only a few available results \cite{Breit-Mensah-2019-NA,Correa-Olivera-2025-arXiv,Kuan-etal-2025-arXiv,Hofmanova-Koley-Sarkar-2022-CPDE,Breit-etal-2025-JCP,Moyo-2023-JDE,Chen-Huang-Wang-2025-arXiv,Dai-etal-2026-arXiv}. These contributions typically assume either multiplicative It\^o noise or standard polytropic pressure laws $P(\rho) = a\rho^\gamma$ with $a>0$ and $\gamma\ge 1$. Our framework removes this constraint, enabling the treatment of a much broader class of physically relevant pressure functions under genuinely mixed noise. Beyond standard polytropic $\gamma$-laws, our theory encompasses piecewise-defined $\gamma$-laws, (piecewise) Chaplygin-type laws, the pressure of white dwarf stars, and other astrophysically relevant regimes. We emphasize that, even in the purely It\^o-forcing setting ($\mathcal{Q}_1=\mathcal{Q}_2=0$),  local existence for these generalized pressure laws in multidimensional case has  not been achieved in the prior literature.   We refer to Remark \ref{Remark-Thm-(u rho)} for more discussions of the advances in our methodology.
	
	Regarding the long-time behavior of the incompressible stochastic damped Euler equations \eqref{SEuler-(u P) damp} on $\K^d$ ($d \ge 2$), establishing global regularity constitutes the primary obstruction. 
	The approach of utilizing noise to prevent blow-up and ensure global existence has been investigated recently; for stochastic fluid models, this was initiated in \cite{Ren-Tang-Wang-2024-POTA} and abstractly formalized in \cite{Tang-Wang-2022-arXiv} (see also \cite{Tang-Yang-2023-AIHP, Tang-Wang-2024-CCM, Bagnara-Mario-Xu-2025-EJP} for further concrete developments). In this work, we advance this line of analysis by identifying a hierarchy of damping–noise conditions under which we obtain, respectively, global-in-time existence, uniform-in-time bounds, and, under the strongest assumptions, decay estimates. 
	
	Prior results on  the existence of an invariant measure in the uniform-in-time bounded case have predominantly relied on viscous or fractional regularization \cite{Kuksin-Shirikyan-2000-CMP,Kuksin-Shirikyan-2025-GAFA,Kuksin-Shirikyan-2012-book,E-Mattingly-Sinai-2001-CMP,Hairer-Mattingly-2006-Annals,Constantin-Glatt-Holtz-Vicol-2014-CMP}, or have been restricted to 2-D vorticity formulations with additive noise \cite{Bessaih-Ferrario-2020-CMP}. To circumvent the need for viscosity, our approach fundamentally relies on the \textit{restricted Feller property under mismatched metrics} (see Definition \ref{Locally Feller definition}).  This abstract framework offers substantial flexibility in selecting weaker metrics, naturally accommodating singular and nonlinear dynamics. While related Krylov--Bogoliubov extensions for mismatched metrics exist (e.g., \cite{CotiZelati-GlattHoltz-Trivisa-2021-AMO,Kim-2005-SIMA,Bessaih-Ferrario-2020-CMP}), the restricted Feller property has not been clearly captured (see Remark \ref{Remark-mismatch compare}). As a primary consequence,  we establish the existence of an invariant probability measure for systems on $\T^d$ driven by genuinely mixed multiplicative noise in all dimensions $d \ge 2$.  Under the most stringent of these conditions, 
	uniqueness of an invariant measure can be also achieved.  This   provides a resolution to the problem \ref{Problem-AS}. In fact, our analysis goes beyond the original conjecture, resolving a notably generalized version of the problem on both $\mathbb T^d$ and $\mathbb R^d$ in all spatial dimensions $d \ge 2$, under genuinely mixed multiplicative noise.
	
	For completeness, we recall that the deterministic compressible Euler equations admit a rich local-in-time theory for strong solutions; see, e.g., \cite{Bahouri-Chemin-Danchin-2011-book,Benzoni-Gavage-Serre-2007-Book,Alinhac-1995-Book,Majda-1984-Book}.

	\section{Notation and Background}\label{Section:Notations}

	\subsection{Function Spaces and Related Notation} 
	
	We start by defining the relevant function spaces.  As before, we denote either $\mathbb{R}$ or $\mathbb{T}  \triangleq  \mathbb{R} / 2\pi \mathbb{Z}$ by $\mathbb{K}$ to unify the analysis. For $1 \leq p < \infty$ and $d, m \geq 1$, $L^p(\mathbb{K}^d; \mathbb{R}^m)$ stands for the standard Lebesgue space of measurable, $p$-th power integrable $\mathbb{R}^m$-valued functions, and $L^\infty(\mathbb{K}^d; \mathbb{R}^m)$ is the space of essentially bounded functions.
	Particularly, the inner product in $L^2(\mathbb{K}^d; \mathbb{R}^m)$ is defined by
	\begin{equation*}
		\left\langle f, g \right\rangle_{L^2}  \triangleq  \sum_{i=1}^m \int_{\mathbb{K}^d} f_i \cdot g_i \, \mathrm{d}x.
	\end{equation*}
	When $m\ge1$ is clear, in the following we denote by
	$\left \langle f,g\right \rangle _{L^{2}}$ the inner product for both
	$f,g\in L^{2}(\mathbb K^{d};\mathbb R^{m})$ and
	$f,g\in L^{2}(\mathbb K^{d};\mathbb R)$.

	Let ${\mathrm{i}}=\sqrt{-1}$.   The Fourier transform
	$\mathscr F$ and inverse Fourier transform
	$\mathscr F^{-1}$ on $\R^d$  are defined by
	\begin{equation*}
		(\mathscr Ff)(\xi ) \triangleq \int _{\mathbb R^{d}}f(x){\mathrm{e}}^{-{
				\mathrm{i}}(x\cdot \xi )}{\,\mathrm{{d}}}x,\quad 
		(\mathscr F^{-1}f)(x) \triangleq  \frac{1}{(2\pi )^{d}}\int _{
			\mathbb R^{d}}f(\xi ){\mathrm{e}}^{{\mathrm{i}} (x\cdot \xi )}{\,\mathrm{{d}}}\xi,\quad (x,\xi)\in \R^d\times\R^d.
	\end{equation*}
	and on $\T^d$ are defined by
	\begin{equation*}
		(\mathscr Ff) (k) \triangleq \int _{\mathbb T^{d}}f(x){\mathrm{e}}^{-{\mathrm{i}}(x
			\cdot k)}{\,\mathrm{{d}}}x,\quad 
		(\mathscr F^{-1}f)(x)= \frac{1}{(2\pi )^{d}}\sum _{k\in
			\mathbb Z^{d}}f(k){\mathrm{e}}^{{\mathrm{i}} (x\cdot k)},\quad (x,k)\in \R^d\times\Z^d.
	\end{equation*}
	Recall that ${\mathbf I}$ is the identity operator. For any
	$s\in \mathbb R$, the Bessel potential 
	$\mathcal D^{s}=({\mathbf I}-\Delta )^{s/2}$ is defined by
	\begin{equation}
		\mathcal D^{s}f \triangleq \mathscr F^{-1}\left((1+|\cdot|^{2})^{s/2}(\mathscr Ff)\right).\label{Define Ds}
	\end{equation}
	For $s\ge 0$, $d,m\ge 1$, we define the Sobolev spaces $H^{s}$ on
	$\mathbb K^{d}$ with values in $\mathbb R^{m}$ as
	\begin{equation*}
		H^{s}(\mathbb R^{d};\mathbb R^{m}) \triangleq 
		\overline{C_{0}^{\infty }(\mathbb R^{d};\mathbb R^{m})}^{\|\cdot \|_{H^{s}}},
		\quad
		H^{s}(\mathbb T^{d};\mathbb R^{m}) \triangleq 
		\overline{C^{\infty }(\mathbb T^{d};\mathbb R^{m})}^{\|\cdot \|_{H^{s}}},\quad 
	\end{equation*}
	where
	\begin{equation*}
		\|f\|_{H^{s}} \triangleq \sqrt{\left \langle f,f\right \rangle _{H^{s}}},\quad 
		\left \langle f,g\right \rangle _{H^{s}} \triangleq \sum _{i=1}^{m}\left
		\langle \mathcal D^{s} f_{i},\mathcal D^{s} g_{i}\right \rangle _{L^{2}}.
	\end{equation*}
	For a multi-index $\alpha = (\alpha_i)_{1\le i\le d} \in \mathbb{N}_0^d  \triangleq  (\mathbb{N} \cup \{0\})^d$, we define
	$|\alpha|_1  \triangleq  \sum_{k=1}^d \alpha_k$, $\partial_x^\alpha  \triangleq  \prod_{k=1}^d \partial_{x_k}^{\alpha_k}.$
	Then for $p \ge 1$, $W^{p, \infty}(\mathbb{K}^d; \mathbb{R}^m)$ is the set of functions $f=(f_j)_{1\le j\le m}: \mathbb{K}^d \to \mathbb{R}^m$ with
	\begin{equation*}
		\|f\|_{W^{p, \infty}}  \triangleq  \sum_{j=1}^m \sum_{|\alpha|_1 = 0}^p \|\partial_x^\alpha f_j\|_{L^\infty} < \infty,
	\end{equation*}
	where $\partial_x^\alpha f$ is understood as the weak derivative of $f$ of order $\alpha$.
	
	Throughout this paper, the systems \eqref{SEuler-(rho u)} and \eqref{SEuler-(u P) damp} involve unknown functions defined on $\mathbb{K}^d$. Nevertheless, we frequently work with functions whose target dimensions vary. Hence, when $d \in \mathbb{N}$ is clear from context, we write
	\begin{equation}
		H_m^s = H^s(\mathbb{K}^d; \mathbb{R}^m), \quad
		W_m^{p,\infty} = W^{p,\infty}(\mathbb{K}^d; \mathbb{R}^m), \quad
		L_m^2 = L^2(\mathbb{K}^d; \mathbb{R}^m), \quad
		s \in \mathbb{R}, \ \ p \ge 1, \ \ m \in \mathbb{N}.
		\label{Hs Lp Wp m}
	\end{equation}
	In particular, to describe the pair $(\rho, u)$ in \eqref{SEuler-(rho u)}, where $\rho$ and $u$ take values in different spaces, we set
	\begin{equation}
		\mathbb{L}^2 \triangleq L_1^2 \times L_d^2, \quad
		\H^s \triangleq H_1^s \times H_d^s, \quad
		\WP \triangleq W_1^{p,\infty} \times W_d^{p,\infty},\quad s \in \mathbb{R},\quad p \ge 1.
		\label{Hs Lp Wp pair}
	\end{equation}
	When only one function is considered and no ambiguity arises regarding the dimensions $d,m$—for instance when one considers \eqref{SEuler-(u P)}—we further simplify notation by writing
	\begin{equation*}
		H^s = H_m^s, \quad
		W^{p,\infty} = W_m^{p,\infty}, \quad
		L^2 = L_m^2, \quad
		s \in \mathbb{R}, \ \ p \ge 1.
	\end{equation*}

	Now, we recall some background in the incompressible case. Let $\delta_{i,j}$ be the Kronecker delta. The Leray projection $\Pi_d$ on $\mathbb K^d$ ($\K=\R$ or $\T$) is defined by  multiplier $\varPi$, i.e.,
	\begin{equation}\label{Pi-d define}
		\Pi_d\triangleq    \mathscr{F}^{-1} \varPi \mathscr{F}, \quad \varPi \triangleq(\varPi_{i,j}  )_{1\le i,j\le d},
	\end{equation} 
	where $\varPi_{i,j}(\xi)\triangleq\delta_{i,j}-\frac{\xi_i \xi_j}{|\xi|^2}$ for $\xi=(\xi_1,\cdots,\xi_d)\in\R^d$ and $\varPi_{i,j}(k)\triangleq\delta_{i,j}-\frac{k_i k_j}{|k|^2}$ for $k=(k_1,\cdots,k_d)\in\T^d$.
	Formally, we have that for smooth vector-valued function $f$,  $\Pi_df=f-\nn\Delta^{-1}({\rm div} f).$
	However, $\Delta$ is not invertible on $H^s(\T^d;\R^d)$. So, when $x\in\T^d$, we restrict $\Pi_d$ to the following zero-average subspace of $H^s(\T^d;\R^d)$:
	\begin{equation*} 
		H^s(\T^d;\R^d)\bigcap\bigg\{f: \int_{\T^d} f(x)\d x=0\bigg\}, \quad  d\in\N.
	\end{equation*} 
	We define the zero-average projection $\Pi_0$ on $L^2(\T^d;\R^d)$ by
	\begin{equation}\label{Pi-0 define} 
		\Pi_0 f(x)\triangleq 
		f(x)- \frac{1}{(2\pi)^d}\int_{\T^d} f(y)\d y, \quad  x\in \T^d, \ f\in L^2(\T^d;\R^d).
	\end{equation}
	It is clear that $\Pi_d$ and $\Pi_0$ are projections, i.e., $\Pi_d^2=\Pi_d$ and $\Pi^2_0=\Pi_0$. Moreover, we note that 
	the Helmholtz--Weyl decomposition holds 
	(cf. \cite{
		Robinson-etal-2016-Book,
		Kuksin-Shirikyan-2012-book}) for $s\ge0$,
	\begin{equation}\label{Hodge decomposition}
		H^s(\R^d;\R^d)=\Pi_d \H^s(\mathbb R^d;\R^d)\oplus\nabla  H^{s+1}(\R^d;\R),\quad \Pi_0H^s(\T^d;\R^d)=\Pi_d\Pi_0 H^s(\mathbb T^d;\R^d)\oplus\nabla \Pi_0H^{s+1}(\T^d;\R).
	\end{equation} 
	Then we define
	\begin{equation}\label{Pi=Pi-d-0}
		\Pi\triangleq\Pi_d \ \ \text{for}\ \K^d=\mathbb{R}^d \ \ \text{and}\ \ 
		\Pi\triangleq\Pi_d\Pi_0\ \ \text{for}\ \K^d=\mathbb{T}^d.
	\end{equation}
	Since 
	$$\Pi_d\Pi_0 f=\Pi_0\Pi_d f=\Pi_d f,\quad f\in H^s(\T^d;\R^d)\bigcap\bigg\{f: \int_{\T^d} f(x)\d x=0\bigg\},$$
	we  see that
	$\Pi$ is also a projection, and 
	we define for $s\ge 0$ and $d\ge2$, 
	\begin{equation}\label{H-div}
		\Hdiv^s(\mathbb K^d;\R^d)\triangleq\Pi H^s(\mathbb K^d;\R^d).
	\end{equation} 
	Using these nations, \eqref{Hodge decomposition} reduces to
	\begin{equation*} 
		H^s(\R^d;\R^d)=\Hdiv^s(\mathbb R^d;\R^d)\oplus\nabla  H^{s+1}(\R^d;\R),\quad \Pi_0H^s(\T^d;\R^d)=\Hdiv^s(\mathbb T^d;\R^d)\oplus\nabla \Pi_0H^{s+1}(\T^d;\R).
	\end{equation*} 
	When $d$ is fixed and ambiguity is unlikely, we will simply write
	$$\Hdiv^s=\Hdiv^s(\mathbb K^d;\R^d), \quad s\ge0.$$
	
	In the following we will  also need to equip $\Hdiv^s$ with a non-standard topology:
	\begin{equation}\label{H-s1 s2}
		\Hdiv^{s_1,s_2}\triangleq\big(\Hdiv^{s_1},\mathscr{T}^{s_2}\big),\quad s_1 \ge s_2,\ \ 
		\mathscr{T}^{s_2}\ \text{is the topology induced by}\ \|\cdot\|_{H^{s_2}}.
	\end{equation}
	If $s_1>s_2$, the space $\Hdiv^{s_1,s_2}$ is not complete and 
	the completion of $\Hdiv^{s_1,s_2}$ in $\mathscr{T}^{s_2}$ is $\Hdiv^{s_2}$.
	
	For a linear operator $\mathcal{A}$ with domain densely contained in $L^2$, the $L^{2}$-adjoint operator of $\mathcal{A}$ is denoted by $\mathcal{A}^{*}$.  
	Then, obviously, we have
	\begin{equation*}
		\Pi_d=\Pi_d^*,\quad  \Pi_0=\Pi_0^*,\quad 
	\end{equation*} 
	Moreover, for all $s\ge0$, since $\Pi$ is a projection and $\Hdiv^s$ is a closed subspace of $H^s$,   we have that for $u,v\in  \Hdiv^s$,
	\begin{equation*}
		\IP{u,v}_{\Hdiv^s}\triangleq \IP{\Pi u, \Pi v}_{H^s}=\IP{u,v}_{H^s},\quad 
		\|u\|^2_{\Hdiv^s}\triangleq\IP{\Pi u, \Pi u}_{H^s}=\|u\|^2_{H^s}.
	\end{equation*} 
	The above facts will be frequently used in the following sections without further notice.

	For topological spaces $E$ and $\widetilde{E}$, the following function spaces and notations are used throughout this work:
	\begin{itemize}[leftmargin=0.79cm]\setlength\itemsep{0.2em}
		\item $\mathscr{B}(E)$ denotes the Borel $\sigma$-algebra of $E$, and $\mathbf P(E)$ is the set of Borel probability measures on $(E,\mathscr{B}(E))$;
		\item $\mathscr{M}(E;\widetilde{E})$ and $C(E;\widetilde{E})$ denote the set of all measurable maps from $E$ to $\widetilde{E}$, and set of continuous functions  from $E$ to $\widetilde{E}$,   respectively;
		\item  $D([0,\infty); \widetilde{E})$ is the set of c\`adl\`ag functions (right-continuous with left limits) from $[0,\infty)$ to $\widetilde{E}$;
		\item When $\widetilde{E}$ is a metric space, $\mathscr{M}_B(E;\widetilde{E})$ and $C_{B}(E;\widetilde{E})$ denote the sets of bounded functions in $\mathscr{M}(E;\widetilde{E})$ and $C(E;\widetilde{E})$, respectively. In particular, when $\widetilde{E}=\R$, we simply write $C(E)=C(E;\R)$ (and similarly for the other spaces).  
		
		\item If $E$ and $\widetilde{E}$ are separable Banach spaces, $\mathscr L(E;\widetilde{E})$ denotes the set of bounded linear operators from $E$ to $\widetilde{E}$.
		\item  $E\hookrightarrow \widetilde{E}$ means that $E$ is continuously and densely embedded into $\widetilde{E}$, while $E \hookrightarrow\hookrightarrow \widetilde{E}$ indicates that the  embedding is compact.
		
	\end{itemize}
	Since the noise coefficient $h$ for $\d \widetilde W$ in \eqref{SEuler-(rho u)} is time-dependent, we introduce the following function space to characterize a broad class of noise coefficients:
	\begin{equation}\label{SCRK} 
		\mathscr{K} \triangleq \big\{K\in\mathscr{M}\big( [0,\infty)\times [0,\infty); (0,\infty)\big): 
		K(x,y)\ \text{is increasing in}\ y\ \text{and locally integrable in}\ x
		\big\}.
	\end{equation}
	For the analysis of long-time behavior, we require Lyapunov-type functions from the following space of nonnegative $C^2$  functions on $[0,\infty)$ that vanish at the origin, are strictly increasing and concave, grow unboundedly at infinity:
	\begin{equation}\label{SCRV} 
		\mathscr{V}\triangleq\left\{f \in C^2([0,\infty);[0,\infty)) :  f(0)=0,\ 
		f'(x)>0,\  f''(x)\le 0,\ \lim_{x\to\infty} f(x)=\infty\right\}.
	\end{equation}

	Throughout the paper,  by a slight abuse of notation, $C$ and $c$ (possibly with subscripts) denote positive constants that may vary from line to line, but whose meaning is clear from context. For two nonnegative quantities $A$ and $B$, we write $A \lesssim B$ if there exists a constant $c > 0$ such that $A \leq c\,B$. For linear operators $\mathcal{A}$ and $\mathcal{B}$, the commutator $[\mathcal{A},\mathcal{B}]$ is defined by $[\mathcal{A},\mathcal{B}] \triangleq \mathcal{A}\mathcal{B} - \mathcal{B}\mathcal{A}$.

	\subsection{Pseudo-differential Operators}\label{Section:PDO}
	
	For two multi-indices $\alpha = (\alpha_i)_{1 \le i \le d},\, \beta = (\beta_i)_{1 \le i \le d}$ in $\mathbb{N}_0^d$ with $\beta \le \alpha$ (that is, $\beta_i \le \alpha_i$ for all $1 \le i \le d$), we define
	\begin{equation*}
		\partial_\xi^\alpha  \triangleq  \prod_{k=1}^d \partial_{\xi_k}^{\alpha_k}, \qquad
		\binom{\alpha}{\beta}  \triangleq  \prod_{i=1}^d \frac{\alpha_i!}{\beta_i! (\alpha_i - \beta_i)!},\qquad
		\alpha ! \triangleq \prod_{i=1}^d\alpha_i!.
	\end{equation*}
	We also recall that, for a multi-index $\alpha = (\alpha_i)_{1\le i\le d} \in\mathbb{N}_0^d$,
	$$
	|\alpha|_1\triangleq \sum_{k=1}^d\alpha_k.
	$$
	
	For any $s\in \mathbb R$, we define the symbol class
	$
	\mathbf S^{s}(\mathbb R^{d}\times \mathbb R^{d};\mathbb{C}^{m\times m})
	\subset C^{\infty }(\mathbb R^{d}\times \mathbb R^{d};\mathbb{C}^{m\times m})
	$
	by
	\begin{align}\label{def Ss R}
		\mathbf S^{s}(\mathbb R^{d}\times \mathbb R^{d};\mathbb{C}^{m\times m})
		\triangleq 
		\left \{\mathscr{p} \, : \,  
		|\mathscr{p}|^{\beta , \alpha ;s}_{\mathbb R^{d}\times \mathbb R^{d}} \triangleq 
		\max_{0\le j\le \beta,\, 0\le \ell\le \alpha}
		\sup_{(x,\xi )\in \mathbb R^{d}\times \mathbb R^{d}}
		\frac{\left |\partial _{x}^{j} \partial _{\xi}^{\ell}
			\mathscr{p} (x, \xi )\right |_{m\times m}}{(1+|\xi |)^{ s-|\ell|_{1}}}<\infty,\quad  \beta ,\alpha \in
		\mathbb N_{0}^{d}\right \}.
	\end{align}
	Here and throughout the paper, $|\cdot |_{m\times m}$ and $|\cdot |$ denote the usual norms on $\mathbb C^{m\times m}$ and $\mathbb R^{d}$, respectively. It is well known that $\mathbf S^{s}(\mathbb R^{d}\times \mathbb R^{d};\mathbb{C}^{m\times m})$ is a Fr\'{e}chet space with respect to the seminorms
	$
	\bigl\{|\cdot |^{\beta , \alpha ;s}_{\mathbb R^{d}\times \mathbb R^{d}}\bigr\}_{\beta , \alpha \in \mathbb N_{0}^{d}}.
	$
	
	To define the toroidal symbol class, we first recall the difference operator $\triangle _{k}^{\alpha}$, defined by
	\begin{equation}
		(\triangle _{k}^{\alpha }g)(k)  \triangleq \sum _{\beta \in
			\mathbb N_{0}^{d}, \beta \leq \alpha} (-1)^{|\alpha -\beta |_{1}}
		\binom{\alpha}{\beta}g(k+\beta ),\quad g:\mathbb{Z}^{d}\to \mathbb{C},
		\ \  k\in \mathbb Z^{d},\ \ \alpha \in \mathbb N_{0}^{d}.
		\label{difference operator}
	\end{equation}
	Then, for $s\in \mathbb R$, the toroidal symbol class of order $s$ is defined by (cf. \cite{Ruzhansky-Turunen-2010-Book})
	\begin{align}\label{def Ss T}
		\mathbf S^{s}(\mathbb T^{d}\times \mathbb Z^{d};\mathbb{C}^{m\times m})
		\triangleq 
		\left \{\mathscr{p} \, : \,
		\begin{aligned}
			&\mathscr{p}(\cdot ,k)\in C^{\infty }(\mathbb T^{d};\mathbb{C}^{m
				\times m}) \quad \text{for all } k\in \mathbb Z^{d},
			\\
			|\mathscr{p}|^{\beta ,\alpha ;s}_{\mathbb T^{d}\times \mathbb Z^{d}} \triangleq\  &
			\max_{0\le j\le \beta,\, 0\le \ell\le \alpha}
			\sup _{(x,k)\in \mathbb T^{d}\times \mathbb Z^{d}}
			\frac{\left |\partial _{x}^{j}\triangle _{k}^{\ell}
				\mathscr{p} (x, k)\right |_{m\times m}}{(1+|k|)^{ s-|\ell |_{1}}}<\infty,\quad \beta ,\alpha \in \mathbb N_{0}^{d}
		\end{aligned} \right \}.
	\end{align}
	Again, $\mathbf S^{s}(\mathbb T^{d}\times \mathbb Z^{d};\mathbb{C}^{m\times m})$ is a Fr\'{e}chet space with respect to the seminorms
	$
	\bigl\{|\cdot |^{\beta , \alpha ;s}_{\mathbb T^{d}\times \mathbb Z^{d}}\bigr\}_{\beta , \alpha \in \mathbb N_{0}^{d}}.
	$
	
	The pseudo-differential operator $\OP(\mathscr p)$ associated with the symbol $\mathscr p$ is defined by
	\begin{align}
		\OP(\mathscr{p}) \triangleq \mathcal P,\qquad
		[\mathcal P f](x) \triangleq  
		\begin{cases}
			\left\{\mathscr F^{-1}\Big[\mathscr{p}(z,\xi) \big(\mathscr Ff\big)(\xi)\Big]\right\}_{z=x}, &\text{if } \mathscr{p}\in \SS^s(\R^d\times \R^d;\mathbb{C}^{m\times m}),\vspace*{4pt}\\
			\left\{\mathscr F^{-1}\Big[\mathscr{p}(z,k) \big(\mathscr Ff\big)(k)\Big]\right\}_{z=x}, &\text{if } \mathscr{p} \in
			\SS^s(\T^d\times \Z^d;\mathbb{C}^{m\times m}).
		\end{cases}
		\label{OP define}
	\end{align}
	
	Throughout this paper, all pseudo-differential operators are assumed to preserve real-valuedness, in the sense that whenever $f$ is real-valued, $\OP(\mathscr p)f$ is also real-valued. Equivalently, we require that, for the frequency variable $\xi\in\R^d$ (when $x\in \mathbb R^{d}$) or $k\in\Z^d$ (when $x\in \mathbb T^{d}$),
	\begin{align}
		\mathscr{p}(x,-\cdot)=\overline{\mathscr{p}(x,\cdot)}.
		\label{Real symbol}
	\end{align}
	
	We denote by $\mathbf S^{s}\big(\mathbb T^{d}\times \mathbb R^{d};\mathbb C\big)$ the class of all symbols $\mathscr{p}\in \mathbf S^{s}\big(\mathbb R^{d}\times \mathbb R^{d};\mathbb C\big)$ such that $\mathscr{p}(\cdot ,\xi)$ is $2\pi$-periodic for every $\xi\in\mathbb R^d$. By \cite[Theorem 5.2]{Ruzhansky-Turunen-2010-JFAA} (see also \cite[Theorem 2.10 and Corollary 2.11]{Delgado-2013-Chapter}), bounded sets in $\SS^s\big(\T^d\times \Z^d;\mathbb C\big)$ correspond precisely to restrictions of bounded sets in $\SS^s\big(\T^d\times \R^d;\mathbb C\big)$. More precisely, if
	\[
	\mathscr O\subset\SS^s(\T^d\times\R^d;\mathbb C)
	\quad \text{is bounded with respect to } 
	\bigl\{|\cdot |^{\beta , \alpha ;s}_{\mathbb R^{d}\times \mathbb R^{d}}\bigr\}_{\beta , \alpha \in \mathbb N_{0}^{d}},
	\]
	then
	\[
	\big\{\widetilde{\mathscr{p}}\ :\ \widetilde{\mathscr{p}} \triangleq \mathscr{p}|_{\mathbb{T}^{d}\times \mathbb{Z}^{d}},\ \mathscr{p}\in\mathscr O\big\}\subset\SS^s(\T^d\times\Z^d;\mathbb C)\ \text{is bounded with respect to}\ 
	\bigl\{|\cdot|_{\mathbb T^{d}\times \mathbb Z^{d}}^{\beta ,\alpha ;s}\bigr\}_{\beta , \alpha \in \mathbb N_{0}^{d}}.
	\]
	Conversely, if
	\[
	\widetilde{\mathscr{O}} \subset\SS^s(\T^d\times\Z^d;\mathbb C)
	\quad \text{is bounded with respect to } \bigl\{|\cdot|_{\mathbb T^{d}\times \mathbb Z^{d}}^{\beta ,\alpha ;s}\bigr\}_{\beta , \alpha \in \mathbb N_{0}^{d}},
	\]
	then there exists a subset ${\mathscr{O}} \subset\SS^s(\T^d\times\R^d;\mathbb C)$ such that
	\[
	\forall\ \widetilde{\mathscr{p}}\in\widetilde{\mathscr O}, \ \exists \ 
	\mathscr{p}\in \mathscr O\ \text{such that}\  \widetilde{\mathscr{p}}=\mathscr{p}|_{\mathbb{T}^{d}\times \mathbb{Z}^{d}},
	\]
	and $\mathscr O$ is bounded with respect to $\bigl\{|\cdot |^{\beta , \alpha ;s}_{\mathbb R^{d}\times \mathbb R^{d}}\bigr\}_{\beta , \alpha \in \mathbb N_{0}^{d}}$.
	
	Applying this argument entrywise, the same conclusion remains valid for matrix-valued symbols. Accordingly, we identify
	\begin{equation*}
		\Big\{{\mathrm{OP}}({\mathscr{p} })\, : \, {\mathscr{p} } \in \SS^{s}\big(\mathbb T^{d}\times \mathbb Z^{d};
		\mathbb C^{m\times m}\big)\Big\}
		=
		\Big\{{\mathrm{OP}}({\mathscr{p} })\, : \, {\mathscr{p} } \in \SS^{s}\big(\mathbb T^{d}\times
		\mathbb R^{d};\mathbb C^{m\times m}\big)\Big\}.
	\end{equation*}
	
	Therefore, when no confusion is possible, we use the unified notation
	\begin{equation}\label{Ss define}
		\mathbf S^{s} \triangleq  
		\left \{\mathscr{p} \in \mathbf S^{s}(\mathbb R^{d}\times \mathbb R^{d}; \mathbb{C}^{m\times m})\, : \, \text{{\eqref{Real symbol}}}\ \text{holds} \right \}
		\quad \text{or} \quad
		\left \{\mathscr{p} \in \mathbf S^{s}(\mathbb T^{d}\times \mathbb Z^{d}; \mathbb{C}^{m\times m})\, : \, \text{{\eqref{Real symbol}}}\ \text{holds} \right \},
	\end{equation}
	and set
	\begin{align}
		{\mathrm{OP}}\mathbf S^{s} \triangleq 
		\big\{{\mathrm{OP}}({\mathscr{p} })\, : \, {\mathscr{p} } \in \mathbf S^{s}\big\},
		\quad 
		\bigl\{|\cdot |^{\beta ,\alpha ;s}\bigr\}_{\beta , \alpha \in \mathbb N_{0}^{d}} \triangleq 
		\bigl\{|\cdot |^{\beta ,\alpha ;s}_{\mathbb R^{d}\times \mathbb R^{d}}\bigr\}_{\beta , \alpha \in \mathbb N_{0}^{d}} 
		\ \text{or} \ 
		\bigl\{|\cdot |^{\beta ,\alpha ;s}_{\mathbb T^{d}\times \mathbb Z^{d}}\bigr\}_{\beta , \alpha \in \mathbb N_{0}^{d}},
		\quad  
		s\in \mathbb R.
		\label{OPS+seminorms}
	\end{align}
	
	In what follows, we also consider symbols depending only on the frequency variable $\xi\in\R^d$ (when $x\in \mathbb R^{d}$) or $k\in\Z^d$ (when $x\in \mathbb T^{d}$). To distinguish this subclass, we define
	\begin{equation*} 
		{\mathrm{OP}}\mathbf S_0^{s} \triangleq 
		\big\{{\mathrm{OP}}({\mathscr{p} })\, : \, {\mathscr{p} } \in \mathbf S_0^{s}\big\},\quad\mathbf S^s_{0} \triangleq \left \{\mathscr{p}\in \mathbf S^{s}\, : \,
		\mathscr{p}(x,\cdot)=\mathscr{p}(\cdot)
		\right \}.
	\end{equation*}
	
	To highlight scalar-valued symbols, we note that $\mathbf S^{s}(\mathbb R^{d}\times \mathbb R^{d}; \mathbb{C})$ and $\mathbf S^{s}(\mathbb T^{d}\times \mathbb Z^{d}; \mathbb{C})$ are obtained formally by taking $m=1$ in \eqref{def Ss R} and \eqref{def Ss T}, with $|\cdot|_{1\times 1}$ identified with the modulus on $\mathbb C$. We then set
	\begin{equation*}
		\mathcal S^{s} \triangleq
		\left \{\mathscr{p} \in \mathbf S^{s}(\mathbb R^{d}\times \mathbb R^{d}; \mathbb{C})\, :\, \text{{\eqref{Real symbol}}}\ \text{holds} \right \}
		\quad \text{or} \quad
		\left \{\mathscr{p} \in \mathbf S^{s}(\mathbb T^{d}\times \mathbb Z^{d}; \mathbb{C})\, : \, \text{{\eqref{Real symbol}}}\ \text{holds} \right \},
	\end{equation*}
	and similarly define $\mathcal S_0^s$ as the scalar counterpart of $\mathbf S_0^s$. As above, we write $\OP\SS_0^s$, $\OP\mathcal S^s$, and $\OP\mathcal S_0^s$ for the corresponding classes of pseudo-differential operators.
	
	By \ref{OP-continuous} in Lemma \ref{LOP}, together with the fact that $\OP$ is injective (cf. \cite[Proposition 1.2, p.~56]{Kumano-go-1981-Book}), we define bounded subsets of $\OP\SS^s$ as follows.
	\begin{Definition}
		A subset $\mathscr O\subset {\mathrm{OP}}\mathbf S^{s}$ is said to be bounded if
		$
		\big\{\mathscr{p}: {\mathrm{OP}}(\mathscr{p})\in \mathscr O\big\}\subset \mathbf S^{s}
		$
		is bounded with respect to the seminorms in \eqref{OPS+seminorms}.
	\end{Definition}

	\subsection{Structure Assumption on the Noise Process}
	Recall that in \eqref{SEuler-(rho u)} and \eqref{SEuler-(u P)}, $W(t)$ and $\widetilde{W}(t)$ are Brownian motions, and $L(t)$ is a pure jump L\'evy process with Poisson random measure $\eta$ counting jumps of size $l$, intensity measure $\nu$, and compensator $\widetilde{\eta} = \eta - \nu \otimes \mathrm{d}t$. We now state the fundamental assumption on the noise process:
	\begin{Hypothesis}\label{Hypo-W L}
		Let $(\Omega,\mathcal{F},\{\mathcal{F}_t\}_{t \ge 0}, \mathbb{P})$ be a complete, right-continuous filtered probability space such that the processes $W(t)$, $\widetilde{W}(t)$, and $L(t)$ are $\{\mathcal{F}_t\}$-adapted and mutually independent. Furthermore, we assume $\nu([-1,1]^c)=0$, which restricts consideration to small jumps with $|l|\le 1$ and excludes large jumps.
	\end{Hypothesis}
	
	Additional assumptions on the operators $\Q_i$ ($i=1,2$), and the coefficients $h$ and $\widetilde h$ will be introduced separately in the subsequent sections about the compressible and incompressible cases, respectively, since they are related to the structure of the  concrete models.

	\section{Cancellation Properties of Pseudo-differential Operators}\label{Section : Cancel}
	
	\subsection{Nearly Skew-Symmetric Operators}\label{Section : Nearly Skew-Symmetric}

	In this paper we mainly focus on operators that are close to skew-adjoint operators.  
	\begin{Definition}\label{Ak class define}
		Let $d, m\ge 1$, and $\alpha\in[0,1]$.  
		We let 
		\[
		\mathbb A^\alpha
		\triangleq
		\left\{
		\mathcal A:\
		\text{there exist }\mathcal L={\mathrm{diag}}(\mathcal L_{1},\cdots ,\mathcal L_{m})\in\mathrm{OP}	\mathbf S^\alpha
		\text{ and }\mathcal G\in\mathrm{OP}\mathbf S^0\text{ such that }
		\mathcal A=\mathcal L+\mathcal G
		\text{ and \ref{Ak:(L G)}  holds}
		\right\},
		\]
		where \ref{Ak:(L G)} is:
		\begin{enumerate}[label={$(\mathbf{A}^\alpha)$},leftmargin=0.79cm]
			\item \label{Ak:(L G)} 
			$\mathcal L_{i}\in {\mathrm{OP}}\mathcal S^{\alpha},\ 
			\mathcal L_{i}+\mathcal L_{i}^{*}\in {\mathrm{OP}}
			\mathcal S^{0} \ (1\le i\le m)$. Moreover, 
			the symbols of $\mathcal L$ and
			$\mathcal G$ are commuting matrices.
		\end{enumerate}
		A family $\mathscr N\subset\mathbb A^\alpha$ is called
		bounded if one can choose a decomposition
		$\mathcal A=\mathcal L_{\mathcal A}+\mathcal G_{\mathcal A}$ for every
		$\mathcal A\in\mathscr N$ such that 
		$
		\{\mathcal L_{i,\mathcal A}:\mathcal A\in\mathscr N,1 \le i \le m\}$
		is bounded in $\mathrm{OP}\mathcal S^\alpha$, 
		$	\{\mathcal L_{i,\mathcal A}+\mathcal L^*_{i,\mathcal A}:\mathcal A\in\mathscr N,1 \le i \le m\}$  is bounded in $\mathrm{OP}\mathcal S^0$,  and $\{\mathcal G_{\mathcal A}:\mathcal A\in\mathscr N\}$
		are bounded in $\mathrm{OP}\mathbf S^0$.
	\end{Definition}
	
	To further distinguish operators with general space-frequency-dependent symbols from those whose symbols depend only on the frequency variable, we introduce an analogous class for the latter case, allowing arbitrary nonnegative order.
	
	\begin{Definition}\label{Bk class define}
		Let $d, m \ge 1$ and $\beta \ge 0$. We define
		\[
		\mathbb{B}^{\beta} \triangleq \left\{\mathcal B:\
		\text{there exist }\mathcal J={\mathrm{diag}}(\mathcal J_{1},\cdots ,\mathcal J_{m})\in\mathrm{OP}\mathbf S_0^\beta
		\text{ and }\mathcal V\in\mathrm{OP}\mathbf S_0^0\text{ such that  }
		\mathcal B=\mathcal J+\mathcal V
		\text{ and \ref{Bk:(J V)}  holds}
		\right\},
		\]
		where condition \ref{Bk:(J V)} is specified as follows:
		\begin{enumerate}[label={$(\mathbf{B}^\beta)$},leftmargin=0.79cm]
			\item\label{Bk:(J V)}  
			$\mathcal J_{i} \in {\mathrm{OP}}\mathcal S_0^{\beta}$ and $\mathcal J_{i} + \mathcal J_{i}^{*} \in {\mathrm{OP}}\mathcal S_0^{0}$ for $1 \le i \le m$. 
			Additionally, the symbols of $\mathcal J$ and $\mathcal V$ are commuting matrices.
		\end{enumerate}
		A subset $\mathscr{N} \subset \mathbb{B}^{\beta}$ is bounded if one can choose a decomposition
		$\mathcal B=\mathcal J_{\mathcal B}+\mathcal V_{\mathcal B}$ for every
		$\mathcal B\in\mathscr N$ such that 
		$
		\{\mathcal J_{i,\mathcal B}:\mathcal B\in\mathscr N,1 \le i \le m\}$
		is bounded in $\mathrm{OP}\mathcal S_0^\beta$, 
		$	\{\mathcal J_{i,\mathcal B}+\mathcal J^*_{i,\mathcal B}:\mathcal B\in\mathscr N,1 \le i \le m\}$  is bounded in $\mathrm{OP}\mathcal S_0^0$,  and $\{\mathcal V_{\mathcal B}:\mathcal B\in\mathscr N\}$
		are bounded in $\mathrm{OP}\mathbf S_0^0$.
	\end{Definition}
	
	To illustrate these abstract operator classes, we now construct examples built from smooth vector fields and coefficient functions. The diagonal components correspond to the main differential operators, which may include fractional derivatives of order $\sigma_i$, while the off-diagonal or additive parts represent lower-order perturbations acting as smoother corrections. In this way, the examples below clarify how operators combining principal (possibly fractional) parts and bounded lower-order terms fit naturally into the classes $\mathbb{A}^\alpha$ and $\mathbb{B}^\beta$, depending on whether their symbols depend on both space and frequency, or only on the latter.
	\begin{Example} 
		Let $d,m\ge 1$. Let $\psi^{(i)}$ and 
		$\phi^{i,j}$ $(1\le i, j \le m)$ be 
		functions such that for $1\le i, j \le m$,
		\begin{equation*}
			\begin{cases}
				\psi^{(i)}\in C_{c}^{\infty}(\mathbb R^{d};\mathbb R^{d}),\
				\ \phi^{i,j}\in C_{c}^{\infty}(\mathbb R^{d};\mathbb R),\  &
				\text{if}\ \  x\in \mathbb R^{d},
				\vspace*{4pt}
				\\
				\psi^{(i)}\in C^{\infty}(\mathbb T^{d};\mathbb R^{d}),\ \
				\phi^{i,j}\in C^{\infty}(\mathbb T^{d};\mathbb R),\
				&\text{if}\ \ x\in \mathbb T^{d}.
			\end{cases}
		\end{equation*}
		Let $\sigma _{i},\, s^{(i,j)}\in \mathbb R$ ($1\le i,j
		\le m$) and define
		\begin{equation*}
			\mathscr H \triangleq {\mathrm{diag}} \left (\big(\psi^{(1)}\cdot
			\nabla \big)({\mathbf I}-\Delta )^{\sigma _{1}},
			\cdots ,\big(\psi^{(m)}\cdot \nabla \big)({
				\mathbf I}-\Delta )^{\sigma _{m}}\right ),\quad 
			\mathscr T \triangleq \big(\phi^{(i,j)}({\mathbf I}-\Delta )^{s^{(i,j)}}\big)_{1\le i,j\le m}.
		\end{equation*}
		We have the following examples:
		\begin{itemize}[leftmargin=0.79cm]\setlength\itemsep{0.2em}
			\item Let $\sigma _{i}+1\in [0,1]$ and $s^{(i,j)}\le 0$ for all
			$1\le i,j\le m$. If either 
			$\psi^{(i)}=\psi^{(j)}$, $\sigma_i=\sigma_j$ ($1\le i,j\le m$), or
			$\phi^{(k,n)}\equiv0$ ($1\le k\neq n\le m$), 
			then
			$\mathcal L \triangleq \mathscr H$ and
			$\mathcal G \triangleq \mathscr T$ satisfy \ref{Ak:(L G)}.
			\item Let $\sigma _{i}+1\ge 0$. With slight abuse of notations, let $\psi^{(i)}\equiv C_i\in\R^d$ and $\phi^{(i,j)}\equiv c_{i,j}\in\R$. Assume that either
			$C_i=C_j$, $\sigma_i=\sigma_j$, $s^{(i,j)}\leq 0$ ($1\le i,j\le m$),
			or
			$c_{k,n}=0$ $(1\le k\neq n\le m)$,
			then
			$\mathcal J \triangleq \mathscr H$ and
			$\mathcal V \triangleq \mathscr T$ satisfy \ref{Bk:(J V)}.
		\end{itemize}
	\end{Example}
	
	In the next example, we move from operators defined through differential expressions to those constructed directly from their symbols. This example thus illustrates how the classes $\mathbb{A}^\alpha$ and $\mathbb{B}^\beta$ can be characterized equivalently at the symbol level. Keep in mind that ${\rm i}$ denotes the imaginary unit.

	\begin{Example} 
		Let $d,m\ge 1, \alpha\in[0,1]$, and $\beta\ge0$. Assume that 
		$$\mathscr{g}=\big(\mathscr{g}^{(i,j)}\big)_{1\le i,j\le m}\in \SS^0,\quad  \mathscr{v}=\big(\mathscr{v}^{(i,j)}\big)_{1\le i,j\le m}\in \SS_0^0,\quad {\rm i}\mathscr{h}_k\in\mathcal{S}^\alpha,\ \ {\rm i}\mathscr{e}_k\in \mathcal{S}_0^{\beta},\quad 1\le k\le m$$  
		such that $\mathscr{e}_k$ and $\mathscr{h}_k$ are $\R$-valued,
		$\mathscr{e}_k(\cdot)=-\mathscr{e}_k(-\cdot)$ and $\mathscr{h}_k(x,\cdot)=-\mathscr{h}_k(x,-\cdot)$.
		Let 
		$$\mathscr{E}_k \triangleq \OP({\rm i}\mathscr{e}_k),\ \ \mathscr{H}_k \triangleq \OP({\rm i}\mathscr{h}_k), \ \ 1\le k\le m.$$ 
		It is easy to verify that $\mathscr{E}_k$ is a skew-adjoint operator, and $\mathscr{H}_k$ satisfies $\mathscr{H}_k + \mathscr{H}_k^* \in {\mathrm{OP}}\mathcal{S}^{0}$. 
		If either $\mathscr{H}_i=\mathscr{H}_j$, $1\le i,j\le m$, 
		or $\mathscr{g}^{(k,n)}\equiv0$, $1\le k\neq n\le m$,  then
		$
		\L \triangleq {\mathrm{diag}} \left (\mathscr{H}_1, \cdots ,\mathscr{H}_m\right )
		\ \text{and}\
		\mathcal G \triangleq \OP(\mathscr{g})
		$ satisfy \ref{Ak:(L G)}. 
		If either $\mathscr{E}_i=\mathscr{E}_j$, $1\le i,j\le m$, 
		or $\mathscr{v}^{(k,n)}\equiv0$, $1\le k\neq n\le m$,  then
		$
		\J \triangleq {\mathrm{diag}} \left (\mathscr{E}_1, \cdots ,\mathscr{E}_m\right)\ \text{and}\
		\mathcal V \triangleq \OP(\mathscr{v})
		$
		satisfy \ref{Bk:(J V)}.
	\end{Example}
	
	\subsection{Main Results and Remarks}\label{Section : Main Cancellation Results}
	
	\begin{Theorem}[\textbf{Cancellation properties of pseudo-differential operators}]\label{Thm-cancel}
		Let $d\ge1$ and $\mu\in[0,\infty)$. 
		\begin{enumerate}[label=\textup{\textbf {(\arabic{enumi})}},leftmargin=0.79cm]\setlength\itemsep{0.2em}
			\item \label{x-dependent case}  Let $\alpha\in[0,1]$. If $\mathscr{N}_1\subset\mathbb{A}^{\alpha}$  is bounded $($cf. Definition \ref{Ak class define}$)$ and $\mathscr{O}_1\subset\OP\mathcal{S}^\mu$ is bounded,  then  we have
			\begin{equation}
				\sup _{\A\in\mathscr{N}_1,\, \mathcal P\in\mathscr{O}_1}\big \langle
				\mathcal P\mathcal A f, \mathcal Pf\big \rangle _{L^{2}}
				\lesssim   \left \|f\right \|^{2}_{H^{\mu}},\quad f\in H_d^{\mu+\alpha},\label{cancel-PA1}
			\end{equation}
			and 
			\begin{align}
				\sup_{\A\in\mathscr{N}_1,\, \mathcal P\in\mathscr{O}_1}\Big|\left
				\langle \mathcal P\mathcal A^{2}f, \mathcal Pf \right \rangle _{L^{2}}
				+ \left \langle \mathcal P\mathcal A f, \mathcal P\mathcal A
				f \right \rangle _{L^{2}} \Big| \lesssim \left \|f\right \|_{H^{\mu}}^{2},
				\quad f\in H_d^{\mu+2\alpha}.\label{cancel-PA2}
			\end{align}
			
			\item \label{x-independent case} Let $\beta\ge0$. If $\mathscr{N}_2\subset\mathbb{B}^{\beta}$  is bounded $($cf. Definition \ref{Bk class define}$)$ and $\mathscr{O}_2\subset\OP\mathcal{S}_0^\mu$ is bounded,  then  we have
			\begin{equation}
				\sup _{\B\in\mathscr{N}_2,\, \mathcal P\in\mathscr{O}_2}\big \langle
				\mathcal P \mathcal B f, \mathcal Pf\big \rangle_{L^{2}}
				\lesssim \left \|f\right \|^{2}_{H^{\mu}},\quad f\in H_d^{\mu+\beta},\label{cancel-PB1}
			\end{equation}
			and
			\begin{equation}
				\sup _{\B\in\mathscr{N}_2,\, \mathcal P\in\mathscr{O}_2}\Big|
				\big \langle \mathcal P\mathcal B^{2}f, \mathcal Pf \big \rangle _{L^{2}}
				+\big \langle \mathcal P\mathcal B f, \mathcal P\mathcal B f
				\big \rangle _{L^{2}} \Big| \lesssim  \left \|f\right \|_{H^{\mu}}^{2},
				\quad f\in H_d^{\mu+2\beta}.\label{cancel-PB2}
			\end{equation}
		\end{enumerate}

	\end{Theorem}

	\begin{Remark}
		Cancellation properties involving pseudo-differential operators have only recently been studied in \cite{Tang-2023-JFA, Tang-Wang-2022-arXiv}. Unlike the fixed operators $\A \in \mathbb{A}^{\alpha}$ and $\B \in \mathbb{B}^{\beta}$ considered in \cite{Tang-2023-JFA, Tang-Wang-2022-arXiv}, here we require cancellation properties for bounded families in $\mathbb{A}^{\alpha}$ and $\mathbb{B}^{\beta}$. This necessity arises because previous work addressed only pseudo-differential noise of the form $\Q_1 u\, \circ\, \mathrm{d}W$, which is equivalent to (cf. \eqref{Stratonovich to Ito}) $\frac{1}{2}\Q_1^2 u \, \mathrm{d}t + \Q_1 u \, \mathrm{d}W.$
		This structure permits separate regularization via
		\[
		\frac{1}{2}J_n^3\Q_1^2 J_n u \, \mathrm{d}t + J_n\Q_1 J_n u \, \mathrm{d}W,\quad J_n\ \text{is given in}\ \eqref{Define Jn},
		\]
		which is \textit{not} equivalent to $J_n\Q_1 J_n u\, \circ\, \mathrm{d}W$ when $\Q_1 \in \mathbb{A}^\alpha$ (because $[\Q_1,J_n]\neq 0$).
		In this scenario, the flexibility of separate regularization means that cancellation properties for the bounded family $\mathscr{O}_1 = \{\D^s J_n\}_{n \ge 1}$ suffice. For example,
		\begin{align*}
			&\langle J_n^3\Q_1^2 J_n u, u \rangle_{H^s} + \langle J_n\Q_1^2 J_n u, J_n\Q_1^2 J_n u \rangle_{H^s} \\
			=\ & \langle (\D^s J_n) \Q_1^2 (J_n u), (\D^s J_n) (J_n u) \rangle_{H^s} + \langle (\D^s J_n) \Q_1 (J_n u), (\D^s J_n) \Q_1 (J_n u) \rangle_{H^s} \\
			\lesssim\ & \|J_n u\|_{H^s}^2 \leq \|u\|_{H^s}^2.
		\end{align*}
		However, when Marcus noise $\Q_2 u \diamond \mathrm{d}L$ (defined via \eqref{Marcus flow Q2}) is included,  separate regularization becomes intricate. Instead, we employ the unified regularization
		\[
		J_n\Q_1 J_n u\, \circ\, \mathrm{d}W + J_n\Q_2 J_n u \diamond \mathrm{d}L,
		\]
		where the Marcus component is interpreted through \eqref{Marcus flow Q2-n}. 
		In summary, the separate regularization applies directly at the It\^o level, smoothing the drift and stochastic parts independently, while the unified regularization performs the smoothing already at the Stratonovich/Marcus level—before later passing to the It\^o form—thus preserving the correct commutator structure and ensuring consistency across all stochastic terms. This approach necessitates cancellation properties for bounded operator families in $\mathbb{A}^\alpha$ (see, e.g., \eqref{Qn cancel-1} and \eqref{Qn cancel-2}).

	\end{Remark}
	
	\subsection{Proof of Theorem  \ref{Thm-cancel}}
	\label{Section : Proof of Thm-cancel}
	In this section we prove Theorem \ref{Thm-cancel} and then apply these cancellation properties
	to  the noise $\Q_1 u\,\circ\, {\rm d}W$ and $\Q_2 u\diamond {\rm d}L$,  clarifying how the structure of the noise can be regularized while preserving cancellation properties uniformly.

	Let
	$\A=\L+\G\in\mathscr{N}_1$, $\B=\J+\V\in\mathscr{N}_2$ and  
	\begin{equation}\label{E M}
		\mathcal L_{\rm sym}\triangleq\mathcal L+\mathcal L^{*},\quad \mathcal J_{\rm sym}\triangleq\mathcal J+\mathcal J^{*}.
	\end{equation}

	\subsubsection{Spatially Dependent Case with Order $\alpha \in [0,1]$}
	
	\begin{proof}[Proof of  Theorem \ref{Thm-cancel} \ref{x-dependent case}]
		
		Since $\mathscr{N}_1$ is bounded in $\mathbb{A}^{\alpha}$ $($cf. Definition \ref{Ak class define}$)$ and $\mathscr{O}_1$ is bounded in $\OP\mathcal{S}^\mu$, we can infer from \ref{OP-continuous} in {Lemma~\ref{LOP}}  that
		\begin{equation}
			\label{P L G E operator norm}
			\sup _{\mathcal P\in\mathscr{O}_1}\left \|\mathcal P\right \|_{
				\mathscr L(H^{r+\mu};H^{r})},\ \
			\sup _{\A=\L+\G\in\mathscr{N}_1}
			\Big\{
			\left \|\mathcal{L}\right \|_{\mathscr L(H^{r+\alpha};H^{r})},\ 
			\left \|\mathcal{G}\right \|_{\mathscr L(H^{r};H^{r})},\ 
			\left \|\mathcal L_{\rm sym}\right \|_{\mathscr L(H^{r};H^{r})}
			\Big\}<
			\infty ,\ \ r\in \mathbb R.
		\end{equation}
		
		\textit{\underline{Proof of  \eqref{cancel-PA1}}.}
		We note that
		\begin{align*}
			\big \langle \mathcal P\L  f, \mathcal Pf
			\big \rangle _{L^{2}}  =\left \langle \big [\mathcal P,\mathcal L \big ] f,  
			\mathcal Pf\right \rangle _{L^{2}}+ \big \langle \mathcal L 
			\mathcal Pf,   \mathcal Pf\big \rangle _{L^{2}}
			=\left \langle \big [\mathcal P, \mathcal L \big ] f,  
			\mathcal Pf\right \rangle _{L^{2}} 
			-\left \langle \mathcal Pf,   \mathcal L \mathcal Pf \right
			\rangle _{L^{2}}+\left \langle \mathcal Pf, \mathcal L_{\rm sym}
			\mathcal Pf \right \rangle _{L^{2}}.
		\end{align*}
		Then we have
		\begin{align*}
			\big \langle \mathcal P\L  f, \mathcal Pf
			\big \rangle _{L^{2}}  
			=  2\left \langle \big [\mathcal P, \mathcal L  \big ] f,  
			\mathcal Pf \right \rangle _{L^{2}} +
			\left \langle \mathcal Pf, \mathcal L_{\rm sym}  \mathcal Pf
			\right \rangle _{L^{2}} 
			- \big \langle \mathcal P\L  f, \mathcal Pf
			\big \rangle _{L^{2}}  .
		\end{align*}
		Hence
		\begin{align}
			\big \langle \mathcal P\L  f, \mathcal Pf
			\big \rangle _{L^{2}}  = \left \langle \big [\mathcal P, \mathcal L \big ] f,  
			\mathcal Pf \right \rangle _{L^{2}} 
			+\frac{1}{2}\left \langle \mathcal Pf, \mathcal L_{\rm sym}
			\mathcal Pf\right \rangle _{L^{2}}.\label{PLf Pf}
		\end{align}
		Applying Lemma~\ref{Lemma-[pn qm]}, we have
		\begin{align}\label{[P L] operator norm}
			\sup_{\A=\L+\G\in\mathscr{N}_1,\,\mathcal P\in \mathscr{O}_1}
			\left\|[\mathcal P,\L]\right\|_{\mathscr L(H^{\mu+\alpha -1};L^{2})}<\infty.
		\end{align} 
		From this and \eqref{P L G E operator norm}, we have
		\begin{align*}
			\sup _{\mathcal P\in\mathscr{O}_1,\, \A=\L+\G\in\mathscr{N}_1}\left|\big \langle \mathcal P\L  f, \mathcal Pf
			\big \rangle _{L^{2}} \right|
			\lesssim  
			\left \|f\right \|^{2}_{H^{\mu}},\quad  \sup _{\mathcal P\in\mathscr{O}_1,\, \A=\L+\G\in\mathscr{N}_1}
			\left|\big \langle \mathcal P\G  f, \mathcal Pf
			\big \rangle _{L^{2}} \right|
			\lesssim  
			\left \|f\right \|^{2}_{H^{\mu}}.
		\end{align*}
		Combining the above two estimates leads to \eqref{cancel-PA1}. 
		
		\textit{\underline{Proof of  \eqref{cancel-PA2}}.}
		Define
		\begin{equation*} 
			\mathcal R_{1}=\mathcal R_{1,\L}\triangleq\big[  
			\mathcal L ,\mathcal L_{\rm sym} \big],\ \
			\mathcal R_{2}=\mathcal R_{2,\G,\L}\triangleq\big[  
			\mathcal G ,   \mathcal L \big],\ \
			\mathcal R_{3}=\mathcal R_{3,\mathcal P,\L}\triangleq\big[\mathcal P,   \mathcal L 
			\big],\ \ \mathcal R_{4}=\mathcal R_{4,\mathcal P,\L}\triangleq\big[\mathcal R_{3,\mathcal P,\L},  
			\mathcal L \big].
		\end{equation*}
		
		\textbf{Claim.}
		There holds the following identity:
		\begin{align}
			\big \langle \mathcal P\mathcal A ^{2}f, \mathcal Pf \big \rangle _{L^{2}}
			+ \big \langle \mathcal P\mathcal A  f, \mathcal P\mathcal A  f
			\big \rangle _{L^{2}} = \sum _{i=1}^{11}I_{i},
			\label{cancellation-identiy}
		\end{align}
		where
		\begin{equation*}
			\begin{cases}
				&I_{1}\triangleq\left \langle \mathcal R_{4}f , \mathcal Pf
				\right \rangle _{L^{2}},\ \
				I_{2}\triangleq\left \langle \mathcal R_{3}f , \mathcal R_{3} f \right \rangle _{L^{2}}, \ \
				I_{3}\triangleq2\left \langle \mathcal Pf , \mathcal L_{\rm sym} \mathcal R_{3}f \right \rangle _{L^{2}},
				\vspace*{4pt}
				\\
				&I_{4}\triangleq-\frac{1}{2}\left \langle \mathcal Pf , \mathcal R_{1}
				\mathcal Pf \right \rangle _{L^{2}}, \ \
				I_{5}\triangleq\frac{1}{2}\left \langle \mathcal Pf , \mathcal L_{\rm sym}^{2}
				\mathcal Pf \right \rangle _{L^{2}},\vspace*{4pt}
				\\
				&
				I_{6}\triangleq2\left \langle \mathcal P \mathcal G  f , 
				\mathcal R_{3} f \right \rangle _{L^{2}},\ \
				I_{7}\triangleq2\left \langle \mathcal P\mathcal G f ,\mathcal L_{\rm sym}
				\mathcal Pf \right \rangle _{L^{2}},\vspace*{4pt}
				\\
				&
				I_{8}\triangleq2\left \langle \mathcal R_{3}\mathcal G 
				f , \mathcal Pf \right \rangle _{L^{2}}, \ \
				I_{9}\triangleq\left \langle \mathcal P\mathcal R_{2} f , \mathcal Pf
				\right \rangle _{L^{2}},
				\vspace*{4pt}
				\\
				&I_{10}\triangleq\left \langle \mathcal P\mathcal G^{2} f ,
				\mathcal Pf\right \rangle _{L^{2}},\ \
				I_{11}\triangleq\left \langle \mathcal P\mathcal G  f , \mathcal P
				\mathcal G f \right \rangle _{L^{2}}.
			\end{cases}
		\end{equation*}

		Now we prove \eqref{cancellation-identiy}. To this end, we note that
		\begin{align}
			&\big \langle \mathcal P\mathcal A ^{2}f, \mathcal Pf
			\big \rangle _{L^{2}} + \big \langle \mathcal P\mathcal A  f,
			\mathcal P\mathcal A  f \big \rangle _{L^{2}}
			\notag
			\\
			=\ & \left \langle \mathcal P\mathcal L ^{2} f , \mathcal Pf
			\right \rangle _{L^{2}} + \left \langle \mathcal P\mathcal L 
			\mathcal G  f , \mathcal Pf \right \rangle _{L^{2}} +\left
			\langle \mathcal P\mathcal G \mathcal L  f , \mathcal Pf
			\right \rangle _{L^{2}} +\left \langle \mathcal P\mathcal G ^{2} f ,
			\mathcal Pf \right \rangle _{L^{2}}
			\notag
			\\
			&+\left \langle \mathcal P\mathcal L  f , \mathcal P\mathcal L  f
			\right \rangle _{L^{2}} +2\left \langle \mathcal P\mathcal L  f ,
			\mathcal P\mathcal G  f \right \rangle _{L^{2}} +\left \langle
			\mathcal P\mathcal G  f , \mathcal P\mathcal G  f \right
			\rangle _{L^{2}}
			\notag
			\\
			\triangleq\ & \sum _{i=1}^{7}\mathfrak{T}_{i}.
			\label{cancellation-rewrite}
		\end{align}
		Due to
		$ \mathcal L ^{*}=-\mathcal L +\mathcal L_{\rm sym} $, we have
		\begin{align*}
			\mathfrak{T}_{1}
			=\ &\left \langle \mathcal P\mathcal L  f , \mathcal L ^{*}
			\mathcal Pf \right \rangle _{L^{2}} +\big \langle \mathcal R_{3} \mathcal L f, \mathcal Pf\big \rangle _{L^{2}}\notag\\
			=\ &-\big \langle \mathcal P\mathcal L  f , \mathcal L 
			\mathcal Pf \big \rangle _{L^{2}} + \left \langle \mathcal P
			\mathcal L  f , \mathcal L_{\rm sym}  \mathcal Pf \right \rangle _{L^{2}} +
			\left \langle \mathcal R_{3} \mathcal L  f ,
			\mathcal Pf \right \rangle _{L^{2}}\notag\\
			=\ &- \left \langle \mathcal P\mathcal L  f , \mathcal P\mathcal L 
			f \right \rangle _{L^{2}} + \left \langle \mathcal P\mathcal L  f,
			\mathcal R_{3} f \right \rangle _{L^{2}} + \left
			\langle \mathcal P\mathcal L  f , \mathcal L_{\rm sym} \mathcal Pf
			\right \rangle _{L^{2}}+ \left \langle \mathcal R_{3}
			\mathcal L  f , \mathcal Pf \right \rangle _{L^{2}}
			\notag\\
			=\ &- \mathfrak{T}_{5} + \left \langle \mathcal P\mathcal L  f,
			\mathcal R_{3} f \right \rangle _{L^{2}} + \left
			\langle \mathcal P\mathcal L  f ,\mathcal L_{\rm sym}  \mathcal Pf
			\right \rangle _{L^{2}}+ \left \langle \mathcal R_{3}
			\mathcal L  f , \mathcal Pf \right \rangle _{L^{2}}.
		\end{align*}
		That is to say,
		\begin{align*}
			\mathfrak{T}_{1}+\mathfrak{T}_{5} = \left \langle \mathcal P
			\mathcal L  f , \mathcal R_{3} f \right \rangle _{L^{2}}
			+ \left \langle \mathcal P\mathcal L  f , \mathcal L_{\rm sym}
			\mathcal Pf \right \rangle _{L^{2}}+ \big \langle 
			\mathcal R_{3} \mathcal L  f , \mathcal Pf \big \rangle _{L^{2}}.
		\end{align*}
		Note that 
		$\mathcal P\mathcal L $ is of order $\mu+\alpha \ge \mu$. Similarly, the
		order of $ \mathcal R_{3} \mathcal L $ may be also larger
		than $\mu$. Therefore, by commuting $\mathcal P$ and $\mathcal L $ and
		using $\mathcal L ^{*}$ again, we have
		\begin{align}
			\mathfrak{T}_{1}+\mathfrak{T}_{5}
			=\ & \left \langle \mathcal L  \mathcal Pf , \mathcal R_{3} f \right \rangle _{L^{2}} + \left \langle \mathcal R_{3} f , \mathcal R_{3} f \right \rangle _{L^{2}}
			+ \left \langle \mathcal R_{3} \mathcal L  f ,
			\mathcal Pf \right \rangle _{L^{2}} + \left \langle \mathcal P
			\mathcal L  f , \mathcal L_{\rm sym}  \mathcal Pf \right \rangle _{L^{2}}
			\notag
			\\
			=\ & - \left \langle \mathcal Pf ,\mathcal L  \mathcal R_{3} f \right \rangle _{L^{2}} + \left \langle \mathcal Pf ,
			\mathcal L_{\rm sym}  \mathcal R_{3} f \right \rangle _{L^{2}}
			+ \left \langle \mathcal R_{3} f , \mathcal R_{3} f \right \rangle _{L^{2}}
			+ \left \langle \mathcal R_{3}\mathcal L  f , \mathcal Pf \right \rangle _{L^{2}} +
			\left \langle \mathcal P\mathcal L  f , \mathcal L_{\rm sym} \mathcal Pf
			\right \rangle _{L^{2}}
			\notag
			\\
			=\ & \left \langle \mathcal R_{4} f , \mathcal Pf
			\right \rangle _{L^{2}} + \left \langle \mathcal Pf , \mathcal L_{\rm sym}
			\mathcal R_{3} f \right \rangle _{L^{2}}  
			+ \left
			\langle \mathcal R_{3} f , \mathcal R_{3} f
			\right \rangle _{L^{2}} + \left \langle \mathcal P\mathcal L  f ,
			\mathcal L_{\rm sym} \mathcal Pf \right \rangle _{L^{2}}.
			\label{cancellation : H1+H5_1}
		\end{align}
		Note that $\mathcal L_{\rm sym} ^{*}=\mathcal L_{\rm sym}$. Then we arrive at
		\begin{align*}
			\left \langle \mathcal P\mathcal L  f , \mathcal L_{\rm sym}
			\mathcal Pf \right \rangle _{L^{2}} 
			=\ & \left \langle \mathcal L \mathcal Pf , \mathcal L_{\rm sym}
			\mathcal Pf \right \rangle _{L^{2}} + \big \langle \mathcal R_{3}f , \mathcal L_{\rm sym}  \mathcal Pf \big \rangle _{L^{2}}\\
			=\ & - \big \langle \mathcal Pf , \mathcal L_{\rm sym}  \mathcal L 
			\mathcal Pf \big \rangle _{L^{2}} -\left \langle \mathcal Pf ,
			\mathcal R_{1} \mathcal Pf \right \rangle _{L^{2}} + \left \langle
			\mathcal Pf , \mathcal L_{\rm sym}^{2}\mathcal Pf \right \rangle _{L^{2}} +
			\left \langle \mathcal R_{3} f , \mathcal L_{\rm sym}
			\mathcal Pf \right \rangle _{L^{2}}\\
			=\ & - \left \langle \mathcal L_{\rm sym}  \mathcal Pf , \mathcal P\mathcal L  f \right \rangle _{L^{2}} 
			- \left \langle \mathcal Pf ,
			\mathcal R_{1} \mathcal Pf \right \rangle _{L^{2}} + \left \langle
			\mathcal Pf , \mathcal L_{\rm sym}^{2} \mathcal Pf \right \rangle _{L^{2}} 
			+2
			\left \langle \mathcal R_{3} f , \mathcal L_{\rm sym}
			\mathcal Pf \right \rangle _{L^{2}}.
		\end{align*}
		Therefore, we arrive at
		\begin{align}
			\left \langle \mathcal P\mathcal L  f , \mathcal L_{\rm sym}  \mathcal Pf
			\right \rangle _{L^{2}}= - \frac{1}{2}\left \langle \mathcal Pf ,
			\mathcal R_{1} \mathcal Pf \right \rangle _{L^{2}} + \frac{1}{2}
			\left \langle \mathcal Pf , \mathcal L_{\rm sym}^{2}\mathcal Pf \right
			\rangle _{L^{2}} + \left \langle \mathcal R_{3} f ,
			\mathcal L_{\rm sym} \mathcal Pf \right \rangle _{L^{2}}.
			\label{cancellation : H1+H5_2}
		\end{align}
		Combining {\eqref{cancellation : H1+H5_1}},
		{\eqref{cancellation : H1+H5_2}} and
		$\mathcal L_{\rm sym}^{*}=\mathcal L_{\rm sym}$ gives
		\begin{align}
			&\mathfrak{T}_{1}+\mathfrak{T}_{5}\notag\\ 
			=\ & 
			\left \langle \mathcal R_{4} f , \mathcal Pf \right \rangle _{L^{2}} + 2\big \langle
			\mathcal L_{\rm sym}  \mathcal Pf , \mathcal R_{3} f
			\big \rangle _{L^{2}}+ \left \langle \mathcal R_{3} f ,
			\mathcal R_{3} f \right \rangle _{L^{2}}
			- \frac{1}{2} \left \langle \mathcal Pf , \mathcal R_{1}
			\mathcal Pf \right \rangle _{L^{2}} +\frac{1}{2} \left \langle
			\mathcal Pf , \mathcal L_{\rm sym}^{2}\mathcal Pf \right \rangle _{L^{2}}.
			\label{cancellation : H1+H5}
		\end{align}
		Similarly, we use
		$ \mathcal L ^{*}=-\mathcal L +\mathcal L_{\rm sym} $,
		$\mathcal R_{2}=\big[\mathcal G  ,\mathcal L \big]$ and
		$ \mathcal R_{3}=\big[\mathcal P, \mathcal L \big]$ to
		derive
		\begin{align*}
			\mathfrak{T}_{2}+\mathfrak{T}_{3} 
			=\ &2\left \langle \mathcal L \mathcal P\mathcal G  f ,
			\mathcal Pf \right \rangle _{L^{2}} 
			+2\left \langle \mathcal R_{3}\mathcal G  f , \mathcal Pf \right \rangle _{L^{2}} +
			\left \langle \mathcal P\mathcal R_{2} f , \mathcal Pf \right
			\rangle _{L^{2}}
			\notag
			\\
			=\ &2\left \langle \mathcal P\mathcal G  f , -\mathcal L 
			\mathcal Pf \right \rangle _{L^{2}} +2\left \langle \mathcal P
			\mathcal G  f , \mathcal L_{\rm sym} \mathcal Pf \right \rangle _{L^{2}}
			+2\left \langle \mathcal R_{3}\mathcal G  f ,
			\mathcal Pf \right \rangle _{L^{2}} +\left \langle \mathcal P
			\mathcal R_{2} f , \mathcal Pf \right \rangle _{L^{2}}
			\notag
			\\
			=\ &2\left \langle \mathcal P\mathcal G  f , \mathcal R_{3} f \right \rangle _{L^{2}}
			-2\left \langle \mathcal P
			\mathcal G  f , \mathcal P\mathcal L  f \right \rangle _{L^{2}} +2
			\left \langle \mathcal P\mathcal G  f , \mathcal L_{\rm sym} \mathcal Pf
			\right \rangle _{L^{2}}
			+2\left \langle \mathcal R_{3}\mathcal G  f ,
			\mathcal Pf \right \rangle _{L^{2}} +\left \langle \mathcal P
			\mathcal R_{2} f , \mathcal Pf \right \rangle _{L^{2}}
			\notag
			\\
			=\ &2\left \langle \mathcal P\mathcal G  f , \mathcal R_{3} f \right \rangle _{L^{2}} -\mathfrak{T}_{6} +2\left
			\langle \mathcal P\mathcal G  f , \mathcal L_{\rm sym} \mathcal Pf
			\right \rangle _{L^{2}}
			+2\left \langle \mathcal R_{3}\mathcal G  f ,
			\mathcal Pf \right \rangle _{L^{2}} +\left \langle \mathcal P
			\mathcal R_{2} f , \mathcal Pf \right \rangle _{L^{2}}.
		\end{align*}
		This implies that
		\begin{align}
			\mathfrak{T}_{2}+\mathfrak{T}_{3}+\mathfrak{T}_{6}
			= 2\left \langle \mathcal P\mathcal G  f , \mathcal R_{3} f \right \rangle _{L^{2}} +2\left \langle \mathcal P
			\mathcal G  f , \mathcal L_{\rm sym} \mathcal Pf \right \rangle _{L^{2}}
			+2\left \langle \mathcal R_{3}\mathcal G  f ,
			\mathcal Pf \right \rangle _{L^{2}} +\left \langle \mathcal P
			\mathcal R_{2} f , \mathcal Pf \right \rangle _{L^{2}}.
			\label{cancellation_H2+H3+H6}
		\end{align}
		On account of {\eqref{cancellation-rewrite}},
		{\eqref{cancellation : H1+H5}} and {\eqref{cancellation_H2+H3+H6}}, we obtain {\eqref{cancellation-identiy}}.
		
		Now, we observe that, 
		if for any given $r\in\R$,
		\begin{align}
			\label{R 1-4 bound}
			\sup_{
				\A=\L+\G\in\mathscr{N}_1,\,\mathcal P\in \mathscr{O}_1}
			\bigg\{ &
			\left \| \mathcal R_{1}\right \|_{\mathscr L(H^{r+\alpha -1};H^{r})},\ 
			\left \| \mathcal R_{2}\right \|_{\mathscr L(H^{r+\alpha -1};H^{r})},\
			\left \| \mathcal R_{3}\right \|_{\mathscr L(H^{r+\mu+\alpha -1};H^{r})},\ 
			\left \| \mathcal R_{4}\right \|_{\mathscr L(H^{r+\mu+2\alpha -2};H^{r})}
			\bigg\}
			<\infty,
		\end{align}
		then \eqref{cancel-PA2} follows directly from \eqref{cancellation-identiy}, \eqref{P L G E operator norm}, and \eqref{R 1-4 bound}. Thus, it remains to establish \eqref{R 1-4 bound}.

		Via  Lemma~\ref{Lemma-[pn qm]} and \eqref{[P L] operator norm}, we see that \eqref{R 1-4 bound} holds true for $\mathcal R_{1}$, $\mathcal R_{2}$, $\mathcal R_{3}$. 
		Moreover, repeating the argument used in the proof of  Lemma \ref{Lemma-[pn qm]}, we see that
		$\{\mathcal R_{3}=\mathcal R_{3,\mathcal{P}, \L}\}_{\A=\L+\G\in\mathscr{N}_1,\,\mathcal P\in \mathscr{O}_1}\subset{\mathrm{OP}}\mathbf S^{\mu+\alpha -1}$ is
		bounded. Note  that
		$\mathcal R_{3}$ is of diagonal form. 
		Finally, by Lemma \ref{Lemma-[pn qm]}, we have
		$$\sup _{\A=\L+\G\in\mathscr{N}_1,\,\mathcal P\in \mathscr{O}_1}
		\left \|\mathcal R_{4}\right \|_{\mathscr L(H^{r+\mu+2\alpha -2};H^{r})}<\infty.$$
		Hence we obtain \eqref{R 1-4 bound}. 
	\end{proof}

	\subsubsection{Spatially Independent Case with Order $\beta \ge0$}
	
	\begin{proof}[Proof of  Theorem \ref{Thm-cancel} \ref{x-independent case}]
		Let $\B=\J+\V\in\mathscr{N}_2$ and  recall that $\mathcal J_{\rm sym}\triangleq\mathcal J+\mathcal J^{*}$ in \eqref{E M}. In this case, we have 
		\begin{equation}\label{P V J M commutes}
			[\mathcal{P}, \mathcal{J}]=[\mathcal{P}, \mathcal{V}]=[\mathcal{P}, \mathcal J_{\rm sym}]=[\mathcal{J}, \mathcal{V}]=[\mathcal{J}, \mathcal{J}^*]=0.
		\end{equation}
		Moreover, similar to the proof of  \eqref{cancel-PA1}, it follows from
		the boundedness of $\mathscr{N}_2\subset\mathbb{B}^{\beta}$ $($cf. Definition \ref{Bk class define}$)$, the boundedness of $\mathscr{O}_2\subset\OP\mathcal{S}_0^\mu$ and Lemma \ref{LOP} that
		\begin{equation}
			\label{P J V M operator norm}
			\sup _{\mathcal P\in\mathscr{O}_2}\left \|\mathcal P\right \|_{
				\mathscr L(H^{r+\mu};H^{r})},\ \
			\sup _{\B=\J+\V\in\mathscr{N}_2}
			\Big\{
			\left \|\mathcal{J}\right \|_{\mathscr L(H^{r+\beta};H^{r})},\ 
			\left \|\mathcal{V}\right \|_{\mathscr L(H^{r};H^{r})},\ 
			\left \|\mathcal J_{\rm sym}\right \|_{\mathscr L(H^{r};H^{r})}
			\Big\}<
			\infty ,\ \ r\in \mathbb R.
		\end{equation}
		
		\textit{\underline{Proof of  \eqref{cancel-PB1}}.} 
		Repeating the argument used in the proof of  \eqref{cancel-PA1} (see also \eqref{PLf Pf}), we obtain
		\begin{align*}
			\big \langle \mathcal P\B f, \mathcal Pf
			\big \rangle _{L^{2}}  =\frac{1}{2}\left \langle \mathcal Pf, \mathcal J_{\rm sym}
			\mathcal Pf\right \rangle _{L^{2}}
			+\left \langle \mathcal P\V f,  
			\mathcal Pf\right \rangle _{L^{2}}.
		\end{align*}
		Combining the above estimate and \eqref{P J V M operator norm}, we obtain  \eqref{cancel-PB1}.

		\textit{\underline{Proof of  \eqref{cancel-PB2}}.} 
		Following the same approach as in the derivation of \eqref{cancellation-identiy} in the proof
		for {\eqref{cancel-PA2}}, and employing \eqref{P V J M commutes}, we identify that
		\begin{align}
			\big \langle \mathcal P\mathcal B^{2}f, \mathcal Pf \big \rangle _{L^{2}}
			+ \big \langle \mathcal P\mathcal B  f, \mathcal P\mathcal B  f
			\big \rangle _{L^{2}} =\, \sum _{i=1}^{4}M_{i},
			\label{cancellation-identiy-2}
		\end{align}
		where
		\begin{equation*}
			M_{1}\triangleq \frac{1}{2}\left \langle \mathcal Pf , 
			\mathcal J_{\rm sym}^{2} \mathcal Pf \right \rangle _{L^{2}},\quad
			M_{2}\triangleq2\left \langle \mathcal P\mathcal V f ,
			\mathcal J_{\rm sym} \mathcal Pf \right \rangle _{L^{2}}, \quad
			M_{3}\triangleq\left \langle \mathcal P\mathcal V^{2} f ,
			\mathcal Pf\right \rangle _{L^{2}},\quad
			M_{4}\triangleq\left \langle \mathcal P\mathcal V f , \mathcal P
			\mathcal V f \right \rangle _{L^{2}}.
		\end{equation*}
		By virtue of \eqref{P J V M operator norm} and \eqref{cancellation-identiy-2}, \eqref{cancel-PB2} holds true.
	\end{proof}

	\subsection{Applications of Cancellation Properties}\label{Section : Appl-cancel}
	
	We first prove the following lemma concerning the existence of regular approximations of $\Q_i$ that preserve the cancellation properties.
	
	For operators $\mathcal{Q}_1,\, \mathcal{Q}_2  \in \mathbb{A}^\alpha \cup \mathbb{B}^\beta$ with symbols taking values in $\mathbb{C}^{d\times d}$. We agree that, once $\Q_1$ and $\Q_2$ are given as elements of  $\mathbb{A}^\alpha \cup \mathbb{B}^\beta$, the chosen decomposition and exponent shall remain fixed,  even if $\mathcal Q_i$ admits another admissible
	description. 
	
	By Lemma \ref{LOP}, we define for $i=1,2$,
	\begin{equation}\label{delta-Qi}
		\Q_i \in \LL\big(H_d^{s+\delta_i};H_d^s\big),\quad s\in\R,\qquad
		\delta_i=\delta_{\Q_i} \triangleq 
		\begin{cases}
			\alpha, & \text{if } \Q_i\in\mathbb{A}^\alpha \text{ and }  \Q_{i}\neq 0,\\[4pt]
			\beta, & \text{if } \Q_i\in\mathbb{B}^\beta \text{ and }  \Q_{i}\neq 0,\\[4pt]
			0, & \text{if } \Q_i=0.
		\end{cases}
	\end{equation}
	
	\begin{Lemma}\label{Lemma:Qn}
		Let $\Q_1,\Q_2\in \mathbb{A}^{\alpha}\cup \mathbb{B}^{\beta}$ with symbols taking values in $\mathbb{C}^{d\times d}$ and with fixed decomposition.   There exist  sequences $\{\Q_{i,n}\}_{n\ge1}$ $(i=1,2)$ such that the following properties hold:
		\begin{enumerate}[label={{\bf (\arabic*)}},leftmargin=0.79cm]\setlength\itemsep{0.2em}
			\item\label{Q-n mollify} For $i=1,2$ and $n\ge1$, $\Q_{i,n}\in\OP\SS^{-\infty}$ if $\Q_i\in\mathbb{A}^{\alpha}$ or $\Q_{i,n}\in\OP\SS_0^{-\infty}$ if $\Q_i\in\mathbb{B}^{\beta}$;

			\item\label{Q-n bouded in A B} $\{\Q_{i,n}\}_{n\ge1}\subset \mathbb{A}^{\alpha}$ is bounded   if $\Q_{i}\in \mathbb{A}^{\alpha}$ and $\{\Q_{i,n}\}_{n\ge1}\subset \mathbb{B}^{\beta}$ is bounded if $\Q_{i}\in \mathbb{B}^{\beta}$; 
			
			\item\label{Qn-Qm} Let $\delta_i=\delta_{\Q_i}$ be given in \eqref{delta-Qi}. Then, for $i=1,2$ and $\theta\ge0$,
			\begin{align}
				\lim_{n\to\infty}\|\Q_{i,n}f-\Q_{i}f\|_{H^{\theta}} =0,\qquad f\in H_d^{\theta+\delta_i}, \label{Qnf-Qf}
			\end{align}
			\begin{align}
				\|\Q_{i,n}-\Q_{i,m}\|_{\LL(H^s,H^{\theta})} 
				\lesssim  (n \land m)^{-(s-\theta-\delta_i)},\qquad s>\theta+\delta_i,\label{Qn-Qm operator norm}
			\end{align}
			\begin{align}
				\|\Q_{i,n}^2-\Q_{i,m}^2\|_{\LL(H^s,H^{\theta})}
				\lesssim   (n \land m)^{-(s-\theta-2\delta_i)},\qquad s>\theta+2\delta_i.\label{Qn2-Qm2 operator norm}
			\end{align}
			
			\item $($\textbf{Uniform cancellation in renormalization}$)$ \label{Q-n cancel} Let $s \ge 0$. There exist constants $C_{i,j}=C_{i,j}(s,d)>0$, for $i,j\in\{1,2\}$, such that
			\begin{equation}\label{Qn cancel-1}
				\sup_{n \ge 1} \left| \langle f, \mathcal{Q}_{i,n} f \rangle_{H^s} \right| \leq C_{i,1} \|f\|^2_{H^s}, \qquad f \in H_d^s, \quad i = 1,2,
			\end{equation}
			and
			\begin{equation}\label{Qn cancel-2}
				\sup_{n \ge 1} \left| \langle \mathcal{Q}_{i,n}^2 f, f \rangle_{H^s} + \langle \mathcal{Q}_{i,n} f, \mathcal{Q}_{i,n} f \rangle_{H^s} \right| \leq C_{i,2} \|f\|^2_{H^s}, \qquad f \in H_d^s, \quad i = 1,2.
			\end{equation}
			
			\item \label{Q-n cancel 0} In particular, 
			if $\Q_i\in\mathbb{B}^\beta$ satisfies $\Q_i+\Q_i^*=0$ for $i=1,2$, then for  any $s \ge 0$,
			\begin{equation*} 
				\sup_{n \ge 1} \left| \langle f, \mathcal{Q}_{i,n} f \rangle_{H^s} \right| =0, \qquad f \in H_d^s, \quad i = 1,2,
			\end{equation*}
			and
			\begin{equation*} 
				\sup_{n \ge 1} \left| \langle \mathcal{Q}_{i,n}^2 f, f \rangle_{H^s} + \langle \mathcal{Q}_{i,n} f, \mathcal{Q}_{i,n} f \rangle_{H^s} \right| =0, \qquad f \in H_d^s, \quad i = 1,2.
			\end{equation*}
		\end{enumerate}

	\end{Lemma}
	\begin{proof}
		Let $\{J_n\}_{n\ge 1}$ be a sequence of Friedrichs mollifiers given in \eqref{Define Jn}. By definition, $J_n\in \OP\SS_0^{-\infty}$ for all $n\ge1$, and $\{J_n\}_{n\ge1}\subset \OP\SS_0^{0}$ is bounded. As a result,
		for $i=1,2$, one can infer from Lemma \ref{LOP} that $$\Q_{i,n}\triangleq J_n\Q_i J_n$$
		satisfies \ref{Q-n mollify} and \ref{Q-n bouded in A B}. For \ref{Qn-Qm}, \eqref{Qnf-Qf} is a direct consequence of Lemma \ref{Lemma-Jn}. When $s-\theta-\delta_i>0$, it follows from Lemmas \ref{Lemma-Jn} and \ref{LOP} that
		\begin{align*}
			\|\Q_{i,n}-\Q_{i,m}\|_{\LL(H^s,H^{\theta})} 
			=\ & \norm{(J_n-J_m)\Q_iJ_n +J_m\Q_i(J_n-J_m)}_{\LL(H^s,H^{\theta})}\notag\\
			\leq \ & \norm{J_n-J_m}_{\LL(H^{s-\delta_i};H^{\theta})}\norm{\Q_iJ_n }_{\LL(H^s;H^{s-\delta_i})}
			+\norm{J_m\Q_i}_{\LL(H^{\theta+\delta_i};H^{\theta})}
			\norm{J_n-J_m}_{\LL(H^s;H^{\theta+\delta_i})}\notag\\
			\lesssim \ & (n \land m)^{-(s-\theta-\delta_i)},
		\end{align*}
		which is \eqref{Qn-Qm operator norm}. Similarly, when  $s>\theta+2\delta_i$, it is easy to see that 
		\begin{align*}
			\|\Q_{i,n}^2-\Q_{i,m}^2\|_{\LL(H^s,H^{\theta})} 
			=\ & \norm{(J_n-J_m)\Q_iJ^2_n\Q_iJ_n +J_m\Q_i(J^2_n-J^2_m)\Q_iJ_n+J_m\Q_iJ^2_m\Q_i(J_n-J_m)}_{\LL(H^s,H^{\theta})}
			\notag\\
			\leq \ & \norm{J_n-J_m}_{\LL(H^{s-2\delta_i};H^{\theta})}\norm{\Q_iJ^2_n\Q_iJ_n}_{\LL(H^s;H^{s-2\delta_i})}\notag\\
			&+\norm{J_m\Q_i}_{\LL(H^{\theta+\delta_i};H^{\theta})}
			\norm{J^2_n-J^2_m}_{\LL(H^{s-\delta_i};H^{\theta+\delta_i})}
			\norm{\Q_iJ_n}_{\LL(H^s;H^{s-\delta_i})}\notag\\
			&+\norm{J_m\Q_iJ^2_m\Q_i}_{\LL(H^{\theta+2\delta_i};H^{\theta})}\norm{J_n-J_m}_{\LL(H^s;H^{\theta+2\delta_i})}\notag\\
			\lesssim \ & (n \land m)^{-(s-\theta-2\delta_i)}.
		\end{align*}
		Hence \eqref{Qn2-Qm2 operator norm} is true. The estimates \eqref{Qn cancel-1} and \eqref{Qn cancel-2} follows from Theorem \ref{Thm-cancel} and the dense embedding $H^{s_1}\hookrightarrow H^{s_2}$ for $s_1\ge s_2$.
		Finally, if $\Q_i\in\mathbb{B}^\beta$, $\Q_i$ commutes with $J_n$, i.e., $[\Q_i,J_n]=0$.  From this and the property $\Q_i+\Q_i^*=0$,  we arrive at
		$$\langle f, \mathcal{Q}_{i,n} f \rangle_{H^s}=\frac{1}{2}\langle J_n f, \mathcal{Q}_{i}J_n f \rangle_{H^s}+\frac{1}{2}\langle J_n f, \mathcal{Q}^*_{i}J_n f \rangle_{H^s}=0,$$
		and 
		$$\langle \mathcal{Q}_{i,n}^2 f, f \rangle_{H^s} + \langle \mathcal{Q}_{i,n} f, \mathcal{Q}_{i,n} f \rangle_{H^s}
		=\langle \mathcal{Q}_{i}J^2_n f, \Q_i^*J_n^2 f \rangle_{H^s} + \langle \mathcal{Q}_{i} J^2_n f, \mathcal{Q}_{i} J^2_n f \rangle_{H^s}=0,$$
		which yields the desired estimates.
	\end{proof}
	
	\begin{Remark}\label{Remark:renormalization Qn}
		The cancellation properties established in Theorem \ref{Thm-cancel} require that $f$ possess additional regularity beyond order $s$. Nevertheless, in  \eqref{Qn cancel-1} and \eqref{Qn cancel-2}, all terms are well defined for $f\in H^s$. Consequently, the sequences $\{\Q_{i,n}\}_{n\ge1}$ ($i=1,2$) can be viewed as renormalized operators that preserve the cancellation structure of \eqref{Qn cancel-1} and \eqref{Qn cancel-2}  uniformly in $n$ for $f\in H^s$. This uniform, renormalized cancellation allows one to handle the singular terms by passage to the limit.
	\end{Remark}

	\subsubsection{Stratonovich Case : $\Q_1 u\,\circ\, {\rm d}W$}\label{Section : Q1n}
	Recall \eqref{eq:Strat-Ito}, i.e.,
	\begin{equation*}
		\mathcal{Q}_1 u\, \circ\, {\rm d}W 
		=
		\frac{1}{2}\mathcal{Q}_1^2 u\d t+\mathcal{Q}_1 u\d W.
	\end{equation*}
	To solve \eqref{SEuler-(rho u)} or \eqref{SEuler-(u P) damp} in the Sobolev space, we require estimates in $H^s$ for the above terms in integral form (see \eqref{Target problem (rho u)}$_{2}$ or \eqref{SEuler-damping-target}$_1$). However, as discussed in Remark \ref{Remark singualrity}, $u\in H^s$ does not guarantee $\mathcal{Q}_1 u\, \circ\, {\rm d}W \in H^s$, since $\mathcal{Q}_1^2 u\in H^{s-2\delta_1}$ and $\mathcal{Q}_1 u\in H^{s-\delta_1}$, where $\delta_1=\delta_{\Q_1}$ is given by \eqref{delta-Qi}. Consequently, we cannot directly apply It\^{o}'s formula to $\|u\|^2_{H^s}$.
	
	Fortunately, \ref{Q-n cancel} in Lemma \ref{Lemma:Qn} shows that $\mathcal{Q}_1 u\, \circ\, {\rm d}W$ can be regularized while preserving cancellation properties uniformly. Specifically, let $\Q_{1,n}$ be the regular approximation of $\Q_1$ provided by Lemma \ref{Lemma:Qn}, and consider the regularized Stratonovich noise
	\begin{equation*} 
		\mathcal{Q}_{1,n} u\,\circ\, {\rm d}W
		=
		\frac{1}{2}\mathcal{Q}_{1,n}^2 u\d t+\mathcal{Q}_{1,n} u\d W.
	\end{equation*}
	Then \ref{Q-n cancel} in Lemma \ref{Lemma:Qn} yields the required uniform estimates.
	
	To pass to the limit in a suitable Sobolev space $H^\theta$ with suitable $\theta<s$, we must also estimate the difference between two mollification levels $n$ and $m$. In particular, we encounter the expressions
	$$
	\langle \mathcal{Q}_{1,n} f-\mathcal{Q}_{1,m} g,f-g\rangle_{H^{\theta}}\d W,
	$$
	and
	$$
	\Big(\langle \mathcal{Q}_{1,n}^2 f-\mathcal{Q}_{1,m}^2 g,f-g\rangle_{H^{\theta}}+
	\langle \mathcal{Q}_{1,n} f-\mathcal{Q}_{1,m} g,\mathcal{Q}_{1,n} f-\mathcal{Q}_{1,m} g\rangle_{H^{\theta}}\Big)\d t.
	$$
	Thus, we require the following stability estimates.
	\begin{Lemma}\label{Lemma-Q1 n m}
		Let $\delta_1=\delta_{\Q_1}$  be as defined in \eqref{delta-Qi} and let $m,n\ge1$.  The following estimates hold true:
		\begin{align}
			&\left|\IP{\mathcal{Q}_{1,n} f-\mathcal{Q}_{1,m} g,f-g}_{H^{\theta}} \right|\notag\\
			\lesssim \ & (n \land m)^{-2(s-\theta-\delta_1)}(\norm{f}_{H^{s}}+\norm{g}_{H^{s}})^2+\|f-g\|^2_{H^{\theta}},\quad f,g\in H_d^s,\ \, s>\theta+\delta_1, 
			\label{Q1n difference Ito 1}
		\end{align}
		and
		\begin{align}
			&\Big|\BIP{\mathcal{Q}_{1,n}^2 f-\mathcal{Q}_{1,m}^2 g,f-g}_{H^{\theta}}+
			\BIP{\mathcal{Q}_{1,n} f-\mathcal{Q}_{1,m} g,\mathcal{Q}_{1,n} f-\mathcal{Q}_{1,m} g}_{H^{\theta}}\Big|\notag\\
			\lesssim\ &  (n \land m)^{-(s-\theta-\delta_1)}(\norm{f}_{H^{s}}+\norm{g}_{H^{s}})^2+\|f-g\|^2_{H^{\theta}},\quad f,g\in H_d^s,\ \,  s>\theta+2\delta_1.
			\label{Q1n difference Ito 2}
		\end{align}
	\end{Lemma}
	\begin{proof}
		We first use  \eqref{Qn-Qm operator norm} and \eqref{Qn cancel-1} to obtain
		\begin{align*}
			&\left|\IP{\mathcal{Q}_{1,n} f-\mathcal{Q}_{1,m} g,f-g}_{H^{\theta}} \right| \\
			\leq  \ & 
			\left|\BIP{\Q_{1,n}f-\Q_{1,m}f,f-g}_{H^{\theta}}\right|
			+ 
			\left| \BIP{\Q_{1,m}f-\Q_{1,m}g,f-g}_{H^{\theta}}\right|\\
			\lesssim \ & 
			\|\Q_{1,n}-\Q_{1,m}\|_{\LL(H^{s};H^\theta)}\|f\|_{H^s}\|f-g\|_{H^{\theta}}
			+ \|f-g\|^2_{H^{\theta}}\\
			\lesssim \ & (n \land m)^{-2(s-\theta-\delta_1)}\|f\|^2_{H^s}+\|f-g\|^2_{H^{\theta}},
		\end{align*}
		which implies \eqref{Q1n difference Ito 1}. As for \eqref{Q1n difference Ito 2}, we note that
		\begin{align}
			&\BIP{\mathcal{Q}_{1,n}^2 f-\mathcal{Q}_{1,m}^2 g,f-g}_{H^{\theta}}+
			\BIP{\mathcal{Q}_{1,n} f-\mathcal{Q}_{1,m} g,\mathcal{Q}_{1,n} f-\mathcal{Q}_{1,m} g}_{H^{\theta}}\notag\\
			=\ & \BIP{\Q_{1,n}^2f -\Q_{1,m}^2f ,f -g}_{H^{\theta}}
			+\BIP{\Q_{1,m}^2f -\Q_{1,m}^2g,f -g}_{H^{\theta}}\notag\\
			&+\BIP{\mathcal{Q}_{1,n} f-\mathcal{Q}_{1,m} f,\mathcal{Q}_{1,n} f-\mathcal{Q}_{1,m} f}_{H^{\theta}}
			+\BIP{\mathcal{Q}_{1,m} f-\mathcal{Q}_{1,m} g,\mathcal{Q}_{1,n} f-\mathcal{Q}_{1,m} f}_{H^{\theta}}\notag\\
			&+ \BIP{\mathcal{Q}_{1,n} f-\mathcal{Q}_{1,m} f,\mathcal{Q}_{1,m} f-\mathcal{Q}_{1,m} g}_{H^{\theta}} 
			+\BIP{\mathcal{Q}_{1,m} f-\mathcal{Q}_{1,m} g,\mathcal{Q}_{1,m} f-\mathcal{Q}_{1,m} g}_{H^{\theta}}\notag\\
			\triangleq \ & \sum_{k=1}^6\mathfrak{A}_k. \label{Q1n difference Ito 1-6}
		\end{align}
		By Lemma \ref{Lemma:Qn}, the linearity of $\Q_{1,n}$ and the fact $s>\theta+2\delta_1$, we have 
		\begin{align*}
			|\mathfrak{A}_2+\mathfrak{A}_6|\lesssim \|f-g\|^2_{H^{\theta}},
		\end{align*}
		\begin{align*}
			|\mathfrak{A}_1| \leq  \norm{\Q_{1,n}^2-\Q_{1,m}^2}_{\LL(H^s;H^{\theta})}\|f\|_{H^{s}}\cdot \|f-g\|_{H^{\theta}} 
			\lesssim (n\wedge m)^{-2(s-\theta-2\delta_1)} (\|f\|_{H^{s}} +\|g\|_{H^{s}})^2 + \|f-g\|^2_{H^{\theta}},
		\end{align*}
		\begin{align*}
			|\mathfrak{A}_3|\leq  \norm{\Q_{1,n}-\Q_{1,m}}^2_{\LL(H^s;H^{\theta})}\|f\|^2_{H^{s}} 
			\lesssim (n\wedge m)^{-2(s-\theta-\delta_1)} \|f\|^2_{H^{s}},
		\end{align*}
		\begin{align*}
			|\mathfrak{A}_4|,|\mathfrak{A}_5|
			\lesssim \|f-g\|_{H^{s}}\norm{\Q_{1,n}-\Q_{1,m}}_{\LL(H^s;H^{\theta})}\|f\|_{H^{s}}
			\lesssim   (n\wedge m)^{-(s-\theta-\delta_1)} (\|f\|^2_{H^{s}} +\|f\|_{H^{s}} \|g\|_{H^{s}}).
		\end{align*}
		Combing the above estimates and \eqref{Q1n difference Ito 1-6}, we obtain \eqref{Q1n difference Ito 2}.
	\end{proof}

	\subsubsection{Marcus Case : $\Q_2 u\diamond {\rm d}L$}\label{Section : Q2n}
	As explained in Section \ref{Section:introduce noise}, to study $\Q_2 u\diamond {\rm d}L$ in Sobolev space $H^s$, we need to 
	consider the Marcus flow \eqref{Marcus flow Q2} in $H^s$ with $l\in\R$, i.e.,
	\begin{equation*}
		\wp(r)\triangleq\wp(r,l,f),\quad 
		\frac{{\rm d}}{{\rm d}r}\wp(r)= l\cdot \Q_2\wp(r),\quad r\in[0,1],  \quad \wp(0) = f\in H^s.
	\end{equation*}
	Then, under \ref{Hypo-W L},  $\int_0^t \Q_2 u(t')\diamond {\rm d}L(t')$ can be expressed as 
	\begin{align*}
		\int_0^t \Q_2 u(t')\diamond {\rm d}L(t') 
		=\ 
		& \int_0^t \int_{|l|\le 1}
		\big\{ \wp(1,l,u(t'-)) - u(t'-)\big\}\,
		\tilde{\eta}({\rm d}l,{\rm d}t') 
		\\ & 
		+\int_0^t \int_{|l|\le 1}
		\big\{ \wp(1,l,u(t'))-u(t')
		-l\cdot \Q_2(u(t'))\big\}\, 
		\nu({\rm d}l)\,{\rm d}t'.
	\end{align*}
	
	As noted in Remark \ref{Remark singualrity}, the equation is singular in $H^s$ because $\Q_2: H^s\to H^{s-\delta_2}$, where $\delta_2=\delta_{\Q_2}$ is given by \eqref{delta-Qi}. Specifically, for $u\in H^s$, the flow $\wp(1,l,u)$ does not remain in $H^s$, implying that the Marcus integral $\int_0^t \Q_2 u(t')\diamond {\rm d}L(t')$ does not preserve $H^s$ regularity. Analogous to the Stratonovich term $\Q_1 u\,\circ\, {\rm d}W$, we regularize the Marcus flow by applying Lemma \ref{Lemma:Qn} to obtain a sequence $\wp_n$ satisfying uniform estimates. Moreover, these uniform estimates extend to a family of Lyapunov-type functions $V\in\mathscr{V}$ defined as in \eqref{SCRV}.

	\begin{Lemma}\label{Lemma-Marcus-n}
		Let $l\in\R$, $s\ge0$ and let \ref{Hypo-Qi} hold. Let $\Q_{2,n}$ denote the regular approximation of  $\Q_2$ provided by Lemma \ref{Lemma:Qn}. Consider the following regularized Marcus flow
		\begin{equation}\label{regular Marcus flow Q2}
			\wp_n(r)\triangleq  \wp_n(r,l,f),\quad  \frac{{\rm d}}{{\rm d}r}\wp_n(r)= l\cdot \Q_{2,n}\wp_n(r),\quad r\in[0,1],  \quad \wp_n(0) = f\in H^s.
		\end{equation}
		Let $C_{2,1} > 0$ and $C_{2,2} > 0$ be the constants introduced in  \ref{Q-n cancel} of Lemma \ref{Lemma:Qn}. 
		Let $V\in \mathscr{V}$ $($see \eqref{SCRV}$)$.
		Define
		\begin{equation}\label{V wp-n f}
			V_s^{\wp_n,f}(r,l)\triangleq V(\|\wp_n(r,l, f)\|_{H^s}^2) - V(\norm{f}_{H^s}^2),\quad r\in[0,1],\quad f\in H^{s}.
		\end{equation}
		Then we have that for $l\in\R$
		\begin{align}
			\sup_{n\ge1}\Abs{V_s^{\wp_n,f}(r,l)}
			\leq \left({\rm e}^{2C_{2,1}|l|r}-1\right) V(\norm{f}_{H^s}^2),\quad r\in[0,1],
			\label{wp-n estimate 1}
		\end{align}
		and
		\begin{align}
			\sup_{n\ge1}\left\{V_s^{\wp_n,f}(1,l)-2l\cdot V'(\|f\|^2_{H^s} )\bIP{\Q_{2,n}f,f}_{H^s}\right\}\leq \frac{C_{2,2}}{2C_{2,1}^2}\left({\rm e}^{2C_{2,1}|l|}-2C_{2,1}|l|-1\right)V(\|f\|^2_{H^s}).
			\label{wp-n estimate 2}
		\end{align}
	\end{Lemma}

	\begin{proof}
		For each $n\ge1$, by Lemma \ref{Lemma:Qn} we see that $\Q_{2,n}\in \mathscr{L}(H^s;H^s)$, which implies the existence and uniqueness of solution $\wp_n\in H^s$ to \eqref{regular Marcus flow Q2}. Now we prove \eqref{wp-n estimate 1} and \eqref{wp-n estimate 2}. 
		Keep in mind that $\wp_n\big(0,l,f\big)=f$, $n\ge1$.
		For simplicity of notation, in the proof we write  $\wp_n(r)$ instead of $\wp_n(r,l,f)$,  omitting the dependence on $l$ and $f$.
		
		It is easy to check that for any $f\in H^s$,
		\begin{align}
			V_s^{\wp_n,f}(r,l) 
			=  \int_0^{r}\Big[\frac{{\rm d}}{{\rm d}r} V(\|\wp_n\|^2_{H^s})\Big](r') \d r'  
			=   2l \int_0^{r} V'(\|\wp_n(r')\|^2_{H^s})\bIP{\wp_n(r'),\Q_{2,n}\wp_n(r')}_{H^s} \d r'.\label{Ito-Marcus-term 1}
		\end{align}
		Then, by
		\eqref{Ito-Marcus-term 1}, \eqref{Qn cancel-1}, the facts $V'(x)>0$ and $V'(x)x\leq V(x)$ (cf. \eqref{SCRV}) we have
		\begin{align*}
			V_s^{\wp_n,f}(r,l)
			\leq   \left|V_s^{\wp_n,f}(r,l) \right| 
			\leq  2C_{2,1}|l| \int_0^{r} V'(\|\wp_n(r')\|^2_{H^s})\|\wp_n(r')\|^2_{H^s} \d r' 
			\leq  2C_{2,1}|l| \int_0^{r} V(\|\wp_n(r')\|^2_{H^s})\d r'.
		\end{align*}
		Thanks to Gr\"{o}nwall's inequality, we have
		\begin{equation}
			\sup_{n\ge1}V(\|\wp_n(r)\|_{H^s}^2) \leq V(\norm{f}_{H^s}^2)\cdot \mathrm{e}^{2C_{2,1}|l|r},
			\quad r\in[0,1].
			\label{Phi norm-2}
		\end{equation}
		From the above two inequalities, we can obtain  \eqref{wp-n estimate 1}.

		Similarly, we use \eqref{Ito-Marcus-term 1} to derive
		\begin{align}
			&V_s^{\wp_n,f}(1,l)-2l\cdot V'(\|f\|^2_{H^s})\bIP{\Q_{2,n} f ,f}_{H^s}\notag\\
			= \ &  2l\cdot \left(\int_0^1  \Big(V'(\|\wp_n(r')\|^2_{H^s})\bIP{\Q_{2,n}\wp_n(r),\wp_n(r)}_{H^s}-V'(\|f\|^2_{H^s})\IP{\Q_{2,n} f ,f}_{H^s}\Big)\d r\right)\notag\\
			= \ &  2l\cdot \int_0^1\int_{0}^{r} \Big[\frac{{\rm d}}{{\rm d}r} g_n\Big](r')\d r' \d r,
			\label{Ito-Marcus-term 2}
		\end{align}
		where 
		$$g_n(\cdot)\triangleq V'(\|\wp_n(\cdot)\|^2_{H^s})\bIP{\Q_{2,n}\wp_n(\cdot),\wp_n(\cdot)}_{H^s}. $$
		Since $\Q_{2,n}$  is a linear mapping, the Fr\'echet derivative of $\Q_{2,n}$, denoted as $\nabla \Q_{2,n}$, satisfies
		\begin{equation*}
			[\nabla \Q_{2,n}(f)]h=\Q_{2,n}h,\quad f,\ h\in H_d^s.
		\end{equation*}   
		This implies that
		\begin{align}
			\frac{{\rm d}}{{\rm d}r} g_n(r)  
			= \ & 2l\cdot V''(\|\wp_n(r)\|^2_{H^s})\bIP{\Q_{2,n}\wp_n(r),\wp_n(r)}^2_{H^s}\notag\\
			&+
			V'(\|\wp_n(\cdot)\|^2_{H^s})\left\{\IP{\Big[\nabla \Q_{2,n}\big(\wp_n(r)\big)\Big] \frac{{\rm d}}{{\rm d}r}\wp_n(r),\wp_n(r)}_{H^s}+\IP{ \Q_{2,n}\wp_n(r), \frac{{\rm d}}{{\rm d}r}\wp_n(r)}_{H^s}\right\}\notag\\
			= \ & 2l\cdot V''(\|\wp_n(r)\|^2_{H^s})\bIP{\Q_{2,n}\wp_n(r),\wp_n(r)}^2_{H^s}\notag\\
			&+l\cdot V'(\|\wp_n(r)\|^2_{H^s})\left\{\IP{
				\Big[\nabla \Q_{2,n}\big(\wp_n(r)\big)\Big] \Q_{2,n}\wp_n(r),\wp_n(r)}_{H^s}+ \bIP{ \Q_{2,n}\wp_n(r), \Q_{2,n}\wp_n(r)}_{H^s}\right\}\notag\\
			=\ &2l \cdot V''(\|\wp_n(r)\|^2_{H^s})\bIP{\Q_{2,n}\wp_n(r),\wp_n(r)}^2_{H^s}\notag\\
			&+l\cdot V'(\|\wp_n(r)\|^2_{H^s})\Big(\bIP{\Q_{2,n}^2\wp_n(r), \wp_n(r)}_{H^s}+\bIP{\Q_{2,n}\wp_n(r), \Q_{2,n}\wp_n(r)}_{H^s}\Big).\label{Ito-Marcus-term 2 gn}
		\end{align}
		Now we use  \eqref{Qn cancel-2}, \eqref{Phi norm-2}, the facts $V''(x)<0$ and $V'(x)x\leq V(x)$ (see \eqref{SCRV}) to obtain
		\begin{align*}
			&V_s^{\wp_n,f}(1,l)-2l\cdot V'(\|f\|^2_{H^s})\bIP{\Q_{2,n} f ,f}_{H^s}\notag\\
			= \ &  2l\cdot \int_0^1\int_{0}^{r} \Big[\frac{{\rm d}}{{\rm d}r} g_n\Big](r')\d r' \d r\notag\\
			\leq  \ &  2l^2\int_0^1\int_{0}^{r}
			V'(\|\wp_n(r)\|^2_{H^s})\Big(\bIP{\Q_{2,n}^2\wp_n(r'), \wp_n(r')}_{H^s}+\bIP{\Q_{2,n}\wp_n(r'), \Q_{2,n}\wp_n(r')}_{H^s} \Big)\d r' \d r \notag\\
			\leq \ &  2C_{2,2}l^2\int_0^1\int_{0}^{r} V'(\|\wp_n(r)\|^2_{H^s})\|\wp_n(r')\|^2_{H^s} \d r' \d r \notag\\
			\leq \ & 2V(\norm{f}_{H^s}^2)C_{2,2}l^2\int_0^1\int_0^{r}  {\rm e}^{2C_{2,1}|l|r'}\d r'\d r,
		\end{align*}
		which implies \eqref{wp-n estimate 2}. The proof is completed.
	\end{proof}

	\begin{Remark}\label{Remark: wp-n estimates} 
		Recall that $V_s^{\wp_n,f}(r,l)$ is defined in \eqref{V wp-n f}.
		We provide the following remarks regarding Lemma \ref{Lemma-Marcus-n}:
		\begin{enumerate}[label={\bf   (\arabic*)},leftmargin=0.79cm]\setlength\itemsep{0.2em}

			\item  If $|l|\le 1$ (see \ref{Hypo-W L}), then we can infer from \eqref{wp-n estimate 1}, \eqref{wp-n estimate 2} and the fact $\int_{|l|\le 1} |l|^2\,\nu({\rm d}l)<\infty$ that for some constants $\mathscr{C}_{2,i}=\mathscr{C}_{2,i}(s,d)>0$ ($i=1,2$),
			\begin{align}
				\sup_{n\ge1}\int_{|l|\le 1}\Abs{V_s^{\wp_n,f}(r,l)}^2\,\nu({\rm d}l)
				\leq \mathscr{C}_{2,1}V^2(\norm{f}_{H^s}^2),\quad r\in[0,1],\label{wp-n estimate 1 C21}
			\end{align}
			and
			\begin{align}
				\sup_{n\ge1}\int_{|l|\le 1}
				\left\{V_s^{\wp_n,f}(1,l)-2l\cdot V'(\|f\|^2_{H^s} )\bIP{\Q_{2,n}f,f}_{H^s}\right\}\,\nu({\rm d}l) \leq \mathscr{C}_{2,2} V(\|f\|^2_{H^s}).\label{wp-n estimate 2 C22}
			\end{align}
			
			\item From \eqref{Ito-Marcus-term 1}, \eqref{Ito-Marcus-term 2}, \eqref{Ito-Marcus-term 2 gn} and \ref{Q-n cancel 0} in Lemma \ref{Lemma:Qn}, we see that if $\Q_2\in\mathbb{B}^\beta$ satisfies $\Q_2+\Q_2^*=0$, then for $r\in[0,1]$ and $n\ge1$,
			\begin{align*}
				\Abs{V_s^{\wp_n,f}(r,l)}
				=  2|l| \int_0^{r} V'(\|\wp_n(r')\|^2_{H^s})\Abs{\bIP{\wp_n(r'),\Q_{2,n}\wp_n(r')}_{H^s}} \d r'=0,
			\end{align*}
			and
			\begin{align*}
				\Abs{V_s^{\wp_n,f}(r,l)-2l\cdot V'(\|f\|^2_{H^s})\bIP{\Q_{2,n} f ,f}_{H^s}} 
				=  \Abs{V_s^{\wp_n,f}(r,l) -0}
				= 0.
			\end{align*}
			
			\item Observe that the function $V(x)=x$ belongs to the class $\mathscr{V}$. For this choice,  \eqref{wp-n estimate 1} simplifies to
			\begin{align*}
				\sup_{n\ge1}\left|\|\wp_n(r,l, f)\|_{H^s}^2 - \norm{f}_{H^s}^2 \right|
				\leq \left({\rm e}^{2C_{2,1}|l|r}-1\right) \norm{f}_{H^s}^2,\quad r\in[0,1],
			\end{align*}
			and the proof of  \eqref{wp-n estimate 2} actually yields the absolute-value bound
			\begin{align*}
				\sup_{n\ge1}\Big|\|\wp_n\big(1,l,f\big)\|^2_{H^s}-\|f\|^2_{H^s} -2l\cdot\bIP{\Q_{2,n}f,f}_{H^s}\Big|\leq \frac{C_{2,2}}{2C_{2,1}^2}\left({\rm e}^{2C_{2,1}|l|}-2C_{2,1}|l|-1\right)\|f\|^2_{H^s},
			\end{align*}
			respectively.
			In particular, if $\Q_2\in\mathbb{B}^\beta$ and satisfies $\Q_2+\Q_2^*=0$, then 
			\begin{align*}
				\Abs{\|\wp_n\big(1,l,f\big)\|^2_{H^s}-\|f\|^2_{H^s} -2l\cdot\bIP{\Q_{2,n}f,f}_{H^s}}=
				\Abs{\|\wp_n\big(1,l,f\big)\|^2_{H^s}-\|f\|^2_{H^s} -0}
				=0.
			\end{align*}
		\end{enumerate}
	\end{Remark}

	In the next lemma we establish the stabilities of the regularized Marcus flow $\wp_n\big(r,l,f\big)$ with respect to the initial data $f$ and different mollifier-layers $n,m$, respectively.
	\begin{Lemma}[Stability of  regularized Marcus flow,  Part 1]\label{Lemma-Stability wp-n}
		Let $\theta\ge0$ and $C_{2,1}>0$ be given in \eqref{Qn cancel-1}. The Marcus flow 
		$\wp_n$ given by \eqref{regular Marcus flow Q2} satisfies the following stabilities estimates:
		\begin{itemize}[leftmargin=0.79cm]\setlength\itemsep{0.2em}
			\item $($Stability with respect to initial data$)$  For all $f,g\in H_d^{\theta}$,  we have
			\begin{align}\label{Marcus flow stability 1}
				\sup_{n\ge1}\|\wp_n\big(r,l,f\big)-\wp_n\big(r,l,g\big)\|^2_{H^{\theta}}\leq \norm{f-g}_{H^{\theta}}^2\cdot \mathrm{e}^{2C_{2,1}|l|r},
				\quad r\in[0,1].
			\end{align} 
			\item $($Stability respect to different mollifier-layers$)$ 
			Let $s>\theta+\delta_2$ with $\delta_2=\delta_{\Q_2}$ given in \eqref{delta-Qi}.
			For all $f\in H_d^s$, we have
			\begin{align}\label{Marcus flow stability 2}
				\|\wp_n\big(r,l,f\big)-\wp_m\big(r,l,f\big)\|^2_{H^{\theta}}\lesssim (n \land m)^{-2(s-\theta-\delta_2)}\|f\|^2_{H^s} \left({\rm e}^{5C_{2,1}|l|r}-{\rm e}^{3C_{2,1}|l|r}\right),
				\quad n,m\ge1,\quad   r\in[0,1].
			\end{align}
		\end{itemize}
	\end{Lemma}
	
	\begin{proof}
		We note that \eqref{Marcus flow stability 1} directly comes from  \eqref{Phi norm-2} and the linearity of $\wp_n(r,l,\cdot)$. Now we prove \eqref{Marcus flow stability 2}. To this end, we let $\wp_n(r)\triangleq  \wp_n(r,l,f)$ and then  observe that $\wp_n(0)=\wp_m(0)$, which implies
		\begin{align*}
			&\|\wp_n(r)-\wp_m(r)\|^2_{H^{\theta}} \notag\\
			= \ &  \left|  \int_0^{r} 
			\left[\frac{{\rm d}}{{\rm d}r} 
			\|\wp_n-\wp_m\|^2_{H^{\theta}}\right](r')\d r'
			\right|\\
			\leq \ &  2|l|  \int_0^{r}
			\left|
			\BIP{\Q_{2,n}\wp_n(r')-\Q_{2,m}\wp_m(r'),\wp_n(r')-\wp_m(r')}_{H^{\theta}}
			\right|\d r'\\
			\leq \ &  2|l|  \int_0^{r}
			\left|
			\BIP{\Q_{2,n}\wp_n(r')-\Q_{2,n}\wp_m(r'),\wp_n(r')-\wp_m(r')}_{H^{\theta}}
			\right|\d r'\\
			&+2|l|  \int_0^{r}
			\left| \BIP{\Q_{2,n}\wp_m(r')-\Q_{2,m}\wp_m(r'),\wp_n(r')-\wp_m(r')}_{H^{\theta}}
			\right|\d r'. 
		\end{align*}
		According to \eqref{Qn cancel-1}, \eqref{Phi norm-2}, and Young's product inequality, it follows that
		\begin{align*}
			&\|\wp_n(r)-\wp_m(r)\|^2_{H^{\theta}}\notag\\
			\leq \ &  2C_{2,1} |l|  \int_0^{r}\|\wp_n(r')-\wp_m(r')\|^2_{H^{\theta}}\d r'\\
			& +2|l|  \int_0^{r}
			\|\Q_{2,n}-\Q_{2,m}\|_{\LL(H^s;H^\theta)}\|\wp_m(r')\|_{H^s}
			\|\wp_n(r')-\wp_m(r')\|_{H^{\theta}}\d r'\\
			\leq \ &  2C_{2,1}|l|  \int_0^{r}\|\wp_n(r')-\wp_m(r')\|^2_{H^{\theta}}\d r'\\
			& +|l|  \int_0^{r}
			\left\{\frac{1}{C_{2,1}}
			\|\Q_{2,n}-\Q_{2,m}\|^2_{\LL(H^s;H^\theta)}\|\wp_m(r')\|^2_{H^s}
			+C_{2,1}\|\wp_n(r')-\wp_m(r')\|^2_{H^{\theta}}\right\}\d r'\\
			\leq \ &  3C_{2,1}|l|  \int_0^{r}\|\wp_n(r')-\wp_m(r')\|^2_{H^{\theta}}\d r' +\|\Q_{2,n}-\Q_{2,m}\|^2_{\LL(H^s;H^\theta)}\|f\|^2_{H^s} \frac{{\rm e}^{2C_{2,1}|l|r}-1}{2C_{2,1}^2}.
		\end{align*}
		By Gr\"{o}nwall's inequality,  
		\begin{align*}
			\|\wp_n(r)-\wp_m(r)\|^2_{H^{\theta}}
			\leq   \|\Q_{2,n}-\Q_{2,m}\|^2_{\LL(H^s;H^\theta)}\|f\|^2_{H^s} \frac{{\rm e}^{2C_{2,1}|l|r}-1}{2C_{2,1}^2}{\rm e}^{3C_{2,1}|l|r}.
		\end{align*}
		If $s>\theta+\delta_2$, then  \eqref{Marcus flow stability 2} comes from the above estimate together with Lemma \ref{Lemma:Qn}. This completes the proof.
	\end{proof}

	\begin{Remark} \label{Remark: wp estimate}
		Let $\theta\ge0$ and assume that \ref{Hypo-Qi} holds. Consider the Marcus flow $\wp$ defined by \eqref{Marcus flow Q2} with $l\in\R$, i.e.,
		\begin{equation*} 
			\wp(r)\triangleq\wp(r,l,f),\quad 
			\frac{{\rm d}}{{\rm d}r}\wp(r)= l\cdot \Q_2\wp(r),\quad r\in[0,1],  \quad \wp(0) = f.
		\end{equation*}
		If $f\in H_d^s$ with $s>\theta+\delta_2$, then Lemma \ref{Lemma-Marcus-n} and \eqref{Marcus flow stability 2} imply that \eqref{Marcus flow Q2} admits a solution $\wp(r)\in H_d^s$, and
		\begin{equation*}
			V(\|\wp(r)\|_{H^s}^2) \leq V(\norm{f}_{H^s}^2)\cdot \mathrm{e}^{2C_{2,1}|l|r},
			\  r\in[0,1],\quad \text{and}\quad \wp=\lim_{n\to\infty}\wp_n\  \text{in}\ H^{\theta},\ \text{where}\  \wp_n \ \text{is given by}\ \eqref{regular Marcus flow Q2}.
		\end{equation*}
		Then, from Lemma \ref{Lemma:Qn} and the proof of  Lemma \ref{Lemma-Marcus-n}, we obtain  
		\begin{align*}
			\Abs{V(\|\wp(r,l, f)\|_{H^\theta}^2) - V(\norm{f}_{H^\theta}^2)}
			\leq \left({\rm e}^{2C_{2,1}|l|r}-1\right) V(\norm{f}_{H^\theta}^2),\quad r\in[0,1],\quad f\in H_d^{s}.
		\end{align*}
		Moreover, for $f\in H_d^{s}$ with $s>\theta+2\delta_2$, we have
		\begin{align*}
			V(\|\wp(1,l, f)\|_{H^\theta}^2) - V(\norm{f}_{H^\theta}^2)-2l\cdot V'(\|f\|^2_{H^\theta} )\bIP{\Q_{2}f,f}_{H^\theta}\leq \frac{C_{2,2}}{2C_{2,1}^2}\left({\rm e}^{2C_{2,1}|l|}-2C_{2,1}|l|-1\right)V(\|f\|^2_{H^\theta}),
		\end{align*}
		and, in particular, when $V(x)=x$,
		\begin{align*}
			\Abs{\|\wp(1,l, f)\|_{H^\theta}^2-\norm{f}_{H^\theta}^2-2l\cdot \bIP{\Q_{2}f,f}_{H^\theta}}\leq \frac{C_{2,2}}{2C_{2,1}^2}\left({\rm e}^{2C_{2,1}|l|}-2C_{2,1}|l|-1\right)V(\|f\|^2_{H^\theta}).
		\end{align*}
		Therefore, under \ref{Hypo-W L},  we have that for $f\in H_d^s$ with $s>\theta+\delta_2$,
		\begin{align}
			\int_{|l|\le 1}\Abs{V(\|\wp(r,l, f)\|_{H^\theta}^2) - V(\norm{f}_{H^\theta}^2)}^2\nu({\rm d}l)
			\leq \mathscr{C}_{2,1}V^2(\norm{f}_{H^\theta}^2),\quad r\in[0,1],\label{wp estimate 1 C21}
		\end{align}
		and for $f\in H_d^s$ with $s>\theta+2\delta_2$,
		\begin{align*}
			\int_{|l|\le 1}
			\textbf{G} \left(V(\|\wp(1,l, f)\|_{H^\theta}^2) - V(\norm{f}_{H^\theta}^2)-2l\cdot V'(\|f\|^2_{H^\theta} )\bIP{\Q_{2}f,f}_{H^\theta}\right)\nu({\rm d}l) \leq \mathscr{C}_{2,2} V(\|f\|^2_{H^\theta}), 
		\end{align*}
		where 
		\begin{equation*}
			\textbf{G}(x)=
			\begin{cases}
				\Abs{x}, & \text{if } V(x)=x,\\
				x, & \text{otherwise}.
			\end{cases}
		\end{equation*}
		Moreover, if $\Q_2\in\mathbb{B}^\beta$ and satisfies $\Q_2+\Q_2^*=0$, then Theorem \ref{Thm-cancel} yields
		$$
		\bIP{g,\Q_{2}g}_{H^\theta}=0,\quad g\in H^{\theta+\delta_2}.
		$$
		Hence, for $r\in[0,1]$ and $f\in H_d^s$ with $s>\theta+\delta_2$,
		\begin{align*}
			\Abs{V(\|\wp(r,l, f)\|_{H^\theta}^2) - V(\norm{f}_{H^\theta}^2)}
			=  2|l| \int_0^{r} V'(\|\wp(r')\|^2_{H^\theta})\Abs{\bIP{\wp(r'),\Q_{2}\wp(r')}_{H^\theta}} \d r'=0,
		\end{align*}
		and consequently,
		\begin{align*}
			\Abs{V(\|\wp(r,l, f)\|_{H^\theta}^2) - V(\norm{f}_{H^\theta}^2)-2l\cdot V'(\|f\|^2_{H^\theta})\bIP{\Q_{2} f ,f}_{H^\theta}} 
			=  \Abs{V(\|\wp(r,l, f)\|_{H^\theta}^2) - V(\norm{f}_{H^\theta}^2)-0}
			= 0.
		\end{align*}
	\end{Remark}
	
	In addition to the estimates in Lemma \ref{Lemma-Stability wp-n}, the following lemma provides useful estimates for applying It\^o's formula to the difference between two regularized Marcus integrals (see \eqref{Znm-Ito}).
	\begin{Lemma}[Stability of regularized Marcus flow,  Part 2]\label{Lemma-Q2 n m}
		Suppose that $0\le \theta<s-3\delta_2$, where $\delta_2=\delta_{\Q_2}$ is given in \eqref{delta-Qi}.  Let \ref{Hypo-W L} hold and let $\wp_n$ be the Marcus flow 
		given by \eqref{regular Marcus flow Q2}. Then there is a function $\lambda_{\cdot,\cdot}:\N\times\N\to(0,1)$ with $\displaystyle \lim_{n,m\to\infty} 
		\lambda_{n,m}=0$ such that  for all $f,g\in H_d^s$ and $r\in[0,1]$,
		\begin{align}\label{M-flow difference Ito 1}
			\left|\|\wp_n\big(r,l,f\big)-\wp_m\big(r,l,g\big)\|^2_{H^{\theta}}-\|f-g\|^2_{H^{\theta}}\right| \lesssim |l|\Big(\lambda_{n,m}(\norm{f}_{H^{s}}+\norm{g}_{H^{s}})^2+\|f-g\|^2_{H^{\theta}}\Big),
		\end{align} 
		\begin{align}\label{M-flow difference Ito 2}
			\Big|\|\wp_n\big(r,l,f\big)-\wp_m\big(r,l,g\big)\|^2_{H^{\theta}} - \|f-g\|^2_{H^{\theta}}  &-  2l\cdot\bIP{\Q_{2,n} f-\Q_{2,m} g ,f-g}_{H^\theta}\Big|\notag\\
			& \lesssim  |l|^2\Big(\lambda_{n,m}(\norm{f}_{H^{s}}+\norm{g}_{H^{s}})^2+\|f-g\|^2_{H^{\theta}}\Big).
		\end{align}
	\end{Lemma}
	\begin{proof}
		Before we prove the estimates, we consider
		$$s_1\in(\theta+\delta_2, s],\quad q_1\in(\theta+\delta_2, s-2\delta_2),\quad
		s_2\in(\theta+2\delta_2, s],\quad q_2\in(\theta+2\delta_2, s-\delta_2).$$
		It is clear that $q_1+\delta_2<s_2$ and $q_2+\delta_2<s$. 
		
		We first prove \eqref{M-flow difference Ito 1}.  Remember that
		$\wp_n\big(0,l,f\big)=f$, $n\ge1$.
		Therefore, we find
		\begin{align*}
			&\left|\|\wp_n\big(r,l,f\big)-\wp_m\big(r,l,g\big)\|^2_{H^{\theta}}-\|f-g\|^2_{H^{\theta}}\right|\notag\\
			= \ &  \left|  \int_0^{r} 
			\left[\frac{{\rm d}}{{\rm d}r} 
			\|\wp_n\big(r,l,f\big)-\wp_m\big(r,l,g\big)\|^2_{H^{\theta}}\right](r')\d r\right|\notag\\
			= \ &  \left|  \int_0^{r} 
			2l\cdot  
			\bIP{\Q_{2,n}\wp_n\big(r',l,f\big)-\Q_{2,m}\wp_m\big(r',l,g\big),\wp_n\big(r',l,f\big)-\wp_m\big(r',l,g\big)}_{H^{\theta}} \d r'
			\right|,\quad r\in[0,1].
		\end{align*}
		Therefore, we arrive at
		\begin{align}
			\left|\|\wp_n\big(r,l,f\big)-\wp_m\big(r,l,g\big)\|^2_{H^{\theta}}-\|f-g\|^2_{H^{\theta}}\right|
			\leq    2|l|\int_0^{r} \sum_{k=1}^{6}|\mathfrak{B}_k(r')|\d r',\quad r\in[0,1],\label{M-flow difference Ito 1-6}
		\end{align}
		where 
		\begin{align*}
			\mathfrak{B}_1(r)\triangleq \ &  \bIP{\Q_{2,n}\wp_n\big(r,l,f\big)-\Q_{2,m}\wp_n\big(r,l,f\big),\wp_n\big(r,l,f\big)-\wp_m\big(r,l,f\big)}_{H^{\theta}},\\
			\mathfrak{B}_2(r)\triangleq \ &  \bIP{\Q_{2,m}\wp_n\big(r,l,f\big)-\Q_{2,m}\wp_m\big(r,l,f\big),\wp_n\big(r,l,f\big)-\wp_m\big(r,l,f\big)}_{H^{\theta}},\\
			\mathfrak{B}_3(r)\triangleq \ & 
			\bIP{\Q_{2,m}\wp_m\big(r,l,f\big)-\Q_{2,m}\wp_m\big(r,l,g\big),\wp_n\big(r,l,f\big)-\wp_m\big(r,l,f\big)}_{H^{\theta}},\\
			\mathfrak{B}_4(r)\triangleq \ &  \bIP{\Q_{2,n}\wp_n\big(r,l,f\big)-\Q_{2,m}\wp_n\big(r,l,f\big),\wp_m\big(r,l,f\big)-\wp_m\big(r,l,g\big)}_{H^{\theta}},\\
			\mathfrak{B}_5(r)\triangleq \ & \bIP{\Q_{2,m}\wp_n\big(r,l,f\big)-\Q_{2,m}\wp_m\big(r,l,f\big),\wp_m\big(r,l,f\big)-\wp_m\big(r,l,g\big)}_{H^{\theta}},\\
			\mathfrak{B}_6(r)\triangleq\ & \bIP{\Q_{2,m}\wp_m\big(r,l,f\big)-\Q_{2,m}\wp_m\big(r,l,g\big),\wp_m\big(r,l,f\big)-\wp_m\big(r,l,g\big)}_{H^{\theta}}.
		\end{align*}
		It follows from \eqref{Qn cancel-1} and Lemma \ref{Lemma-Stability wp-n} that for all $r\in[0,1]$,
		\begin{align*}
			|\mathfrak{B}_2(r)|\lesssim \|\wp_n\big(r,l,f\big)-\wp_m\big(r,l,f\big)\|^2_{H^{\theta}}\lesssim  (n \land m)^{-2(s-\theta-\delta_2)}\|f\|^2_{H^s} \left({\rm e}^{5C_{2,1}|l|r}-{\rm e}^{3C_{2,1}|l|r}\right),
		\end{align*}
		\begin{align*}
			|\mathfrak{B}_{6}(r)|\lesssim \|\wp_m\big(r,l,f\big)-\wp_m\big(r,l,g\big)\|^2_{H^{\theta}}\lesssim {\rm e}^{2C_{2,1}|l|r}\|f-g\|^2_{H^{\theta}}.
		\end{align*}
		Keep in mind that $\wp_n(r,l,\cdot)$ is linear. 
		Using the Cauchy product inequality, \eqref{Qn-Qm operator norm}, \eqref{Qn2-Qm2 operator norm} and Lemma \ref{Lemma-Stability wp-n}, it follows that for all $r\in[0,1]$,
		\begin{align*}
			|\mathfrak{B}_1(r)|
			\lesssim  \ &(n \land m)^{-(s_1-\theta-\delta_2)}\|\wp_n\big(r,l,f\big)\|_{H^{s_1}}\cdot\|\wp_n\big(r,l,f\big)-\wp_m\big(r,l,f\big)\|_{H^{\theta}}\\
			\lesssim \ &  (n \land m)^{-(s_1-\theta-\delta_2)}\norm{f}_{H^{s_1}}\mathrm{e}^{C_{2,1}|l|r}\cdot (n \land m)^{-(s_1-\theta-\delta_2)}\|f\|_{H^{s_1}} \left({\rm e}^{5C_{2,1}|l|r}-{\rm e}^{3C_{2,1}|l|r}\right)^{1/2}\\
			\lesssim \ &  (n \land m)^{-2(s_1-\theta-\delta_2)}\norm{f}^2_{H^{s_1}}\left({\rm e}^{7C_{2,1}|l|r}-{\rm e}^{5C_{2,1}|l|r}\right)^{1/2},
		\end{align*}
		\begin{align*}
			|\mathfrak{B}_3(r)|
			\lesssim  \ &\|\wp_m\big(r,l,f\big)-\wp_m\big(r,l,g\big)\|_{H^{q_1}}  \cdot\|\wp_n\big(r,l,f\big)-\wp_m\big(r,l,f\big)\|_{H^{\theta}}\\
			\lesssim \ &  \norm{f-g}_{H^{q_1}}  \mathrm{e}^{C_{2,1}|l|r}\cdot (n \land m)^{-(s_1-\theta-\delta_2)}\|f\|_{H^{s_1}} \left({\rm e}^{5C_{2,1}|l|r}-{\rm e}^{3C_{2,1}|l|r}\right)^{1/2}\\
			\lesssim \ &  (n \land m)^{-2(s_1-\theta-\delta_2)}\norm{f}_{H^{q_1}}\norm{f}_{H^{s_1}}\left({\rm e}^{7C_{2,1}|l|r}-{\rm e}^{5C_{2,1}|l|r}\right)^{1/2},
		\end{align*}
		\begin{align*}
			|\mathfrak{B}_{4}(r)| 
			\lesssim\ & (n \land m)^{-2(s_1-\theta-\delta_2)}\|\wp_n\big(r,l,f\big)\|^2_{H^{s_1}}
			+\|\wp_m\big(r,l,f\big)-\wp_m\big(r,l,g\big)\|^2_{H^{\theta}}\\
			\lesssim \ &   (n \land m)^{-2(s_1-\theta-\delta_2)}\norm{f}^2_{H^{s_1}}\mathrm{e}^{2C_{2,1}|l|r}
			+\norm{f-g}^2_{H^{\theta}}  \mathrm{e}^{2C_{2,1}|l|r},
		\end{align*}
		\begin{align*}
			|\mathfrak{B}_5(r)|
			\lesssim  \ &\|\wp_n\big(r,l,f\big)-\wp_m\big(r,l,f\big)\|^2_{H^{q_1}}  +\|\wp_m\big(r,l,f\big)-\wp_m\big(r,l,g\big)\|^2_{H^{\theta}}\\
			\lesssim \ &   (n \land m)^{-(s_2-q_1-\delta_2)}\norm{f}_{H^{s_2}}\left({\rm e}^{5C_{2,1}|l|r}-{\rm e}^{3C_{2,1}|l|r}\right)
			+\norm{f-g}^2_{H^{\theta}}  \mathrm{e}^{2C_{2,1}|l|r}.
		\end{align*}
		Combining the above estimates and \eqref{M-flow difference Ito 1-6}, we obtain \eqref{M-flow difference Ito 1}.
		
		We now turn to the proof of  \eqref{M-flow difference Ito 2}. In the same manner, we obtain
		\begin{align*}
			&\left|\|\wp_n\big(r,l,f\big)-\wp_m\big(r,l,g\big)\|^2_{H^{\theta}}-\|f-g\|^2_{H^{\theta}} -2l\bIP{\Q_{2,n} f-\Q_{2,m} g ,f-g}_{H^\theta}\right|\notag\\
			= \ &  \left|  \int_0^{r} 
			\left[\frac{{\rm d}}{{\rm d}r} 
			\|\wp_n\big(r,l,f\big)-\wp_m\big(r,l,g\big)\|^2_{H^{\theta}}\right](r')\d r'
			-2l\bIP{\Q_{2,n} f-\Q_{2,m} g ,f-g}_{H^{\theta}}\right|\\
			= \ &  \left|  \int_0^{r} 
			2l\cdot \Big\{
			\bIP{\Q_{2,n}\wp_n\big(r',l,f\big)-\Q_{2,m}\wp_m\big(r',l,g\big),\wp_n\big(r',l,f\big)-\wp_m\big(r',l,g\big)}_{H^{\theta}}
			-\bIP{\Q_{2,n} f-\Q_{2,m} g ,f-g}_{H^{\theta}}\Big\}\d r'
			\right|\\
			\leq \ & 2|l|\int_0^{r} \int_{0}^{r'} \left|\left[\frac{{\rm d}}{{\rm d}r}  
			\BIP{\Q_{2,n}\wp_n\big(r,l,f\big)-\Q_{2,m}\wp_m\big(r,l,g\big),\wp_n\big(r,l,f\big)-\wp_m\big(r,l,g\big)}_{H^{\theta}}\right](r'')\right| \d r'' \d r'.
		\end{align*}
		Similar to the proof of  \eqref{Ito-Marcus-term 2 gn}, we have
		\begin{align*}
			&\frac{{\rm d}}{{\rm d}r} 
			\bIP{\Q_{2,n}\wp_n\big(r,l,f\big)-\Q_{2,m}\wp_m\big(r,l,g\big),\wp_n\big(r,l,f\big)-\wp_m\big(r,l,g\big)}_{H^{\theta}}\\
			= \ & l\cdot \IP{
				\Big[\nabla \Q_{2,n}\big(\wp_n\big(r,l,f\big)\big)\Big]
				\Q_{2,n}\wp_n\big(r,l,f\big)-\Big[\nabla \Q_{2,n}\big(\wp_n\big(r,l,g\big)\big)\Big]
				\Q_{2,m}\wp_m\big(r,l,g\big),\wp_n\big(r,l,f\big)-\wp_m\big(r,l,g\big)}_{H^{\theta}}\\
			&+ l\cdot \BIP{\Q_{2,n}\wp_n\big(r,l,f\big)-\Q_{2,m}\wp_m\big(r,l,g\big), \Q_{2,n}\wp_n\big(r,l,f\big)-\Q_{2,m}\wp_m\big(r,l,g\big)}_{H^{\theta}}\\
			= \ & l\cdot \BIP{
				\Q_{2,n}^2\wp_n\big(r,l,f\big)-
				\Q_{2,m}^2\wp_m\big(r,l,g\big),\wp_n\big(r,l,f\big)-\wp_m\big(r,l,g\big)}_{H^{\theta}}\\
			&+ l\cdot \BIP{\Q_{2,n}\wp_n\big(r,l,f\big)-\Q_{2,m}\wp_m\big(r,l,g\big), \Q_{2,n}\wp_n\big(r,l,f\big)-\Q_{2,m}\wp_m\big(r,l,g\big)}_{H^{\theta}}.
		\end{align*}
		Therefore, we obtain that for $r\in[0,1]$,
		\begin{align*}
			\Big|\|\wp_n\big(r,l,f\big)-\wp_m\big(r,l,g\big)\|^2_{H^{\theta}}-\|f-g\|^2_{H^{\theta}} -2l\bIP{\Q_{2,n} f-\Q_{2,m} g, f-g}_{H^\theta}\Big| 
			\leq   2|l|^2\int_0^r\int_{0}^{r'} \sum_{k=1}^{12}|\mathfrak{L}_k(r'')|\d r'' \d r',
		\end{align*}
		where the terms involving $\Q_{2,n}^2$ or $\Q_{2,m}^2$ are
		\begin{align*}
			\mathfrak{L}_1(r)\triangleq\ & \BIP{
				\Q_{2,n}^2\wp_n\big(r,l,f\big)-
				\Q_{2,m}^2\wp_n\big(r,l,f\big),\wp_n\big(r,l,f\big)-\wp_m\big(r,l,f\big)}_{H^{\theta}},\\
			\mathfrak{L}_2(r)\triangleq\ & \BIP{
				\Q_{2,m}^2\wp_n\big(r,l,f\big)-
				\Q_{2,m}^2\wp_m\big(r,l,f\big),\wp_n\big(r,l,f\big)-\wp_m\big(r,l,f\big)}_{H^{\theta}},\\
			\mathfrak{L}_3(r)\triangleq\ & \BIP{
				\Q_{2,m}^2\wp_m\big(r,l,f\big)-
				\Q_{2,m}^2\wp_m\big(r,l,g\big),\wp_n\big(r,l,f\big)-\wp_m\big(r,l,f\big)}_{H^{\theta}},\\
			\mathfrak{L}_4(r)\triangleq\ & \BIP{
				\Q_{2,n}^2\wp_n\big(r,l,f\big)-
				\Q_{2,m}^2\wp_n\big(r,l,f\big),\wp_m\big(r,l,f\big)-\wp_m\big(r,l,g\big)}_{H^{\theta}},\\
			\mathfrak{L}_5(r)\triangleq\ &\BIP{
				\Q_{2,m}^2\wp_n\big(r,l,f\big)-
				\Q_{2,m}^2\wp_m\big(r,l,f\big),\wp_m\big(r,l,f\big)-\wp_m\big(r,l,g\big)}_{H^{\theta}},\\
			\mathfrak{L}_6(r)\triangleq\ & \BIP{
				\Q_{2,m}^2\wp_m\big(r,l,f\big)-
				\Q_{2,m}^2\wp_m\big(r,l,g\big),\wp_m\big(r,l,f\big)-\wp_m\big(r,l,g\big)}_{H^{\theta}},
		\end{align*}
		and the remaining terms are
		\begin{align*}
			\mathfrak{L}_{7}(r)\triangleq \ & \BIP{
				\Q_{2,m}\wp_n\big(r,l,f\big)-\Q_{2,m}\wp_m\big(r,l,f\big), \Q_{2,m}\wp_n\big(r,l,f\big)-\Q_{2,m}\wp_m\big(r,l,f\big)}_{H^{\theta}},\\
			\mathfrak{L}_{8}(r)\triangleq \ & \BIP{
				\Q_{2,m}\wp_m\big(r,l,f\big)-\Q_{2,m}\wp_m\big(r,l,g\big), \Q_{2,m}\wp_m\big(r,l,f\big)-\Q_{2,m}\wp_m\big(r,l,g\big)}_{H^{\theta}},\\
			\mathfrak{L}_9(r)\triangleq \ &  \BIP{
				\Q_{2,n}\wp_n\big(r,l,f\big)-\Q_{2,m}\wp_n\big(r,l,f\big), \Q_{2,n}\wp_n\big(r,l,f\big)-\Q_{2,m}\wp_n\big(r,l,f\big)}_{H^{\theta}},\\
			\mathfrak{L}_{10}(r)\triangleq \ & 2\BIP{
				\Q_{2,m}\wp_n\big(r,l,f\big)-\Q_{2,m}\wp_m\big(r,l,f\big), \Q_{2,n}\wp_n\big(r,l,f\big)-\Q_{2,m}\wp_n\big(r,l,f\big)}_{H^{\theta}},\\
			\mathfrak{L}_{11}(r)\triangleq \ &  2\BIP{
				\Q_{2,m}\wp_m\big(r,l,f\big)-\Q_{2,m}\wp_m\big(r,l,g\big), \Q_{2,n}\wp_n\big(r,l,f\big)-\Q_{2,m}\wp_n\big(r,l,f\big)}_{H^{\theta}},\\
			\mathfrak{L}_{12}(r)\triangleq \ & 2\BIP{
				\Q_{2,m}\wp_m\big(r,l,f\big)-\Q_{2,m}\wp_m\big(r,l,g\big), \Q_{2,m}\wp_n\big(r,l,f\big)-\Q_{2,m}\wp_m\big(r,l,f\big)}_{H^{\theta}}.
		\end{align*}
		Thanks to \eqref{Qn cancel-2} and Lemma \ref{Lemma-Stability wp-n}, we obtain that for all $r\in[0,1]$,
		\begin{align*}
			|\mathfrak{L}_2(r)+\mathfrak{L}_{7}(r)|\lesssim \|\wp_n\big(r,l,f\big)-\wp_m\big(r,l,f\big)\|^2_{H^{\theta}}\lesssim  (n \land m)^{-2(s-\theta-\delta_2)}\|f\|^2_{H^s} \left({\rm e}^{5C_{2,1}|l|r}-{\rm e}^{3C_{2,1}|l|r}\right),
		\end{align*}
		\begin{align*}
			|\mathfrak{L}_6(r)+\mathfrak{L}_{8}(r)|\lesssim \|\wp_m\big(r,l,f\big)-\wp_m\big(r,l,g\big)\|^2_{H^{\theta}}\lesssim {\rm e}^{2C_{2,1}|l|r}\|f-g\|^2_{H^{\theta}}.
		\end{align*}
		Using the Cauchy product inequality, \eqref{Qn-Qm operator norm}, \eqref{Qn2-Qm2 operator norm} and Lemma \ref{Lemma-Stability wp-n}, we find that for all $r\in[0,1]$
		\begin{align*}
			|\mathfrak{L}_1(r)|
			\lesssim \ & (n \land m)^{-(s_2-\theta-2\delta)}\|\wp_n\big(r,l,f\big)\|_{H^{s_2}}\cdot\|\wp_n\big(r,l,f\big)-\wp_m\big(r,l,f\big)\|_{H^{\theta}}\\
			\lesssim\ & (n \land m)^{-(s_2-\theta-2\delta)}\norm{f}_{H^{s_2}}\mathrm{e}^{C_{2,1}|l|r}\cdot (n \land m)^{-(s_1-\theta-\delta_2)}\|f\|_{H^{s_1}} \left({\rm e}^{5C_{2,1}|l|r}-{\rm e}^{3C_{2,1}|l|r}\right)^{1/2}\\
			\lesssim\ &  (n \land m)^{-(s_1+s_2-2\theta-3\delta)}\|f\|_{H^{s_1}}\|f\|_{H^{s_2}} \left({\rm e}^{7C_{2,1}|l|r}-{\rm e}^{5C_{2,1}|l|r}\right)^{1/2},
		\end{align*}
		\begin{align*}
			|\mathfrak{L}_3(r)| 
			\lesssim \ & \|\wp_m\big(r,l,f\big)-\wp_m\big(r,l,g\big)\|_{H^{s_2}}\cdot
			\|\wp_n\big(r,l,f\big)-\wp_m\big(r,l,f\big)\|_{H^{\theta}}\\
			\lesssim \ &  \norm{f-g}_{H^{s_2}} \mathrm{e}^{C_{2,1}|l|r}\cdot(n \land m)^{-(s_1-\theta-\delta_2)}\|f\|_{H^{s_1}} \left({\rm e}^{5C_{2,1}|l|r}-{\rm e}^{3C_{2,1}|l|r}\right)^{1/2}\\
			\lesssim \ & (n \land m)^{-(s_1-\theta-\delta_2)}
			\left(\norm{f}_{H^{s_1}}\norm{f}_{H^{s_2}}+\norm{f}_{H^{s_1}}\norm{g}_{H^{s_2}}\right)\left({\rm e}^{7C_{2,1}|l|r}-{\rm e}^{5C_{2,1}|l|r}\right)^{1/2},
		\end{align*}
		\begin{align*}
			|\mathfrak{L}_4(r)|
			\lesssim \ & (n \land m)^{-2(s_2-\theta-2\delta)}\|\wp_n\big(r,l,f\big)\|^2_{H^{s_2}}+\|\wp_m\big(r,l,f\big)-\wp_m\big(r,l,g\big)\|^2_{H^{\theta}}\\
			\lesssim\ & (n \land m)^{-2(s_2-\theta-2\delta)}\norm{f}^2_{H^{s_2}}\mathrm{e}^{2C_{2,1}|l|r}+\|f-g\|^2_{H^{\theta}}\mathrm{e}^{2C_{2,1}|l|r},
		\end{align*}
		\begin{align*}
			|\mathfrak{L}_5(r)|
			\lesssim \ & \|\wp_n\big(r,l,f\big)-\wp_m\big(r,l,f\big)\|_{H^{q_2}}^2+\|\wp_m\big(r,l,f\big)-\wp_m\big(r,l,g\big)\|^2_{H^{\theta}}\\
			\lesssim\ & (n \land m)^{-2(s-q_2-\delta_2)}\norm{f}^2_{H^{s}}\left({\rm e}^{7C_{2,1}|l|r}-{\rm e}^{5C_{2,1}|l|r}\right)+\|f-g\|^2_{H^{\theta}}\mathrm{e}^{2C_{2,1}|l|r}.
		\end{align*}
		\begin{align*}
			|\mathfrak{L}_9(r)|
			\lesssim  (n \land m)^{-2(s_1-\theta-\delta_2)}\|\wp_n\big(r,l,f\big)\|_{H^{s_1}}^2
			\lesssim   (n \land m)^{-2(s_1-\theta-\delta_2)}\norm{f}^2_{H^{s_1}}\mathrm{e}^{2C_{2,1}|l|r},
		\end{align*}
		\begin{align*}
			|\mathfrak{L}_{10}(r)| 
			\lesssim\ & \|\wp_n\big(r,l,f\big)-\wp_m\big(r,l,f\big)\|_{H^{q_1}}\cdot  (n \land m)^{-(s_1-\theta-\delta_2)}\|\wp_n\big(r,l,f\big)\|_{H^{s_1}}\\
			\lesssim \ & (n \land m)^{-(s_2-q_1-\delta_2)}\norm{f}_{H^{s_2}}\left({\rm e}^{5C_{2,1}|l|r}-{\rm e}^{3C_{2,1}|l|r}\right)^{1/2} \cdot (n \land m)^{-(s_1-\theta-\delta_2)}\norm{f}_{H^{s_1}}\mathrm{e}^{C_{2,1}|l|r}\\
			\lesssim \ &  (n \land m)^{-(s_1+s_2-q_1-\theta-\delta_2)}\norm{f}_{H^{s_1}}\norm{f}_{H^{s_2}}\left({\rm e}^{7C_{2,1}|l|r}-{\rm e}^{5C_{2,1}|l|r}\right)^{1/2},
		\end{align*}
		\begin{align*}
			|\mathfrak{L}_{11}(r)| 
			\lesssim\ & \|\wp_m\big(r,l,f\big)-\wp_m\big(r,l,g\big)\|_{H^{q_1}}\cdot  (n \land m)^{-(s_1-\theta-\delta_2)}\|\wp_n\big(r,l,f\big)\|_{H^{s_1}}\\
			\lesssim \ & \norm{f-g}_{H^{q_1}}  \mathrm{e}^{C_{2,1}|l|r} \cdot (n \land m)^{-(s_1-\theta-\delta_2)}\norm{f}_{H^{s_1}}\mathrm{e}^{C_{2,1}|l|r}\\
			\lesssim \ &  (n \land m)^{-(s_1-\theta-\delta_2)}\left(\norm{f}_{H^{q_1}}\norm{f}_{H^{s_1}}+\norm{g}_{H^{q_1}}\norm{f}_{H^{s_1}}\right)\mathrm{e}^{2C_{2,1}|l|r},
		\end{align*}
		\begin{align*}
			|\mathfrak{L}_{12}(r)| 
			\lesssim\ & \|\wp_m\big(r,l,f\big)-\wp_m\big(r,l,g\big)\|_{H^{q_1}}\cdot  \|\wp_n\big(r,l,f\big)-\wp_m\big(r,l,f\big)\|_{H^{q_1}}\\
			\lesssim \ & \norm{f-g}_{H^{q_1}}  \mathrm{e}^{C_{2,1}|l|r} \cdot (n \land m)^{-(s_2-q_1-\delta_2)}\norm{f}_{H^{s_2}}\left({\rm e}^{5C_{2,1}|l|r}-{\rm e}^{3C_{2,1}|l|r}\right)^{1/2}\\
			\lesssim \ &  
			(n \land m)^{-(s_2-q_1-\delta_2)}
			\left(\norm{f}_{H^{q_1}}\norm{f}_{H^{s_2}}+\norm{g}_{H^{q_1}}\norm{f}_{H^{s_2}}\right)\left({\rm e}^{7C_{2,1}|l|r}-{\rm e}^{5C_{2,1}|l|r}\right)^{1/2}.
		\end{align*}
		Collecting the above estimates
		with noting that $|l|\le1$ (cf. \ref{Hypo-W L}) and $r\in[0,1]$, we obtain  \eqref{M-flow difference Ito 2}. 
	\end{proof}

	\section{Stochastic Compressible Case}\label{Section : Compressible case}
	
	\subsection{Assumptions and Target Model}\label{Section : Compressible-Assumptions}
	Based on the two classes $\mathbb{A}^{\alpha}$ and $\mathbb{B}^{\beta}$ (see Definitions \ref{Ak class define} and \ref{Bk class define}, respectively), we introduce the following assumptions, which ensure that the operators $\mathcal Q_1$ and $\mathcal Q_2$ in \eqref{SEuler-(rho u)} and \eqref{SEuler-(u P) damp} are nearly skew-adjoint:
	\begin{Hypothesis} \label{Hypo-Qi}
		Assume that $\mathcal Q_1,\mathcal Q_2\in \mathbb{A}^{\alpha}\cup \mathbb{B}^{\beta}$, with $d=m\ge2$  in \eqref{Ss define}.  
	\end{Hypothesis}
	
	As before,  once  $\Q_1$ and $\Q_2$ are taken from $\mathbb{A}^\alpha \cup \mathbb{B}^\beta$ under \ref{Hypo-Qi},their associated decompositions and exponents are kept fixed, even if $\mathcal Q_i$ ($i=1,2$) admits an alternative admissible decomposition.

	Under \ref{Hypo-W L}, as in \eqref{Marcus flow H} and \eqref{Marcus integral define},
	we can define
	\begin{align*}
		\int_0^t \Q_2 u(t') \diamond {\rm d}L(t')
		\triangleq \ & \int_0^t\int_{|l|\le1} \Big\{\wp\big(1,l,u(t'-)\big)-u(t'-)\Big\}\, \widetilde{\eta}({\rm d}l,{\rm d}t')\\
		&+\int_0^t\int_{|l|\le1} \Big\{\wp\big(1,l,u(t')\big)-u(t')-l\cdot \Q_2 u(t')\Big\}\, \nu({\rm d}l){\rm d}t',\quad t>0,
	\end{align*}
	where the Marcus flow 
	$\wp$ is the solution to the following equation depending on jump size $l$ and initial data $f$:
	\begin{equation}\label{Marcus flow Q2}
		\wp(r)\triangleq\wp(r,l,f),\quad 
		\frac{{\rm d}}{{\rm d}r}\wp(r)= l\cdot \Q_2\wp(r),\quad r\in[0,1],  \quad \wp(0) = f.
	\end{equation}
	From this,
	\eqref{Stratonovich to Ito}, and \ref{Hypo-Qi}, we can to reformulate \eqref{SEuler-(rho u)}  as
	\begin{equation}\label{Target problem (rho u)}
		\left\{\begin{aligned}
			\rho(t)-\rho_0\ +&\int_0^t\Div(\rho u) \d t' =0,\quad \rho(t)>0,\\
			u(t)-u_0 \ 
			+&\int_0^{t}  \bigg[(u\cdot\nabla)u+\frac{\nabla P(\rho)}{\rho}-\frac{1}{2} \mathcal{Q}_1^2u\bigg](t') \d t'\\
			= & \int_0^{t} \left(\Q_{1}u(t') \d W(t')+ h\big(t',\rho(t'),u(t')\big)\d \widetilde W(t')\right )\\
			&+\int_0^t\int_{|l|\le1} \Big\{\wp\big(1,l,u(t'-)\big)-u(t'-)\Big\}\, \widetilde{\eta}({\rm d}l,{\rm d}t')\\
			&+\int_0^t\int_{|l|\le1} \Big\{\wp\big(1,l,u(t')\big)-u(t')-l\cdot \Q_2 u(t')\Big\}\, \nu({\rm d}l){\rm d}t',\\
			\rho(0)=\rho_0,\ \ & 
			u(0)=u_0.
		\end{aligned}
		\right.
	\end{equation}

	\begin{Remark}\label{Remark singualrity}
		Recall the state space $\H^s=H_1^s\times H_d^s$  (cf. \eqref{Hs Lp Wp m} and \eqref{Hs Lp Wp pair}).
		We emphasize that \eqref{Target problem (rho u)} defines a singular evolution system in $\H^s$. To understand this singularity, observe that:
		
		\begin{enumerate}[label={\bf (\alph*)},leftmargin=0.79cm]\setlength\itemsep{0.2em}
			\item The terms $(\Div (\rho u), (u\cdot\nn)u)$ do not preserve $\H^s$ regularity: if $(\rho,u)\in \H^s$, then $(\Div (\rho u), (u\cdot\nn)u)$ only belong to $\H^{s-1}$, not gaining regularity. Similarly, depending on the pressure, $\frac{\nabla P(\rho)}{\rho}$ may not preserve $H_d^s$ regularity either.

			\item Recall that $\delta_i=\delta_{\Q_i}$ is given in \eqref{delta-Qi}. 
			If $\alpha,\beta>0$, then, for \eqref{Target problem (rho u)}$_2$, the regularity condition $u\in H_d^s$ is insufficient to ensure that the corresponding terms lie in $H_d^s$:
			\[
			\Q_1 u(t')\in H_d^{s-\delta_1},\qquad \Q_1^2 u(t')\in H_d^{s-2\delta_1},
			\]
			and therefore, in general,
			\[
			\int_0^t -\tfrac{1}{2}\,\mathcal{Q}_1^2 u(t')\,\mathrm{d}t' \notin H_d^{s},\qquad
			\int_0^{t} \Q_{1}u(t')\, \mathrm{d}W(t') \notin H_d^{s}.
			\]
			Similarly, \eqref{Marcus flow Q2} exhibits singular behavior in $H_d^s$: although the initial datum lies in $H_d^s$, the coefficient $\Q_2$ maps $H_d^s$ to $H_d^{s-\delta_2}$, which implies that the associated Marcus flow loses regularity at the level of the noise amplitude.
		\end{enumerate}
		
		These singularities present an essential challenge: knowing merely that $(\rho,u)\in\H^s$, one cannot directly apply It\^o's formula to $\|(\rho, u)\|^2_{\H^s}$ because the $\H^s$ inner products involved are not well-defined. 
		Therefore, we must carefully design a regularization (or renormalization) procedure so that key properties—such as the cancellation estimates in Theorem \ref{Thm-cancel}—hold uniformly with respect to the regularization parameter, before passing to the limit. For details on this regularization (or renormalization)  approach, we refer to  Remarks  \ref{Remark:renormalization Qn} and \ref{Remark singualrity IC} , Lemmas \ref{Lemma:Qn}, \ref{Lemma-Marcus-n}, and \ref{Lemma : Xn T estimates}, and Section \ref{Section : SEuler-global-proof}.
	\end{Remark}

	We now turn to the general pressure law, beginning with the deterministic case to clarify the main idea. 
	To address the challenge of solving the compressible Euler system \eqref{Euler (rho u)} for a broad class of equations of state within a unified framework, we must overcome the degeneracy that classical symmetrization techniques suffer at the vacuum boundary.

	To achieve this, we introduce a function $\rr(\cdot)$ and examine the evolution of the quantity
	$$\vrho=\vrho(t,x)=[\rr(\rho)](t,x),\ \ (t,x)\in[0,\infty)\times\mathbb{K}^d,$$
	rather than the density $\rho(t,x)$ itself. This formulation is expected to facilitate the analysis of more general pressure functions $P(\rho)$.

	To make this rigorous, let $f^{-1}$ denote the inverse of an invertible function $f$, and define the admissible class of transformation functions as
	\begin{equation}\label{Set of transform}
		\left\{\begin{aligned}
			&{\bf R}_{(\rr_0,\rr_\infty)} \triangleq \Big\{
			\rr\in C^\infty\big((0,\infty);(\rr_0,\rr_\infty)\big)\,:\, \rr'>0,\ \ \rr^{-1}\in C^\infty\big((\rr_0,\rr_\infty);(0,\infty)\big)
			\Big\},\\
			&\rr_0 \triangleq \lim_{x\to0}\rr(x)\ge-\infty,\ \ \rr_\infty \triangleq \lim_{x\to\infty}\rr(x)\le\infty.
		\end{aligned}
		\right.
	\end{equation}
	Since $\rr(\cdot)$ and $\rr^{-1}(\cdot)$ are assumed to be smooth, the new variable $\vrho$ inherits the regularity of $\rho$. 
	
	In this paper, we impose the following structural and analytic assumptions on the pressure function $P(\rho)$ and the associated transformation $\rr(\rho)$.
	
	\begin{Hypothesis}\label{Hypo-Pressure}
		There exists a transformation $\rr\in {\bf R}_{(\rr_0,\rr_\infty)}$ such that the following conditions hold true: 
		\begin{enumerate}[label={\bf (A$_{\bf (5.2)}^{\arabic*}$)},leftmargin=*]
			\setlength\itemsep{0.2em}
			\item\label{Structural Identity P' r'} \textbf{Structural Identity:} The pressure $P: [0,\infty) \to \mathbb{R}$ satisfies
			\begin{equation*}
				P'(\rho)=\big(\rho\, \rr'(\rho)\big)^2 \quad \text{for} \quad \rho>0.
			\end{equation*}
			
			\item\label{Regularity-Sound} \textbf{Regularity of the  Sound Speed:} Let $\Theta(y) \triangleq \rr'\big(\rr^{-1}(y)\big)\rr^{-1}(y)$ for $y\in(\rr_0,\rr_\infty)$. The function $\Theta(\cdot)$ satisfies one of the following two conditions: 
			
			\begin{enumerate}[label={\bf (A$_{\bf (5.2)}^{2.\arabic*}$)},leftmargin=*]
				\setlength\itemsep{0.2em}
				\item\label{Constant Sound} \textbf{Constant case:} $\Theta(\cdot)\equiv \sqrt{P(1)}$ on $(\rr_0,\rr_\infty)=(-\infty,\infty)$.
				\item\label{General Sound} \textbf{General case:} $(\rr_0,\rr_\infty)\neq(-\infty,\infty)$, and there exists a $C^\infty$ extension of $\Theta$ from $(\rr_0,\rr_\infty)$ to $\R$, denoted by $\Lambda$, such that
				\begin{equation*}
					\Lambda:\ \R\to\R,\quad \Lambda(0)=0,\quad \Lambda\big|_{(\rr_0,\rr_\infty)}=\Theta,\quad \sup_{y\in\R}|\Lambda'(y)|<\infty.
				\end{equation*}
				Moreover, for all $s>\frac{d}{2}$, there exist a non-decreasing function $\Phi_{\Lambda}:[0,\infty)\to[0,\infty)$ and a constant $C_{s}>0$ such that the composition operator satisfies the Sobolev estimates:
				\begin{equation*} 
					\norm{\Lambda(\vrho)}_{H^s}\leq C_{s} \Phi_{\Lambda}(\|\vrho\|_{L^\infty})\norm{\vrho}_{H^s} \quad \forall \vrho\in H^s,
				\end{equation*}
				and 
				\begin{equation*} 
					\norm{\Lambda(\vrho_1)-\Lambda(\vrho_2)}_{H^s}\leq  C_{s} \Phi_{\Lambda}(\|\vrho_1\|_{H^s}+\|\vrho_2\|_{H^s})\norm{\vrho_1-\vrho_2}_{H^s}\quad 
					\forall \vrho_1,\vrho_2\in H^s.
				\end{equation*}
			\end{enumerate}
		\end{enumerate}
	\end{Hypothesis}

	\begin{Remark}\label{Remark-Hypo-Pressure}
		Some remarks regarding the physical relevance and mathematical implications of \ref{Hypo-Pressure} are given below.

		\begin{enumerate}[label={{\bf (\alph*)}},leftmargin=0.79cm]\setlength\itemsep{0.2em}
			
			\item \textbf{The system for new variable with constraint.}  \ref{Structural Identity P' r'} induces that for  sufficiently regular solutions $(\rho,u)$ to \eqref{Euler (rho u)},  the choice $(\vrho,u)=(\rr(\rho),u)$ satisfies
			\begin{equation}\label{Euler (q u) Theta}
				\left\{
				\begin{aligned}
					&\partial_t \vrho+\Theta(\vrho) \, \Div u+ u\cdot\nabla \vrho=0,\\
					&\partial_t u+(u\cdot\nn)u+\Theta(\vrho)\,\nabla \vrho=0,\\
					& \vrho\in(\rr_0,\rr_\infty),\\
					&\vrho(0, x) = \vrho_0(x)=\rr(\rho_0(x)),\\
					&u(0, x) = u_0(x).
				\end{aligned}
				\right.
			\end{equation}

			Here we note that the constraint $\vrho \in (\rr_0, \rr_\infty)$ for the case $(\rr_0, \rr_\infty) \neq (-\infty, \infty)$ may introduce additional complications in constructing a solution to \eqref{Euler (q u) Theta}. Therefore, when $(\rr_0, \rr_\infty) \neq (-\infty, \infty)$, we separate the equations from the constraint $\vrho \in (\rr_0, \rr_\infty)$. Specifically, assuming \ref{Regularity-Sound}, we can attempt to find a function $\rr \in {\bf R}_{(\rr_0, \rr_\infty)}$ for a certain interval $(\rr_0, \rr_\infty)$ and solve the following unconstrained equations  
			\begin{equation}\label{Euler (q u) Xi}
				\left\{
				\begin{aligned}
					&\partial_t \vrho+\Lambda(\vrho) \, \Div u+ u\cdot\nabla \vrho=0,\\
					&\partial_t u+(u\cdot\nn)u+\Lambda(\vrho)\,\nabla \vrho=0,\\
					&\vrho(0, x) = \vrho_0(x),\\
					&u(0, x) = u_0(x).
				\end{aligned}
				\right.
			\end{equation}
			Finally, to solve the compressible Euler equations \eqref{Euler (rho u)}, we will verify $\rr_0 < \vrho(t) < \rr_\infty$ (at least locally in time) \textit{a posteriori}.
			As for higher-order regularity estimate, by assuming \ref{Structural Identity P' r'} with either \ref{Constant Sound} or \ref{General Sound}, this difficulty can be overcome.  For instance, if \ref{Constant Sound} holds, integrating by parts eliminates the singular term, yielding
			\begin{equation*}
				\sqrt{P(1)} \bIP{ \Div u, \vrho}_{H^s} + \sqrt{P(1)} \bIP{\nabla \vrho, u}_{H^s} = 0.
			\end{equation*}
			If \ref{General Sound} holds, as demonstrated in the proof of  \eqref{F(X) X-2} (see \eqref{q=r(rho) no singularity}), the following term
			\begin{equation*}
				\bIP{ \Lambda(\vrho)\, \nabla \cdot u, \vrho}_{H^s} + \bIP{\Lambda(\vrho)\, \nabla \vrho, u}_{H^s}
			\end{equation*}
			can be controlled by a function  including a nonlinear component in $\|(\vrho,u)\|_{\WLip}$ and a linear part in $\|(\vrho,u)\|^2_{\H^s}$. Therefore, in the stochastic context, the cut-off technique in SPDEs can be applied (see \eqref{GnR HR}). Given these facts and the locally Lipschitz condition of $\Lambda$ in $H^s$ (see \ref{General Sound}), it becomes feasible to use the compactness argument by first constructing a proper regularized problem with a cut-off and then finding a convergent subsequence.

			\item \textbf{Compared to symmetrization and Makino transformation.}   
			In contrast to the classical symmetrization, under our new variables $X = (\vrho, u)^T$, the system \eqref{Euler (q u) Theta} is in symmetric quasi-linear form where the off-diagonal coefficients are exactly $\Theta(\vrho)$ (it turns to be the sound speed under new variable, see explanation below). Furthermore, for the specific case of a polytropic gas $P(\rho) = a\rho^\gamma$, integrating our structural identity \ref{Structural Identity P' r'} exactly recovers the classical Makino transformation (see Example \ref{Example-gamma} below). Our abstract transformation principle thus extends the Makino transform  to non-polytropic equations of state.
			
			When $\rho \to 0$, the behavior of $\Theta(\vrho)$ depends entirely on the specific equation of state. For instance, in the concrete examples (see Section \ref{Section : Pressure Examples}), one can see  that $\Theta(\vrho) \to 0$ for a polytropic gas (see Example \ref{Example-gamma}),  meaning the coupling between density and velocity waves simply vanishes at vacuum. In a Chaplygin gas (see Example \ref{Example-Chaplygin}), $\Theta(\vrho)$ may diverge as $\vrho \to -\infty$. However, in all cases, the corresponding $L^2$ energy functional is standard and never degenerates: it is always $\int_{\K^d} (|\vrho|^2 + |u|^2) \d x$.  
			
			Since the possible singularity is now entirely in the equation coefficients rather than in the energy functional. 
			\ref{General Sound} can be understood as a condition to control such singularity. Indeed, by introducing the smooth extension $\Lambda$ with controlled Sobolev estimates in \ref{General Sound}, we rigorously tame these coefficient singularities, enabling a unified local-in-time theory even for highly singular pressure laws.  
			Physically, the condition $\sup_{y\in\R}|\Lambda'(y)|<\infty$ in \ref{General Sound} is well-motivated: materials at extremely high densities typically exhibit finite rigidity, and this uniform bound excludes unphysical ``super-hard'' states with a diverging power law. Moreover, this assumption precludes a rapid divergence of the sound speed.   In Section \ref{Section : Pressure Examples}, we demonstrate that \ref{Hypo-Pressure} encompasses many well-known cases and accommodates more intricate situations, some of which have not been previously studied in the stochastic setting.

			\item  \textbf{Connection to the enthalpy and sound speed.}
			Under \ref{Hypo-Pressure}, the enthalpy and sound speed (see \eqref{enthalpy+sound}) are related via $\rr(\rho)$ by the identity
			\begin{equation*}
				\frac{\nabla \hbar(\rho)}{c_{{\rm sound}}(\rho)}
				=\nabla \rr(\rho).
			\end{equation*}
			Moreover, since $c_{{\rm sound}}(\rho)=\rho \rr'(\rho)$, the conditions $\rr'>0$ and $\rho > 0$ ensure a strictly positive sound speed, which is physically meaningful. More crucially, the structural identity \ref{Structural Identity P' r'} implies that $\rho\, \rr'(\rho) = \sqrt{P'(\rho)}$, which is exactly the local sound speed $c_{{\rm sound}}(\rho)$. Consequently, the function $\Theta(\vrho)$ defined in \ref{Regularity-Sound} represents the sound speed expressed in terms of the new variable $\vrho$, i.e., 
			$\Theta(\vrho)=c_{{\rm sound}}(\rho).$
			This explains why condition \ref{Regularity-Sound} characterizes the regularity of the  sound speed. 
			
			\item \textbf{Connection to the acoustic nonlinearity parameter.} It is worth noting that, in addition to the relation $\Theta(\vrho)=c_{{\rm sound}}(\rho)$ mentioned above, the derivative $\Theta'(\vrho)$ is intimately linked to another fundamental physical quantity in nonlinear acoustics. 
			For a given reference state $\rho=\rho_0$, the dimensionless parameter 
			\begin{equation}
				\frac{1}{2}\rho_0 \frac{P''(\rho_0)}{P'(\rho_0)}=\frac{{\rm d}\log c_{{\rm sound}}(\rho)}{{\rm d}\log \rho}\bigg|_{\rho=\rho_0},\label{eq : acoustics parameter}
			\end{equation}
			is usually denoted as $B/A$ in nonlinear acoustics; see \cite{Beyer-2024-Chapter}. This parameter characterizes the intrinsic nonlinearity of the equation of state: it relates the second-order pressure-density response to the linear one and describes how the sound speed varies with density (or pressure). It governs harmonic generation, waveform distortion, and the tendency of finite-amplitude waves toward shock formation.
			
			Under \ref{Hypo-Pressure}, recalling that $\Theta(\vrho) \triangleq  \rr'\big(\rr^{-1}(\vrho)\big)\rr^{-1}(\vrho)$ with $\vrho \triangleq  \rr(\rho)$, a direct computation yields
			$$\Theta'(\vrho) =\frac{1}{2} \frac{\rho P''(\rho)}{P'(\rho)}= \frac{{\rm d}\log c_{{\rm sound}}(\rho)}{{\rm d}\log \rho}.$$
			Comparing this with \eqref{eq : acoustics parameter}, we see that $\Theta'(\vrho)$ is the pointwise functional analogue of $\frac{1}{2}\rho_0 (P''(\rho_0)/P'(\rho_0))$, i.e., exactly half of the nonlinear acoustics parameter given in \eqref{eq : acoustics parameter}. Hence, $\Theta'(\vrho)$ quantifies how the relative strength of the quadratic term in the equation of state varies with the local density $\rho$. It provides a pointwise measure of thermodynamic nonlinearity across different background states, rather than only at a single reference state.
			
			\item \label{Thea=C0-Remark C0 (a b)} 
			\textbf{The case of constant sound.} If, in \ref{Constant Sound}, we assume that for some $C_0\in\R\setminus\{0\}$, $\Theta(y)=\rr'\big(\rr^{-1}(y)\big)\rr^{-1}(y)\equiv C_0$ for all $y\in(\rr_0,\rr_\infty)$, then $\rr'(x)=C_0/x$ for all $x\in(0,\infty)$. Solving for $\rr$ yields $\rr(\rho)=C_0\log(\rho)+C_1$ with $C_1\in\R$. By \ref{Structural Identity P' r'}, we observe that $P(\rho)=C_0^2\rho$, requiring $C_0$ to be $\sqrt{P(1)}$. Besides, although $\rr(\rho)$ is not uniquely determined (due to the arbitrary constant $C_1\in\R$), the interval $(\rr_0,\rr_\infty)=(-\infty,\infty)$ is uniquely fixed. Note that we \textbf{cannot} include this case into \ref{General Sound} since the non-zero constant function $\Theta(y)\equiv C_0\notin H^s(\R;\R)$.

		\end{enumerate}
	\end{Remark}

	To solve \eqref{Target problem (rho u)}, we need conditions on $h(t,\rho,u)$ that are compatible with the transformation based on \ref{Hypo-Pressure}.

	\begin{Hypothesis}
		\label{Hypo-h}
		Let $d \ge 2$, $p \ge 1$, $\sigma > \frac{d}{2} + p$, and let $\rr(\cdot)$ be given by \ref{Hypo-Pressure}. Assume that
		\begin{equation}
			h(t, \rho, u) = z(t, \rr(\rho), u)\label{h z}
		\end{equation}
		for some measurable function $z: [0, \infty) \times \H^\sigma \to \H^\sigma$. Moreover, suppose that there is a function $K \in \mathscr{K}$ $($cf. \eqref{SCRK}$)$ such that the for all $t\ge 0$ and $(\vrho,u), (\vrho_1,u_1), (\vrho_2,u_2)\in \H^{\sigma}$,
		\begin{align*}
			\norm{z(t,\vrho,u)}^2_{\H^{\sigma}} \leq  K(t,\|(\vrho,u)\|_{\WP}
			)(1+\|(\vrho,u)\|^{2}_{\H^{\sigma}}),
		\end{align*}
		\begin{align*}
			\norm{z(t,\vrho_1,u_1)- z(t, \vrho_2,u_2)}^{2}_{\H^{\sigma}} \le K(t,
			\|(\vrho_1,u_1)\|_{\H^{\sigma}}+\|(\vrho_2,u_2)\|_{\H^{\sigma}})\|(\vrho_1-\vrho_2,u_1-u_2)\|^{2}_{\H^{\sigma}}.
		\end{align*}
		
	\end{Hypothesis}

	\begin{Remark}
		Broadly speaking,  \ref{Hypo-h} can be interpreted as a family of assumptions—depending on parameters $p$ and $\sigma$—on the function $h(t,\rho,u)=z(t,\rr(\rho,u))$.  
		
		\begin{enumerate}[label={\bf (\alph*)},leftmargin=0.79cm]\setlength\itemsep{0.2em}
			\item The parameter $p\ge1$ serves as an index describing the growth factor $\|(\vrho,u)\|_{\WP}$ of $z\in \H^\sigma$ (where $\sigma>\frac{d}{2}+p$ ensures $\H^\sigma\hookrightarrow\WP$), and it affects the blow-up criterion of solutions, as shown in \eqref{Blow-up criterion (rho u)} below.
			\item The parameter $\sigma$ determines the Sobolev space $\H^\sigma$ in which $z$ is locally Lipschitz.
		\end{enumerate}
	\end{Remark}

	In Remark \ref{Remark-Hypo-Pressure}, we observe that \ref{Hypo-Pressure} induces a transformation that converts \eqref{Euler (rho u)} into \eqref{Euler (q u) Xi}. Similarly, in the stochastic case, since the driving noise terms only appear in the velocity equation \eqref{Target problem (rho u)}$_2$ and the transformation acts only on the density $\rho$, one can directly apply the transformation to the density equation and no further quadratic variation on $\rho$ arises. In this way, \ref{Hypo-Pressure} and \eqref{h z} in \ref{Hypo-h} allow us to transform \eqref{Target problem (rho u)} into a constrained problem. To illustrate this, let $A^T$ denote the transpose of a matrix/vector $A$, and let $\rr \in {\bf R}_{(\rr_0,\rr_\infty)}$ (cf. \eqref{Set of transform}). For $X = (\vrho, u)^T = (\rr(\rho), u)^T$, we define
	\begin{equation}\label{F-Euler}
		F(X) \triangleq 
		\begin{cases}
			\left(
			\begin{array}{c}
				\sqrt{P(1)}\, \Div u + u \cdot \nabla \vrho \\
				(u \cdot \nn)u + \sqrt{P(1)}\,\nabla \vrho \\
			\end{array}
			\right), & \text{if}\ \Theta(\vrho)\equiv \sqrt{P(1)}, \vspace*{0.3cm}\\
			\left(
			\begin{array}{c}
				\Lambda(\vrho) \, \Div u + u \cdot \nabla \vrho \\
				(u \cdot \nn)u + \Lambda(\vrho)\,\nabla \vrho \\
			\end{array}
			\right), & \text{if}\ \Theta(\vrho) = \Lambda\big|_{(\rr_0,\rr_\infty)}(\varrho),
		\end{cases}
	\end{equation}
	where $\Lambda$ is defined in \ref{Regularity-Sound}. Given this definition of $F$ and by virtue of \ref{Hypo-Pressure}, solving \eqref{Target problem (rho u)} is equivalent to solving the Cauchy problem:
	\begin{equation}\label{Target problem X=(q u)}
		\left\{
		\begin{aligned}
			&\d X + F(X)\d t
			= \QQ_{1}X\,\circ\, {\rm d} W+\QQ_{2}X\diamond {\rm d}L + \ZZ(t,X)\d \widetilde W,\\
			&\QQ_{i} \triangleq 
			{\rm diag}(0,\Q_i)\ (i=1,2),\quad 
			\ZZ(t,X) \triangleq 
			\left(
			\begin{array}{c}
				0\\
				z(t,\vrho, u)\\
			\end{array}
			\right),\\
			&X(0)=\Xi \triangleq (\vrho_0,u_0)^T=(\rr(\rho_0),u_0)^T,
		\end{aligned}\right.
	\end{equation}
	subject to the constraint
	\begin{equation}\label{a<q<b}
		\mathbb{P}\big(\rr_0<\vrho(t)<\rr_\infty,\ 0\le t< \tau^*\big)=1,\quad \tau^*\ \text{is the lifespan of}\ X^T=(\vrho,u)\quad (\text{cf. Definition \ref{solution definition (q u)} below}).
	\end{equation}

	Note that \eqref{Target problem X=(q u)} is interpreted in the following sense:
	\begin{align}\label{Target problem X=(q u) full}
		X(t) -\Xi \
		+ & \int_0^{t}  \Big\{F(X)-\frac{1}{2} 
		\QQ_1^2X\Big\} (t')\d t'\notag\\
		= & 
		\int_0^{t} \QQ_1X(t')\d W(t')
		+
		\int_0^{t} \ZZ(t',X(t')) \d \widetilde W(t')\notag\\
		&+\int_0^t\int_{|l|\le1} \Big\{\M\big(1,l,X(t'-)\big)-X(t'-)\Big\}\, \widetilde{\eta}({\rm d}l,{\rm d}t')\notag\\
		&+\int_0^t\int_{|l|\le1} \Big\{\M\big(1,l,X(t')\big)-X(t')-l\cdot \QQ_2 X(t')\Big\}\, \nu({\rm d}l){\rm d}t',\quad t\in[0,\tau^*),
	\end{align} 
	where, according to \ref{Hypo-W L}, large jumps are excluded, and
	\begin{equation*}
		\M(r)\triangleq \M(r,l,Y)
	\end{equation*} 
	denotes the solution to the initial value problem
	\begin{equation}
		\frac{{\rm d}}{{\rm d}r}\M(r)= l\cdot  \QQ_2 \M(r),\quad r\in[0,1],  \quad \M(0) = Y\in \H^s.\label{Marcus flow QQ2}
	\end{equation}

	In this paper, we focus on \textbf{classical} solutions, meaning that equations \eqref{Target problem (rho u)}$_1$ and \eqref{Target problem (rho u)}$_2$ hold as equalities in $C(\K^d)$.
	To this end, we require solutions $X \in \H^s$ for some sufficiently large $s > 0$. We now introduce the following definition of pathwise classical solutions to \eqref{Target problem X=(q u)}:
	
	\begin{Definition}\label{solution definition (q u)} 
		Let $\tau^*$ be a stopping time with $\p(\tau^*>0)=1$, and let $(X,\tau^*) \triangleq (X(t))_{t\in [0,\tau^*)}$ be a progressively measurable process on $\H^s$.
		\begin{enumerate}[label={{\rm (\arabic*)}},leftmargin=0.79cm]\setlength\itemsep{0.2em}
			\item $($Maximal $\H^s$ classical solution$)$.  The pair $(X,\tau^*)$ is called a maximal $\H^s$ classical solution to \eqref{Target problem X=(q u)} if $\pas$ the following conditions are satisfied:
			\begin{itemize}
				\item \eqref{Target problem X=(q u) full} holds as an equality in $C(\K^d)$ for all $t\in[0,\tau^*)$;
				\item $\sup_{t \in [0,T]}\|X(t)\|_{\H^s}<\infty$ for every $T<\tau^*$;
				\item $\limsup_{t\uparrow\tau^*}\|X(t)\|_{\H^s}=\infty$  on $\{\tau^*<\infty\}.$
			\end{itemize}
			In particular, if $\p(\tau^*=\infty)=1$, the maximal solution is called global.
			
			\item $($Admissible $\H^s$ classical solution$)$.
			A maximal solution $(X,\tau^*)$ to \eqref{Target problem X=(q u)} is called admissible if \eqref{a<q<b} holds, i.e.,
			\begin{equation*} 
				\p\big(\rr_0<\vrho(t)<\rr_\infty \text{ for all } t\in [0,\tau^*)\big)=1.
			\end{equation*} 
		\end{enumerate}
	\end{Definition}

	\begin{Remark}\label{Remark : measurability-upgrade}
		We first note that under the conditions in Definition \ref{solution definition (q u)} and  \ref{Hypo-W L}, the Debut theorem guarantees that the first hitting time of any Borel set of $\H^s$ by the progressively measurable process $X$ is a well-defined stopping time. Furthermore, we remark that, if one prefers,  the progressive measurability of $X$ can  be required in a weaker topology. To see this at an abstract level,  let $(Y, \|\cdot\|_Y)$ be a separable Banach space equipped with a weaker Hausdorff topology $\mathscr{T}$. 
		Since the identity map into $(Y, \mathscr{T})$ is continuous and bijective, $(Y, \mathscr{T})$ is a Lusin space. 
		Then, for any pointwise defined stochastic process $Z = \{Z(t)\}_{t \ge 0}$ taking values in $Y$, if $Z$ is progressively measurable with respect to $(Y, \mathscr{T})$, it is necessarily progressively measurable with respect to the norm topology $(Y, \|\cdot\|_Y)$. Indeed, this holds because the Borel $\sigma$-algebras associated with comparable Lusin topologies are identical (see, e.g., \cite[Page 107]{Schwartz-1973-Book} or \cite[Theorem 6.8.6]{Bogachev-2007-Books}). In Definition~\ref{solution definition (q u)}, the condition $\sup_{t\in[0,T]}\|X(t)\|_{\H^s}<\infty$ almost surely for every $T<\tau^*$ implies that $X(t)$ takes values in $\H^s$ $\pas$ Therefore, once the progressive measurability of $X$ is established in a weaker topology (such as the $\H^\theta$-topology for some $\theta < s$), its progressive measurability as an $\H^s$-valued process follows automatically by viewing this $\H^\theta$-topology as the weaker topology $\mathscr{T}$ on $\H^s$. This is precisely the strategy we will employ later (see \ref{Convergence of X-n T} in Lemma \ref{Lemma : Xn T estimates}, where we obtain $X$ via a limit procedure in the space of c\`{a}dl\`{a}g paths taking values in $\H^\theta$ with $\theta<s$). 
	\end{Remark}

	Since $\rr\in{\bf R}_{(\rr_0,\rr_\infty)}$ is an increasing $C^\infty$ function (cf. \eqref{Set of transform}), 
	the transformed state space   
	\begin{equation}\label{Hsr space}
		\H^s_{\rr}\triangleq \rr^{-1}(H_1^s)\times  H_d^s \quad \text{with}\quad 
		\rr^{-1}(H_1^s) \triangleq \Big\{f\,:\, \rr(f)\in H_1^{s}\Big\}
	\end{equation}
	is well-defined.  Let $(X,\tau^*)$, where $X =  (\rr(\rho), u)^T$, be an admissible $\H^s$ classical solution to \eqref{Target problem X=(q u)}.  Then   $((\rho,u),\tau^{*})$ is a maximal $\H^s_{\rr}$ classical solution to \eqref{SEuler-(rho u)} in the following sense:
	
	\begin{Definition}\label{solution definition (rho u)} 
		Let $\tau^*$ be a stopping time satisfying $\p(\tau^*>0)=1$, and let $((\rho,u),\tau^{*})$ be a progressively measurable process on $\H_{\rr}^s$ defined for $t\in[0,\tau^*)$. We call $((\rho,u),\tau^{*})$ a maximal $\H_{\rr}^s$ classical solution to \eqref{Target problem (rho u)} if the following conditions hold $\pas$:

		\begin{enumerate}[label={\rm (\arabic*)},leftmargin=0.79cm]\setlength\itemsep{0.2em}
			
			\item  \eqref{Target problem (rho u)}  holds in $C(\K^d)$ for all $t\in[0,\tau^*)$;
			\item  $\sup_{[0,T]}\|(\rr(\rho),u)(t)\|_{\H^s}<\infty$ for every $T<\tau^*$;
			\item $ 
			\limsup_{t\uparrow\tau^*}\|(\rr(\rho), u)(t)\|_{\H^s}=\infty\  \text{on}\  \{\tau^*<\infty\}$.
		\end{enumerate}
	\end{Definition}

	\subsection{Examples of a Broad Class of Pressure Laws}\label{Section : Pressure Examples}
	We now provide examples that satisfy \ref{Hypo-Pressure}. It is worth noting that, even in specific cases, several remain unexplored within the stochastic framework.   
	
	\begin{Example}[\textbf{\(\gamma\)-law on} $\K^d$] \label{Example-gamma}
		Consider the \(\gamma\)-law pressure \(P(\rho)=a\rho^{\gamma}\) with \(a > 0,\ \gamma \ge 1\). We verify \ref{Hypo-Pressure} for this pressure law:
		
		\begin{itemize}[leftmargin=0.79cm]\setlength\itemsep{0.2em}
			\item \textbf{Case of \(\gamma > 1\).}  \ref{Structural Identity P' r'} can be verified by taking
			\[
			\rr(\rho)=\frac{2\sqrt{a\gamma}}{\gamma-1}\rho^{\frac{\gamma-1}{2}}, \quad  {\bf R}_{(\rr_0,\rr_\infty)}={\bf R}_{(0,\infty)}.
			\] 
			Moreover, it is easy to see that $\Theta(x)=C_{a,\gamma} x$
			for some constant \(C_{a,\gamma}\) depending on \(a,\gamma\). Taking \(\Lambda(x)=C_{a,\gamma} x\) with domain \(\mathbb{R}\), one can verify that \ref{Regularity-Sound} holds.
			To recover \(\rho=\rr^{-1}(\vrho)\) in this case, we require \(\vrho=\rr(\rho) \in (0,\infty)\). 
			This transformation has been proposed in \cite{Makino-1986-chapter}.
			
			\item \textbf{Case of \(\gamma = 1\).} As discussed in \ref{Thea=C0-Remark C0 (a b)} of Remark \ref{Remark-Hypo-Pressure}, \ref{Hypo-Pressure} holds with
			\[
			\rr(\rho)=\sqrt{a}\log(\rho),\quad {\bf R}_{(\rr_0,\rr_\infty)}={\bf R}_{(-\infty,\infty)},\quad \Theta(x)\equiv \sqrt{a}.
			\]
			Since \(\rr(y)\to-\infty\) as \(y\to0\), we can reconstruct \(\rho\) provided \(|\vrho|=|\rr(\rho)|<\infty\).
		\end{itemize}
	\end{Example}
	
	\begin{Example}[\textbf{(Generalized) Chaplygin Gas States on} $\K^d$] \label{Example-Chaplygin}
		Consider the pressure law \(P(\rho)=-a \rho^{1-2\kappa}\) for \(\kappa\in(1/2,1]\) and \(a>0\). This corresponds to the generalized Chaplygin gas, which arises in cosmological theories \cite{Bento-Bertolami-Sen-2002-PRD,Ferreira-Avelino-2018-PRD}. The case \(\kappa=1\) gives the pure Chaplygin gas. 
		To verify \ref{Hypo-Pressure}, we take (ensuring the monotonicity of \(\rr\) as required by \eqref{Set of transform})
		\[
		\rr(\rho)=-\frac{\sqrt{a}\sqrt{2\kappa-1}}{\kappa}\rho^{-\kappa},\quad {\bf R}_{(\rr_0,\rr_\infty)}={\bf R}_{(-\infty,0)},\quad \Theta(x)=-\kappa x.
		\]
		Taking \(\Lambda(x)=-\kappa x\) with domain \(\mathbb{R}\) satisfies the required conditions.
	\end{Example}

	Practically relevant smooth piecewise-defined $\gamma$-pressure laws and Chaplygin gas pressure models are encompassed as special cases of \ref{Hypo-Pressure}.
	
	\begin{Example}[\textbf{Piecewise $\gamma$-pressure law/Chaplygin gas on} $\T^d$]\label{Example-Mixed}
		We illustrate only the case of piecewise $\gamma$-pressure law, since the approach for demonstrating the piecewise Chaplygin gas follows similarly.

		Let $k\ge2, k\in \N$, $b_i\in\R$ with $i=0,1,\cdots,k-1$ and $c_j\in\R$ with $j=1,2,\cdots,k-1$ satisfy
		$$0=b_0<1<c_1<b_1<c_2<b_2<\cdots
		<b_{k-2}<c_{k-1}<b_{k-1}<\infty.$$ 
		Let $\gamma_i\ge1$ ($i=2,\cdots,k-1$), $\gamma_1,\,\gamma_k>1$ and $a_j>0$ ($j=1,2,\cdots,k$) such that
		$a_ic_i^{\gamma_i}<a_{i+1}b_i^{\gamma_{i+1}}$, $i=1,\cdots,k-1.$
		Suppose that $P$ is $C^\infty$, $P'>0$,
		\begin{equation*}
			P\big|_{[b_{i-1},c_i]}(\rho)=a_i\rho^{\gamma_i}\ (i=1,2,\cdots,k-1),\quad \text{and}\quad 
			P\big|_{[b_{k-1},\infty)}(\rho)=a_k\rho^{\gamma_k}.
		\end{equation*}
		Direct calculation shows that \ref{Structural Identity P' r'} is verified with 
		$$\rr(\rho) \triangleq \int_0^\rho\frac{\sqrt{P'(y)}}{y}\d y,\quad {\bf R}_{(\rr_0,\rr_\infty)}={\bf R}_{(0,\infty)}.$$ 
		Note that on $x\in(0,c_1]$, we have
		$\frac{\sqrt{P'(x)}}{x}=\sqrt{a_1\gamma_1}\,x^{(\gamma_1-3)/2}$.
		Since $\gamma_1>1$, the function $\rr(\rho)$ is well-defined.
		
		Moreover, by considering $\Theta(y)\triangleq \rr'(\rr^{-1}(y))\rr^{-1}(y)$, and noticing that the integration across transition intervals naturally accumulates constant offsets, we obtain constants $A_i>0$ and $B_i\in\R$ ($i=1,\dots,k$), with $B_1=0$, such that the following piecewise affine form holds:
		\begin{equation*} 
			\Theta\big|_{(0,\, \rr(c_1)]}(y)= A_1 y, \quad 
			\Theta\big|_{[\rr(b_{i-1}),\, \rr(c_i)]}(y)= 
			\begin{cases} 
				A_i y + B_i,\ \ &\text{if}\ \gamma_i>1,\vspace*{4pt}\\ 
				B_i,\ \ &\text{if}\ \gamma_i=1, 
			\end{cases} 
			\quad \text{and}\quad \Theta\big|_{[\rr(b_{k-1}),\infty)}(y)= A_k y + B_k.
		\end{equation*}
		When $\mathbb{K}=\mathbb{T}$, condition \ref{Regularity-Sound} is verified. Indeed, although $\Theta(y)$ is piecewise affine rather than strictly linear, its derivative $\Theta'(y)$ is bounded by piecewise constants. Therefore, one can smoothly extend $\Theta$ to a function $\Lambda$ on $\R$ such that the uniform bound $\sup_{y\in\R}|\Lambda'(y)|<\infty$ is preserved.
		An illustrative case of Example \ref{Example-Mixed} for piecewise $\gamma$-pressure law with $(b_0,b_1,b_2)=(0,2,4)$, $(c_1,c_2)=(1,3)$,  $(a_1,\gamma_1)=(2,5/3)$, $(a_2,\gamma_2)=(3,1)$ and $(a_3,\gamma_3)=(2,3/2)$ is given by the following Figure \ref{Figure-Pressure}:

		\begin{center}
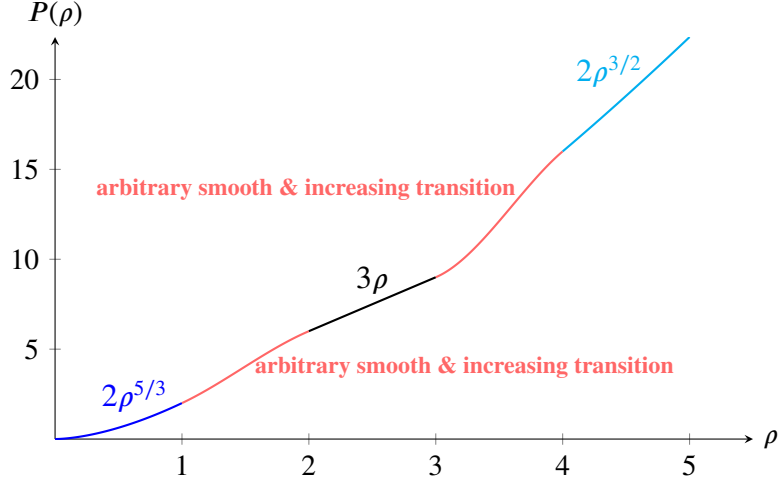

			\begin{tikzpicture}
				\begin{axis}[
					domain=0:5,
					samples=500,
					smooth,
					axis lines=middle,
					xlabel=$\rho$,
					ylabel=$P(\rho)$,
					xmin=0, xmax=5.5,
					every axis y label/.style={at=(current axis.above origin),anchor=south},
					every axis x label/.style={at=(current axis.right of origin),anchor=west},
					height=7.2cm, width=11cm,
					]
					
					\addplot [
					domain=0:1,
					samples=100,
					thick,
					color=blue
					] {2*x^(5/3)} 
					node[pos=0.5, above] {{\large $2\rho^{5/3}$}}
					;

					\addplot [
					domain=1:2,
					samples=100,
					thick,
					color=red!60
					]
					{(-5/3)*x^3 + (22/3)*x^2 -(19/3)*x + (8/3)}
					node[pos=0.5, right] {{\footnotesize \textbf{arbitrary smooth \& increasing transition}}}
					;

					\addplot [
					domain=2:3,
					samples=100,
					thick,
					color=black 
					] {3*x} 
					node[pos=0.5,  above] {{\large $3\rho$}}
					;

					\addplot [
					domain=3:4,
					samples=100,
					thick,
					color=red!60
					] 
					{-5*x^3 + 54*x^2 - 186*x + 216}
					node[pos=0.7, left] {{\footnotesize \textbf{arbitrary smooth \& increasing transition}}}
					;
					
					\addplot [
					domain=4:5,
					samples=100,
					thick,
					color=cyan
					] {2*x^(3/2)} 
					node[pos=0.7, left] {{\large $2\rho^{3/2}$}}
					;
				\end{axis}
			\end{tikzpicture}

			\captionof{figure}{Any smooth $P(\rho)$ such that $P'(\rho)>0$, $P(\rho)=2\rho^{5/3}$ on $[0,1]$, $P(\rho)=3\rho$ on $[2,3]$ and $P(\rho)=2\rho^{3/2}$ on $[4,\infty)$ can be  included in \ref{Hypo-Pressure}.}\label{Figure-Pressure}

		\end{center}
		
	\end{Example}

	\begin{Example}[Pressure law of a white dwarf star on $\T^d$]\label{Example-white dwarf star}
		In other astrophysical contexts, such as the modeling of white dwarf stars (cf.~\cite{Strauss-Wu-2020-Nonlinearity,Chen-etal-2024-CMP}), the constitutive pressure can deviate significantly from the polytropic form and may be given, for instance, by
		$$
		P(\rho) = c_1 \int_0^{c_2 \rho^{1/3}} 
		\frac{y^4}{\sqrt{c_3 + y^2}}\,\mathrm{d}y, 
		\qquad \rho>0,
		$$
		where $c_1$, $c_2$, and $c_3$ are positive constants. This pressure exhibits the asymptotic behaviors 
		$
		P(\rho) \sim  a_1\rho^{5/3}
		$ when $\rho \to 0$ and
		$
		P(\rho) \sim a_2\rho^{4/3}
		$ when $\rho \to \infty$
		for suitable constants $a_1,a_2>0$. Accordingly, we define
		$$
		\rr(\rho) \triangleq \int_0^\rho\frac{\sqrt{P'(y)}}{y}\d y
		= \sqrt{3 c_1 c_2^3 \sqrt{c_3}} \int_0^{\frac{c_2 \rho^{1 / 3}}{\sqrt{c_3}}} \frac{1}{\left(1+x^2\right)^{1 / 4}}\d x.
		$$
		Clearly, $\rr(\rho)$ satisfies  \ref{Structural Identity P' r'}  with \({\bf R}_{(\rr_0,\rr_\infty)}={\bf R}_{(0,\infty)}\). A direct computation shows that \(\Theta(y)\triangleq \rr'\big(\rr^{-1}(y)\big)\rr^{-1}(y)\) satisfies
		$$
		\Theta(y)
		= \frac{\sqrt{3 c_1 c_2^3 \sqrt{c_3}}}{3}\frac{x}{\left(1+x^2\right)^{1 / 4}},\quad 
		x = g^{-1}\Big(\frac{y}{\sqrt{3 c_1 c_2^3 \sqrt{c_3}}}\Big),\quad 
		g(x)=\int_0^x\frac{1}{(1+r^2)^{1/4}}\d r,\quad x>0.
		$$
		Since $g$ can be extended to a smooth odd diffeomorphism from $\R$ onto $\R$, $\Theta$  admits a smooth extension $\Lambda: \R\to\R$ that satisfies
		$$
		\Lambda(0)=\lim_{y\to0}\Theta(y)=0,\quad \sup_{y\in\R}|\Lambda'(y)|\le \lim_{y\to0}\Theta'(y)=1/3,
		$$
		which verifies \ref{General Sound}.
	\end{Example}
	
	\begin{Remark}
		In Examples \ref{Example-Mixed} and \ref{Example-white dwarf star}, the function $\Lambda$ is nonlinear with $\Lambda'(0)\neq 0$. Consequently, these examples are not applicable when $\K=\R$. More precisely, for such a nonlinear $\Lambda$, the estimate $\|\Lambda(\vrho_1)-\Lambda(\vrho_2)\|_{H^s}$ from \ref{Regularity-Sound} may fail when $\K=\R$.
	\end{Remark}
	
	\begin{Example}[Pressure with nonlinear acoustics parameter vanishing in the vacuum limit on $\K^d$]\label{Example-vacuum limit}
		Consider the pressure class 
		\begin{align*}
			\mathfrak{P}\triangleq
			\left \{P\in C^1([0,\infty); \mathbb{R})\cap C^\infty((0,\infty); \mathbb{R}) \, : \,
			\begin{aligned}
				&0<\liminf_{\rho\to\infty}\frac{\rho P''(\rho)}{P'(\rho)} \le \limsup_{\rho\to\infty}\frac{\rho P''(\rho)}{P'(\rho)}<\infty,\\
				P'(\rho)&>0\ \text{for}\ \rho>0,\quad P'(\rho)=\frac{1}{(\log \rho)^4}\ \text{for} \ \rho\in(0,1/2).
			\end{aligned} \right \}.
		\end{align*}
		Let $P\in \mathfrak{P}$. Then \( \rr(\rho)\triangleq \int_0^\rho\frac{\sqrt{P'(y)}}{y}\d y\)
		is well defined on $(0,\infty)$, and $\lim_{\rho\to\infty}\rr(\rho)=\infty$. Therefore, $\rr\in {\bf R}_{(\rr_0,\rr_\infty)}$ with $(\rr_0,\rr_\infty)=(0,\infty)$, ensuring \ref{Structural Identity P' r'}. Moreover, in this case,
		$$\rr(\rho)\big|_{(0,\frac{1}{2})}=\frac{-1}{\log \rho},\quad  
		\Theta(y)\big|_{(0,\frac{1}{\log 2})}=y^2,\quad \limsup_{y\to\infty}\Theta'(y)= \limsup_{\rho\to\infty}\frac{\rho P''(\rho)}{2P'(\rho)}<\infty.
		$$
		We first extend $\Theta$ to $[0,\infty)$ by setting $\Theta(0)=0$, and then extend the resulting function to $\R$ as an even function $\Lambda$. Since $\Lambda(0)=\Lambda'(0)=0$, it follows directly that \ref{General Sound} holds for this $\Lambda$.
		
		Recall that $\rho \frac{P''(\rho)}{2P'(\rho)}$ serves as the nonlinear acoustics parameter (cf. Remark \ref{Remark-Hypo-Pressure}). Direct computation then yields
		$$
		P\in\mathfrak{P}\Longrightarrow P'(0)=0,\quad \lim_{\rho\to0^+}\frac{{\rm d}\log c_{{\rm sound}}(\rho)}{{\rm d}\log \rho}=0,\quad 0<\liminf_{\rho\to\infty}\frac{{\rm d}\log c_{{\rm sound}}(\rho)}{{\rm d}\log \rho}\leq \limsup_{\rho\to\infty}\frac{{\rm d}\log c_{{\rm sound}}(\rho)}{{\rm d}\log \rho}<\infty.
		$$
		The above conditions model very ``soft'' physical states: at extremely low densities or near vacuum, fluid responses become exceptionally sluggish, the sound speed tends to zero, and the pressure increases only very slowly with density. As $\rho$ increases, no ``super-hard'' states—characterized by a diverging power law or an excessively rapid divergence of the speed of sound—occur. Consequently, $\mathfrak{P}$ provides an effective framework for simulating fluids or plasmas under extreme conditions, such as strongly cooled outer boundaries, highly rarefied regions, or tails approaching vacuum.
	\end{Example}

	\subsection{Local-in-time Theory for the Compressible Case}
	\label{Section : Result Thm-compressible}
	
	The stochastic compressible Euler system \eqref{SEuler-(rho u)} depends on the orders of $\Q_1$ and $\Q_2$. To state our local-in-time theory for \eqref{SEuler-(rho u)}, we introduce
	\begin{equation}
		\label{zeta-Q12}
		\zeta=\zeta_{\Q_1,\Q_2} \triangleq 
		\begin{cases}
			\alpha, & \text{if } \Q_1,\ \Q_2\in\mathbb{A}^\alpha \text{ and } \exists\, i\in\{1,2\} \text{ such that } \Q_{i}\neq 0,\\[4pt]
			\beta, & \text{if } \Q_1,\ \Q_2\in\mathbb{B}^\beta \text{ and } \exists\, i\in\{1,2\} \text{ such that } \Q_{i}\neq 0,\\[4pt]
			\alpha\lor \beta, & \text{if } \Q_1\in\mathbb{A}^\alpha,\  \Q_2\in\mathbb{B}^\beta \text{ or } \Q_1\in\mathbb{B}^\beta,\  \Q_2\in\mathbb{A}^\alpha,\ \text{and } 
			\Q_i\neq 0,\ i=1,2,\\[4pt]
			0, & \text{if } \Q_1=\Q_2=0,
		\end{cases}
	\end{equation}
	where $\alpha \in [0,1]$ and $\beta \ge 0$ are given in \ref{Hypo-Qi} (see \eqref{Ak class define} and \eqref{Bk class define}).
	
	We recall the transform set ${\bf R}_{(\rr_0,\rr_\infty)}$ defined in \eqref{Set of transform} and the transformed state space from \eqref{Hsr space}:
	\begin{equation*}
		\H^s_{\rr}\triangleq \rr^{-1}(H_1^s)\times H_d^s \quad \text{with}\quad 
		\rr^{-1}(H_1^s) \triangleq \Big\{f\,:\, \rr(f)\in H_1^{s}\Big\}, \quad \rr\in{\bf R}_{(\rr_0,\rr_\infty)}.
	\end{equation*}
	
	\begin{Theorem}[\textbf{Local-in-time theory for pathwise classical solutions to \eqref{SEuler-(rho u)}}]
		\label{Thm-(rho u)}
		Let $d \geq 2$ and $p\ge 1$.  
		Assume \ref{Hypo-W L}, \ref{Hypo-Qi}, \ref{Hypo-Pressure},  and \ref{Hypo-h} with $\sigma> \frac{d}{2} + p$. Let $s > \frac{d}{2} + p + \max \{3\zeta,1\}$, where $\zeta=\zeta_{\Q_1,\Q_2}$ is given in \eqref{zeta-Q12}. Let $(\rho_0,u_0)\in\H^s_{\rr}$ be an $\mathcal{F}_{0}$-measurable random variable.
		
		\begin{enumerate}[label={\bf (\arabic*)},leftmargin=0.79cm]\setlength\itemsep{0.2em}
			
			\item  
			If \ref{Constant Sound} holds, then the Cauchy problem \eqref{SEuler-(rho u)} 
			admits a unique maximal $\H^s_{\rr}$ classical solution $((\rho,u),\tau^{*})$ in the sense of Definition \ref{solution definition (rho u)}. Moreover, 
			\begin{equation}\label{X D(0,tau)}
				(\rr(\rho), u)\in D([0,\tau^*);\H^{s'})\quad \forall s'<s\quad \pas,
			\end{equation}
			the lifetime $\tau^*$ is independent of $s$, and the following blow-up criterion holds:
			\begin{equation}\label{Blow-up criterion (rho u)}
				\limsup_{t \rightarrow \tau^*} \|(\rr(\rho), u)(t)\|_{\WP} = \infty \quad \text{a.s. on } \{\tau^* < \infty\}.
			\end{equation}
			
			\item  If \ref{General Sound} holds and we further assume that 
			$$\text{either}\quad \p\Big(0=\rr_0<\underline{\rr}\le \vrho_0\le \overline{\rr}<\rr_\infty=\infty\Big)=1\quad \text{or}
			\quad \p\Big(-\infty=\rr_0<\underline{\rr}\le\vrho_0\le\overline{\rr}<\rr_\infty=0\Big)=1.$$
			Then \eqref{SEuler-(rho u)} also admits a unique maximal $\H^s_{\rr}$ classical solution $((\rho,u),\tau^{*})$ satisfying \eqref{X D(0,tau)} in the sense of Definition \ref{solution definition (rho u)}. Furthermore, the lifetime $\tau^*$ is independent of $s$ and the solution satisfies the blow-up criterion \eqref{Blow-up criterion (rho u)}.
		\end{enumerate}
	\end{Theorem}

	\begin{Remark}[\textit{Methodological Advancements}]\label{Remark-Thm-(u rho)}
		Theorem \ref{Thm-(rho u)} establishes a unified existence theory for \eqref{SEuler-(rho u)} that applies simultaneously to both $\mathbb{T}^d$ and $\mathbb{R}^d$. Notably, this approach obviates the need for the classical martingale formulation---which typically relies on tightness, Skorokhod's representation, and the Yamada--Watanabe principle---standardly employed for constructing solutions on $\mathbb{T}^d$ (cf.\ \cite{Breit-Feireisl-Hofmanova-2018-CPDE, Breit-Hofmanova-2016-IUMJ, Hofmanova-Koley-Sarkar-2022-CPDE}). Furthermore, in contrast to much of the existing literature on stochastic fluid models, our results accommodate arbitrary $\mathcal{F}_0$-measurable initial data $u_0$ without imposing any \textit{a priori} moment bounds. This level of generality is achieved by systematically performing the  conditional expectation $\mathbb{E}[\bullet|\mathcal{F}_0]$. Lastly, we emphasize that the absence of viscous dissipation renders the classical Gelfand-triple framework for monotone SPDEs (e.g., \cite{Prevot-Rockner-2007-book, DaPrato-Zabczyk-2014-Book}) inapplicable to our singular setting.
	\end{Remark}

	\subsection{Proof of Theorem \ref{Thm-(rho u)}}
	\label{Section : Proof of Compressible}

	Theorem \ref{Thm-(rho u)} follows from Propositions \ref{Proposition-(q u)} (which handle the unconstrained version \eqref{Target problem X=(q u)}) and \ref{Proposition-(q u)-admissible} (which verifies \eqref{a<q<b}) stated below. To prove these propositions, we approximate the nonlinear term $F$ defined in \eqref{F-Euler}. Recall that $\H^s= H_1^s \times H_d^s$ (see \eqref{Hs Lp Wp pair}).

	\begin{Lemma}\label{Lemma:Fn}
		Let $F(\cdot)$ be defined in \eqref{F-Euler}. There exists a sequence of locally Lipschitz functions $\{F_n:\H^s\to \H^s\}_{n\ge1}$ with $s>\frac{d}{2}+1$ such that the following properties hold:
		\begin{itemize}[leftmargin=0.79cm]\setlength\itemsep{0.2em}
			\item If $s>\frac{d}{2}+1$ and $\theta\in(\frac{d}{2},s-1)$, $\{X_n,X\}_{n\ge1}\subset \H^s$ is bounded, and $\lim_{n\to\infty}X_n=X$ in $\H^\theta$, then
			$$\lim_{n\to\infty} F_n(X_n)=F(X)\quad \text{in}\quad \H^{\theta-1}.$$
			\item If \ref{Constant Sound} holds, then 
			\begin{equation*}
				\sup_{n\ge1}\Abs{\IP{F_n(X), X}_{\H^s}}\lesssim  \big(\|\nabla u\|_{L^\infty}+\|\nabla \vrho\|_{L^\infty}\big)\|X\|_{\H^s}^2,\quad X=(\vrho,u)^T\in \H^{s},\quad s>\frac{d}{2}+1.
			\end{equation*}
			Moreover, if $s>\frac{d}{2}+2$ and $\theta\in(\frac{d}{2}+1,s-1)$, then 
			\begin{align*}
				\IP{F_n(X)-F_m(Y),X-Y}_{\H^\theta} 
				\lesssim \Big(1+\|X\|^4_{\H^s}+\|Y\|^4_{\H^s}\Big)
				\left((n\wedge m)^{-2(s-1-\theta)}+\|X-Y\|^2_{\H^\theta}\right), \quad X,Y\in\H^s.
			\end{align*}
			
			\item If \ref{General Sound} holds, then 
			\begin{equation*}
				\sup_{n\ge1}\Abs{\IP{F_n(X), X}_{\H^s}}\lesssim \big(1+\Phi_\Lambda(\|\vrho\|_{L^\infty})\big)\big(\|\nabla u\|_{L^\infty}+\|\nabla \vrho\|_{L^\infty}\big)\|X\|_{\H^s}^2,\quad X=(\vrho,u)^T\in \H^{s}.
			\end{equation*}
			Additionally, if $s>\frac{d}{2}+2$ and $\theta\in\left(\frac{d}{2}+1,s-1\right)$, then for all $m,n\ge1$,
			\begin{align*}
				&\IP{F_n(X)-F_m(Y),X-Y}_{\H^\theta} \\
				\lesssim\ &  \Big(1+ \Phi^2_{\Lambda}\big(\|X\|_{\H^{s}}+\|Y\|_{\H^{s}}\big)+\|X\|^4_{\H^s}+\|Y\|^4_{\H^s}\Big)
				\left((n\wedge m)^{-2(s-1-\theta)}+\|X-Y\|^2_{\H^\theta}\right),\quad X,Y\in\H^s.
			\end{align*}
			
		\end{itemize}
		
	\end{Lemma}
	\begin{proof}
		Let $J_n$ be defined in \eqref{Define Jn} and define $F_n(\cdot)\triangleq J_nF(J_n\cdot).$
		Then the results come from the definition of $F(\cdot)$ in \eqref{F-Euler},  Lemmas \ref{Lemma-F} and \ref{Lemma-F-difference}.
	\end{proof}

	\subsubsection{The Unconstrained Problem \eqref{Target problem X=(q u)}}\label{Section : Unconstrained Problem}
	We will first ignore \eqref{a<q<b} and focus   on the Cauchy problem \eqref{Target problem X=(q u)} alone.  
	The main result for the unconstrained problem \eqref{Target problem X=(q u)} is 
	\begin{Proposition}
		\label{Proposition-(q u)}
		Let $d \geq 2$. 
		Assume \ref{Hypo-W L}, \ref{Hypo-Pressure}, \ref{Hypo-Qi}, and \ref{Hypo-h} with $p\ge 1$ and $\sigma > \frac{d}{2} + p$. Let $s > \frac{d}{2} + p + \max \{3\zeta,1\}$, where $\zeta=\zeta_{\Q_1,\Q_2}$ is given in \eqref{zeta-Q12}.  Then, for any $\H^{s}$-valued $\mathcal{F}_{0}$-measurable random variable $\Xi$, the Cauchy problem \eqref{Target problem X=(q u)} admits a unique maximal $\H^s$ classical solution $(X,\tau^{*})$ in the sense of Definition~\ref{solution definition (q u)}. Moreover,
		\begin{equation*}
			X\in D([0,\tau^*);\H^{s'})\quad \forall\, s'<s\quad \pas  
		\end{equation*}
		and the following blow-up criterion holds:
		\begin{equation}
			\label{X blow-up criterion}
			\limsup _{t \rightarrow \tau ^{*}}\|X(t)\|_{\WP} = \infty   \  \text{on}\ \{\tau^*<\infty\}\ \ \mathrm{a.s.}
		\end{equation}
		
	\end{Proposition}

	The proof is divided into several steps.

	\paragraph{\textbf{Step 1: Approximation Scheme.}}
	
	Let $\chi\in C^\infty([0,\infty);[0,1])$ such that
	$\chi=1$ on $[0,1]$ and $\chi=0$ on $[2,\infty)$. For any $R\ge1$, define $\chi_R\in C^\infty([0,\infty);[0,1])$ as
	\[
	\chi_R(x)\triangleq\chi(x/R).
	\]
	Then it is clear that $\|\chi'_R\|_{L^\infty}\leq \|\chi'\|_{L^\infty}$. 
	
	Let $\{\QQ_{i,n}={\rm diag}(0,\Q_{i,n})\}_{n\ge1} $ ($i=1,2$), where $\{\Q_{i,n}\}_{n\ge1}$ is given by Lemma \ref{Lemma:Qn}. Let  $\{F_n\}_{n\ge 1}$ be the sequence given by Lemma  \ref{Lemma:Fn}. For any $n,R\ge1$, we consider the following approximate  scheme for \eqref{Target problem X=(q u)}:
	\begin{equation}\label{eq : appro compressible}
		{\rm d}X + \chi_R\big(\|X-\Xi\|_{\WP}\big)F_n (X)\d t
		=\QQ_{1,n}X\,\circ\, {\rm d} W+\QQ_{2,n}X\diamond {\rm d}L + \chi_R\big(\|X-\Xi\|_{\WP}\big)\ZZ(t,X)\d \widetilde W,\quad X(0)=\Xi. 
	\end{equation}
	As before,  by  \eqref{Stratonovich to Ito},   we have 
	\begin{align}
		\int_0^t \QQ_{1,n} X\,\circ\, {\rm d} W(t')=\int_0^t \frac{1}{2} \QQ_{1,n}^2X(t')\d t'+\int_0^t \QQ_{1,n} X(t')\d W(t').\label{Stratonovich integral Q1-n}
	\end{align}
	Besides, the Marcus integral $\int_0^t \QQ_{2,n} X(t')\diamond {\rm d} L(t')$   is defined as
	\begin{align}
		&\int_0^t \QQ_{2,n} X(t')\diamond {\rm d} L(t') \notag\\
		\triangleq \ &
		\int_0^t\int_{|l|\le1} \Big\{\M_n\big(1,l,X(t'-)\big)-X(t'-)\Big\}\, \widetilde{\eta}({\rm d}l,{\rm d}t')\notag\\
		&+\int_0^t\int_{|l|\le1} \Big\{\M_n\big(1,l,X(t')\big)-X(t')-l\cdot \QQ_{2,n} X(t')\Big\}\, \nu({\rm d}l){\rm d}t',\label{Marcus integral Q2-n}
	\end{align}
	where 
	$$
	\M_n(r)\triangleq \M_n(r,l,Y)
	$$
	solves the following $\H^s$-valued flow:
	\begin{equation}\label{Marcus flow Q2-n}
		\frac{{\rm d}}{{\rm d}r}\M_n(r)= l\cdot \QQ_{2,n} \M_n(r),\quad r\in[0,1],  \quad \M_n(0) = Y\in \H^s.
	\end{equation}
	For $Y\in\H^s$, we define
	\begin{equation}\label{GnR HR}
		G_{n,R}(Y) \triangleq   -\chi_R\big(\|Y-\Xi\|_{\WP}\big)J_n F(J_nY)
		+\frac{1}{2} \QQ_{1,n}^2Y,\quad
		H_R(t,Y)\triangleq \chi_R\big(\|Y-\Xi\|_{\WP}\big)\ZZ(t,Y).
	\end{equation}
	In summary, by \eqref{Stratonovich integral Q1-n}, \eqref{Marcus integral Q2-n}, \eqref{GnR HR}, and \ref{Hypo-W L},  the approximation scheme can be rewritten as
	\begin{align}\label{approximation scheme C}
		X(t)-\Xi\  = \ & \int_0^t G_{n,R}(X(t'))\d t'+\int_0^t \QQ_{1,n} X(t')\d W(t')+\int_0^tH_R(t',X(t'))\d \widetilde W(t')\notag\\
		\ & + \int_0^t\int_{|l|\le1} \Big\{\M_n\big(1,l,X(t'-)\big)-X(t'-)\Big\}\, \widetilde{\eta}({\rm d}l,{\rm d}t')\notag\\
		\ & +\int_0^t\int_{|l|\le1} \Big\{\M_n\big(1,l,X(t')\big)-X(t')-l\cdot \QQ_{2,n} X(t')\Big\}\, \nu({\rm d}l){\rm d}t'.
	\end{align}
	
	\begin{Remark}
		From the definition of $\QQ_{2,n}$, Lemma \ref{Lemma:Qn}, and \eqref{regular Marcus flow Q2}, we obtain
		\begin{equation*}
			\M_n(r,l,Y)=  (f,\wp_{n}(r,l,g))^T,\quad Y=(f,g),
		\end{equation*} 
		where $\wp_{n}$ is given by \eqref{regular Marcus flow Q2}. 
		Consequently, all estimates on $\wp_{n}$ with values in $H_d^s$ (established in Lemmas \ref{Lemma-Marcus-n}, \ref{Lemma-Stability wp-n}, and \ref{Lemma-Q2 n m}) also hold for $\M_n(r)$ with values in $\H^s$. In what follows, we will frequently use this fact without further mention.
	\end{Remark}
	\begin{Lemma}\label{Lemma: existence of XnR}
		Let the conditions in Proposition \ref{Proposition-(q u)} hold. Then, for any $n,R\ge1$ and for any $\F_0$-measurable $\H^s$-valued random variable $\Xi$, the Cauchy problem
		\eqref{approximation scheme C} admits a unique solution $X_n^{(R)}(t)\in D([0,\infty);\H^s)$ such that
		\begin{align}\label{approximation solution XnR}
			X_n^{(R)}(t)-\Xi\  = \ & \int_0^t G_{n,R}(X_n^{(R)}(t'))\d t'+\int_0^t \QQ_{1,n} X_n^{(R)}(t')\d W(t')+\int_0^tH_R(t',X_n^{(R)}(t'))\d \widetilde W(t')\notag\\
			\ & + \int_0^t\int_{|l|\le1} \Big\{\M_n\big(1,l,X_n^{(R)}(t'-)\big)-X_n^{(R)}(t'-)\Big\}\, \widetilde{\eta}({\rm d}l,{\rm d}t')\notag\\
			\ & +\int_0^t\int_{|l|\le1} \Big\{\M_n\big(1,l,X_n^{(R)}(t')\big)-X_n^{(R)}(t')-l\cdot \QQ_{2,n} X_n^{(R)}(t')\Big\}\, \nu({\rm d}l){\rm d}t'.
		\end{align}
	\end{Lemma}
	
	\begin{proof}
		For any $n,R\ge1$, the truncated coefficients $G_{n,R}(\cdot)$ and $H_{R}(t,\cdot)$ are locally Lipschitz continuous in $\cdot\in\H^s$ locally uniform in $t$. Since the regularized Marcus flow $\M_n$ solves a linear equation, $\M_n\big(1,l,\cdot\big)$ is Lipschitz continuous in $\cdot\in \H^s$. Therefore, for any deterministic initial data, \eqref{approximation scheme C} has a unique solution  (see for instance \cite[Theorem 4.9]{Mandrekar-Rudiger-2006-Stoch}) with c\`adl\`ag path. Moreover, we note that the truncated coefficients $G_{n,R}(\cdot)$ and $H_{R}(t,\cdot)$  exhibit linear growth; consequently, the solution is global, that is, it belongs to
		$D([0,\infty);\H^s)$.  		
		Because
		$\F_0$ is independent of the system,  for any $\F_0$-measurable $\H^s$-valued random variable $\Xi$, 
		\eqref{approximation scheme C} also admits a unique solution $X_n^{(R)}(t)$, and $X_n^{(R)}(t)\in D([0,\infty);\H^s)$.
	\end{proof}

	\paragraph{\textbf{Step 2: Estimates for Approximate Solutions.}} 
	\begin{Lemma}\label{Lemma : Xn T estimates} Let the conditions in Proposition \ref{Proposition-(q u)} hold and let $\theta\in\left(\frac{d}{2}+p,{s-\max \{3\zeta,1\}}\right)$. Then $\pas$ the following properties hold:
		\begin{enumerate}[label={\bf (\arabic*)},leftmargin=0.79cm]\setlength\itemsep{0.2em}
			\item \label{Global cut-off estimate} There is a function $\hh K: [0,\infty)\times [0,\infty)\to (0,\infty)$, increasing in both variables, such that  for any $R\ge1$ and $T>0$,
			\begin{align}
				\sup_{n\ge1}\E\Big[\sup_{t\in[0,T]}\|X_n^{(R)}(t)\|^2_{\H^s}\Big|\F_0\Big] 
				\leq \hh K(T,2R+\|\Xi\|_{\WP})(1+\|\Xi\|_{\H^s}^2).\label{Xn uniform bound} 
			\end{align} 
			\item \label{Convergence of X-n TnmN}  Let
			$$\tau_N^{n,m}(R) \triangleq  N\land \inf\big\{t\ge 0: \|X_n^{(R)}(t)\|_{\H^s}\lor\|X_m^{(R)}(t)\|_{\H^s}\ge N\big\},\quad n,m\ge 1, \ N\ge1.$$
			Then 
			\begin{equation} \label{Cauchy in E-M}
				\lim_{n\rightarrow\infty}\sup_{m\ge n}
				\E\left[\sup_{t\in[0,\tau_N^{n,m}(R)]} \|X_n^{(R)}(t)-X_m^{(R)}(t)\|^2_{\H^\theta}\bigg|\F_0\right]=0,\quad N>0.
			\end{equation}
			
			\item \label{Convergence of X-n T} There exists an $\mathcal{F}_t$-progressively measurable $\H^s$-valued process $X^{(R)}=(X^{(R)}(t))_{t\ge 0}$ satisfying 
			\begin{equation}\label{X Hs bound}
				\E\Big[\sup_{t\in [0,T]} \|X^{(R)}(t)\|_{\H^s}^2\Big|\F_0\Big]\le \hh K(T,2R+\|\Xi\|_{\WP})(1+\|\Xi\|^2_{\H^s}),\quad T>0,
			\end{equation} 
			such that for some subsequence of $\{X_n^{(R)}\}_{n\ge1}$ $($still denoted by $\{X_n^{(R)}\}_{n\ge1}$ for simplicity$)$ and for all $s'\in[\theta,s)$,
			\begin{equation}\label{Xn to X}
				\lim_{n\to\infty}\sup_{t\in[0,T]}\|X_n^{(R)}(t)-X^{(R)}(t)\|_{\H^{s'}}=0\quad \text{and}\quad  X_n^{(R)}\xrightarrow[]{n\to \infty}X^{(R)} \  {\rm in}\ D([0,T];\H^{s'}),\quad T>0
				\quad \pas
			\end{equation}
			
		\end{enumerate} 
	\end{Lemma}
	\begin{proof}
		\ref{Global cut-off estimate}  
		Using It\^o's formula for \eqref{approximation solution XnR}, we arrive at
		\begin{align*}
			&\|X_n^{(R)}(t)\|^2_{\H^s}- \|\Xi\|^2_{\H^s} \\
			=\ & 2\int_0^t \left(\IP{G_{n,R}(X_n^{(R)}(t'),X_n^{(R)}(t')}_{\H^s}+\|\QQ_{1,n} X_n^{(R)}(t')\|_{H^s}^2+ \|H_{R}(t',X_n^{(R)}(t'))\|_{\H^s}^2\right)\d t'\\
			&+2 \int_0^t\IP{\QQ_{1,n} X_n^{(R)}(t'),X_n^{(R)}(t')}_{\H^s}\d W(t')
			+\int_0^t\IP{H_{R}(t',X_n^{(R)}(t')),X_n^{(R)}(t')}_{\H^s}\d \widetilde W(t')\\
			&+ \int_0^t \int_{|l|\le 1} \Big[\norm{\M_n\big(1,l,X_n^{(R)}(t'-)\big)}^2_{\H^s}-  \norm{X_n^{(R)}(t'-)}^2_{\H^s}\Big] \widetilde{\eta}({\rm d}l, {\rm d}t')  \\
			&+ \int_0^t \int_{|l|\le 1}\Big[\norm{\M_n\big(1,l,X_n^{(R)}(t')\big)}^2_{\H^s}-  \norm{X_n^{(R)}(t')}^2_{\H^s}- 2l \cdot\bIP{\QQ_{2,n} X_n^{(R)}(t'),X_n^{(R)}(t')}_{\H^s}\Big] \nu({\rm d}l) {\rm d}t'\\
			\triangleq& \sum_{i=1}^5 I_{i,n}(t).
		\end{align*}

		By \eqref{F(X) X-1}, \eqref{F(X) X-2} (with $J_n X_n$ replacing $X$), \ref{Hypo-h} (with $\sigma=s$), Theorem \ref{Thm-cancel} (with $\mathscr{O}_1\equiv \{\D^s\}$, $\mathscr{N}_1=\{J_n\Q_1J_n\}_{n\ge1}$), there is a function $\mathring{K}\in\mathscr K$ such that for all $X\in \H^s$, $n\ge1$ and $t\ge 0$,
		\begin{align}
			&2\IP{G_{n,R}(X), X}_{\H^s} +\|\QQ_{1,n} X\|_{\H^s}^2+\|H_{R}(t,X)\|_{\H^s}^2 \nonumber\\
			=\ & \chi_R\big(\|X-\Xi\|_{\WP}\big)\left(-2\IP{F(J_nX), J_nX}_{\H^{s}}+\|\ZZ(t,X)\|_{\H^{s}}^2\right)   +\IP{\QQ_{1,n}X,X}_{\H^{s}}+\|\QQ_{1,n} X\|_{\H^s}^2\nonumber\\
			\le\ & \mathring{K}(t,\|\Xi\|_{\WP}+2R) (1+\|X\|_{\H^s}^2),\label{Ito-1}
		\end{align} 
		\begin{align} 
			\bIP{\QQ_{1,n} X, X}_{\H^s}^2 \lesssim  \|X\|_{\H^s}^4,\label{Ito-2}
		\end{align}
		and 
		\begin{align} 
			\bIP{H_{R}(t,X), X}_{\H^s}^2 =  \chi^2_R\big(\|X(t)-\Xi\|_{\WP}\big)\IP{\ZZ(t,X),X}_{\H^{s}}^2\leq \mathring{K}(t,\|\Xi\|_{\WP}+2R) \big(1+ \|X\|_{\H^s}^4\big).\label{Ito-3}
		\end{align}
		Then, as in Remark \ref{Remark: wp estimate},
		we infer from \eqref{wp-n estimate 1} and \eqref{wp-n estimate 2} (with $V(x)=x$) and \ref{Hypo-W L} that
		\begin{align}
			\norm{\M_n(1, l, X)}_{\H^s}^2 - \norm{X}_{\H^s}^2 \lesssim \left({\rm e}^{2C_{2,1}|l|}-1\right)\norm{X}_{\H^s}^2
			\lesssim |l|  \norm{X}_{\H^s}^2,\label{Ito-4}
		\end{align}
		and 
		\begin{align}
			\|\M_n\big(1,l,X\big)\|^2_{\H^s}-\|X\|^2_{\H^s} -2l\bIP{\QQ_{2,n}X,X}_{\H^s}
			\lesssim \frac{C_{2,2}}{2C_{2,1}^2}\left({\rm e}^{2C_{2,1}|l|}-2C_{2,1}|l|-1\right)\norm{X}_{\H^s}^2
			\lesssim |l|^2\|X\|^2_{\H^s}.
			\label{Ito-5}
		\end{align}
		By \eqref{Ito-1} and \eqref{Ito-5}, we obtain
		\begin{align*}
			&\E \left[\sup_{t \in [0, T]} |I_{1,n}(t)|+|I_{5,n}(t)|\Big|\F_0\right] \\
			\lesssim \ &  
			\E\left[\int_0^{T} \left(\mathring{K}(t',\|\Xi\|_{\WP}+2R)  
			\Big(1+\|X_n^{(R)}(t')\|_{\H^s}^2\Big)+\|X_n^{(R)}(t')\|_{\H^s}^2\right)\d t'\bigg|\F_0\right]\\  
			\lesssim\ & \int_0^T\mathring{K}(t,\|\Xi\|_{\WP}+2R)\d t+ \int_0^T\left(1+\mathring{K}(t,\|\Xi\|_{\WP}+2R) \right)\E\left[\sup_{t'\in [0,t]} \|X_n^{(R)}(t')\|_{\H^s}^2\Big|\F_0\right]\d t.
		\end{align*}
		Note that
		$I_{2,n}(t)$, $I_{3,n}(t)$ and $I_{4,n}(t)$  are martingales. For any $T>0$,  on account of the Burkholder-Davis-Gundy (BDG) inequality, \eqref{Ito-2}, \eqref{Ito-3}, \eqref{Ito-4} and the fact $\int_{|l|\le 1} |l|^2\nu({\rm d}l)<\infty$, we find constants $c_1,c_2>0$ such that
		\begin{align*}
			\E \left[\sup_{t \in [0, T]} |I_{2,n}(t)|\Big|\F_0\right] 
			\lesssim \ &  \E \left[\left(\int_0^T\IP{\QQ_{1,n} X_n^{(R)}(t'),X_n^{(R)}(t')}^2_{\H^s}{\rm d}t'\right)^{\frac{1}{2}}\bigg|\F_0\right]\\
			\leq  \ & c_1\E \left[\left(\int_0^T\|X_n^{(R)}(t')\|_{\H^s}^4{\rm d}t'\right)^{\frac{1}{2}}\bigg|\F_0\right]\\
			\leq\ & \frac{1}{6}\E\left[\sup_{t \in [0, T]} \|X_n^{(R)}(t)\|_{\H^s}^2\Big|\F_0\right]
			+c_2\int_0^T\E\left[ \sup_{t' \in [0, t]} \|X_n^{(R)}(t')\|_{\H^s}^2\Big|\F_0\right]{\rm d}t,
		\end{align*}
		\begin{align*}
			\E \left[\sup_{t \in [0, T]} |I_{3,n}(t)|\big|\F_0\right]
			\lesssim \ &  \E \left[\left(\int_0^T\IP{H_{R}(t',X_n^{(R)}(t')),X_n^{(R)}(t')}^2_{\H^s}{\rm d}t'\right)^{\frac{1}{2}}\bigg|\F_0\right]\\
			\le \ &
			c_1 \E\bigg[\bigg(\int_0^{T}\mathring{K}(t',\|\Xi\|_{\WP}+2R) 
			\Big(1+\|X_n^{(R)}(t')\|_{\H^s}^4\Big)\d t'\bigg)^{\frac 1 2}\bigg|\F_0\bigg]\\  
			\le\ & \frac{1}{6}\E\left[\sup_{t \in [0, T]} \norm{X_n^{(R)}(t)}_{\H^s}^2\Big|\F_0\right] + \int_0^T\mathring{K}(t,\|\Xi\|_{\WP}+2R)\d t\\
			&+ c_2 \int_0^T\mathring{K}(t,\|\Xi\|_{\WP}+2R) \E\left[\sup_{t'\in [0,t]} \|X_n^{(R)}(t')\|_{\H^s}^2\Big|\F_0\right]\d t,
		\end{align*}
		\begin{align*}
			\E \left[\sup_{t \in [0, T]} |I_{4,n}(t)|\Big|\F_0\right]
			\lesssim \ & \E \left[\left(\int_0^T\int_{|l|\le 1} \Big[\norm{\M_n\big(1,l,X_n^{(R)}(t')\big)}^2_{\H^s}-  \norm{X_n^{(R)}(t')}^2_{\H^s}\Big]^2\nu({\rm d}l) {\rm d}t'\right)^{\frac{1}{2}}\bigg|\F_0\right]\\
			\leq  \ & c_1\E \left[\left(\int_0^T\norm{X_n^{(R)}(t')}_{\H^s}^4{\rm d}t'\right)^{\frac{1}{2}}\bigg|\F_0\right]\\
			\leq\ & \frac{1}{6}\E\left[\sup_{t \in [0, T]}\|X_n^{(R)}(t)\|_{\H^s}^2\Big|\F_0\right]
			+c_2\int_0^T\E\left[ \sup_{t' \in [0, t]} \|X_n^{(R)}(t')\|_{\H^s}^2\Big|\F_0\right]{\rm d}t.
		\end{align*}
		Thus, we arrive at
		\begin{align*}
			&\E\bigg[\sup_{t\in[0,T]}\|X_n^{(R)}(t)\|^2_{\H^s}\bigg|\F_0\bigg]- 2\|\Xi\|^2_{\H^s}\\
			\le\ &  2\int_0^T\mathring{K}(t,\|\Xi\|_{\WP}+2R)\d t + 2\int_0^T\Big(1+2 c_2+\mathring{K}(t,\|\Xi\|_{\WP}+2R) \Big)\E\left[\sup_{t'\in [0,t]} \|X_n^{(R)}(t')\|_{\H^s}^2\Big|\F_0\right]\d t.
		\end{align*}
		By Gr\"{o}nwall's inequality, there exists a function $\hh K: [0,\infty)\times [0,\infty)\to (0,\infty)$ increasing in both variables such that 
		\eqref{Xn uniform bound} holds true.

		\ref{Convergence of X-n TnmN} 
		For $n,m\ge 1$, $t\ge0$ and $|l|\in[0,1]$, we 
		let  
		\begin{equation*}
			Z_{n,m}^{(R)}(t)\triangleq X_n^{(R)}(t)-X_m^{(R)}(t),
		\end{equation*}
		\begin{equation*}
			\textbf{B}^{(R)}_{n,m}(r,l,t)\triangleq \M_n\big(r,l,X_n^{(R)} (t)\big)-\M_m\big(r,l,X_m^{(R)} (t)\big)-Z_{n,m}^{(R)} (t),\quad r\in[0,1],
		\end{equation*}
		\begin{equation*}
			\textbf{C}^{(R)}_{n,m}(r,l,t)\triangleq \textbf{B}^{(R)}_{n,m}(r,l,t)
			-l \cdot \left(\QQ_{2,n} X_n^{(R)} (t)-\QQ_{2,m} X_m^{(R)} (t)\right),\quad r\in[0,1].
		\end{equation*}
		By \eqref{approximation scheme C}, the linearity of operators $\QQ_1$, $\QQ_2$ and the map $\M_n(1,l,\cdot)$, we have that
		\begin{align}\label{problem Znm-integral}
			Z_{n,m}^{(R)} (t)  
			=\ & \int_0^t \mathfrak{i}_{1}^{(n,m,R)}(t') \d t'
			+\int_0^t\int_{|l|\le1} \textbf{C}^{(R)}_{n,m}(1,l,t')\, \nu({\rm d}l){\rm d}t'\notag\\
			&+\int_0^t 
			\mathfrak{i}_{2}^{(n,m,R)}(t') \d W(t') 
			+\int_0^t\mathfrak{i}_{3}^{(n,m,R)}(t')\d \widetilde W(t') 
			+ \int_0^t\int_{|l|\le1} \textbf{B}^{(R)}_{n,m}(1,l,t'-)\, \widetilde{\eta}({\rm d}l,{\rm d}t'),
		\end{align}
		where for $t\ge0$,
		\begin{align*}
			&\mathfrak{i}_{1}^{(n,m,R)}(t)\triangleq G_{n,R}\big(X_n^{(R)}(t)\big)-G_{m,R}\big(X_m^{(R)}(t)\big),\\
			&\mathfrak{i}_{2}^{(n,m,R)}(t)\triangleq \QQ_{1,n} X_n^{(R)}(t) -\QQ_{1,m} X_m^{(R)}(t),\\
			&\mathfrak{i}_{3}^{(n,m,R)}(t)\triangleq H_{R}(t,X_n^{(R)}(t))-H_{R}(t,X_m^{(R)}(t)).
		\end{align*}
		Then we apply It\^o's formula (cf. Lemma \ref{Ito formula}) to \eqref{problem Znm-integral} to obtain that for $t\ge0$,
		\begin{align}
			\|Z_{n,m}^{(R)} (t)\|_{\H^\theta}^2 
			=\ &
			2\int_0^{t}\underbrace{\left\{\IP{ \mathfrak{i}_{1}^{(n,m,R)}(t'),Z_{n,m}^{(R)} (t')}_{\H^\theta}+\|\mathfrak{i}_{2}^{(n,m,R)}(t')\|^2_{\H^\theta}
				+\|\mathfrak{i}_{3}^{(n,m,R)}(t')\|^2_{\H^\theta}\right\}}_{\triangleq\ \mathfrak{R}_1^{(n,m,R)}(t')}
			\d t'\notag\\
			&+2\int_0^{t}
			\underbrace{\IP{\mathfrak{i}_{2}^{(n,m,R)}(t'),Z_{n,m}^{(R)} (t')}_{\H^\theta}}_{\triangleq\ \mathfrak{R}_2^{(n,m,R)}(t')}
			\d W(t')
			+2\int_0^{t}
			\underbrace{\IP{\mathfrak{i}_{3}^{(n,m,R)}(t'),Z_{n,m}^{(R)} (t')}_{\H^\theta}}_{\triangleq\ \mathfrak{R}_3^{(n,m,R)}(t')}
			\d \widetilde W(t')\notag\\
			&+ \int_0^{t}\int_{|l|\le1}
			\underbrace{ \left\{\Big\|\M_n\big(1,l,X_n^{(R)} (t'-)\big)-\M_m\big(1,l,X_m^{(R)} (t'-)\big)\Big\|^2_{\H^\theta} 
				-\norm{Z_{n,m}^{(R)} (t'-)}^2_{\H^\theta}\right\}}_{\triangleq\ \mathfrak{R}_4^{(n,m,R)}(t'-)}
			\, \widetilde \eta(\mathrm{d}l,\mathrm{d}t')\notag\\
			&+\int_0^{t}\int_{|l|\le1}
			\underbrace{-2l\cdot \IP{\QQ_{2,n} X_n^{(R)} (t')-\QQ_{2,m} X_m^{(R)} (t'),Z_{n,m}^{(R)} (t')}_{\H^\theta}}_{\triangleq\ \mathfrak{R}_5^{(n,m,R)}(t')}
			\,\nu({\rm d}l) {\rm d}t'\notag\\
			&+ \int_0^{t}\int_{|l|\le1} 
			\left\{\Big\|\M_n\big(1,l,X_n^{(R)} (t')\big)-\M_m\big(1,l,X_m^{(R)} (t')\big)\Big\|^2_{\H^\theta} 
			-\norm{Z_{n,m}^{(R)} (t')}^2_{\H^\theta}\right\}
			\,\nu({\rm d}l) {\rm d}t'.\label{Znm-Ito}
		\end{align}
		Thanks to the embedding $\H^s\hookrightarrow\H^\theta\hookrightarrow \WP$ (as $s>\theta>\frac{d}{2}+p$), 
		Lemmas \ref{Lemma-Jn}, \ref{Lemma-F} and \ref{Lemma-F-difference}, and the property  $\chi_R\in C^{\infty}([0,\infty);[0,1])$,  we have
		\begin{align*}
			&\IP{\chi_R\big(\|X_n^{(R)} (t)-\Xi\|_{\WP}\big)J_n F(J_nX_n^{(R)} (t))-\chi_R\big(\|X_m^{(R)} (t)-\Xi\|_{\WP}\big)J_m F(J_m X_m^{(R)} (t)),\ Z_{n,m}^{(R)} (t)}_{\H^\theta}\\
			=\ & \IP{\Big(\chi_R\big(\|X_n^{(R)} (t)-\Xi\|_{\WP}\big)-\chi_R\big(\|X_m^{(R)} (t)-\Xi\|_{\WP}\big)\Big)J_n F(J_nX_n^{(R)} (t)),\ Z_{n,m}^{(R)} (t)}_{\H^\theta}\\
			& +\IP{\chi_R\big(\|X_m^{(R)} (t)-\Xi\|_{\WP}\big)\Big(J_n F(J_nX_n^{(R)} (t))-J_m F(J_m X_m^{(R)} (t))\Big),\ Z_{n,m}^{(R)} (t)}_{\H^\theta}\\
			\lesssim \ & 
			\bigg(1+\Phi^2_{\Lambda}\Big(\|X_n^{(R)} (t)\|_{\H^{s}}+\|X_m^{(R)} (t)\|_{\H^{s}}\Big)
			+\|X_n^{(R)} (t)\|^4_{\H^s}+\|X_m^{(R)} (t)\|^4_{\H^s}\bigg)\left[(n\wedge m)^{-2(s-1-\theta)}+\|Z_{n,m}^{(R)} (t)\|_{\H^\theta}^2\right].
		\end{align*}
		Similarly, we can infer from \ref{Hypo-h} (with $\sigma=\theta$) that
		\begin{align*}
			\|\mathfrak{i}_{3}^{(n,m,R)}(t)\|^2_{\H^{\theta}}
			\lesssim 
			K\Big(t,\|X_n^{(R)} (t)\|_{\H^s}+\|X_m^{(R)} (t)\|_{\H^s}\Big)
			\Big(1+\|X_n^{(R)} (t)\|^2_{\H^s}\Big)
			\|Z_{n,m}^{(R)} (t)\|^{2}_{\H^{\theta}},
		\end{align*}
		where $K(\cdot,\cdot)\in\mathscr{K}$ (cf. \ref{Hypo-h} and \eqref{SCRK}). 
		From the above two estimates, \ref{Hypo-h},  
		Lemma \ref{Lemma-Q1 n m}, we can find a constant $C_N>0$ and a function $\lambda_{\cdot,\cdot}:\N\times\N\to(0,1)$ with $\displaystyle \lim_{n,m\to\infty} 
		\lambda_{n,m}=0$ such that for all $T\in [0,N]$,
		\begin{align}
			2\E\left[\int_0^{T\land \tau_N^{n,m}(R)}|\mathfrak{R}_1^{(n,m,R)}(t)|\d t\bigg|\F_0\right]\leq C_N\lambda_{n,m}+ C_N
			\E\left[\int_0^{T}\sup_{t'\in [0,t\land \tau_N^{n,m}(R)]}\|Z_{n,m}^{(R)} (t')\|^{2}_{\H^{\theta}}\d t\bigg|\F_0\right].\label{R-1-nmR estiamte}
		\end{align}
		In the same way, by using
		the BDG inequality, Lemmas \ref{Lemma-Q1 n m} and \ref{Lemma-Q2 n m}, \ref{Hypo-h}, we can  find constants $c_1,c_2, C_N>0$ such that for all $T\in [0,N]$,
		\begin{align}
			& 2\E\left[\sup_{t \in [0, T\land \tau_N^{n,m}(R)]} \bigg|\int_0^{t}\mathfrak{R}_2^{(n,m,R)}(t')\d W(t')\bigg|\bigg|\F_0\right]\notag\\
			\lesssim \ &  \E \left[\left(\int_0^{T\land \tau_N^{n,m}(R)}\Big[\mathfrak{R}_2^{(n,m,R)}(t)\Big]^2\d t \right)^{\frac{1}{2}}\bigg|\F_0\right]\notag\\
			\leq  \ & c_1\E \left[\left(\int_0^T \bigg(C_N\lambda^2_{n,m}+\sup_{t'\in [0,t\land \tau_N^{n,m}(R)]}\|Z_{n,m}^{(R)} (t)\|_{\H^\theta}^4\bigg)\d t\right)^{\frac{1}{2}}\bigg|\F_0\right]\notag\\
			\leq\ &   \frac{1}{6}\E\left[\sup_{t \in [0, T\land \tau_N^{n,m}(R)]} \|Z_{n,m}^{(R)} (t')\|_{\H^\theta}^2\Big|\F_0\right]
			+c_2\int_0^T\E\left[\sup_{t'\in [0,t\land \tau_N^{n,m}(R)]}\|Z_{n,m}^{(R)} (t')\|_{\H^\theta}^2\bigg|\F_0\right]{\rm d}t+C_N\lambda_{n,m},\label{R-2-nmR estiamte}
		\end{align}
		\begin{align}
			& 2\E\left[\sup_{t \in [0, T\land \tau_N^{n,m}(R)]} \bigg|\int_0^{t}\mathfrak{R}_3^{(n,m,R)}(t')
			\d \widetilde W(t')\bigg|\bigg|\F_0\right]\notag\\
			\lesssim \ &  \E \left[\left(\int_0^{T\land \tau_N^{n,m}(R)}\Big[\mathfrak{R}_3^{(n,m,R)}(t)\Big]^2\d t \right)^{\frac{1}{2}}\bigg|\F_0\right]\notag\\
			\le \ &
			c_1 \E\bigg[\bigg(\int_0^{T}K\Big(t,\|X_n^{(R)} (t)\|_{\H^s}+\|X_m^{(R)} (t)\|_{\H^s}\Big)
			\Big(1+\|X_n^{(R)} (t)\|^2_{\H^s}\Big)
			\|Z_{n,m}^{(R)} (t)\|^{4}_{\H^{\theta}}\d t'\bigg)^{\frac 1 2}\bigg|\F_0\bigg]\notag\\  
			\le\ &  \frac{1}{6}\E\left[\sup_{t \in [0, T\land \tau_N^{n,m}(R)]} \|Z_{n,m}^{(R)} (t)\|_{\H^\theta}^2\Big|\F_0\right]
			+C_N\int_0^T\E\left[\sup_{t'\in [0,t\land \tau_N^{n,m}(R)]}\|Z_{n,m}^{(R)} (t')\|_{\H^\theta}^2\bigg|\F_0\right]{\rm d}t,\label{R-3-nmR estiamte}
		\end{align}
		\begin{align}
			& \E\left[\sup_{t \in [0, T\land \tau_N^{n,m}(R)]} \bigg|\int_0^{t}\int_{|l|\le1}\mathfrak{R}_4^{(n,m,R)}(t'-)
			\, \widetilde \eta(\mathrm{d}l,\mathrm{d}t')\bigg|\bigg|\F_0\right]\notag\\
			\lesssim \ & \E \left[\left(\int_0^{T\land \tau_N^{n,m}(R)}
			\int_{|l|\le 1} \Big[\mathfrak{R}_4^{(n,m,R)}(t)\Big]^2\, \nu({\rm d}l) {\rm d}t\right)^{\frac{1}{2}}\bigg|\F_0\right]\notag\\
			\leq  \ & c_1\E \left[\left(\int_0^{T\land \tau_N^{n,m}(R)}
			\int_{|l|\le1}|l|^2
			\Big(\lambda^2_{n,m}
			\big(\|X_n^{(R)} (t)\|_{\H^s}+\|X_m^{(R)} (t)\|_{\H^s}\big)^4+\|Z_{n,m}^{(R)} (t)\|^4_{\H^{\theta}}\Big) 
			\nu({\rm d}l) {\rm d}t\right)^{\frac{1}{2}}\bigg|\F_0\right]\notag\\
			\leq  \ & \frac{1}{6}\E\left[\sup_{t \in [0, T\land \tau_N^{n,m}(R)]} \|Z_{n,m}^{(R)} (t)\|_{\H^\theta}^2\Big|\F_0\right]
			+c_2\int_0^T\E\left[\sup_{t'\in [0,t\land \tau_N^{n,m}(R)]}\|Z_{n,m}^{(R)} (t')\|_{\H^\theta}^2\bigg|\F_0\right]{\rm d}t+C_N\lambda_{n,m},\label{R-4-nmR estiamte}
		\end{align}
		and 
		\begin{align}
			& \E\left[ \sup_{t \in [0, T\land \tau_N^{n,m}(R)]}\bigg|\int_0^{t}\int_{|l|\le1}
			\left(\mathfrak{R}_4^{(n,m,R)}(t')+\mathfrak{R}_5^{(n,m,R)}(t')\right)
			\, \nu({\rm d}l) {\rm d}t'\bigg|\bigg|\F_0\right]\notag\\
			\lesssim \ & \E \left[\int_0^{T\land \tau_N^{n,m}(R)}
			\int_{|l|\le 1} |l|^2\Big(\lambda_{n,m}
			\big(\|X_n^{(R)} (t)\|_{\H^s}+\|X_m^{(R)} (t)\|_{\H^s}\big)^2+\|Z_{n,m}^{(R)} (t)\|^2_{\H^{\theta}}\Big) \,\nu({\rm d}l) {\rm d}t\bigg|\F_0\right]\notag\\
			\leq  \ & C_N\lambda_{n,m}+ C_N
			\E\left[\int_0^{T}\sup_{t'\in [0,t\land \tau_N^{n,m}(R)]}\|Z_{n,m}^{(R)} (t')\|^{2}_{\H^{\theta}}\d t\bigg|\F_0\right].\label{R-4+5-nmR estiamte}
		\end{align}
		Collecting \eqref{R-1-nmR estiamte},  \eqref{R-2-nmR estiamte},  \eqref{R-3-nmR estiamte},  \eqref{R-4-nmR estiamte}, and  \eqref{R-4+5-nmR estiamte}, we find constant $C_N>0$ such that
		\begin{align*}
			\E \left[\sup_{t\in[0,\,T\land\tau_N^{n,m}(R)]} 
			\|Z_{n,m}^{(R)} (t)\|^2_{\H^\theta}\Big|\F_0\right] 
			\leq  C_N\lambda_{n,m}+C_N\int_0^T\E\left[\sup_{t'\in [0,\, t\land \tau_N^{n,m}(R)]}\|Z_{n,m}^{(R)} (t')\|_{\H^\theta}^2\bigg|\F_0\right]{\rm d}t.
		\end{align*}
		By Gr\"onwall's inequality and noting that $\lambda_{n,m}\to 0$ as $n,m\to\infty$, we obtain  \eqref{Cauchy in E-M}.

		\ref{Convergence of X-n T} We fix a time $T>0$ and let $N>T$. 
		Applying \ref{Global cut-off estimate} of Lemma \ref{Lemma : Xn T estimates} and utilizing Chebyshev's inequality, we obtain
		\begin{align*} 
			&\p(\tau_N^{n,m}(R)<T|\F_0)\\
			\le\ & \p\bigg(\sup_{t\in [0,T]} \|X_n^{(R)}(t)\|_{\H^s}\ge N\bigg|\F_0\bigg)+ \p\bigg(\sup_{t\in [0,T]} \|X_m^{(R)}(t)\|_{\H^s}\ge N\bigg|\F_0\bigg)\\
			\le\ & \frac {2 \hh K(T,2R+\|\Xi\|_{\WP})(1+\|\Xi\|^2_{\H^s})}{N^2}.
		\end{align*} 
		Consequently, for any $ N>T$, $n,m\ge 1$, and $\epsilon>0$,
		\begin{align*}
			&\p\bigg(\sup_{t\in[0,T]}\|X_n^{(R)}(t)-X_m^{(R)}(t)\|_{\H^\theta}>\epsilon\bigg|\F_0\bigg)\\
			\le\ & \p(\tau_N^{n,m}(R)<T|\F_0) + \p\bigg(\sup_{t\in[0,\tau_N^{n,m}(R)]}\|X_n^{(R)}(t)-X_m^{(R)}(t)\|_{\H^\theta}>\epsilon\bigg|\F_0\bigg)\\
			\leq\ &\frac {2 \hh K(T,2R+\|\Xi\|_{\WP})(1+\|\Xi\|^2_{\H^s})}{N^2}+\p\bigg(\sup_{t\in[0,\tau_N^{n,m}(R)]}\|X_n^{(R)}(t)-X_m^{(R)}(t)\|_{\H^\theta}>\epsilon\bigg|\F_0\bigg).
		\end{align*} 
		By applying \ref{Convergence of X-n TnmN} of Lemma \ref{Lemma : Xn T estimates}, and first letting $n,m\to\infty$ followed by $N\to\infty$, we obtain 
		\begin{equation*}
			\lim_{n,m\rightarrow\infty} \p\bigg(\sup_{t\in[0,T]}\|X_n^{(R)}(t)-X_m^{(R)}(t)\|_{\H^\theta}>\epsilon\bigg|\F_0\bigg)=0,\quad  \epsilon,\,T>0. 
		\end{equation*}
		By the reverse Fatou lemma (since the probabilities are bounded by $1$), this implies 
		\begin{align*} 
			&\limsup_{n,m\rightarrow\infty} \p\bigg(\sup_{t\in[0,T]}\|X_n^{(R)}(t)-X_m^{(R)}(t)\|_{\H^\theta}>\epsilon\bigg)\\
			=\ &\limsup_{n,m\rightarrow\infty} \E\bigg[\p\bigg(\sup_{t\in[0,T]}\|X_n^{(R)}(t)-X_m^{(R)}(t)\|_{\H^\theta}>\epsilon\bigg|\F_0\bigg)\bigg]\\
			\le\ & \E\bigg[\limsup_{n,m\rightarrow\infty} \p\bigg(\sup_{t\in[0,T]}\|X_n^{(R)}(t)-X_m^{(R)}(t)\|_{\H^\theta}>\epsilon\bigg|\F_0\bigg)\bigg]=0,\quad \epsilon,\, T>0.
		\end{align*} 
		Therefore, by passing to a suitable subsequence (which we still denote by $\{X_n^{(R)}\}_{n\ge1}$ for simplicity), the sequence is almost surely Cauchy in $\mathscr{M}_{B}([0,T];\H^\theta)$, yielding
		\begin{equation}\label{Xn Cauchy uniform-t H-theta}
			\lim_{n,m\to\infty}\sup_{t\in[0,T]}\|X_n^{(R)}(t)-X_m^{(R)}(t)\|_{\H^\theta}=0\quad \pas
		\end{equation}
		Therefore, there is a limit process $X^{(R)}\in \mathscr{M}_B([0,T];\H^\theta)$, defined for all $t \in [0,T]$, such that
		\begin{equation}\label{Xn to X uniform-t H-theta-limit}
			\lim_{n\to\infty}\sup_{t\in[0,T]}\|X_n^{(R)}(t)-X^{(R)}(t)\|_{\H^\theta}=0\quad \pas
		\end{equation}
		Since $\{X_n^{(R)}(t)\}_{n\ge1}\subset D([0,T];\H^{\theta})$, $D([0,T];\H^{\theta})$ is a closed subspace of $\mathscr{M}_B([0,T];\H^\theta)$,  and uniform-in-$t$ convergence implies convergence in the Skorokhod topology of $D([0,T];\H^{\theta})$ (see Lemma \ref{Lemma: D-convergence}),  it follows that $X^{(R)}\in D([0,T];\H^{\theta})$ and 
		\begin{equation}\label{Xn to X D H-theta}
			X_n^{(R)}\xrightarrow[]{n\to \infty}X^{(R)} \quad {\rm in}\ D([0,T];\H^{\theta})\quad \pas
		\end{equation}

		Next, we demonstrate that $X^{(R)}$ is an $\H^s$-valued process defined for all $t \in [0,T]$. 
		Note that  \eqref{Xn uniform bound} combined with Fatou's lemma yields
		\begin{align*}
			\E\Big[ \liminf_{n\to\infty} \sup_{t\in[0,T]} \|X_n^{(R)}(t)\|^2_{\H^s} \Big| \F_0 \Big]
			\le \liminf_{n\to\infty} \E\Big[ \sup_{t\in[0,T]} \|X_n^{(R)}(t)\|^2_{\H^s} \Big| \F_0 \Big]   < \infty \quad \pas
		\end{align*}
		This implies that 
		\begin{equation}\label{uniform_Hs_bound_liminf}
			\liminf_{n\to\infty} \sup_{t\in[0,T]} \|X_n^{(R)}(t)\|_{\H^s} < \infty \quad \pas
		\end{equation}
		
		On the probability-one set where both \eqref{Xn to X uniform-t H-theta-limit} and \eqref{uniform_Hs_bound_liminf} hold, let us fix an arbitrary $t \in [0,T]$. There exists a subsequence $n_k$ (which generally depends on both $t$ and $\omega$) such that 
		\begin{equation*}
			\lim_{k\to\infty} \|X_{n_k}^{(R)}(t)\|_{\H^s} = \liminf_{n\to\infty} \|X_n^{(R)}(t)\|_{\H^s} \le \liminf_{n\to\infty} \sup_{t\in[0,T]} \|X_n^{(R)}(t)\|_{\H^s} < \infty.
		\end{equation*}
		This dependence on $t$ and $\omega$ poses \textbf{no} issue; we are not using this subsequence to construct the limit process, but rather to verify the spatial regularity of $X^{(R)}(t)$ \textit{a posteriori} for every $t$.
		
		Since the sequence $\{X_{n_k}^{(R)}(t)\}_{k\ge 1}$ is bounded in the reflexive space $\H^s$, we can extract a further weakly convergent subsequence. However, by \eqref{Xn to X uniform-t H-theta-limit}, the sequence $X_n^{(R)}(t)$ converges strongly to $X^{(R)}(t)$ in the $\H^\theta$-topology. Due to the embedding $\H^s\hookrightarrow\H^\theta$, the weak limit in $\H^s$ must coincide with the already existing state $X^{(R)}(t)$. Consequently, $X^{(R)}(t) \in \H^s$, and the weak lower semicontinuity of the $\H^s$-norm guarantees that
		\begin{equation}\label{pointwise Hs bound}
			\|X^{(R)}(t)\|_{\H^s} \le \liminf_{n\to\infty} \|X_n^{(R)}(t)\|_{\H^s} \le \liminf_{n\to\infty} \sup_{t\in[0,T]} \|X_n^{(R)}(t)\|_{\H^s} < \infty.
		\end{equation}
		Since $t \in [0,T]$ was arbitrary, \eqref{pointwise Hs bound} holds for \textit{all} $t \in [0,T]$. 
		
		As discussed in Remark~\ref{Remark : measurability-upgrade}, the progressive measurability of $X^{(R)}$ in the $\H^\theta$-topology automatically upgrades to progressive measurability in $\H^s$. Finally, taking the supremum over $t \in [0,T]$ in \eqref{pointwise Hs bound} and squaring both sides, we find
		\begin{equation*}
			\sup_{t\in[0,T]}\|X^{(R)}(t)\|_{\H^s}^2 \le \liminf_{n\to\infty} \sup_{t\in[0,T]}\|X_n^{(R)}(t)\|_{\H^s}^2 \quad \pas
		\end{equation*}
		Taking the conditional expectation given $\F_0$ on both sides, applying Fatou's lemma, and utilizing the uniform bound \eqref{Xn uniform bound}, we obtain \eqref{X Hs bound}.

		By \eqref{Xn to X D H-theta} and \eqref{Xn Cauchy uniform-t H-theta},
		it remains to prove \eqref{Xn to X} with $s' \in (\theta, s)$. For any $\epsilon > 0$, we first claim that
		\begin{equation}\label{eq:interp-convergence-p}
			\lim_{n\to\infty} \p\left( \sup_{t\in[0,T]}\|X_n^{(R)}(t) - X^{(R)}(t)\|_{\H^{s'}} > \epsilon \bigg|\F_0\right) = 0.
		\end{equation}
		Indeed, the interpolation inequality yields that there exists an $\alpha \in (0,1)$ such that $s' = \alpha \theta + (1-\alpha)s$ and
		\begin{align*}
			\sup_{t\in[0,T]}\|X_n^{(R)} - X^{(R)}\|_{\H^{s'}}  \leq  C \left(\sup_{t\in[0,T]}\|X_n^{(R)} - X^{(R)}\|_{\H^\theta}\right)^\alpha \left(\sup_{t\in[0,T]}\|X_n^{(R)} - X^{(R)}\|_{\H^s}\right)^{1-\alpha}  \triangleq C Y_n^\alpha Z_n^{1-\alpha}.
		\end{align*}
		For any large constant $M > 0$, we can split the probability:
		\begin{align*}
			\p\left(C Y_n^\alpha Z_n^{1-\alpha} > \epsilon\Big|\F_0\right) 
			&= \p\left(C Y_n^\alpha Z_n^{1-\alpha} > \epsilon, \, Z_n > M\Big|\F_0\right) + \p\left(C Y_n^\alpha Z_n^{1-\alpha} > \epsilon, \, Z_n \le M\Big|\F_0\right) \\
			&\le \p\left(Z_n > M\Big|\F_0\right) + \p\left(C Y_n^\alpha M^{1-\alpha} > \epsilon\Big|\F_0\right).
		\end{align*}
		Noting that $Z_n^2 \le 2\sup_{t\in[0,T]} \|X_n^{(R)}(t)\|_{\H^s}^2 + 2\sup_{t\in[0,T]} \|X^{(R)}(t)\|_{\H^s}^2$, we apply Chebyshev's inequality, \eqref{Xn uniform bound}, and \eqref{X Hs bound} to deduce
		$$ \p\left(C Y_n^\alpha Z_n^{1-\alpha} > \epsilon\Big|\F_0\right) \le \frac{4\hh K(T,2R+\|\Xi\|_{\WP})(1+\|\Xi\|^2_{\H^s})}{M^2} + \p\left(Y_n > \left(\frac{\epsilon}{C M^{1-\alpha}}\right)^{1/\alpha}\Big|\F_0\right). $$
		Now, take the limit as $n \to \infty$. By \eqref{Xn to X uniform-t H-theta-limit}, the second term goes to $0$ as $n \to \infty$. We are left with:
		$$ \limsup_{n\to\infty} \p\left(C Y_n^\alpha Z_n^{1-\alpha} > \epsilon\Big|\F_0\right) \le \frac{4\hh K(T,2R+\|\Xi\|_{\WP})(1+\|\Xi\|^2_{\H^s})}{M^2}. $$
		Since $M > 0$ is arbitrary, we can let $M \to \infty$ to obtain \eqref{eq:interp-convergence-p}. The same reasoning that led to \eqref{Xn Cauchy uniform-t H-theta} also implies that for a further subsequence (still denoted by $\{X_n^{(R)}\}_{n\ge1}$),
		\begin{equation*} 
			\lim_{n\to\infty}\sup_{t\in[0,T]}\|X_n^{(R)}(t)-X^{(R)}(t)\|_{\H^{s'}}=0\quad \pas,
		\end{equation*}
		which implies
		$X_n^{(R)}\xrightarrow[]{n\to \infty}X^{(R)}$ in $D([0,T];\H^{s'})$ $\pas$
		We obtain \eqref{Xn to X} for fixed $s'$ and $T>0$. By diagonal argument, a subsequence can be find  uniformly in $s'$ and $T>0$.
	\end{proof}
	
	\paragraph{\textbf{Step 3: Global Uniqueness for the Cut-off Problem.}}
	
	\begin{Lemma}\label{cut-off global solution} 
		Let $\M(\cdot,\cdot,\cdot)$ be defined in \eqref{Marcus flow QQ2}.
		Under the conditions of Proposition \ref{Proposition-(q u)}, for any $R \ge 1$ the Cauchy problem
		\begin{align}
			X^{(R)}(t)-\Xi  = \ & \int_0^t 
			\Big\{\chi_R\big(\|X^{(R)}(t')-\Xi\|_{\WP}\big) F(X^{(R)}(t'))
			-\frac{1}{2}\QQ_1^2 X^{(R)}(t')\Big\}\, \d t' \notag\\
			\ & + \int_0^t  \QQ_1 X^{(R)}(t')\, \d W(t')
			+ \int_0^t \chi_R\big(\|X^{(R)}(t')-\Xi\|_{\WP}\big) H(t',X^{(R)}(t'))\, \d \widetilde W(t') \notag\\
			\ & + \int_0^t \int_{|l|\le1} \Big\{\M\big(1,l,X^{(R)}(t'-)\big)-X^{(R)}(t'-)\Big\}\, \widetilde{\eta}({\rm d}l,{\rm d}t') \notag\\
			\ & + \int_0^t \int_{|l|\le1} \Big\{\M\big(1,l,X^{(R)}(t')\big)-X^{(R)}(t')-l\cdot \QQ_2 X^{(R)}(t')\Big\}\, \nu({\rm d}l)\, {\rm d}t',
			\quad t>0,\label{eq : cut-off problem}
		\end{align}
		admits a unique global solution in the sense of the cut-off version of Definition \ref{solution definition (q u)}, that is, in the sense of Definition \ref{solution definition (q u)} with \eqref{Target problem X=(q u)} replaced by \eqref{eq : cut-off problem}.
	\end{Lemma}

	\begin{proof}
		We will show that $X^{(R)}$ given by \ref{Convergence of X-n T} of  Lemma \ref{Lemma : Xn T estimates} is the unique solution.  Let $\theta\in\left(\frac{d}{2}+p,{s-\max \{3\zeta,1\}}\right)$ and $T>0$. 
		From \ref{Hypo-h} (with $\sigma=\theta$), \eqref{Xn to X}, and Lemma \ref{Lemma:Fn}, we infer that $\pas$,
		\begin{equation*}
			\lim_{n\to\infty}\sup_{t\in[0,T]}\Big\|\chi_R\big(\|X_n^{(R)}(t)-\Xi\|_{\WP}\big) H(t,X_n^{(R)}(t))- \chi_R\big(\|X^{(R)}(t)-\Xi\|_{\WP}\big) H(t,X^{(R)}(t))\Big\|_{\H^\theta}=0,
		\end{equation*}
		and
		\begin{equation*}
			\lim_{n\to\infty}\sup_{t\in[0,T]}\Big\|\chi_R\big(\|X_n^{(R)}(t)-\Xi\|_{\WP}\big) F_n(X_n^{(R)}(t))- \chi_R\big(\|X^{(R)}(t)-\Xi\|_{\WP}\big) F(X^{(R)}(t))\Big\|_{\H^{\theta-1}}=0.
		\end{equation*}
		Here, $\{\QQ_{i,n}={\rm diag}(0,\Q_{i,n})\}_{n\ge1}$ for $i=1,2$, where $\{\Q_{i,n}\}_{n\ge1}$ is given by Lemma \ref{Lemma:Qn}. 
		Similarly, for $i=1,2$, using \ref{Hypo-Qi}, \eqref{Marcus flow QQ2}, \eqref{Marcus flow stability 2}, and \eqref{Xn to X}, we obtain that $\pas$ for $i=1,2$,
		\begin{equation*}
			\lim_{n\to\infty} \sup_{t\in[0,T]}\Big\|\QQ_{i,n} X_n^{(R)}(t)- \QQ_i  X^{(R)}(t)\Big\|_{\mathbb{L}^{2}}=0, \quad  
			\lim_{n\to\infty} \sup_{t\in[0,T]}\Big\|\QQ^2_{i,n}  X_n^{(R)}(t)- \QQ^2_i  X^{(R)}(t)\Big\|_{\mathbb{L}^{2}}=0,
		\end{equation*}
		and 
		\begin{equation*}
			\lim_{n\to\infty}\sup_{t\in[0,T]}\Big\| \M_n(r,l,X^{(R)}_n(t))-\M(r,l,X^{(R)}(t))\Big\|_{\mathbb{L}^{2}}=0,\quad  r\in[0,1].
		\end{equation*}
		Therefore, passing to the limit as $n\to\infty$ in \eqref{approximation solution XnR}, we conclude that \eqref{eq : cut-off problem} holds as an equation in $\mathbb{L}^{2}$. Since $X^{(R)}(t)\in \H^s$ with $s > \frac{d}{2} + p + \max \{3\zeta,1\}$ (cf. \eqref{X Hs bound}), all terms in \eqref{eq : cut-off problem} lie in $\H^\theta$, and the embedding $\H^\theta\hookrightarrow C(\K^d)$ implies that \eqref{eq : cut-off problem} is in fact an equation in $C(\K^d)$. The details are standard and are omitted to save space. 
		
		Now we prove the uniqueness of solution.
		Let $Y^{(R)}$ be another solution to \eqref{eq : cut-off problem} with $Y^{(R)}(0)=\Xi$ such that \eqref{X Hs bound} holds for $Y^{(R)}$ replacing $X^{(R)}$.  
		For  $t\ge0$, $r\in[0,1]$ and $|l|\le 1$, we 
		let  
		\begin{equation*}
			Z^{(R)}(t)\triangleq X^{(R)}(t)-Y^{(R)}(t),
			\quad \textbf{B}^{(R)}(r,l,t)\triangleq \M\big(r,l,Z^{(R)} (t)\big)-Z^{(R)} (t),\quad \textbf{C}^{(R)}(r,l,t)\triangleq \textbf{B}^{(R)}(r,l,t)
			-l \cdot \QQ_2 Z^{(R)} (t),
		\end{equation*}
		and 
		\begin{equation*}
			G_{R}\big(X^{(R)}(t)\big)\triangleq \chi_R\big(\|X^{(R)}(t)-\Xi\|_{\WP}\big) F(X^{(R)}(t))
			-\frac{1}{2}\QQ_1^2X^{(R)}(t).
		\end{equation*}
		It follows  from \eqref{approximation scheme C} and the linearity of operators $\QQ_1$, $\QQ_2$ and the map $\M(1,l,\cdot)$ that  
		\begin{align*} 
			Z^{(R)} (t)  
			=\ & \int_0^t \mathfrak{i}_{1}^{(R)}(t') \d t'
			+\int_0^t\int_{|l|\le1} \textbf{C}^{(R)}(1,l,t')\, \nu({\rm d}l){\rm d}t'\notag\\
			&+\int_0^t 
			\mathfrak{i}_{2}^{(R)}(t') \d W(t') 
			+\int_0^t\mathfrak{i}_{3}^{(R)}(t')\d \widetilde W(t') 
			+ \int_0^t\int_{|l|\le1} \textbf{B}^{(R)}(1,l,t'-)\, \widetilde{\eta}({\rm d}l,{\rm d}t').
		\end{align*}
		where for $t>0$,
		\begin{align*}
			\mathfrak{i}_{1}^{(R)}(t)\triangleq G_{R}\big(X^{(R)}(t)\big)-G_{R}\big(Y^{(R)}(t)\big),\quad \mathfrak{i}_{2}^{(R)}(t)\triangleq  \QQ_1  Z^{(R)}(t),\quad
			\mathfrak{i}_{3}^{(R)}(t)\triangleq H_{R}\left(t',X^{(R)}(t)\right)-H_{R}\left(t',Y^{(R)}(t)\right).
		\end{align*}
		Similar to \eqref{Znm-Ito}, we 
		use It\^o's formula (Lemma \ref{Ito formula}) to obtain that for $t\ge0$ and $\theta\in\left(\frac{d}{2}+p,{s-\max \{3\zeta,1\}}\right)$,
		\begin{align*}
			\|Z^{(R)} (t)\|_{\H^\theta}^2
			=\ &
			2\int_0^{t}\underbrace{\left\{\IP{ \mathfrak{i}_{1}^{(R)}(t'),Z^{(R)} (t')}_{\H^\theta}+\|\mathfrak{i}_{2}^{(R)}(t')\|^2_{\H^\theta}
				+\|\mathfrak{i}_{3}^{(R)}(t')\|^2_{\H^\theta}\right\}}_{\triangleq\ \mathfrak{R}_1^{(R)}(t')}
			\d t'\\
			&+2\int_0^{t}
			\underbrace{\IP{\mathfrak{i}_{2}^{(R)}(t'),Z^{(R)} (t')}_{\H^\theta}}_{\triangleq\ \mathfrak{R}_2^{(R)}(t')}
			\d W(t')
			+2\int_0^{t}
			\underbrace{\IP{\mathfrak{i}_{3}^{(R)}(t'),Z^{(R)} (t')}_{\H^\theta}}_{\triangleq\ \mathfrak{R}_3^{(R)}(t')}
			\d \widetilde W(t')\\
			&+ \int_0^{t}\int_{|l|\le1}
			\underbrace{ \left\{\Big\|\M\big(1,l,Z^{(R)} (t'-)\big)\Big\|^2_{\H^\theta} 
				-\norm{Z^{(R)} (t'-)}^2_{\H^\theta}\right\}}_{\triangleq\ \mathfrak{R}_4^{(R)}(t'-)}
			\, \widetilde \eta(\mathrm{d}l,\mathrm{d}t')\\
			&+\int_0^{t}\int_{|l|\le1}
			\underbrace{-2l\cdot \IP{\QQ_2 Z^{(R)} (t'),Z^{(R)} (t')}_{\H^\theta}}_{\triangleq \mathfrak{R}_5^{(R)}(t')}
			\,\nu({\rm d}l) {\rm d}t'\\
			&+ \int_0^{t}\int_{|l|\le1} 
			\left\{\Big\|\M\big(1,l,Z^{(R)} (t')\big)\Big\|^2_{\H^\theta} 
			-\norm{Z^{(R)} (t')}^2_{\H^\theta}\right\}
			\,\nu({\rm d}l) {\rm d}t'.
		\end{align*}
		Let
		\begin{equation}
			\tau_N(R) \triangleq  N\land \inf\big\{t\ge 0: \|X^{(R)}(t)\|_{\H^s}\lor\|Y^{(R)}(t)\|_{\H^s}\ge N\big\},\quad N\ge1.\label{tau-N in cut-off uniqueness}
		\end{equation}
		Since $X^{(R)}, Y^{(R)}$ are solutions in the sense of Definition \ref{solution definition (q u)}, the stopping time $\tau_N(R)$ in \eqref{tau-N in cut-off uniqueness} is well-defined.
		Then,  the estimate for  $\E [\sup_{t\in[0,\tau_N(R)]} 
		\|Z^{(R)} (t)\|^2_{\H^\theta}|\F_0] $ has been essentially carried out in
		the analysis of \eqref{R-1-nmR estiamte},  \eqref{R-2-nmR estiamte},  \eqref{R-3-nmR estiamte},  \eqref{R-4-nmR estiamte} and  \eqref{R-4+5-nmR estiamte}. Actually, 
		from \eqref{Xn to X} and the fact $\lim_{n,m\to \infty}\lambda_{n,m}=0$, 
		we can take $n=m\to\infty$ in the analysis of \eqref{R-1-nmR estiamte},  \eqref{R-2-nmR estiamte},  \eqref{R-3-nmR estiamte},  \eqref{R-4-nmR estiamte} and  \eqref{R-4+5-nmR estiamte} to find constants $c_1,c_2, C_N>0$ such that for all $T\in [0,N]$,
		\begin{align*}
			2\E\left[\int_0^{T\land \tau_N(R)}|\mathfrak{R}_1^{(R)}(t)|\d t\bigg|\F_0\right]\leq C_N
			\E\left[\int_0^{T}\sup_{t'\in [0,t\land \tau_N(R)]}\|Z^{(R)} (t')\|^{2}_{\H^{\theta}}\d t\bigg|\F_0\right],
		\end{align*}
		\begin{align*}
			&  2\E\left[\sup_{t \in [0, T\land \tau_N(R)]}\bigg|\int_0^{t}\mathfrak{R}_2^{(R)}(t')\d W(t')\bigg|\bigg|\F_0\right]\notag\\
			\leq\ &   \frac{1}{6}\E\left[\sup_{t \in [0, T\land \tau_N(R)]} \|Z^{(R)} (t)\|_{\H^\theta}^2\Big|\F_0\right]
			+c_2\int_0^T\E\left[\sup_{t'\in [0,t\land \tau_N(R)]}\|Z^{(R)} (t')\|_{\H^\theta}^2\bigg|\F_0\right]{\rm d}t,
		\end{align*}
		\begin{align*}
			& 2\E\left[\sup_{t \in [0, T\land \tau_N(R)]} \bigg|\int_0^{t}\mathfrak{R}_3^{(R)}(t')
			\d \widetilde W(t')\bigg|\bigg|\F_0\right]\notag\\
			\le\ &  \frac{1}{6}\E\left[\sup_{t \in [0, T\land \tau_N(R)]} \|Z^{(R)} (t)\|_{\H^\theta}^2\Big|\F_0\right]
			+C_N\int_0^T\E\left[\sup_{t'\in [0,t\land \tau_N(R)]}\|Z^{(R)} (t')\|_{\H^\theta}^2\bigg|\F_0\right]{\rm d}t,
		\end{align*}
		\begin{align*}
			&  \E\left[\sup_{t \in [0, T\land \tau_N(R)]}\bigg|\int_0^{t}\int_{|l|\le1}\mathfrak{R}_4^{(R)}(t'-)
			\, \widetilde \eta(\mathrm{d}l,\mathrm{d}t')\bigg|\bigg|\F_0\right]\notag\\
			\leq  \ & \frac{1}{6}\E\left[\sup_{t \in [0, T\land \tau_N(R)]} \|Z^{(R)} (t)\|_{\H^\theta}^2\Big|\F_0\right]
			+c_2\int_0^T\E\left[\sup_{t'\in [0,t\land \tau_N(R)]}\|Z^{(R)} (t')\|_{\H^\theta}^2\bigg|\F_0\right]{\rm d}t, 
		\end{align*}
		and 
		\begin{align*}
			\E\left[\sup_{t \in [0, T\land \tau_N(R)]} \bigg|\int_0^{t}\int_{|l|\le1}
			\left(\mathfrak{R}_4^{(R)}(t')+\mathfrak{R}_5^{(R)}(t')\right)
			\, \nu({\rm d}l) {\rm d}t'\bigg|\bigg|\F_0\right] 
			\leq C_N
			\E\left[\int_0^{T}\sup_{t'\in [0,t\land \tau_N(R)]}\|Z^{(R)} (t')\|^{2}_{\H^{\theta}}\d t\bigg|\F_0\right].
		\end{align*}
		Using the above estimates,  we have 
		\begin{align*}
			\E \left[\sup_{t\in[0,\tau_N(R)]} 
			\|Z^{(R)} (t)\|^2_{\H^\theta}\Big|\F_0\right] 
			\leq  C_N\int_0^T\E\left[\sup_{t'\in [0,t\land \tau_N(R)]}\|Z^{(R)} (t')\|_{\H^\theta}^2\bigg|\F_0\right]{\rm d}t.
		\end{align*}
		By Gr\"onwall's inequality, the monotone convergence theorem, the fact that $\lim_{N\to\infty}\tau_N(R)=\infty$ and the arbitrariness of $T>0$, we infer that, for every  $Y^{(R)}(t)= X^{(R)}(t)$ for all $t\ge 0$.   
	\end{proof}
	
	\paragraph{\textbf{Step 4: Conclude the Proof of Proposition \ref{Proposition-(q u)}.}}
	We first note that uniqueness of $\H^s$-valued solution to \eqref{Target problem X=(q u)} can be carried out similar to the cut-off case in Lemma \ref{cut-off global solution}.
	For any $R\ge 1$, let $X^{(R)}$ be the unique global solution to the Cauchy problem \eqref{eq : cut-off problem} such that \eqref{X Hs bound} is verified (by Lemma \ref{cut-off global solution}). Let 
	\begin{equation*}
		\tau(R) \triangleq  \inf\big\{t\ge 0: \|X^{(R)}(t)-\Xi\|_{\WP}\ge R\big\}. 
	\end{equation*}
	According to \eqref{Xn to X},  we have that $X^{(R)} \in D([0,\infty);\H^{s'})$ with $s'\in[\theta,s)$.
	Due to the right continuity of $X^{(R)}(t)$ in $\H^\theta$ and the embedding $\H^\theta\hookrightarrow {\WP}$, we have $\p(\tau(R)>0)=1$ for all $R>0.$ Since $\chi_R(\|X^{(R)}(t)-\Xi\|_{\WP})=1$ for $t\le \tau(R)$, \eqref{Target problem X=(q u) full} coincides with 
	\eqref{eq : cut-off problem} up to time $\tau(R)$. This together with \eqref{X Hs bound} implies that 
	$X^{(R)}$ satisfies \eqref{Target problem X=(q u)} up to $\tau(R)$.    By the uniqueness of solutions to \eqref{Target problem X=(q u)}, we see that $\tau(R)$ is increasing in $R$, and 
	$$X^{(R)}(t)= X^{(R+1)}(t),\qquad t\le \tau(R), \quad  R\ge 1\quad \p\text{-a.s.}$$
	We set
	\begin{align} 
		X(t) \triangleq \sum_{R=1}^\infty \textbf{1}_{[\tau(R-1), \tau(R))}(t) X^{(R)}(t),\quad t\in [0,\tau^*),\qquad \tau^*  \triangleq \lim_{R\to\infty} \tau(R),\qquad \tau(0) \triangleq 0.\label{X=sum XR} 
	\end{align}
	Then one can conclude that $X\in D([0,\tau^*);\H^{s'})$ with $s'<s$ and $X$ satisfies \eqref{Target problem X=(q u)} up to $\tau^*$. Moreover, according to the definitions of $\tau^*$ and $\tau(R)$,  we see that \eqref{X blow-up criterion} holds true, and hence $(X,\tau^{*})$  is actually a maximal solution.

	\subsubsection{Solving \eqref{Target problem X=(q u)} under Constraint \eqref{a<q<b}}
	\label{Section : Constrained Problem}
	It remains to verify that the maximal $\mathbb{H}^s$ solution is admissible in the sense of Definition
	\ref{solution definition (q u)}, that is, it satisfies the constraint \eqref{a<q<b}.   Recall that
	$\rr \in {\bf R}_{(\rr_0,\rr_\infty)}$ (cf. \eqref{Set of transform}).
	The key result in this step is

	\begin{Proposition}
		\label{Proposition-(q u)-admissible}
		Let $(X,\tau^*)$ be the unique maximal $\H^s$ classical solution to \eqref{Target problem X=(q u)} given by Proposition \ref{Proposition-(q u)}.
		\begin{enumerate}[label={\bf (\arabic*)},leftmargin=0.79cm]\setlength\itemsep{0.2em}
			
			\item \label{Prop-(q u)-admissible-Theta=C0}
			When\ref{Constant Sound} holds, $(X,\tau^*)$ is   admissible in the sense of Definition \ref{solution definition (q u)}.
			
			\item  \label{Propo-(q u)-admissible-Theta->Lambda} When \ref{General Sound} holds, we further assume 
			$$\text{either}\quad \p\Big(0=\rr_0<\underline{\rr}\le \vrho_0\le \overline{\rr}<\rr_\infty=\infty\Big)=1\quad \text{or}
			\quad \p\Big(-\infty=\rr_0<\underline{\rr}\le\vrho_0\le\overline{\rr}<\rr_\infty=0\Big)=1.$$
			Then $(X,\tau^*)$ is also admissible in the sense of Definition \ref{solution definition (q u)}.
			
		\end{enumerate}
	\end{Proposition}

	We first establish a maximum principle for a transport-type equation with a linearly growing term $\upvartheta(\cdot)$. For the well-known case $\upvartheta(x)=x$, we refer, for example, to \cite{Maslennikova-Bogovskii-1994-chapter,Feireisl-2004-Book}.
	
	\begin{Lemma}[Maximum principle]\label{Maximum principle-Lemma}
		Assume  $\upvartheta:\R\to\R$ satisfies, for some $C>0$, $|\upvartheta(x)|\leq C |x|$ for $x\in\R.$
		Let $d\ge2$, $a_0<b_0$ and $T^*>0$. Assume that $v: [0,T^*)\times\K^d\mapsto v(t,x)\in \R^d$ such that  $v\in L_{\rm loc}^1([0,T^*);\Wlip_d\cap L^2_d)$ and
		$$\int_0^{t}\|v(t')\|_{\Wlip}\d t'<\infty,\quad t\in[0,T^*).$$
		If  $\partial_tf\in C([0,T^*);\Wlip_1)$ and  $f$  satisfies
		\begin{equation*}
			\partial_tf+\upvartheta(f) \, \Div v+ v\cdot\nabla  f=0,\quad
			f(0,x)=f_0(x),\quad a_0\le f_0(\cdot)\le b_0,
		\end{equation*}
		then the following estimates hold true for $(t,x)\in[0,T^*)\times \K^d$:
		\begin{itemize}[leftmargin=0.79cm]\setlength\itemsep{0.2em}
			\item {\bf (Positive solution)} When  $a_0>0$,  
			\begin{equation}\label{c0 exp-<f<d0 exp}
				a_0 \exp \left\{-\int_0^tC\|\Div v(t')\|_{L^\infty}\d t'\right\} \leq f(t,x)\leq b_0 \exp\left\{\int_0^tC\|\Div v(t')\|_{L^\infty}\d t'\right\}.
			\end{equation}
			\item {\bf (Negative solution)} When  $b_0<0$,  
			\begin{equation}\label{c0 exp<f<d0 exp-}
				a_0 \exp \left\{\int_0^tC\|\Div v(t')\|_{L^\infty}\d t'\right\} \leq f(t,x) \leq b_0 \exp\left\{\int_0^t-C\|\Div v(t')\|_{L^\infty}\d t'\right\}.
			\end{equation}
		\end{itemize}

	\end{Lemma}

	\begin{proof}
		We first consider the case $\K=\R$.  Since $v(t,\cdot)\in \Wlip\cap L^2$,
		the differential equation for unknown variable $\upchi(t,x):[0,T^*)\times \R^d\to \R^d$
		\begin{equation*}
			\frac{{\rm d}}{{\rm d} t} \upchi(t,x)=\frac{{\rm d}}{{\rm d}t} (\upchi_i(t,x))_{1\le i\le d}=v(t,\upchi(t,x))=(v_i(t,\upchi(t,x)))_{1\le i\le d},\ \ \upchi(0,x)=x\in\R^d
		\end{equation*}
		has a unique solution $\upchi(t,x)$. Let $L(t) \triangleq \|v(t)\|_{\Wlip}$ for $t\in[0,T^*)$. We have
		that for $x,y\in\R^d$,
		\begin{equation*}
			|x-y|-\int_0^tL(t')|\upchi(t',x)-\upchi(t',y)|\d t'\leq |\upchi(t,x)-\upchi(t,y)|\leq |x-y|+\int_0^tL(t')|\upchi(t',x)-\upchi(t',y)|\d t',
		\end{equation*}
		which means that
		\begin{equation*} 
			|x-y| \exp\left\{\int_0^t-L(t')\d t'\right\}\leq |\upchi(t,x)-\upchi(t,y)|\leq |x-y| \exp\left\{\int_0^tL(t')\d t'\right\}.
		\end{equation*}
		Consequently,  $\upchi(t,\cdot)$ is a Lipschitz homeomorphism of $\R^d$ and one can obtain the pointwise-in-$x$ estimate on $f(t,x)$ by considering  $f(t,\upchi(t,x))$.

		Now we observe that for $(t,x)\in[0,T^*)\times \R^d$,
		\begin{align*}
			\frac{{\rm d}}{{\rm d} t} f(t,\upchi(t,x))=\partial_t f(t,\upchi(t,x))+\nabla f(t,\upchi(t,x)) \cdot v(t,\upchi(t,x)) =-\upvartheta(f)(t,\upchi(t,x))\, \Div v(t,\upchi(t,x)).
		\end{align*}

		Let $a_0>0$, and let
		$$T^+_{x} \triangleq T^*\land \inf\{t>0: f(t,\upchi(t,x))\le0\},\quad x\in\R^d.$$ Then on $[0,T^+_{x})$, $f>0$ and
		\begin{align*}
			-Cf(t,\upchi(t,x))\|\Div v(t)\|_{L^\infty}\leq\frac{{\rm d}}{{\rm d} t} f(t,\upchi(t,x))\leq  Cf(t,\upchi(t,x))\|\Div v(t)\|_{L^\infty}.
		\end{align*}
		Hence 
		\begin{align*}
			-C\int_0^t\|\Div v(t')\|_{L^\infty}\d t'\leq\int_{f_0}^{f(t,\upchi(t,x))}\frac{ {\rm d}y}{y}\leq  C\int_0^t\|\Div v(t')\|_{L^\infty}\d t',\quad (t,x)\in[0,T^+_{x})\times \R^d.
		\end{align*}
		Then we have that for  $(t,x)\in[0,T_{x}^+)\times \R^d$,
		$$a_0 \exp \left\{-\int_0^tC\|\Div v(t')\|_{L^\infty}\d t'\right\}\leq f(t,x)\leq b_0 \exp\left\{\int_0^tC\|\Div v(t')\|_{L^\infty}\d t'\right\}.$$

		It follows from $a_0 \exp \left\{-\int_0^tC\|\Div v(t')\|_{L^\infty}\d t'\right\}>0$ and the continuity of $f(\cdot,\upchi(\cdot,x))$ that $T_x^+=T^*$ for all $x\in\R^d$. Hence we obtain  \eqref{c0 exp-<f<d0 exp}.
		
		When $b_0<0$, similarly, we define for all $x\in \R^d$ that 
		$$T_{x}^- \triangleq T^*\land \inf\{t>0: f(t,\upchi(t,x))\ge0\}.$$ Then on $[0,T_{x}^-)$, $f<0$ and
		\begin{align*}
			Cf(t,\upchi(t,x))\|\Div v(t)\|_{L^\infty}\leq\frac{{\rm d}}{{\rm d} t} f(t,\upchi(t,x))\leq  -Cf(t,\upchi(t,x))\|\Div v(t)\|_{L^\infty}.
		\end{align*}
		Similar to the above case,  we have that for $(t,x)\in[0,T_{x}^-)\times \R^d$, 
		$$a_0 \exp \left\{\int_0^tC\|\Div v(t')\|_{L^\infty}\d t'\right\} \leq f(t,x)\leq b_0 \exp\left\{\int_0^t-C\|\Div v(t')\|_{L^\infty}\d t'\right\},$$
		which  implies  $T_{x}^-=T^*$ for all $x\in\R^d$. Therefore, \eqref{c0 exp<f<d0 exp-} holds. 
		
		For the case $\K=\T$, 
		one can also obtain the pointwise-in-$x$ estimate on $f(t,x)$ by considering  $f(t,\upchi(t,x))$. In the same way we can prove \eqref{c0 exp-<f<d0 exp} and \eqref{c0 exp<f<d0 exp-}.
	\end{proof}

	\begin{proof}[Proof of  Propostion \ref{Proposition-(q u)-admissible}]
		By part  \ref{Thea=C0-Remark C0 (a b)} of Remark \ref{Remark-Hypo-Pressure}, if \ref{Constant Sound} holds, $(\rr_0,\rr_\infty)=(-\infty,\infty)$, so that \ref{Prop-(q u)-admissible-Theta=C0}  of  Proposition \ref{Proposition-(q u)-admissible} automatically holds true.
		We only need to prove that \ref{Propo-(q u)-admissible-Theta->Lambda}  of Proposition \ref{Proposition-(q u)-admissible}. For all $s'<s$, since $(\vrho,u)\in D\big([0,\tau^*);\H^{s'}\big)$, we can infer from \eqref{Target problem X=(q u)}$_1$ that $\partial_t \vrho\in C([0,\tau^*);H_1^{s'-1})$ $\pas$ and satisfies
		\begin{equation}\label{q Eq maximum principle}
			\partial_t \vrho+\Lambda(\vrho) \, \Div u+ u\cdot\nabla \vrho=0, \quad t<\tau^*\quad \pas
		\end{equation}
		Furthermore, \eqref{X blow-up criterion} implies 
		$$\int_0^{t}\|u(t')\|_{\Wlip_d}\d t'\leq \limsup_{t\to \tau^*}\int_0^{t}\|X(t')\|_{\WP}\d t'<\infty,\quad t\in[0,\tau^*)\quad \pas$$
		Therefore, the desired results comes from pathwisely applying Lemma \ref{Maximum principle-Lemma}  (with $\upvartheta=\Lambda$) to \eqref{q Eq maximum principle}.
	\end{proof}

	\subsubsection{Concluding the Proof of Theorem  \ref{Thm-(rho u)}}
	Now, 
	Theorem \ref{Thm-(rho u)} is a direct consequence of Propositions \ref{Proposition-(q u)} and \ref{Proposition-(q u)-admissible}.

	\section{A Generalized Krylov--Bogoliubov Criterion}\label{Section : ASSES}

	To establish the existence of invariant probability measures for the stochastic damped Euler equations \eqref{SEuler-(u P) damp}, we first review the classical Krylov--Bogoliubov criterion and isolate the underlying topological difficulties mentioned in Section \ref{Section : novel-introduction}.
	
	Consider a Markov transition semigroup $\{\mathscr{P}_{t}\}_{t\ge0}$ on a Polish space $(E,d_E)$, admitting the representation
	\begin{equation}\label{Markov semigroup-kernel}
		[\mathscr{P}_{t} f](x) = \int_{E} f(y)\, \pi(t,x,{\rm d}y), \quad t \geq 0, \quad f \in \mathscr{M}_B(E),
	\end{equation}
	where $\pi$ denotes the associated transition probability kernel on $E$. Recall that $\mathscr{B}(E)$ denotes the Borel $\sigma$-algebra on $E$, and $\mathbf{P}(E)$ be  the set of Borel probability measures on $(E,\mathscr{B}(E))$. 
	Operating at the measure level, we utilize the standard pairing $\mu(\varphi) \triangleq \int_E \varphi(x)\,\mu(\mathrm{d}x)$ and define the adjoint (or push-forward) operator $\mathscr{P}^*_{t}$ of $\mathscr{P}_{t}$ as
	\begin{align*}
		\mathscr{P}^*_{t}: \mathbf{P}(E)\to \mathbf{P}(E),\qquad
		[\mathscr{P}^*_{t}\mu](A) = \int_E \pi(t,x,A)\,\mu({\rm d}x),\quad \mu \in \mathbf{P}(E),\quad A \in \mathscr{B}(E),
	\end{align*}
	
	\begin{Definition}\label{IM definition}
		A Markov semigroup $\{\mathscr{P}_{t}\}_{t\ge0}$ is called \textbf{Feller} on $E$ if $\mathscr{P}_{t} (C_B(E)) \subset C_B(E)$ for all $t\ge0$. A probability measure $\mu\in\mathbf{P}(E)$ is called an \textbf{invariant probability measure} for $\{\mathscr{P}_{t}\}_{t\ge0}$ if  
		\begin{equation}\label{eq:invariance}
			\mathscr{P}^*_{t}\mu=\mu, \quad \text{for all } t > 0.
		\end{equation}
	\end{Definition}
	
	We briefly recall the classical Krylov--Bogoliubov criterion (cf.\ \cite[Theorem 3.1.1]{DaPrato-Zabczyk-1996-Book}). For an initial state $x_0 \in E$, let $\delta_{x_0}$ be the Dirac measure concentrated at $x_0$. Define the family of time-averaged measures
	\begin{equation}\label{eq:approx_measures}
		\nu_{T}^{x_0} \triangleq \frac{1}{T} \int_{0}^{T} \mathscr{P}^*_{t} \delta_{x_0}\,\mathrm{d}t, \quad T>0.
	\end{equation}
	The classical theorem ensures that $\{\mathscr{P}_{t}\}_{t\ge0}$ admits an invariant probability measure provided that:\\ \textbf{(a)} $\{\nu_T^{x_0}\}_{T>0}$ possesses a weakly convergent subsequence in $\mathbf{P}(E)$ as $T \to \infty$;\\ 
	\textbf{(b)} $\{\mathscr{P}_{t}\}_{t\ge0}$ is globally Feller on $E$.
	
	\subsection{Problem Formulation, Assumptions, and Main Result}\label{Section : Mismatched topo}
	
	The application of the classical Krylov--Bogoliubov theorem necessitates both tightness and the Feller property within the \textbf{same} underlying Polish space $(E,d_E)$. However, for strongly nonlinear stochastic fluid models, such as the stochastic damped Euler equations (see Section \ref{Section : novel-introduction}), establishing uniform bounded regularity, compactness, and global continuous dependence on initial data simultaneously in a single phase space is generally unfeasible. This obstruction is well documented; for instance, in \cite{Constantin-Glatt-Holtz-Vicol-2014-CMP}, the addition of fractional diffusion was required to restore the regularization intrinsically absent in the purely damped 2-D Euler system. 
	
	To illustrate the dynamical mechanism underlying this degeneracy, consider the 1-D inviscid advection term $u \partial_x u$. When analyzing the continuous dependence on initial data for solutions in a Sobolev space $H^s$, estimating the difference between two solutions in a strictly weaker $H^\theta$-metric yields bounds of the form
	\begin{equation*}
		\big| \langle u \partial_x u - v \partial_x v, u-v \rangle_{H^{\theta}} \big| \lesssim \big(\|u\|_{H^{\theta+1}}+\|v\|_{H^{\theta+1}}\big) \|u-v\|^2_{H^\theta}.
	\end{equation*}
	To close the energy estimate for $\E\|u-v\|^2_{H^\theta}$, one is compelled to restrict $\theta \le s-1$ and employ appropriate stopping times to truncate the stronger norms. In the general theory of SPDEs with strong transport nonlinearities, these stopping times may degenerate to zero for unbounded initial data. Consequently, continuous dependence---and by extension, the Feller property---can only be rigorously guaranteed on bounded closed balls of $(H^s, \|\cdot\|_{H^s})$ when measured with respect to the weaker $H^\theta$-metric. 
	
	To abstractly formalize this metric-mismatch, we introduce the following structural hypothesis.
	
	\begin{Hypothesis}\label{Hypo-generalized KB}
		Let $(E, d_E)$ be a Polish space. Let $\widetilde{d}_E$ be another metric defined on $E$ strictly weaker than $d_E$, i.e., the $d_E$-topology is strictly stronger than the $\widetilde{d}_E$-topology. Moreover, we assume that $d_E(x_0,\cdot)$ is lower semicontinuous with respect to the $\widetilde{d}_E$-topology for any fixed reference point $x_0\in E$.
	\end{Hypothesis}
	
	In the context of stochastic inviscid flows, $(E, d_E)$ represents the high-regularity state space where solutions are globally defined but lack compactness, whereas $\widetilde{d}_E$ generates the weaker topology in which weak convergence of probability measures can be established via \textit{a priori} bounds. 
	
	Let the metrics $d_E$ and $\widetilde{d}_E$ be specified as in \ref{Hypo-generalized KB}. For notational clarity, we denote the strong space simply by $E \triangleq (E, d_E)$, and the weak space by $E_{\widetilde{d}_E} \triangleq (E,\widetilde{d}_E)$.
	We introduce the following definition to capture the localized nature of the continuous dependence under mismatched metrics.
	
	\begin{Definition}[\textbf{Restricted Feller Property under Mismatched Metrics}]\label{Locally Feller definition} 
		Let $x_0 \in E$ be a given reference point and define the strong closed ball
		\begin{equation}\label{eq: ball B-E-R}
			B^{d_E}_R(x_0) \triangleq \big\{x \in E : d_E(x, x_0) \le R\big\}.
		\end{equation}
		A Markov semigroup $\{\mathscr{P}_t\}_{t\ge 0}$ on $E$ satisfies the \textbf{restricted Feller property under mismatched metrics} (with respect to $\widetilde{d}_E$ and $d_E$, anchored at $x_0$) if the following two conditions hold:
		\begin{enumerate}[label={\rm(\roman*)}, leftmargin=0.79cm]\setlength\itemsep{0.2em}
			\item for every bounded Borel-measurable function $\phi \in \mathscr{M}_B(E)$ and $x \in E$, the map $t \mapsto [\mathscr{P}_t \phi](x)$ is Borel measurable on $[0, \infty)$;
			\item for any $t\ge 0$, $R>0$, and $\varphi \in C_B(E_{\widetilde{d}_E})$, the restriction 
			$\mathscr{P}_{t}\varphi\big|_{B^{d_E}_R(x_0)}$
			is continuous with respect to the metric $\widetilde{d}_E$.
		\end{enumerate}
	\end{Definition}
	
	To handle abstract Markov semigroups lacking globally continuous trajectories, we incorporate the requirement of temporal measurability directly into Definition \ref{Locally Feller definition}. By formulating condition (i) for arbitrary bounded measurable functions $\mathscr{M}_B(E)$, we rigorously ensure that the time-averaged measures $\nu_T^{x_0}$ defined in Theorem \ref{Thm-generalized KB} are well-defined probability measures on $\mathscr{B}(E)$. In the context of stochastic evolution equations (such as the fluid models discussed herein), this condition is naturally satisfied. Specifically, when the transition semigroup is induced by a solution map $\mathscr{S}_t$ via $[\mathscr{P}_t \phi](x) = \E[\phi(\mathscr{S}_t x)]$ (see \eqref{Markov semigroup-SES} below), the 
	joint measurability of $(\omega,t)$ of the stochastic flow ensures the Borel measurability of $t \mapsto [\mathscr{P}_t \phi](x)$ for any $\phi \in \mathscr{M}_B(E)$ via Fubini's theorem.
	The following theorem provides a generalized Krylov--Bogoliubov criterion adapted to systems exhibiting this restricted Feller property under mismatched metrics.
	
	\begin{Theorem}[\textbf{Existence Criterion}]\label{Thm-generalized KB}
		Suppose \ref{Hypo-generalized KB} holds. Let $\nu_{T}^{x_0}$ be defined by \eqref{eq:approx_measures} for the fixed initial point $x_0 \in E$ given in Definition \ref{Locally Feller definition}. Assume the following conditions are satisfied:
		\begin{enumerate}[label={{\bf (\arabic*)}},leftmargin=0.79cm]\setlength\itemsep{0.2em}
			\item\label{Cond-Tightness} \textbf{$($Weak Convergence$)$} There exist a sequence $T_n\to\infty$ and a probability measure $\nu\in\mathbf{P}(E_{\widetilde{d}_E})$ such that $\nu_{T_n}^{x_0} \to \nu$ weakly in $\mathbf{P}(E_{\widetilde{d}_E})$ as $n\to\infty$. 		
			\item\label{Cond-tail control} \textbf{$($Uniform Tail Control$)$} The time-averaged measures satisfy the following uniform tail estimate
			\begin{equation*}
				\lim_{R\to\infty} \sup_{n\ge 1} \nu_{T_n}^{x_0}\big(E \setminus B^{d_E}_R(x_0)\big) = 0,
			\end{equation*}
			where $B^{d_E}_R(x_0)$ is the closed ball defined in \eqref{eq: ball B-E-R}.
			
			\item\label{Cond-RestrictedFeller} \textbf{$($Restricted Feller Property under Mismatched Metrics$)$} The semigroup $\{\mathscr{P}_t\}_{t\ge 0}$ satisfies the restricted Feller property under mismatched metrics (cf.\ Definition \ref{Locally Feller definition}).
		\end{enumerate}
		Then $\nu\in\mathbf{P}(E)$ is an invariant probability measure for $\{\mathscr{P}_{t}\}_{t\ge0}$.
	\end{Theorem}
	
	\begin{Remark}\label{Remark-generalized KB} 
		We provide further context regarding the conditions in Theorem \ref{Thm-generalized KB} and address several underlying topological nuances:
		\begin{itemize}[leftmargin=0.79cm]\setlength\itemsep{0.2em}
			\item Regarding condition \ref{Cond-Tightness}, while tightness is typically formulated in complete spaces, its application within the generally incomplete space $E_{\widetilde{d}_E}$ is mathematically justified. In an arbitrary metrizable space, a tight family of probability measures inherently possesses a weakly convergent subsequence whose limit is a Radon measure supported on the space itself (see Lemmas \ref{Lemma-tightness} and \ref{Lemma-convergence of measures}). Furthermore, as explained in Remark \ref{Remark : measurability-upgrade}, the Borel $\sigma$-algebras associated with comparable Lusin topologies on a Polish space coincide (see, e.g., \cite[Page 107]{Schwartz-1973-Book} or \cite[Theorem 6.8.6]{Bogachev-2007-Books}); hence, $\mathscr{B}(E_{\widetilde{d}_E}) = \mathscr{B}(E)$. Consequently, the limit measure $\nu \in \mathbf{P}(E_{\widetilde{d}_E})$ obtained from condition \ref{Cond-Tightness} is a well-defined probability measure on the strong Borel $\sigma$-algebra $\mathscr{B}(E)$. This circumvents the need for a completion of $E_{\widetilde{d}_E}$ and directly verifies the support condition $\nu(E)=1$.
			
			\item Condition \ref{Cond-tail control} compensates for the localized nature of the restricted Feller property, and is typically verified via Chebyshev's inequality combined with \textit{a priori} bounds. In SPDE applications, whenever every $d_E$-bounded set is intrinsically precompact in $E_{\widetilde{d}_E}$---which is naturally the case under compact Sobolev embeddings---the closed balls $B^{d_E}_R(x_0)$ are compact in $E_{\widetilde{d}_E}$. In such scenarios, condition \ref{Cond-tail control} implies  tightness in $\mathbf{P}(E_{\widetilde{d}_E})$ (see Section \ref{Section: Probability Measures}).  
			
			\item Finally,  \ref{Hypo-generalized KB} requires the metric function $d_E(x_0, \cdot)$ to be lower semicontinuous with respect to $\widetilde{d}_E$. This structural assumption ensures that the $d_E$-closed ball $B^{d_E}_R(x_0)$ remains a closed subset within the topological space $E_{\widetilde{d}_E}$. This closure property is crucial, as it permits the application of the Tietze extension theorem, thereby bridging the local continuity of the restricted mapping $\mathscr{P}_{t}\varphi |_{B^{d_E}_R(x_0)}$ with global weak convergence.
		\end{itemize}
	\end{Remark}
	
	\subsection{Proof of Theorem \ref{Thm-generalized KB}}
	\label{Section : gKB-proof}
	
	\begin{proof}[Proof of Theorem \ref{Thm-generalized KB}]
		By Remark \ref{Remark-generalized KB}, the topological equality $\mathscr{B}(E_{\widetilde{d}_E}) = \mathscr{B}(E)$ guarantees that $\nu \in \mathbf{P}(E)$. To establish invariance, we proceed in two stages. First, we claim that for any arbitrary test function $\varphi \in C_{B}(E_{\widetilde{d}_E})$ and $t'>0$,
		\begin{equation}\label{eq:Phi_limit}
			\lim_{n \to \infty} \int_{E} [\mathscr{P}_{t'}\varphi](x)\, \nu_{T_n}^{x_0}(\mathrm{d}x) = \int_{E} [\mathscr{P}_{t'}\varphi](x)\, \nu(\mathrm{d}x).
		\end{equation}
		Define $\psi(x) \triangleq [\mathscr{P}_{t'}\varphi](x)$. Since $\mathscr{P}_{t'}:\mathscr{M}_B(E)\to\mathscr{M}_B(E)$ is a Markov propagator and $\varphi \in C_B(E_{\widetilde{d}_E}) \subset \mathscr{M}_B(E)$, the function $\psi$ is globally bounded by $\|\varphi\|_{L^\infty}$ and is Borel-measurable on $E$, rendering the integrals in \eqref{eq:Phi_limit} well-defined.
		
		However, condition \ref{Cond-RestrictedFeller} only guarantees that $\psi$ is $\widetilde{d}_E$-continuous when restricted to strong $d_E$-balls; it may not globally belong to $C_B(E_{\widetilde{d}_E})$. To circumvent this, let us fix $\varepsilon > 0$. By the tail condition \ref{Cond-tail control}, there exists a sufficiently large radius $R_0 > 0$ such that
		\begin{equation}\label{eq:tail_control_proof}
			\sup_{n \ge 1} \nu_{T_n}^{x_0}\big(E \setminus B^{d_E}_{R_0}(x_0)\big) < \varepsilon.
		\end{equation}
		The lower semicontinuity of $d_E(x_0, \cdot)$ with respect to $\widetilde{d}_{E}$ implies that the ball $B^{d_E}_{R_0}(x_0)$ is closed in the metrizable space $E_{\widetilde{d}_E}$. Here, condition \ref{Cond-RestrictedFeller} ensures that the restriction $\psi|_{B^{d_E}_{R_0}(x_0)}$ is $\widetilde{d}_{E}$-continuous. Since metrizable spaces are perfectly normal, the Tietze extension theorem yields a globally continuous function $\Phi_{R_0} \in C_B(E_{\widetilde{d}_E})$ satisfying
		\begin{equation*}
			\Phi_{R_0}(x) = \psi(x) \quad \text{for all } x \in B^{d_E}_{R_0}(x_0), \quad \text{and} \quad \sup_{x \in E}|\Phi_{R_0}(x)| \le \sup_{x \in B^{d_E}_{R_0}(x_0)}|\psi(x)| \le \|\varphi\|_{L^\infty}.
		\end{equation*}
		
		With this extension in hand, we naturally decompose the error into components localized inside and outside $B^{d_E}_{R_0}(x_0)$:
		\begin{equation}\label{eq:Tietze_split}
			\begin{aligned}
				\left| \int_{E} \psi \, \mathrm{d}\nu_{T_n}^{x_0} - \int_{E} \psi \, \mathrm{d}\nu \right| 
				\le \int_{E \setminus B^{d_E}_{R_0}(x_0)} |\psi - \Phi_{R_0}| \, \mathrm{d}\nu_{T_n}^{x_0}  
				+ \left| \int_{E} \Phi_{R_0} \, \mathrm{d}\nu_{T_n}^{x_0} - \int_{E} \Phi_{R_0} \, \mathrm{d}\nu \right| 
				+ \int_{E \setminus B^{d_E}_{R_0}(x_0)} |\Phi_{R_0} - \psi| \, \mathrm{d}\nu. 
			\end{aligned}
		\end{equation}
		Globally, we observe that $|\psi(x) - \Phi_{R_0}(x)| \le 2\|\varphi\|_{L^\infty}$. Thus, by \eqref{eq:tail_control_proof}, the first term on the right-hand side is strictly bounded by $2\|\varphi\|_{L^\infty} \varepsilon$. 
		
		For the third term, since the complement $E \setminus B^{d_E}_{R_0}(x_0)$ is an open set in $E_{\widetilde{d}_E}$, Lemma \ref{Lemma-convergence of measures} implies
		\begin{equation*}
			\nu\big(E \setminus B^{d_E}_{R_0}(x_0)\big) \le \liminf_{n \to \infty} \nu_{T_n}^{x_0}\big(E \setminus B^{d_E}_{R_0}(x_0)\big) \le \sup_{n \ge 1} \nu_{T_n}^{x_0}\big(E \setminus B^{d_E}_{R_0}(x_0)\big) < \varepsilon.
		\end{equation*}
		Hence, the third term is also bounded by $2\|\varphi\|_{L^\infty} \varepsilon$. 
		
		As for the middle term, the global continuity of $\Phi_{R_0} \in C_B(E_{\widetilde{d}_E})$ and the assumed weak convergence ensure that it vanishes as $n \to \infty$. Taking the limit superior in \eqref{eq:Tietze_split} then yields
		\begin{equation*}
			\limsup_{n \to \infty} \left| \int_{E} [\mathscr{P}_{t'}\varphi](x) \, \mathrm{d}\nu_{T_n}^{x_0} - \int_{E} [\mathscr{P}_{t'}\varphi](x) \, \mathrm{d}\nu \right| \le 4\|\varphi\|_{L^\infty} \varepsilon.
		\end{equation*}
		Since $\varepsilon > 0$ is arbitrary, we conclude that \eqref{eq:Phi_limit} rigorously holds.
		
		Equipped with \eqref{eq:Phi_limit}, we can now prove the invariance for functions in $C_B(E_{\widetilde{d}_E})$. By the definition of $\nu_{T_n}^{x_0}$, Fubini's theorem (justified by the temporal measurability in Definition \ref{Locally Feller definition}), and the semigroup property $\mathscr{P}_{t} \circ \mathscr{P}_{t'} = \mathscr{P}_{t+t'}$ (which holds universally for elements in $\mathscr{M}_B(E)$, independent of topological continuity), we obtain
		\begin{align}
			\int_{E} [\mathscr{P}_{t'}\varphi](x)\, \nu(\mathrm{d}x) 
			=\ & \lim_{n \to \infty} \frac{1}{T_n} \int_{0}^{T_n} [\mathscr{P}_{t+t'} \varphi](x_0) \, \mathrm{d} t \nonumber\\
			=\ & \lim_{n \to \infty} \frac{1}{T_n} \left(\int_{0}^{T_n} [\mathscr{P}_{r} \varphi](x_0) \, \mathrm{d} r + \int_{T_n}^{T_n+t'} [\mathscr{P}_{r} \varphi](x_0) \, \mathrm{d} r - \int_{0}^{t'} [\mathscr{P}_{r} \varphi](x_0) \, \mathrm{d} r\right) \nonumber\\
			=\ & \lim_{n \to \infty} \int_{E} \varphi(x)\, \nu_{T_n}^{x_0}(\mathrm{d}x) \nonumber \\
			=\ & \int_{E} \varphi(x)\, \nu(\mathrm{d} x), \quad t'\ge0,\label{eq:invariance_final}
		\end{align}
		where the boundary temporal integrals vanish asymptotically since $\varphi$ is uniformly bounded. The final equality follows directly from the weak convergence of $\nu_{T_n}^{x_0} \to \nu$ in $\mathbf{P}(E_{\widetilde{d}_E})$. 
		
		Note that \eqref{eq:invariance_final} can be reformulated as 
		\begin{align*}
			\int_{E} \varphi(x)\, [\mathscr{P}^*_{t'}\nu](\mathrm{d}x) = \int_{E} \varphi(x)\, \nu(\mathrm{d} x).
		\end{align*}
		Since $E_{\widetilde{d}_E}$ is a metric space, $C_B(E_{\widetilde{d}_E})$ forms a measure-determining class (also known as a separating class), see \cite[Theorem 1.2]{Billingsley-1999-Book}. Consequently, $\mathscr{P}^*_{t'}\nu \equiv \nu$ as probability measures on the Borel $\sigma$-algebra $\mathscr{B}(E_{\widetilde{d}_E})$.  As pointed out in Remark \ref{Remark-generalized KB}, we have the identity of Borel $\sigma$-algebras: $\mathscr{B}(E_{\widetilde{d}_E}) = \mathscr{B}(E)$. Thus, $\mathscr{P}^*_{t'}\nu \equiv \nu$ as measures on the strong measurable space $(E, \mathscr{B}(E))$. This completes the proof.
	\end{proof}

	\begin{Remark}[\emph{Detailed Comparison with Existing Frameworks}]\label{Remark-mismatch compare}
		While alternative abstract frameworks (e.g., \cite{CotiZelati-GlattHoltz-Trivisa-2021-AMO,Kim-2005-SIMA}) and specific applications (\cite{Bessaih-Ferrario-2020-CMP}) have addressed systems exhibiting Feller-type properties under weaker topologies, these existing approaches do not accommodate the \emph{restricted Feller property under mismatched metrics} developed herein.
		
		\begin{itemize}[leftmargin=0.79cm]\setlength\itemsep{0.2em}
			\item In \cite{Kim-2005-SIMA}, the Feller property is postulated on a larger, complete Banach space. By contrast, \ref{Hypo-generalized KB} does not require the weaker topological space $E_{\widetilde{d}_E}$ to be complete, thereby imposing strictly milder topological prerequisites. 
			
			\item The criterion developed in \cite{CotiZelati-GlattHoltz-Trivisa-2021-AMO} considers a space $E$ endowed with two metrics $d_s$ and $d_w$ (where $d_s$ is stronger than $d_w$), but its abstract formulation exhibits two notable structural limitations. First, it explicitly requires any $d_s$-bounded set to be precompact in the $d_w$-topology. While this condition naturally holds for Sobolev spaces on bounded domains, it inherently fails on unbounded domains (such as $\mathbb{R}^d$) due to the lack of compact embeddings. Our framework fully decouples this topological restriction: Theorem \ref{Thm-generalized KB} operates without requiring spatial precompactness, substituting it instead with a dynamical uniform tail control (Condition \ref{Cond-tail control}). Although we focus on bounded domains in the concrete applications herein, this structural decoupling preserves the mathematical flexibility to accommodate whole-space problems. Second, to theoretically ensure that a limiting measure obtained via the weaker topology uniquely defines a valid probability measure on the strong state space $(E,d_s)$, one fundamentally requires the measurable inclusion $\mathscr{B}\big((E,d_s)\big) \subset \mathscr{B}\big((E,d_w)\big)$. Although this inclusion fortuitously holds in many concrete SPDE applications (e.g., within Polish Sobolev spaces), it fails for general metric spaces. Without explicitly incorporating this prerequisite, an abstract framework cannot theoretically guarantee that the limiting measure truly returns to the original strong space. By formalizing precise topological constraints in  \ref{Hypo-generalized KB}, our formulation clarifies this measure-theoretic subtlety, rigorously ensuring that the constructed invariant measure genuinely resides on the \emph{original} strong Borel $\sigma$-algebra.
			
			\item For the 2D stochastic damped Euler equations in vorticity form driven by additive noise, the authors of \cite{Bessaih-Ferrario-2020-CMP} established the existence of an invariant probability measure defined on $\mathscr{B}\big((L^\infty, \mathscr{T}_{\rm bw*})\big)$---the Borel $\sigma$-algebra generated by the bounded weak-$\ast$ topology. However, it remains unclear whether such a measure can be well-defined on the strong Borel $\sigma$-algebra $\mathscr{B}\big((L^\infty, \|\cdot\|_{L^\infty})\big)$. By contrast, when applying Theorem \ref{Thm-generalized KB} to the higher-dimensional stochastic damped Euler equations under genuinely mixed multiplicative noise, we construct an invariant probability measure explicitly defined on 
			strong Borel $\sigma$-algebra
			\(\mathscr B(\Hdiv^s,\|\cdot\|_{\Hdiv^s})\), i.e., belonging to
			$\mathbf{P}\big((\Hdiv^{s},\|\cdot\|_{\Hdiv^s})\big)$. Consequently, by virtue of the continuous embedding $\Hdiv^{s}\hookrightarrow L^\infty$, this measure is naturally supported on $L^\infty$ (see Theorem \ref{Thm-(u P)-long-time} below).
		\end{itemize}
	\end{Remark}
	
	\subsection{Application to Abstract Singular Stochastic Evolution Systems}\label{Section : SSES}
	
	Theorem \ref{Thm-generalized KB} establishes an  abstract criterion for the existence of invariant probability measures tailored to Markov semigroups subject to the restricted Feller property under mismatched metrics. In this section, we demonstrate that this  machinery applies to a \textbf{broad} class of of stochastic fluid models. To highlight the universality of this approach, we first extract the core mechanism responsible for the loss of regularity in inviscid fluid dynamics and formulate an abstract singular stochastic evolution system. We then apply Theorem \ref{Thm-generalized KB} to   deduce the existence of invariant probability measures for such systems.
	
	Let $(\mathbf{H},\langle\cdot,\cdot\rangle_{\mathbf{H}})$ be a separable Hilbert space, continuously and densely embedded in a separable Fr\'echet space $(\widetilde{\mathbf{H}},d_{\widetilde{\mathbf{H}}})$, where $d_{\widetilde{\mathbf{H}}}$ is a translation-invariant metric. We consider the following singular stochastic evolution system for an $\mathbf{H}$-valued unknown process $X(t)$ with $t \geq 0$:
	\begin{equation}
		X(t) = X_0 + \int_0^t b_1(X(t')) \, \mathrm{d}t' + \int_0^t b_2(X(t')) \, \mathrm{d}W(t') + \int_0^t \int_{|l| \le 1} b_3(l,X(t'-)) \, \widetilde{\eta}(\mathrm{d}l,\mathrm{d}t')+ \int_0^t b_4(X(t')) \, \mathrm{d}\widetilde{W}(t'), \label{SES}
	\end{equation} 
	where the driving noises and coefficients satisfy the following standard structural assumptions:
	\begin{itemize}[leftmargin=0.79cm]\setlength\itemsep{0.2em}
		\item $\eta$ is an $\{\mathcal{F}_t\}$-adapted Poisson random measure counting jumps on the punctured ball $\{l \in \mathbb{R} : 0 < |l| \le 1\}$, with an intensity measure $\nu$ supported on $[-1,1] \setminus \{0\}$, and associated compensator $\widetilde{\eta}=\eta-\nu\otimes\mathrm{d}t$;
		\item $W(t)$ and $\widetilde{W}(t)$ are standard 1-D $\{\mathcal{F}_t\}$-adapted Brownian motions, and $W(t)$, $\widetilde{W}(t)$, and $\eta$ are mutually independent.
		
		\item the unknown process $X:\Omega\times[0,\infty)\to\mathbf{H}$ is $\{\mathcal{F}_t\}$-progressively measurable;
		\item $b_1,\ b_2 \in \mathscr{M}\big(\mathbf{H}; \widetilde{\mathbf{H}}\big)$, $b_3\in \mathscr{M}\big([-1,1]\setminus\{0\}\times \mathbf{H}; \widetilde{\mathbf{H}}\big)$, and $b_4\in \mathscr{M}\big(\mathbf{H}; \mathbf{H}\big)$.
	\end{itemize}
	
	Because the coefficients $b_1$, $b_2$, and $b_3$ map into $\widetilde{\mathbf{H}}$ rather than preserving the state space $\mathbf{H}$, the system \eqref{SES} is intrinsically singular. The severity of this singularity is quantified precisely by the regularity gap between $\mathbf{H}$ and $\widetilde{\mathbf{H}}$.  
	
	\begin{Definition}\label{Global-solution-definition-SES} 
		An $\{\mathcal{F}_t\}$-progressively measurable $\mathbf{H}$-valued process $X=(X(t))_{t\ge0}$ with càdlàg paths in $\widetilde{\mathbf{H}}$ is called a global solution to \eqref{SES} with initial data $X_0 \in \mathbf{H}$ if the following conditions hold $\mathbb{P}$-$\mathrm{a.s.}$:
		\begin{enumerate}[label={\rm(\arabic*)},leftmargin=0.79cm]\setlength\itemsep{0.2em}
			\item $\sup_{t\in[0,T]}\|X(t)\|_{\mathbf{H}}<\infty$ for every $T>0$;
			\item the integral equation \eqref{SES} holds identically for all $t\ge0$ as an equality in $\widetilde{\mathbf{H}}$.
		\end{enumerate}
	\end{Definition}
	Assuming global existence and pathwise uniqueness of solutions to \eqref{SES}, we denote by $X(t,x_0)$ the state of the solution at time $t$ originating from a deterministic initial state $x_0 \in \mathbf{H}$. Thanks to the time-homogeneous nature of the coefficients and the independent, stationary increments of the driving noises $W$ and $\eta$, the solution family naturally induces a Markov transition semigroup $\{\mathscr{P}_{t}\}_{t\ge0}$ on $\mathbf{H}$, whose action on bounded measurable functions $f \in \mathscr{M}_B(\mathbf{H})$ is defined via
	\begin{align}\label{Markov semigroup-SES}
		[\mathscr{P}_{t} f](x_0) \triangleq \mathbb{E} \big[f\big(X(t,x_0)\big)\big],\quad t \ge 0, \quad x_0 \in \mathbf{H}.
	\end{align}
	
	\begin{Remark}[Scope of the Abstract System of the Abstract System \eqref{SES}]\label{Remark-SES-examples}
		The singular system \eqref{SES} captures the fundamental mechanism underlying the loss of regularity in inviscid fluid models. For such nonlinear singular dynamics, the flow cannot intrinsically gain higher regularity than its initial data; consequently, the coefficients fail to map the state space into itself, an obstruction characterized by the strict embedding $\mathbf{H}\subsetneq \widetilde{\mathbf{H}}$. In this sense, \eqref{SES} encompasses a vast class of fluid dynamics equations. For instance, we observe that, under suitable structural hypotheses (e.g., \ref{Hypo-W L}, \ref{Hypo-Q-Pi} and \ref{Hypo-h div} in Section \ref{Section : Incompressible case}), specific choices of the coefficients $b_i$ precisely reduce the abstract equation \eqref{SES} to the targeted stochastic damped Euler equations \eqref{SEuler-(u P) damp} (see \eqref{Target problem (u P) damp} below):
		\begin{align*} 
			&b_1(u) = -\bigg[\Pi\big((u\cdot\nabla)u\big)+\Upgamma u-\frac{1}{2}\mathcal{Q}_1^2u\bigg]
			+\int_{|l|\le1} \Big\{\wp\big(1,l,u\big)-u-l\cdot\mathcal{Q}_2u\Big\}\,\nu(\mathrm{d}l), \\
			&b_2(u) = \mathcal{Q}_1u, \quad b_3(l,u) = \wp\big(1,l,u\big)-u,\quad b_4(u)=\widetilde h (u).
		\end{align*} 
		Here, $\Pi$ denotes the Leray projection (restricted to zero-average functions on $\mathbb{T}^d$, see \eqref{Pi=Pi-d-0}), $\mathcal{Q}_1$ and $\mathcal{Q}_2$ are spatial pseudo-differential noise operators, $\Upgamma\ge0$ is the physical damping coefficient, and $\wp$ represents the Marcus canonical flow associated with the jump amplitude $\mathcal{Q}_2$ (see \eqref{Marcus flow Q2}). To construct global strong solutions in the sense of Definition~\ref{Global-solution-definition-SES}, one typically operates within the high-regularity Sobolev space $\mathbf{H}=H_{\mathrm{div}}^{s}$ with $s\gg 1.$ After the derivative loss inherent to the transport term $(u\cdot\nabla)u$ and the noise operators, the system naturally takes values in a lower-regularity space $\widetilde{\mathbf{H}}=H_{\mathrm{div}}^{\theta},$ where $\theta < s$ depends on the concrete model. If $\theta > d/2$, then $\widetilde{\mathbf{H}}\hookrightarrow C(\mathbb{K}^d)$, ensuring that solutions retain classical spatial interpretations. 
		Furthermore, our unified abstract formulation intentionally bypasses any reliance on smoothing properties (e.g., from full or fractional viscosity) by treating any potential viscous term identically as a singular operator causing regularity loss (i.e., $H_{\mathrm{div}}^s\ni u\mapsto (-\Delta)^{\ell/2}u \in H_{\mathrm{div}}^{s-\ell}$). This viewpoint bridges the gap, allowing for the simultaneous analysis of invariant measures for both stochastic inviscid and viscous fluid models (see Remarks \ref{Remark: local-(u P) viscosity} and \ref{Remark: global-(u P) viscosity} below).
	\end{Remark}
	
	To study the long-time statistical behavior of \eqref{SES}, we introduce the following natural compactness assumption:
	\begin{Hypothesis}\label{Hypo-SES H E}
		The continuous embedding of the separable Hilbert space $(\bfH,\IP{\cdot,\cdot}_{\bfH})$ into the Fr\'{e}chet space $(\widetilde{\bfH},d_{\widetilde{\bfH}})$ is compact, i.e., $\bfH \hookrightarrow\hookrightarrow \widetilde{\bfH}$.
	\end{Hypothesis}
	
	\begin{Theorem}[\textbf{Invariant Probability Measures for Abstract Singular Systems}]\label{Thm-SES-IM}
		Let \ref{Hypo-SES H E} hold. Suppose there exists a deterministic initial state $x_0 \in \bfH$ such that the following conditions are met:
		\begin{enumerate}[label={\bf (\arabic*)},leftmargin=0.79cm]\setlength\itemsep{0.2em}
			\item The system \eqref{SES} admits a pathwise unique global solution $X(t,x_0)$ starting from $x_0$ which satisfies a uniform-in-time \textit{a priori} bound:
			\begin{equation}\label{E V X}
				\sup_{t\ge0}\mathbb{E}\Big[f\big(\|X(t,x_0)\|_{\bfH}^2\big)\Big] < \infty,
			\end{equation}
			for some strictly increasing, continuous function $f:[0,\infty)\to [0,\infty)$ with $f(0)=0$ and $\lim_{r\to\infty} f(r)=\infty$.
			
			\item\label{restricted mismatched feller SES} The Markov semigroup $\{\mathscr{P}_{t}\}_{t\ge0}$ defined by \eqref{Markov semigroup-SES} satisfies the restricted Feller property under mismatched metrics (cf. Definition \ref{Locally Feller definition}) on the state space $E=\bfH$ equipped with  $\widetilde{d}_E=d_{\widetilde{\bfH}}$.
		\end{enumerate}
		Then there exists an invariant probability measure for $\{\mathscr{P}_{t}\}_{t\ge0}$ supported on $\bfH$ in the sense of Definition \ref{IM definition}.
	\end{Theorem}
	
	\begin{proof}[Proof of Theorem \ref{Thm-SES-IM}]
		It is clear that \ref{Hypo-generalized KB} holds for $E=\mathbf{H}$ and $\widetilde{d}_E=d_{\widetilde{\mathbf{H}}}$. Since the Markov propagator $\{\mathscr{P}_{t}\}_{t\ge0}$ is assumed to satisfy the restricted Feller property under mismatched metrics, Condition \ref{Cond-RestrictedFeller} of Theorem \ref{Thm-generalized KB} is immediately met.  
		Consider the family of time-averaged measures $\{\nu_{T}^{x_0}\}_{T > 0} \subset \mathbf{P}(\mathbf{H})$ defined on the strong Borel $\sigma$-algebra $\mathscr{B}(\mathbf{H})$:
		\begin{align*}
			\nu_{T}^{x_0}(\cdot) \triangleq \frac{1}{T} \int_{0}^{T} \mathscr{P}^*_{t} \delta_{x_0} \mathrm{d}t, \quad T>0.
		\end{align*}
		Let $\mathbf{H}_{d_{\widetilde{\mathbf{H}}}} \triangleq (\mathbf{H},d_{\widetilde{\mathbf{H}}})$. Because $\mathbf{H}$ is a reflexive Hilbert space, its closed norm balls $B^{\mathbf{H}}_R \triangleq \big\{v \in \mathbf{H} : \|v\|_{\mathbf{H}} \le R\big\}$ are weakly compact. Furthermore, since the compact embedding $\mathbf{H} \hookrightarrow\hookrightarrow \widetilde{\mathbf{H}}$ maps weakly convergent sequences in $\mathbf{H}$ to strongly convergent sequences in $\widetilde{\mathbf{H}}$, the weakly compact sets $B^{\mathbf{H}}_R$ are strongly compact in the topology of $\mathbf{H}_{d_{\widetilde{\mathbf{H}}}}$. Consequently, verifying tightness in ${\bf P}(\mathbf{H}_{d_{\widetilde{\mathbf{H}}}})$ reduces precisely to establishing uniform tail control in the strong norm (Condition \ref{Cond-tail control} of Theorem \ref{Thm-generalized KB}). To apply Theorem \ref{Thm-generalized KB}, it thus suffices to verify the tightness of $\{\nu_{T}^{x_0}\}_{T > 0}$ in ${\bf P}(\mathbf{H}_{d_{\widetilde{\mathbf{H}}}})$.
		
		We establish this tightness via Chebyshev's inequality. Let $C \triangleq \sup_{t\ge0}\mathbb{E}\big[f(\|X(t,x_0)\|_{\mathbf{H}}^2)\big]$, which is finite by assumption \eqref{E V X}. For any threshold radius $R>0$, we bound the measure of the complement $\mathbf{H} \setminus B^{\mathbf{H}}_R$ as follows:
		\begin{align}
			\nu_{T}^{x_0}(\mathbf{H} \setminus B^{\mathbf{H}}_R) 
			=\nu_{T}^{x_0}\Big(\big\{v \in \mathbf{H}: f(\|v\|_{\mathbf{H}}^2) > f(R^2)\big\}\Big) 
			=\leq  \frac{1}{f(R^2)} \int_{\mathbf{H}} f(\|v\|_{\mathbf{H}}^2) \, \nu_{T}^{x_0}(\mathrm{d} v) 
			\leq \frac{C}{f(R^2)}. \label{tight of nu-(T,-n)}
		\end{align}
		Since $\lim_{r\to\infty} f(r) = \infty$, for any arbitrarily small $\varepsilon>0$, choosing $R$ sufficiently large guarantees that $C / f(R^2) \leq \varepsilon$. This yields $\nu_{T}^{x_0}(\mathbf{H} \setminus B^{\mathbf{H}}_R) \leq \varepsilon$ uniformly for all $T > 0$, confirming the required tightness of $\{\nu_{T}^{x_0}\}_{T > 0}$ in ${\bf P}(\mathbf{H}_{d_{\widetilde{\mathbf{H}}}})$. 
		
		By Lemma \ref{Lemma-convergence of measures}, Condition \ref{Cond-Tightness} of Theorem \ref{Thm-generalized KB} is satisfied. Having fulfilled all prerequisites of Theorem \ref{Thm-generalized KB}, we conclude the existence of an invariant probability measure for the abstract singular system.
	\end{proof}
	
	\begin{Remark}
		Condition \eqref{E V X} in Theorem \ref{Thm-SES-IM} functions as a uniform-in-time coercive Lyapunov bound rather than a stringent higher-moment energy estimate. The function $f$ is solely required to be increasing, continuous, and unbounded; it may grow \emph{arbitrarily slowly}, as it enters the analysis only through Chebyshev's inequality to secure the tightness bound \eqref{tight of nu-(T,-n)}. Similarly, Condition \ref{restricted mismatched feller SES} in Theorem \ref{Thm-SES-IM} imposes minimal topological constraints. By completely circumventing the need for a Feller property within a single topology---a rigid limitation inherent to the classical Krylov--Bogoliubov criterion---this framework relies exclusively on the restricted Feller property under mismatched metrics. Consequently, in applications, Theorem \ref{Thm-SES-IM} affords \textbf{substantial flexibility} in the choice of the weaker metric $d_{\widetilde{\mathbf{H}}}$, naturally accommodating singular and nonlinear dynamics. 
		As a concrete realization of this abstract machinery, Theorems \ref{Thm-(u P)-long-time} and \ref{Thm-(u P)-IM} establish the existence of invariant probability measures for the stochastic damped Euler equations \eqref{SEuler-(u P) damp} on the torus $\mathbb{T}^d$ across all spatial dimensions $d \geq 2$. Furthermore, as highlighted in Remark \ref{Remark-SES-examples}, because this abstract framework does \emph{not} rely on the parabolic smoothing properties of the Laplacian (treating viscosity merely as an alternative mechanism of regularity loss), Theorem \ref{Thm-SES-IM} applies equally well to fluid models driven by fractional or  degenerate dissipations (see Remarks \ref{Remark: local-(u P) viscosity} and \ref{Remark: global-(u P) viscosity} below).
	\end{Remark}
	
	\section{Stochastic Incompressible Case}\label{Section : Incompressible case}
	
	\subsection{Assumptions and Target Model}\label{Section : Incompressible-Assumptions}
	
	The stochastic damped Euler equations \eqref{SEuler-(u P) damp} are understood, similarly to \eqref{Target problem (rho u)}, in the following sense:
	\begin{equation}\label{SEuler-damping-target}
		\left\{\begin{aligned}
			u(t)-u_0 \ 
			+\ &\int_0^{t}  \bigg[(u\cdot\nabla)u+\Upgamma u-\frac{1}{2} \mathcal{Q}_1^2u\bigg](t') \d t'+\nabla P(t)-\nabla P(0)\\
			=\ & \int_0^{t} \left(\Q_{1}u(t') \d W(t')+  \widetilde h\big(u(t')\big)\d \widetilde W(t')\right )\\
			&+\int_0^t\int_{|l|\le1} \Big\{\wp\big(1,l,u(t'-)\big)-u(t'-)\Big\}\, \widetilde{\eta}({\rm d}l,{\rm d}t')\\
			&+\int_0^t\int_{|l|\le1} \Big\{\wp\big(1,l,u(t')\big)-u(t')-l\cdot \Q_2 u(t')\Big\}\, \nu({\rm d}l){\rm d}t',\\
			{\rm div}\, u =0,\ \ \ & \\
			u(0)=u_0,\ \, \, &
		\end{aligned}
		\right.
	\end{equation}
	where $\wp$ denotes the solution to \eqref{Marcus flow Q2}, i.e., 
	\begin{equation*} 
		\wp(r)\triangleq\wp(r,l,f),\quad 
		\frac{{\rm d}}{{\rm d}r}\wp(r)= l\cdot \Q_2\wp(r),\quad r\in[0,1],  \quad \wp(0) = f.
	\end{equation*}
	In \eqref{SEuler-damping-target}, the pressure $\nabla P$ can be eliminated using the Leray projection $\Pi_d$ defined in \eqref{Pi-d define}, but care must be taken with the spatial domain. As previously noted, $\Pi_d$ is not well-defined in $H^s(\mathbb{T}^d;\mathbb{R}^d)$ because $\Delta$ is not invertible. Hence, in the periodic setting, we consider \eqref{SEuler-damping-target} under the conditions $\Pi_0u=u$ and $\Pi_0P=0$.  With this structure, when $x\in\T^d$, we restrict the system to the zero-average subspace by applying the zero-average projection $\Pi_0$ (cf. \eqref{Pi-0 define}) to \eqref{SEuler-damping-target}.

	To conclude, on account of the decomposition \eqref{Hodge decomposition},  applying the projection $\Pi$ defined in \eqref{Pi=Pi-d-0}, i.e.,
	\begin{equation*} 
		\Pi\triangleq\Pi_d \ \ \text{on}\ x\in\mathbb{R}^d \ \ \text{and}\ \ 
		\Pi\triangleq\Pi_d\Pi_0\ \ \text{on}\  x\in\mathbb{T}^d,
	\end{equation*}
	we can formulate  \eqref{SEuler-damping-target} as
	\begin{align}\label{Target problem (u P)-remark}
		\left\{\begin{aligned}
			u(t)-u_0 \ 
			+&\int_0^{t}  \bigg[\Pi(u\cdot\nabla)u+\Upgamma u-\frac{1}{2} \Pi\mathcal{Q}_1^2u\bigg](t') \d t'\\
			= & \int_0^{t} \left(\Pi\Q_{1}u(t') \d W(t')+ \Pi  \widetilde h\big(u(t')\big)\d \widetilde W(t')\right )\\
			&+\int_0^t\int_{|l|\le1} \Big\{\Pi\wp\big(1,l,u(t'-)\big)-\Pi u(t'-)\Big\}\, \widetilde{\eta}({\rm d}l,{\rm d}t')\\
			&+\int_0^t\int_{|l|\le1} \Big\{\Pi\wp\big(1,l,u(t')\big)-\Pi u(t')-l\cdot \Pi\Q_2 u(t')\Big\}\, \nu({\rm d}l){\rm d}t'
			,\\
			u(0)=u_0.\, \, &
		\end{aligned}
		\right.
	\end{align}
	However, the product of $\Pi$ and $\Q_i$ is generally very difficult to handle. 
	We therefore introduce the space $\Hdiv^d$ (as defined in \eqref{H-div}):
	\begin{equation*}
		\Hdiv^s(\mathbb K^d;\R^d)\triangleq\Pi H^s(\mathbb K^d;\R^d),
	\end{equation*}
	and formulate the following assumption:
	\begin{Hypothesis}
		\label{Hypo-Q-Pi}
		Suppose that for $d=m\ge2$ in \eqref{Ss define}, the operators  $\Q_i$  with $i=1,2$ satisfy  \ref{Hypo-Qi} and
		$$
		\Pi\Q_i\big|_{\Hdiv^s}=\Q_i\big|_{\Hdiv^s}.
		$$
	\end{Hypothesis}
	For the It\^{o}-forcing part, we impose the following assumption on $\widetilde h(u)$:
	
	\begin{Hypothesis}
		\label{Hypo-h div}
		Let $d\ge 2$, $p\ge 1$, and $\sigma > \frac{d}{2}+p$, $\widetilde h: \Hdiv^\sigma\to \Hdiv^\sigma$. Moreover,  there exists a non-decreasing function
		$\widetilde K\in \mathscr{M}([0,\infty);[0,\infty))$ such that 
		for all $t\ge 0$ and $u, u_1, u_2\in H^{\sigma}$,
		\begin{align*}
			\|\widetilde h(u)\|^2_{H^{\sigma}} 
			\leq  \widetilde K\big(\|u\|_{W^{p,\infty}}\big)(1+\|u\|^{2}_{H^{\sigma}}),
		\end{align*}
		\begin{align*}
			\|\widetilde h(u_1)- \widetilde h( u_2)\|^{2}_{H^{\sigma}} \le 
			\widetilde K\big(\|u_1\|_{H^{\sigma}}+\|u_2\|_{H^{\sigma}}\big)
			\|u_1-u_2\|^{2}_{H^{\sigma}}.
		\end{align*}
	\end{Hypothesis}

	\begin{Remark} 
		There are at least two reasons why \ref{Hypo-Q-Pi} is essential.
		\begin{enumerate}[label={\bf (\alph*)},leftmargin=0.79cm]\setlength\itemsep{0.2em}
			\item From the perspective of the Marcus integral: if $\wp$ satisfies \eqref{Marcus flow Q2}, it is \textbf{unknown} whether $\frac{{\rm d}}{{\rm d}r}\Pi\wp(r) = l \cdot \Pi\Q_2\Pi\wp(r)$ holds. If this equality fails, then \eqref{Target problem (u P)-remark} becomes \textbf{self-inconsistent} because $\Pi\wp$ would \textbf{not} be the Marcus flow corresponding to $\Pi\Q_2$. 
			
			\item From the perspective of the cancellation properties in \eqref{Cancel-introduction} (see Theorem \ref{Thm-cancel}): the formal argument requires the identity $\Pi\Q_i^2 = (\Pi\Q_i)^2$ to hold in \eqref{Target problem (u P)-remark}. This identity generally \textbf{does not} hold. Moreover, because the multiplier of $\Pi_d$ lacks smoothness at $\xi=0$ (cf. \eqref{Pi-d define}), even if the formal equality $\Pi\Q_i^2 = (\Pi\Q_i)^2$ were valid, the operator $\Pi\Q_i$ might not be a pseudo-differential operator. Consequently, the cancellation properties in \eqref{Cancel-introduction} may fail.
		\end{enumerate}
		The role of \ref{Hypo-Q-Pi} is to circumvent the two issues described above. We also note that, if $\Q_i$ is of the following form
		$$\Q_i={\mathrm{diag}}( \mathcal T_{i}, \cdots,  \mathcal T_{i})=\mathcal T_{i}\mathbf{I}$$
		for some $\mathcal T_{i}\in \OP\mathcal{S}_0^\beta$ (recall that the subscript 0 means that the symbol is independent of $x$, see Section \ref{Section:Notations}) such that $\mathcal T_{i}+\mathcal T^*_{i}\in\OP\mathcal{S}_0^0$, $i=1,2$, then $\Q_1,\Q_2\in\mathbb{B}^\beta$ and satisfy \ref{Hypo-Q-Pi} since $\Pi$ commutes with $\Q_i$.
	\end{Remark}

	Under  \ref{Hypo-Q-Pi} and \ref{Hypo-h div}, finding $\Hdiv^s$-valued solutions to \eqref{Target problem (u P)-remark} is equivalent to finding $\Hdiv^s$-valued solutions to  
	\begin{equation}\label{Target problem (u P) damp}
		\left\{\begin{aligned}
			u(t)-u_0 \ 
			+&\int_0^{t}  \bigg[\Pi(u\cdot\nabla)u+\Upgamma u-\frac{1}{2} \mathcal{Q}_1^2u\bigg](t') \d t'\\
			= & \int_0^{t} \left(\Q_{1}u(t') \d W(t')+   \widetilde h\big(u(t')\big)\d \widetilde W(t')\right )\\
			&+\int_0^t\int_{|l|\le1} \Big\{\wp\big(1,l,u(t'-)\big)-u(t'-)\Big\}\, \widetilde{\eta}({\rm d}l,{\rm d}t')\\
			&+\int_0^t\int_{|l|\le1} \Big\{\wp\big(1,l,u(t')\big)-u(t')-l\cdot \Q_2 u(t')\Big\}\, \nu({\rm d}l){\rm d}t',\\
			u(0)=u_0.\, \, &
		\end{aligned}
		\right.
	\end{equation}
	where $\wp$ is the solution to \eqref{Marcus flow Q2}.

	As in the compressible case, we focus on \textit{classical} solutions requiring that  \eqref{Target problem (u P) damp}$_1$  hold as an equation in $C(\K^d)$.  
	We now precisely state the following definition of pathwise classical solutions to \eqref{Target problem (u P) damp}:
	
	\begin{Definition}\label{solution definition (u P)} 
		Let $\tau^*$ be a stopping time with $\p(\tau^*>0)=1$, and let $(u,\tau^*) \triangleq (u(t))_{t\in [0,\tau^*)}$ be a progressively measurable process on $\Hdiv^s$.
		The pair $(u,\tau^*)$ is called a maximal $\Hdiv^s$ classical solution to \eqref{Target problem (u P) damp} if $\pas$ the following conditions are satisfied:
		\begin{itemize}
			\item \eqref{Target problem (u P) damp} holds as an equality in $C(\K^d)$ for all $t\in[0,\tau^*)$;
			\item  $\sup_{[0,T]}\|u(t)\|_{H^s}<\infty$ for all $T<\tau^*$;
			\item $\limsup_{t\uparrow\tau^*}\|u(t)\|_{\H^s}=\infty$ a.s. on $\{\tau^*<\infty\}.$
		\end{itemize}
		In particular, if $\p(\tau^*=\infty)=1$, the solution is called global.
		
	\end{Definition}
	After establishing the maximal solution, our next goal is to quantify the long-time regimes induced by the \textbf{damping-noise interaction} (\textbf{DNI})
	\begin{equation*} 
		\Upgamma u, \quad \mathcal{Q}_1 u\, \circ\, {\rm d} W,\quad \mathcal{Q}_2 u \diamond {\rm d} L, \quad \text{and} \quad \widetilde h( u)\d \widetilde{W}. 
	\end{equation*}
	
	\begin{Remark}\label{Remark singualrity IC}
		To analyze the long-time behavior of solutions under varying \textbf{DNI}, we introduce the function class $\mathscr{V}$ defined in \eqref{SCRV} as Lyapunov-type functions, following the approach in \cite{Ren-Tang-Wang-2024-POTA} (see also \cite{Tang-Wang-2024-CCM,Tang-Yang-2023-AIHP}). The following points are noteworthy:
		
		\begin{enumerate}[leftmargin=0.79cm]\setlength\itemsep{0.2em}
			\item As noted in Remark \ref{Remark singualrity}, the stochastic damped Euler equations \eqref{Target problem (u P) damp} form a singular evolution system in the Sobolev space $\Hdiv^s$, since their coefficients do not map $\Hdiv^s$ into itself. Consequently, for a local solution $(u,\tau)$, one cannot directly apply It\^o's formula to $\|u\|^2_{H^s}$ to establish global-in-time estimate, as the required inner products are \textbf{not} well-defined.   For systems without Marcus noise, a known  approach in the literature is to apply a mollifier $J_n$ (for example, as defined in \eqref{Define Jn}) to the equation, estimate $\|J_n u\|^2_{H^s}$, and then pass to the limit (see, e.g., \cite{Miao-Rohde-Tang-2024-SPDE,Tang-Yang-2023-AIHP,Li-Liu-Tang-2021-SPA}). However, when a Marcus integral is present, applying $J_n$ to the term $\wp\big(1,l,u(t')\big)-u(t')-l\cdot\Q_2u(t')$ yields
			\[
			J_n\wp\big(1,l,u(t')\big)-J_nu(t')-l\cdot J_n\Q_2u(t'),
			\]
			which is \textbf{inconsistent}, because $J_n\wp\big(1,l,u(t')\big)$ is generally \textbf{not} the Marcus flow corresponding to the operator $J_n\Q_2$ with initial condition $J_nu(t')$.
			
			\item Suppose a blow-up criterion is available in terms of the $\Wlip$-norm; that is, the $H^s$-norm blows up if and only if the $W^{1,\infty}$-norm blows up. Assume further that It\^o's formula can be applied to $\|u\|^2_{H^\theta}$ for some suitable $\theta$ satisfying $\Hdiv^s\hookrightarrow\Hdiv^{\theta}\hookrightarrow \Wlip$, yielding a global $H^\theta$ bound. Invoking the blow-up criterion then implies global existence in $H^s$. 
			This argument, however, does not yield a direct bound on the $H^s$-norm itself. When an explicit global $H^s$ estimate is required, we therefore adopt a different approach. We first construct an approximation scheme whose solutions satisfy uniform $H^s$ bounds, and we impose the conditions ensuring global $H^s$ estimates already at the approximation level. Passing to the limit then yields a solution to the original equation that inherits the $H^s$ estimate, thereby providing the desired global $H^s$ bound. The details are presented in Section~\ref{Section : SEuler-global-proof}.

			\item We record two constants used in the analysis of \textbf{DNI}. 
			For a regularized version of \eqref{Target problem (u P) damp}, in which each $\Q_{i}$ is replaced by its regularized counterpart $\Q_{i,n}$ ($i=1,2$), we show that there exist constants $\mathscr{A}_i=\mathscr{A}_i(s,d)>0$ for $i=1,2$ (see \eqref{Qn cancel-2} and \eqref{wp-n estimate 2 C22} above) such that, for every $f\in H^s$ with $s\ge0$,
			\begin{equation}\label{Qn cancel-2 A1}
				\sup_{n \ge 1} \left| \langle \mathcal{Q}_{1,n}^2 f, f \rangle_{H^s}
				+ \langle \mathcal{Q}_{1,n} f, \mathcal{Q}_{1,n} f \rangle_{H^s} \right|
				\leq \mathscr{A}_1 \|f\|^2_{H^s}.
			\end{equation}
			Moreover, for every $V\in\mathscr{V}$,
			\begin{align}\label{wp-n estimate 2 C22 A2}
				\sup_{n\ge1}\int_{|l|\le 1}
				\Big\{V\!\big(\|\wp_n(r,l,f)\|_{H^s}^2\big)
				-V\!\big(\norm{f}_{H^s}^2\big)
				-2l\cdot V'\!\big(\|f\|^2_{H^s}\big)\bIP{\Q_{2,n}f,f}_{H^s}\Big\}\,\nu({\rm d}l)
				\leq \mathscr{A}_2\,V(\|f\|^2_{H^s}).
			\end{align}
			Furthermore, if $\Q_i\in\mathbb{B}^\beta$ and $Q_i+\Q_i^*=0$, then  we have (cf.\ \ref{Q-n cancel 0} in Lemma~\ref{Lemma:Qn} and Remark~\ref{Remark: wp-n estimates})
			\begin{equation*}
				\left| \langle \mathcal{Q}_{1,n}^2 f, f \rangle_{H^s} + \langle \mathcal{Q}_{1,n} f, \mathcal{Q}_{1,n} f \rangle_{H^s} \right| = 0,\quad n\ge1,
			\end{equation*}
			and
			\begin{align*}
				\int_{|l|\le 1}
				\Big\{V(\|\wp_n(r,l, f)\|_{H^s}^2) - V(\norm{f}_{H^s}^2)-2l\cdot V'(\|f\|^2_{H^s} )\bIP{\Q_{2,n}f,f}_{H^s}\Big\}\,\nu({\rm d}l)
				= 0,\quad n\ge1.
			\end{align*}

			A comparison with \eqref{Qn cancel-2} and \eqref{wp-n estimate 2 C22} yields
			\begin{equation}\label{A1=C12 A2=C22}
				\mathscr{A}_1 = C_{1,2}, \quad \mathscr{A}_2 = \mathscr{C}_{2,2}.
			\end{equation}
			The notations $C_{1,2}$ and $\mathscr{C}_{2,2}$ are used to distinguish particular cases that arise in frameworks involving several scenarios; see \eqref{Qn cancel-1}, \eqref{Qn cancel-2}, \eqref{wp-n estimate 1 C21}, and \eqref{wp-n estimate 2 C22}.  In the present setting, however, only the two estimates   \eqref{Qn cancel-2 A1} and \eqref{wp-n estimate 2 C22 A2}, together with their counterparts \eqref{Qn cancel-2} and \eqref{wp-n estimate 2 C22}, are needed.  It is therefore convenient to write $\mathscr{A}_1$ and $\mathscr{A}_2$ instead. For a more detailed discussion, see  Remarks~\ref{Remark:renormalization Qn}, \ref{Remark: wp-n estimates} and \ref{Remark: wp estimate}, Lemmas~\ref{Lemma:Qn} and \ref{Lemma-Marcus-n}, and Section~\ref{Section : SEuler-global-proof}.
		\end{enumerate}
	\end{Remark}
	
	From Remark \ref{Remark singualrity IC} and Lemma \ref{Lemma:Qn}, we define
	\begin{align}\label{SCR ai}
		\mathscr{a}_i=\mathscr{a}_i(s,d) \triangleq  
		\begin{cases}
			0, &\text{if }\ \text{\ref{Hypo-Qi} holds and }\ \Q_i\in\mathbb{B}^\beta,\ \ \Q_i+\Q_i^*=0,\vspace*{4pt}\\
			\mathscr{A}_i, &\text{otherwise},
		\end{cases}
	\end{align}
	where $\mathscr{A}_1$ and $\mathscr{A}_2$ denote the constants appearing in \eqref{Qn cancel-2 A1} and \eqref{wp-n estimate 2 C22 A2}, respectively. Thus $\mathscr{a}_1=\mathscr{a}_1(s,d)$ and $\mathscr{a}_2=\mathscr{a}_2(s,d)$ quantify the strength (in $H^s$) of the noise terms $\mathcal{Q}_1 u\, \circ\, {\rm d} W$ and $\mathcal{Q}_2 u \diamond {\rm d} L$. 
	
	Given $V\in\mathscr{V}$ and $\widetilde h$ satisfying \ref{Hypo-h div} with $p\ge1$, we define $\mathbf{V}_{\sigma}^{\Q_1,\,\Q_2,\,\widetilde h}$ by
	\begin{align}  
		\mathbf{V}_{\sigma}^{\Q_1,\, \Q_2,\, \widetilde h}(f) 
		\triangleq\ & V'(\|f\|^2_{H^{\sigma}}) \bigg\{  
		\big(\mathscr{a}_1 + 2\mathscr{c}_{\sigma,d} \|f\|_{\Wlip} \big) \|f\|^2_{H^{\sigma}}    + \|\widetilde h(f)\|^2_{H^{\sigma}} \bigg\}  \notag\\
		& + 2V''(\|f\|^2_{H^{\sigma}})\bIP{\widetilde h(f), f}_{H^{\sigma}}^2+\mathscr{a}_2 V(\|f\|^2_{H^{\sigma}}), \qquad f\in \Hdiv^{\sigma},\quad \sigma > \tfrac{d}{2} + p, \label{V Lyapunov type term}
	\end{align}  
	where $\mathscr{a}_1$ and $\mathscr{a}_2$ are given in \eqref{SCR ai}, and $\mathscr{c}=\mathscr{c}_{\sigma,d}>0$ is taken from Lemma \ref{uux u M} and denotes the growth coefficient of $\IP{\Pi(u\cdot\nn)u,u}_{H^{\sigma}}$.
	
	Let $p\ge1$ be fixed as in \ref{Hypo-h div}.
	We now characterize the \textbf{DNI} conditions implicitly using  $2\Upgamma V'(\|f\|^2_{H^{\sigma}} )\|f\|^2_{H^{\sigma}}$ and $\mathbf{V}_{\sigma}^{\Q_1,\, \Q_2,\, \widetilde h}$ in \eqref{V Lyapunov type term}. Recall that $\mathscr{V}$ is defined in \eqref{SCRV}.
	
	\begin{Hypothesis}
		\label{Hypo-DNI-1}
		There exist $V\in \mathscr{V}$ and constants $\mathscr{G}_1,\mathscr{G}_2>0$ such that
		\begin{align*}
			\mathbf{V}_{\sigma}^{\Q_1,\, \Q_2,\, \widetilde h}(f)
			\leq \mathscr{G}_1\,V(\|f\|^2_{H^{\sigma}})+\mathscr{G}_2,\qquad f\in  \Hdiv^{\sigma},\quad \sigma>d/2+p.
		\end{align*}
	\end{Hypothesis}

	\begin{Hypothesis}
		\label{Hypo-DNI-2}
		$\Q_2\in\mathbb{B}^\beta$ and $Q_2+\Q_2^*=0$. Moreover, there exists $V\in \mathscr{V}$ such that
		\begin{align*}
			\mathbf{V}_{\sigma}^{\Q_1,\, \Q_2,\, \widetilde h}(f)-2\Upgamma V'(\|f\|^2_{H^{\sigma}} )\|f\|^2_{H^{\sigma}} 
			\leq 0,\qquad f\in  \Hdiv^{\sigma},\quad \sigma>d/2+p.
		\end{align*}
	\end{Hypothesis}
	
	\begin{Hypothesis}
		\label{Hypo-DNI-3}
		$\Q_2\in\mathbb{B}^\beta$ and $Q_2+\Q_2^*=0$. Moreover, there exists  $V\in \mathscr{V}$ with $\sup_{x\in\R}|V''(x)|<\infty$ and a constant $\mathscr{G}_3>0$ such that 
		\begin{align*} 
			\mathbf{V}_{\sigma}^{\Q_1,\, \Q_2,\, \widetilde h}(f) -2\Upgamma V'(\|f\|^2_{H^{\sigma}} )\|f\|^2_{H^{\sigma}} 
			\leq -\mathscr{G}_3  V\left(\frac{\|f\|^2_{W^{p,\infty}}}{M^2}\right),\qquad f\in  \Hdiv^{\sigma},\quad \sigma>d/2+p.
		\end{align*}
		Here $M=M_{\sigma,d}>0$ is the embedding constant for $\|\cdot\|_{W^{p,\infty}}\leq M_{s,d}\|\cdot\|_{H^s}$ when $s>d/2+p$.
	\end{Hypothesis}

	\begin{Remark}\label{Remark-DNI-Compare}
		Applying It\^o's formula, we can formally derive the following {\it a priori} estimate (a regularization process is required for rigorous validation; see Remark \ref{Remark singualrity IC}):  
		\begin{equation*} 
			\mathbb{E} \Big[ V(\|u(t)\|^2_{H^\sigma}) \big| \F_0 \Big]-V(\|u_0\|^2_{H^\sigma})  
			\leq \int_0^t\mathbb{E} \Big[ \left(\mathbf{V}_{\sigma}^{\Q_1,\, \Q_2,\, \widetilde h}(u(t')) -2\Upgamma V'(\|u(t')\|^2_{H^{\sigma}} )\|u(t')\|^2_{H^{\sigma}}\right)  \Big| \F_0 \Big]\d t'.
		\end{equation*}
		The logical relationships between these conditions are clearly expressed by
		\begin{center}
			\ref{Hypo-DNI-1} $\Longleftarrow$ \ref{Hypo-DNI-2} $\Longleftarrow$ \ref{Hypo-DNI-3}.
		\end{center}
		To make this precise, we provide the following comparative analysis:
		\begin{enumerate}[leftmargin=0.79cm]\setlength\itemsep{0.2em}
			\item \textbf{Weak dissipation of \textbf{DNI} via \ref{Hypo-DNI-1}.}  
			A dissipative mechanism is described in \ref{Hypo-DNI-1}: when the growth of $\widetilde h(\cdot)$ exceeds a certain threshold, it counteracts the growth contributions from other terms (due to $V''<0$), leading to linear growth in $\mathbb{E}\big[V(\|u(t)\|^2_{H^\sigma})|\mathcal{F}_0\big]$ for the solution $u$, despite the growth observed in $\mathbb{E}\big[\|u(t)\|^2_{H^\sigma}|\mathcal{F}_0\big]$. Consequently, global existence is guaranteed under \ref{Hypo-DNI-1}; related results in a more general setting are first provided in \cite{Ren-Tang-Wang-2024-POTA,Tang-Wang-2022-arXiv}. Nevertheless, the dissipation remains weak because \ref{Hypo-DNI-1} still permits the possibility of a ``blow-up at $t=\infty$''.  
			
			\item \textbf{Moderate dissipation of \textbf{DNI} via \ref{Hypo-DNI-2}.} \ref{Hypo-DNI-2} produces a  stronger dissipative regime for analyzing the long-time behavior of solutions to \eqref{Target problem (u P) damp}. In this regime, solutions satisfy uniform-in-time estimates, which in the time-homogeneous setting on $\T^d$ makes it possible to establish the existence of an invariant probability measure.
			
			\item \textbf{Strong dissipation of \textbf{DNI} via \ref{Hypo-DNI-3}.} \ref{Hypo-DNI-3} addresses the regime where the \textbf{DNI} generates sufficiently strong dissipation to ensure uniqueness of the invariant probability measure.

			\item \textbf{The role of the damping coefficient $\Upgamma$.} Under \ref{Hypo-DNI-1}, the value $\Upgamma=0$ is allowed; thus, no dissipation from the damping term is required. In concrete examples (see Example \ref{Example-decay-noise} below) satisfying \ref{Hypo-DNI-2} and \ref{Hypo-DNI-3}, the parameter $\Upgamma$ must be sufficiently large in the case $V(\cdot)=\log(1+\cdot)$. In particular, for examples satisfying \ref{Hypo-DNI-3}, the required $\Upgamma$ must be larger than that in the example verifying \ref{Hypo-DNI-2}.
		\end{enumerate}
		
	\end{Remark}

	\subsection{Examples of the Damping-Noise Interactions (\textbf{DNI}s)}
	
	\begin{Example}\label{Example-decay-noise} 
		Since \ref{Hypo-DNI-3} and \ref{Hypo-DNI-2} are stronger than \ref{Hypo-DNI-1},  
		We only give an example satisfying \ref{Hypo-h div} and \ref{Hypo-DNI-3} or \ref{Hypo-h div} and  \ref{Hypo-DNI-2}. As before, for brevity, the examples are given for $\K=\R$ since the example on $\K=\T$ can be constructed in the same way.

		As before,  $\mathscr{c}=\mathscr{c}_{\sigma,d}$ is the constant given in Lemma \ref{uux u M}, and $M=M_{\sigma,d}>0$ denotes the embedding constant for $\|\cdot\|_{W^{p,\infty}}\leq M_{\sigma,d}\|\cdot\|_{H^{\sigma}}$ when $\sigma>d/2+p$. Suppose that
		$\mathfrak{g}(x)\in C^1\left([0,\infty); (0,\infty)\right)$ satisfies
		\begin{equation*}
			\mathfrak{g}'>0,\quad \lim_{x\to\infty}\mathfrak{g}(x)=\infty, \quad 
			A\triangleq\sup_{x\in [0,\infty)}\frac{2\mathscr{c} x}{\mathfrak{g}^2\big(\frac{x}{M}\big)}<1.
		\end{equation*}
		Let 
		\begin{align*}
			\widetilde h(f) =\mathfrak{g}\left(\frac{\|f\|_{\Wlip}}{M}\right)f,\qquad f\in  \Hdiv^{\sigma},\quad \sigma>d/2+p.
		\end{align*}
		We now show that, with appropriate choices of $\Upgamma$ and $\mathscr{G}_3$, the function $\widetilde h$ given above satisfies \ref{Hypo-DNI-3} with $V(x)=\log(1+x)$. Indeed, since $\Q_2\in\mathbb{B}^\beta$ with $\Q_2+\Q_2^*=0$, we have $\mathscr{a}_2=0$ (cf. \eqref{SCR ai}). 
		Then, when $2\Upgamma>\mathscr{a}_1$, we obtain
		\begin{align*}  
			& \mathbf{V}_{\sigma}^{\Q_1,\, \Q_2,\, \widetilde h}(f) -2\Upgamma V'(\|f\|^2_{H^{\sigma}} )\|f\|^2_{H^{\sigma}} \\
			=\ & \frac{\left(\mathscr{a}_1-2\Upgamma \right)\|f\|_{H^{\sigma}}^2}{1+\|f\|_{H^{\sigma}}^2}
			+\frac{2\mathscr{c} \|f\|_{W^{1,\infty}}\|f\|_{H^{\sigma}}^2+
				\mathfrak{g}^2(M^{-1}\|f\|_{W^{p,\infty}})\|f\|_{H^{\sigma}}^2}{1+\|f\|_{H^{\sigma}}^2}
			-\frac{2\mathfrak{g}^2(M^{-1}\|f\|_{W^{p,\infty}})\|f\|_{H^{\sigma}}^4}{(1+\|f\|_{H^{\sigma}}^2)^2}\\
			\leq\ &\frac{\left(\mathscr{a}_1-2\Upgamma \right)\|f\|_{H^{\sigma}}^2}{1+\|f\|_{H^{\sigma}}^2}
			+\mathfrak{g}^2(M^{-1}\|f\|_{W^{p,\infty}}) \underbrace{\left(\frac{(A+1)\|f\|_{H^{\sigma}}^2}{1+\|f\|_{H^{\sigma}}^2}
				-\frac{2\|f\|_{H^{\sigma}}^4}{(1+\|f\|_{H^{\sigma}}^2)^2}\right)}_{\triangleq H(\|f\|^2_{H^\sigma})}.
		\end{align*}
		Let $y = \|f\|_{H^{\sigma}}^2$ and $z = M^{-2}\|f\|_{W^{p,\infty}}^2$, which implies $z \le y$. 
		A direct computation shows $H(y) = \frac{(A+1)y + (A-1)y^2}{(1+y)^2}$. 
		Since $A < 1$, we have $\lim_{y \to \infty} H(y) = A-1 < 0$. Thus, there exists $Y_0 > 0$ such that $H(y) \le \frac{A-1}{2} < 0$ for all $y \ge Y_0$. 
		
		\textbf{Case 1: $y < Y_0$.} Note that $H(y) \le (A+1)\frac{y}{1+y}$. Since $z \le y < Y_0$, $\mathfrak{g}^2(z^{1/2})$ is uniformly bounded by $G_{\rm max} \triangleq \mathfrak{g}^2(Y_0^{1/2})$. The entire expression is bounded by $\big(\mathscr{a}_1 - 2\Upgamma + G_{\rm max}(A+1)\big)\frac{y}{1+y}$. By choosing $2\Upgamma$ sufficiently large, we can ensure this is controlled by $-\mathscr{G}_3 \frac{1+Y_0}{1+y}y \le -\mathscr{G}_3 y \le -\mathscr{G}_3 z \le -\mathscr{G}_3 \log(1+z)$.
		
		\textbf{Case 2: $y \ge Y_0$.} The expression is bounded by $\frac{\mathscr{a}_1 - 2\Upgamma}{1+Y_0}y + \frac{A-1}{2}\mathfrak{g}^2(z^{1/2}) \le \frac{A-1}{2} \mathfrak{g}^2(z^{1/2})$. Because $A-1 < 0$ and $\lim_{x\to \infty}\frac{\log(1+x^2)}{\mathfrak{g}^2(x)}=0$, the strong dissipation provided by the fast-growing noise dominates. We can choose $\mathscr{G}_3 > 0$ sufficiently small such that $\frac{A-1}{2} \mathfrak{g}^2(z^{1/2}) \le -\mathscr{G}_3 \log(1+z)$ for all $z \ge 0$.
		
		Consequently, universally for all $f \in \Hdiv^\sigma$, we obtain:
		\begin{align*}
			\mathbf{V}_{\sigma}^{\Q_1,\, \Q_2,\, \widetilde h}(f) -2\Upgamma V'(\|f\|^2_{H^{\sigma}} )\|f\|^2_{H^{\sigma}} 
			\leq   -\mathscr{G}_3 \log(1+M^{-2}\|f\|_{W^{p,\infty}}^2).
		\end{align*}
		Therefore, there exists $\mathscr{G}_3>0$ such that when $2\Upgamma\gg \mathscr{a}_1+\mathscr{G}_3$, $\widetilde h$ satisfies \ref{Hypo-DNI-3} with $V(x)=\log(1+x)$. Similarly, when $2\Upgamma\gg \mathscr{a}_1$, \ref{Hypo-DNI-2} holds with $V(x)=\log(1+x)$.
	\end{Example}
	
	\subsection{Local and Long-Time Behavior for Damped Incompressible Case}
	\label{Section : Result Thm-DNI}
	
	\begin{Theorem}[\textbf{Local pathwise classical solutions to \eqref{Target problem (u P) damp}}]
		\label{Thm-(u P)-local}
		Let $d \geq 2$, $p\ge 1$ and $\Upgamma\ge0$.  
		Assume \ref{Hypo-W L}, \ref{Hypo-Q-Pi},  and \ref{Hypo-h div} with $\sigma> \frac{d}{2} + p $. Let $s> \frac{d}{2} + p + \max \{3\zeta,1\}$ where $\zeta=\zeta_{\Q_1,\Q_2}$ is given in \eqref{zeta-Q12}.  
		Let $u_0\in \Hdiv^s$ be $\mathcal{F}_{0}$-measurable. Then the Cauchy problem \eqref{Target problem (u P) damp} 
		has a unique maximal $\Hdiv^s$ classical solution $(u,\tau^{*})$ in the sense of Definition \ref{solution definition (u P)}. Moreover, 
		\begin{equation*} 
			u\in D([0,\tau^*);\Hdiv^{s'}),\quad s'<s\quad \pas,
		\end{equation*}
		the lifetime $\tau^*$ is independent of $s$, and the following blow-up criterion holds:
		\begin{equation}\label{Blow-up criterion u}
			\limsup_{t \rightarrow \tau^*} \|u(t)\|_{\Wp} = \infty \quad \text{a.s. on } \{\tau^* < \infty\}.
		\end{equation}
	\end{Theorem}
	
	Since $\Q_1$ and $\Q_2$ significantly extend classical transport-type operators, Theorem \ref{Thm-(u P)-local} for \eqref{SEuler-(u P) damp} with noise structure $\mathcal{Q}_1 u\, \circ\, {\rm d} W
	+\mathcal{Q}_2 u \diamond {\rm d} L 
	+ \widetilde h( u)\d \widetilde{W}$ extends many existing results in the literature that focus on a single noise type.
	
	\begin{Theorem}[\textbf{Long-time dynamics of \eqref{Target problem (u P) damp}}]\label{Thm-(u P)-long-time}
		Let the conditions of Theorem \ref{Thm-(u P)-local} hold and
		$$\theta\in\left(\frac{d}{2}+p,{s-\max \{3\zeta,1\}}\right).$$
		\begin{enumerate}[label={\bf{(\arabic*)}},leftmargin=0.79cm]\setlength\itemsep{0.2em}
			\item\label{Thm-(u P)-estimate V} $($\textbf{Global regularity and continuous dependence on initial data}$)$  Under \ref{Hypo-DNI-1}, the Cauchy problem \eqref{Target problem (u P) damp} admits a global $\Hdiv^s$ solution. For $V\in\mathscr{V}$  as specified in \ref{Hypo-DNI-1}, we have
			\begin{equation}\label{u Hs bound V-1}
				\mathbb{E} \left[V(\|u(t)\|_{H^s}^2)\Big|\mathcal{F}_0\right]  
				\leq \Big(V(\|u_0\|^2_{H^s})+\mathscr{G}_2t\Big) {\rm e}^{\mathscr{G}_1t},
				\quad t>0.
			\end{equation}
			Moreover, 
			if $\{u_{0,m}\}_{m\ge1}$ are $\mathcal{F}_0$-measurable $\Hdiv^s$-valued random variables such that 
			\begin{equation}
				\sup_{m\ge1}\|u_{0,m}\|_{H^s}\le R\quad \text{for some}\  R>0 \quad  \text{and}\quad  \lim_{m\to\infty}\|u_{0,m}-u_0\|_{H^{\theta}}=0 \quad\pas,\label{u-0-m condition}
			\end{equation}
			then the corresponding sequence of solutions $\{u_{m}\}_{m\ge1}$ with $u_m|_{t=0}=u_{0,m}$ satisfies
			\begin{equation}\label{H-theta stabilty}
				\lim_{m\to\infty}
				\mathbb{E} \left[1\land \|u_{m}(t)-u(t)\|^2_{H^{\theta}}\Big|\mathcal{F}_0\right]=0, \qquad t>0.
			\end{equation}
			
			\item\label{Thm-(u P)-IM}  $($\textbf{Uniform-in-time estimate and invariant probability measure}$)$ If \ref{Hypo-DNI-1} is strengthened to \ref{Hypo-DNI-2}, then $u$ satisfies 
			\begin{equation}\label{u Hs bound V-2}
				\mathbb{E}\left[V(\|u(t)\|_{H^s}^2) \Big| \mathcal{F}_0\right]  
				\leq V(\|u_0\|^2_{H^s}), \qquad t\ge0.
			\end{equation}
			Moreover, on $\T^d$, there exists an
			invariant 
			measure $\mu$ 
			in the sense of Definition \ref{IM definition} with $E=\Hdiv^s$. 
			
			\item\label{Thm-(u P)-decay} $($\textbf{Uniqueness of invariant probability measure}$)$ Let $M=M_{\theta,d}>0$ be a constant satisfying $\|\cdot\|_{W^{p,\infty}}\leq M \|\cdot\|_{H^\theta}$ for $\theta>d/2+p$. If \ref{Hypo-DNI-3} holds, then  
			\begin{equation}\label{u Wlip -> 0}
				\int_0^t\mathbb{E}  \left[V\left(\frac{\|u(t')\|^2_{W^{p,\infty}}}{M^2}\right)\Bigg|\mathcal{F}_0\right]\d t'
				\leq  \frac{V(\| u_0\|^2_{H^\theta})}{\mathscr{G}_3}\left(1-{\rm e}^{-\mathscr{G}_3t}\right),
				\quad t>0,
			\end{equation}
			As a consequence of \eqref{u Wlip -> 0} and $V(0)=0$, it follows that
			the Dirac measure centered at $0$ is the \textbf{unique invariant probability measure} on both $\T^d$ and $\R^d$ when $\widetilde h(0)=0$.
		\end{enumerate}
	\end{Theorem}

	\begin{Remark}
		Regularization effects of fast-growing noise that guarantee global existence have been investigated. For fluid-type SPDEs, the related techniques appear to be  pioneered in \cite{Ren-Tang-Wang-2024-POTA}, see the subsequent development in an abstract framework \cite{Tang-Wang-2022-arXiv}, as well as \cite{Tang-2023-JFA,Tang-Yang-2023-AIHP,Tang-Wang-2024-CCM} for further concrete advances. Motivated by these works, we identify a hierarchy of conditions \ref{Hypo-DNI-1}, \ref{Hypo-DNI-2}, and \ref{Hypo-DNI-3}, which yield \ref{Thm-(u P)-estimate V}, \ref{Thm-(u P)-IM}, and \ref{Thm-(u P)-decay} in Theorem \ref{Thm-(u P)-long-time}, respectively. Broadly speaking, these conditions exploit the dissipativity far from the origin induced by fast-growing noise, as well as the dissipativity near the origin provided by damping. It is worth emphasizing that applying conditions \ref{Hypo-DNI-1}, \ref{Hypo-DNI-2}, and \ref{Hypo-DNI-3} at the technical level is more subtle than it may first appear. For instance, when Marcus noise is present, one cannot directly use these conditions to obtain a global \textit{a priori} estimate, and an additional limiting procedure is required; see Remark \ref{Remark singualrity IC}. We also note that dissipation from damping is not needed in \ref{Hypo-DNI-1}; see Remark \ref{Remark-DNI-Compare}.  
	\end{Remark}

	\subsection{Proof of Theorem \ref{Thm-(u P)-local}}\label{Section : SEuler-local-proof}

	In this section we prove Theorem \ref{Thm-(u P)-local}.
	Let $J_n$ be the mollifier given in \eqref{Define Jn}.
	By Lemma \ref{Lemma-Jn}, for $n\ge1$ and $s\ge0$, the operator $J_n$ maps $\Hdiv^s$ to $\Hdiv^s$, where $\Hdiv^s$ is defined in \eqref{H-div}. Accordingly, we consider the following approximation scheme for \eqref{Target problem (u P) damp}$_1$:
	\begin{align*}
		\d u+ \left\{\chi_R\big(\|u-u_0\|_{\Wp}\big)J_n\Pi[(J_n u\cdot\nn)J_n u]+\Upgamma u\right\}\d t
		= \Q_{1,n}X\,\circ\, {\rm d} W+\Q_{2,n}X\diamond {\rm d}L + \chi_R\big(\|u-u_0\|_{\Wp}\big)\widetilde h(u)\d \widetilde W,
	\end{align*}
	where $\{\Q_{1,n}\}_{n\ge1}$ and $\{\Q_{2,n}\}_{n\ge1}$ the sequences from Lemma \ref{Lemma:Qn}, and $\chi_R\in C^{\infty}([0,\infty);[0,1])$ is the cut-off function used in \eqref{eq : appro compressible} with $R>1$.  Under    \ref{Hypo-W L}, \ref{Hypo-Q-Pi}, and \ref{Hypo-h div}, we rewrite \eqref{Target problem (u P) damp} as
	\begin{equation}\label{approximation scheme IC}
		\left\{\begin{aligned}
			u(t)-u_0 \ 
			+&\int_0^{t}  \left\{\chi_R\big(\|u-u_0\|_{\Wp}\big)J_n\Pi[(J_n u\cdot\nn)J_n u]+\Upgamma u-\frac{1}{2} \mathcal{Q}_{1,n}^2u\right\}(t') \d t'\\
			= & \int_0^{t} \left(\Q_{1,n}u(t') \d W(t')+ \chi_R\big(\|u-u_0\|_{\Wp}\big)  \widetilde h\big(u(t')\big)\d \widetilde W(t')\right )\\
			&+\int_0^t\int_{|l|\le1} \Big\{\wp_n\big(1,l,u(t'-)\big)-u(t'-)\Big\}\, \widetilde{\eta}({\rm d}l,{\rm d}t')\\
			&+\int_0^t\int_{|l|\le1} \Big\{\wp_n\big(1,l,u(t')\big)-u(t')-l\cdot \Q_{2,n} u(t')\Big\}\, \nu({\rm d}l){\rm d}t',\\
			u(0)=u_0.\, \, &
		\end{aligned}
		\right.
	\end{equation}
	where the Marcus flow $\wp_n$ is defined by \eqref{regular Marcus flow Q2}, i.e.,
	\begin{equation*} 
		\wp_n(r)\triangleq\wp_n(r,l,f),\quad 
		\frac{{\rm d}}{{\rm d}r}\wp_n(r)= l\cdot \Q_{2,n}\wp_n(r),\quad r\in[0,1],  \quad \wp_n(0) = f.
	\end{equation*}
	
	\begin{proof}[Proof of  Theorem \ref{Thm-(u P)-local}]
		The proof of  Theorem \ref{Thm-(u P)-local} is very similar to that for Proposition \ref{Proposition-(q u)} and we only provide a sketch.
		
		Let $R>1$. Since the modified non-linear term $\chi_R\big(\|u-u_0\|_{\Wp}\big)J_n\Pi [(J_n u\cdot\nn)J_n u]$ can be estimated analogously to the term $G_{n,R}(X)$ defined in \eqref{GnR HR}, the proof of Lemma \ref{Lemma: existence of XnR} guarantees that  \eqref{approximation scheme IC} admits a global-in-time solution sequence $\{u_{R,n}\}_{n\ge1}\subset D([0,\infty);\Hdiv^s)$ $\pas$ 
		
		Repeating the proof of Lemma \ref{Lemma : Xn T estimates}, we conclude that  there exists a subsequence of $\{u_{R,n}\}_{n\ge1}$ (still denoted by $\{u_{R,n}\}_{n\ge1}$) that converges to a limit process $u_R$ in the following sense:
		\begin{equation*} 
			\lim_{n\to\infty}\sup_{t\in[0,T]}\|u_{R,n}(t)-u_R(t)\|_{H^{s'}}=0 \quad \text{and} \quad u_{R,n} \xrightarrow[n\to \infty]{} u_R \quad \text{in} \quad D([0,T];\Hdiv^{s'}), \quad T>0, \quad s'<s \quad \pas
		\end{equation*}
		Following the analysis in the proof of Lemma \ref{cut-off global solution}, we then show that the limit yields a unique global solution to the cut-off version of \eqref{Target problem (u P) damp} (obtained by replacing $J_n$ with the identity map $\I$ in \eqref{approximation scheme IC}). 
		
		Finally, removing the cut-off $\chi_R\big(\|u-u_0\|_{\Wp}\big)$ as described in \eqref{X=sum XR} produces a local solution to \eqref{Target problem (u P) damp} satisfying the blow-up criterion \eqref{Blow-up criterion u}. Uniqueness comes from the argument used in the proof of  Lemma \ref{cut-off global solution}.
	\end{proof}
	
	\begin{Remark}\label{Remark: local-(u P) viscosity}
		As in Remark~\ref{Remark:renormalization Qn}, the proof of Theorem~\ref{Thm-(u P)-local} relies on a structure-preserving mollification of the singular terms. This methodology is robust: the operator $(-\Delta)^{\ell/2}u$ with $\ell \in (0,2]$ can be seamlessly incorporated into the system. Rather than exploiting its parabolic smoothing effect, our framework treats diffusion purely as an additional singular term that induces regularity loss (see Remark \ref{Remark-SES-examples}). By applying the mollification $J_n(-\Delta)^{\ell/2}J_n u$ and imposing the adjusted regularity threshold $s > \frac{d}{2} + p + \max \{3\zeta, \ell, 1\}$, the results established in Theorem~\ref{Thm-(u P)-local} remains valid. Similarly,  this approach also accommodates systems with partial or degenerate viscosity (e.g., dissipation restricted exclusively to certain spatial directions).
	\end{Remark}

	\subsection{Proof of  Theorem \ref{Thm-(u P)-long-time}}\label{Section : SEuler-global-proof}

	We begin by noting the following observation. With the local solution $(u,\tau)$ obtained in Theorem \ref{Thm-(u P)-local} at hand, and as noted in Remarks \ref{Remark singualrity} and \ref{Remark singualrity IC}, the singularities in \eqref{Target problem (u P) damp} preclude the direct application of It\^o's formula to $\|u\|^2_{H^s}$ to derive a global-in-time estimate. More critically, it is not clear how to apply the mollifier $J_n$ (cf. \eqref{Define Jn}) to the equation to obtain an estimate for $\|J_n u\|^2_{H^s}$ and subsequently pass to the limit. Indeed, in the presence of a Marcus integral, applying $J_n$ to the term $\wp\big(1,l,u(t')\big)-u(t')-l\cdot\Q_2u(t')$ yields
	\[
	J_n\wp\big(1,l,u(t')\big)-J_nu(t')-l\cdot J_n\Q_2u(t'),
	\]
	which is inconsistent: $J_n\wp\big(1,l,u(t')\big)$ is generally \textbf{not} the Marcus flow associated with the operator $J_n\Q_2$ and initial condition $J_nu(t')$. 
	
	To obtain an $H^s$ bound, as mentioned in Remark \ref{Remark singualrity IC}, we adopt a different strategy: we first construct an approximation scheme whose solutions satisfy a uniform $H^s$ estimate by enforcing the conditions required for a global $H^s$ bound at the approximation level. Then, passing to the limit yields a solution to the original equation that satisfies an $H^s$ estimate, thereby providing the desired global $H^s$ bound.   More precisely, we choose the approximation scheme for \eqref{Target problem (u P) damp} to be \eqref{approximation scheme IC} without a cut-off, i.e.,
	\begin{equation}\label{approximation scheme IC global}
		\left\{\begin{aligned}
			u(t)-u_0 \ 
			+&\int_0^{t}  \left\{J_n\Pi[(J_n u\cdot\nn)J_n u]+\Upgamma u-\frac{1}{2} \mathcal{Q}_{1,n}^2u\right\}(t') \d t'\\
			= & \int_0^{t} \left(\Q_{1,n}u(t') \d W(t')+  \widetilde h\big(u(t')\big)\d \widetilde W(t')\right )\\
			&+\int_0^t\int_{|l|\le1} \Big\{\wp_n\big(1,l,u(t'-)\big)-u(t'-)\Big\}\, \widetilde{\eta}({\rm d}l,{\rm d}t')\\
			&+\int_0^t\int_{|l|\le1} \Big\{\wp_n\big(1,l,u(t')\big)-u(t')-l\cdot \Q_{2,n} u(t')\Big\}\, \nu({\rm d}l){\rm d}t',\\
			u(0)=u_0.\, \ &
		\end{aligned}
		\right.
	\end{equation}

	\begin{Lemma}\label{Lemma: approxiamtion IC global} 
		Let $d \geq 2$ and $p\ge 1$.  
		Assume \ref{Hypo-W L}, \ref{Hypo-Q-Pi}, \ref{Hypo-h div} with $\sigma> \frac{d}{2} + p$, and \ref{Hypo-DNI-1}. Let $s> \frac{d}{2} + p + \max \{3\zeta,1\}$ where $\zeta=\zeta_{\Q_1,\Q_2}$ is given in \eqref{zeta-Q12}.  For any $\mathcal{F}_{0}$-measurable $u_0\in \Hdiv^s$ and $n\ge1$,  \eqref{approximation scheme IC global} admits a global solution $u_n\in D([0,\infty);\Hdiv^s)$ such that 
		\begin{equation}\label{un Hs bound V}
			\sup_{n\ge1} \E\left[V(\|u_n(t)\|_{H^s}^2)\Big|\F_0\right]  
			\leq   \Big(V(\|u_0\|^2_{H^s})+\mathscr{G}_2t\Big) {\rm e}^{\mathscr{G}_1t},
			\quad t>0.
		\end{equation}
		Moreover, there exists an $\mathcal{F}_t$-progressively measurable $\Hdiv^s$-valued process $u=(u(t))_{t\ge 0}$  and a subsequence of $\{u_n\}_{n\ge1}$  $($still labeled as $\{u_n\}_{n\ge1}$ for simplicity$)$ such that for any $T>0$ and $s'<s$,
		\begin{equation}\label{un to u IC}
			\lim_{n\to\infty}\sup_{t\in[0,T]}\|u_{n}(t)-u(t)\|_{H^{s'}}=0\quad \text{and}\quad  u_{n}\xrightarrow[]{n\to \infty}u \  {\rm in}\ D([0,T];\Hdiv^{s'}),\quad T>0
			\quad \pas
		\end{equation}
	\end{Lemma}
	\begin{proof}
		Similar to the proof of  Lemma \ref{Lemma: existence of XnR}, 
		for any $n\ge1$, all the coefficients in \eqref{approximation scheme IC global} are locally Lipschitz continuous in $\Hdiv^s$ locally uniform in $t$.  Therefore, \eqref{approximation scheme IC global} admits a maximal solution $u_n\in D([0,\tau^*_n);\Hdiv^s)$. It sufficies to prove \eqref{un Hs bound V}.
		
		Define the stopping time
		\begin{equation} 
			\tau_{n,N} \triangleq N \land \inf \big\{ t \ge 0 : 
			\| u_n(t) \|^2_{H^s} > N \big\}, \quad N \ge 1.\label{eq: un tau-n-N}
		\end{equation}
		and recall \eqref{V wp-n f}, i.e.,
		$
		V_s^{\wp_n,f}(r,l)\triangleq V(\|\wp_n(r,l, f)\|_{H^s}^2) - V(\norm{f}_{H^s}^2)$ with $r\in[0,1]$ and $f\in H^{s}$. 
		For the function $V\in\mathscr{V}$ given by
		\ref{Hypo-DNI-1}, 
		we use It\^o's formula (Lemma \ref{Ito formula}), 
		to derive 
		that for any $t>0$,
		\begin{align*}
			&\E  \left[V(\| u_n(t\wedge \tau_{n,N})\|^2_{H^s})\big|\F_0\right]
			-V(\| u_0\|^2_{H^s}) \notag \\
			=\ &\E \bigg[
			\int_0^{t\wedge \tau_{n,N}}V'(\|u_n(t')\|^2_{H^s})
			\Big(-2\IP{J_n [(J_n u_n(t')\cdot\nn)J_n u_n(t')],u_n(t')}_{H^s}
			-2\Upgamma \|u_n(t')\|^2_{H^s}\Big)\, {\rm d}t'\bigg|\F_0\bigg]\notag\\
			&+\E \bigg[
			\int_0^{t\wedge \tau_{n,N}}V'(\|u_n(t')\|^2_{H^s})
			\Big(\IP{\Q_{1,n}^2u_n(t') , u_n(t')}_{H^s}+\|\Q_{1,n} u_n(t')\|^2_{H^s}
			+\| \widetilde h(u_n(t')) \|^2_{H^s}\Big)\, {\rm d}t'\bigg|\F_0\bigg]\notag\\
			&+\E \bigg[
			\int_0^{t\wedge \tau_{n,N}} 2V''(\|u_n(t')\|^2_{H^s}) \Big(\IP{\Q_{1,n}u_n(t'), u_n(t')}^2_{H^{s}}
			+\IP{  \widetilde h(u_n(t')),u_n(t')}_{H^s}^2 \Big)\, {\rm d}t'\bigg|\F_0\bigg]\notag\\
			&+ \E \left[
			\int_0^{t\wedge \tau_N}\int_{|l|\le1} 
			\bigg\{V_s^{\wp_n,u_n(t')}(1,l)-2l\cdot V'(\|u_n(t')\|^2_{H^s})\IP{\Q_{2,n} u_n(t'),u_n(t')}_{H^s}\bigg\}\, \nu({\rm d}l){\rm d}t'\bigg|\F_0\right].
		\end{align*}
		Recall the function
		$\mathbf{V}_{s}^{\Q_1,\, \Q_2,\, \widetilde h}(f)$ defined in \eqref{V Lyapunov type term}, namely, 
		\begin{align*}  
			\mathbf{V}_{s}^{\Q_1,\, \Q_2,\, \widetilde h}(f)
			\triangleq V'(\|f\|^2_{H^s}) \bigg\{  
			\big(\mathscr{a}_1 + 2\mathscr{c}  \|f\|_{\Wlip} \big) \|f\|^2_{H^s}    + \|\widetilde h(f)\|^2_{H^s} \bigg\}  +2V''(\|f\|^2_{H^s}) \bIP{\widetilde h(f), f }_{H^s}^2+\mathscr{a}_2 V(\|f\|^2_{H^{s}}).
		\end{align*}  
		Using the properties of $V$ (cf. \eqref{SCRV}), \ref{Hypo-h div}, \eqref{Qn cancel-2},  \eqref{wp-n estimate 2 C22}, \eqref{A1=C12 A2=C22}, and Lemma
		\ref{uux u M}, we arrive at
		\begin{align*}
			&\E  \left[V(\| u_n(t\wedge \tau_{n,N})\|^2_{H^s})\big|\F_0\right]
			-V(\| u_0\|^2_{H^s}) \notag \\
			\leq \ &  \E \bigg[
			\int_0^{t\wedge \tau_{n,N}}V'(\|u_n(t')\|^2_{H^s})
			\Big(\big(\mathscr{a}_1- 2\Upgamma + 2\mathscr{c}  \|u_n(t')\|_{\Wlip} \big) \|u_n(t')\|^2_{H^s}+\|\widetilde h(u_n) \|^2_{H^s} \Big)\, {\rm d}t'\bigg|\F_0\bigg]\notag\\
			&+ \E \bigg[
			\int_0^{t\wedge \tau_{n,N}}
			2V''(\| u_n(t')\|^2_{H^s})
			\IP{\widetilde h(u_n(t')),  u_n(t')}^2_{H^{s}}\,{\rm d}t'\Big|\F_0\bigg] +\mathscr{a}_2 \E \left[
			\int_0^{t\wedge \tau_{n,N}} 
			V(\| u_n(t')\|^2_{H^s}) \d t'\bigg|\F_0\right]\notag\\
			\leq \ &  \E \bigg[
			\int_0^{t\wedge \tau_{n,N}}\left(\mathbf{V}_{s}^{\Q_1,\, \Q_2,\, \widetilde h}(u_n(t'))-2\Upgamma V'(\|u_n(t')\|^2_{H^{s}} )\|u_n(t')\|^2_{H^{s}} \right)\d t'\bigg|\F_0\bigg],
		\end{align*}
		which together with \ref{Hypo-DNI-1} yields that
		\begin{align*}
			\E  \left[V(\|u_n(t\wedge \tau_{n,N})\|^2_{H^s})\big|\F_0\right]
			-V(\| u_0\|^2_{H^s})  
			\leq  
			\mathscr{G}_1\int_0^{t} \E \left[V(\|u_n(t'\land\tau_{n,N})\|^2_{H^s})
			\Big|\F_0\right]\, {\rm d}t'+\mathscr{G}_2t.
		\end{align*}
		Thanks to Gr\"{o}nwall's inequality, we have
		\begin{equation}\label{Hs bound V tau-N}
			\sup_{n,\, N\ge 1}\E  \left[V(\|u_n(t\wedge \tau_{n,N})\|^2_{H^s})\Big|\F_0\right]
			\leq \Big(V(\|u_0\|^2_{H^s})+\mathscr{G}_2t\Big) {\rm e}^{\mathscr{G}_1t}, 
			\quad t>0.
		\end{equation}
		Accordingly, by the c\`adl\`ag property of $u_n(t)$ in $\Hdiv^s$, we find that for $N> t>0$,
		\begin{align} 
			\p\big(\tau_n^*<t |\F_0\big)  \le\,  \p\big(\tau_{n,N}<t|\F_0\big) 
			\le \frac{\Big(V(\|u_0\|^2_{H^s})+\mathscr{G}_2t\Big) {\rm e}^{\mathscr{G}_1t}}{V(N)}.\label{eq: un tau-n-N estimate}
		\end{align} 
		Letting $N\uparrow\infty$ and then $t\uparrow\infty$ 
		yields
		$\p(\tau_n^*<\infty|\F_0)=0$. 
		Hence we obtain $\p(\tau_n^*<\infty)=0.$ Then \eqref{un Hs bound V} is a consequence of \eqref{Hs bound V tau-N}. Note that $\Hdiv^s$ is a closed subspace of $H^s$.
		Given \eqref{un Hs bound V}, the convergence \eqref{un to u IC}  can be established by repeating the arguments used in the proof of   \ref{Convergence of X-n T} in Lemma \ref{Lemma : Xn T estimates}, For  brevity, the details are omitted.
	\end{proof}
	
	\subsubsection{Global Regularity: Proof of Theorem \ref{Thm-(u P)-long-time} \ref{Thm-(u P)-estimate V}} 
	
	Following the analysis in proving Lemma \ref{cut-off global solution}, we show that the limit $u$ obtained in Lemma \ref{Lemma: approxiamtion IC global} is the unique global solution to \eqref{Target problem (u P) damp}. Moreover, \eqref{u Hs bound V-1} follows from \eqref{un to u IC} and \eqref{un Hs bound V}. The details are omitted for brevity. It suffices to prove \eqref{H-theta stabilty}. 
	It follows  from the linearity of operators $\Q_1$, $\Q_2$ and the Marcus flow $\wp(1,l,\cdot)$ that  $z_m\triangleq u_m-u$ satisfies
	\begin{align*} 
		z_m (t)  
		=\ & \int_0^t \left\{-\Pi[(u_m\cdot\nn)u_m]+\Pi[(u\cdot\nn)u]-\Upgamma z_m+\frac{1}{2}\Q_1^2z_m\right\} (t') \d t'\notag\\
		&+\int_0^t\int_{|l|\le1} \left\{\wp\big(1,l,z_m(t')\big)-z_m (t')-l \cdot \Q_2 z_m (t')\right\}\, \nu({\rm d}l){\rm d}t'\notag\\
		&+\int_0^t 
		\Q_1 z_m(t') \d W(t') 
		+\int_0^t\left\{\widetilde h(u_m(t'))-\widetilde h(u(t'))\right\}\d \widetilde W(t')\notag \\
		&+ \int_0^t\int_{|l|\le1} \left\{\wp\big(1,l,z_m(t'-)\big)-z_m (t'-)\right\}\, \widetilde{\eta}({\rm d}l,{\rm d}t'),\qquad z_m(0)=0.
	\end{align*}
	Define the stopping times:
	\begin{equation}\label{eq: tau-m-N tau-N}
		\tau_m^{N}\triangleq N \land\inf\{t\ge 0: \|u_{m}(t)\|^2_{H^s} > N\},\quad \tau^{N}\triangleq N \land \inf\{t\ge 0: \|u(t)\|^2_{H^s}> N\},\ \ m, N\ge1
	\end{equation}

	Note that the fixed-time expectation bound \eqref{u Hs bound V-1} yields no path information, and it is not clear whether $u_m$ and $u$ have c\`adl\`ag paths in $\Hdiv^s$. Hence,  we  cannot  \textbf{directly} estimate
	$\p(\tau_m^N<t |\F_0)$ and $\p(\tau^N<t |\F_0)$ by using Chebyshev's inequality.

	Before proceeding, we justify the measurability of $\tau^N$ defined in \eqref{eq: un tau-n-N} for $N\ge t$.  
	We recall the approximate solutions $u_n$ satisfying \eqref{un to u IC}. Although $u$ is not \textit{a priori} known to possess c\`adl\`ag paths in the strong $H^s$-topology, the strong convergence in \eqref{un to u IC} ensures that $u$ is right-continuous in $\Hdiv^{s'}$. Combined with the lower semicontinuity of the $H^s$-norm with respect to the $H^{s'}$-topology (which readily follows from the weak compactness in $H^s$ alongside the uniqueness of limits in $H^{s'}$), the mapping $t \mapsto \|u(t)\|_{H^s}^2$ is inherently right-lower semicontinuous. Consequently, any strict exceedance of $N$ is maintained over a short interval to the right, which allows us to reduce the uncountable infimum to a countable union over rationals:
	\begin{equation*}
		\left\{ \tau^N < t \right\} = \bigcup_{q \in \mathbb{Q} \cap [0,t)} \Big\{ \|u(q)\|_{H^s}^2 > N \Big\},\quad N\ge t.
	\end{equation*}
	Since $u$ is pointwise measurable, this countable union rigorously guarantees that $\tau^N$ is a well-defined $\F_t$-stopping time.
	
	Now we verify the strict event inclusion 
	\begin{equation}\label{eq: strict inclusion}
		\left\{ \tau^N < t \right\} \subseteq \liminf_{n\to\infty} \left\{ \tau_{n,N} < t \right\}.
	\end{equation} 
	Fix an arbitrary time $t \le N$. 
	Suppose $\omega \in \{\tau^N < t\}$. By definition of the infimum, there exists a time $r \in [0,t)$ such that $\|u(r)\|_{H^s}^2 > N$. Relying on the pointwise lower semicontinuity, we obtain:
	\begin{equation*}
		\liminf_{n\to\infty} \|u_n(r)\|_{H^s}^2 \ge \|u(r)\|_{H^s}^2 > N.
	\end{equation*}
	Crucially, this strict inequality guarantees the existence of an integer $n_0(\omega)$ such that for all $n \ge n_0(\omega)$, $\|u_n(r)\|_{H^s}^2 > N$. Since $r < t \le N$, the truncation by $N$ for $u_n$ is similarly inactive, which intrinsically implies $\tau_{n,N} \le r < t$. Thus, for all $n \ge n_0(\omega)$, we must have $\tau_{n,N} \le r < t$.  This yields the exact event inclusion \eqref{eq: strict inclusion}.

	Applying Fatou's lemma to the indicator functions associated with \eqref{eq: strict inclusion}, and utilizing Chebyshev's inequality along with the uniform bound \eqref{Hs bound V tau-N}, we obtain:
	\begin{align*}
		& \p\left(\tau^N < t \;\big|\; \F_0\right) \\
		\le\ & \p\left( \liminf_{n\to\infty} \left\{ \tau_{n,N} < t \right\} \;\bigg|\; \F_0 \right) \\
		\le\ &  \liminf_{n\to\infty} \frac{\E \left[V(\|u_n(t\wedge \tau_{n,N})\|^2_{H^s})\Big|\F_0\right]}{V(N)} \\
		\le\ &  \frac{\Big(V(\|u_0\|^2_{H^s})+\mathscr{G}_2t\Big){\rm e}^{\mathscr{G}_1t}}{V(N)}.
	\end{align*}
	Similarly, for $u_m$, we can also carry out the estimate at approximate level and then obtain 
	\begin{align*}
		\p\left(\tau_m^N < t \;\big|\; \F_0\right) 
		\le \frac{\Big(V(\|u_{0,m}\|^2_{H^s})+\mathscr{G}_2t\Big){\rm e}^{\mathscr{G}_1t}}{V(N)}.
	\end{align*}
	Therefore, we obtain that 
	\begin{align*}
		\p\left(\tau_m^N\land \tau^N<t\big|\F_0\right) \leq 
		\frac{\Big(V(\|u_{0,m}\|^2_{H^s})+V(\|u_0\|^2_{H^s})+2\mathscr{G}_2t\Big){\rm e}^{\mathscr{G}_1t}}{V(N)},\quad m\ge 1,\quad N>t.
	\end{align*}
	Using Lemma \ref{KP-commutator} leads 
	to the following
	inequality:
	\begin{align*}
		\left|\bIP{\Pi[(u_m\cdot\nn)u_m]-\Pi[(u\cdot\nn)u],u_m-u}_{H^{\theta}}\right|
		\lesssim \left(\|u\|_{H^{\theta+1}}+\|u_m\|_{H^{\theta+1}}\right)
		\|u_m-u\|^2_{H^\theta}.
	\end{align*}
	Similar to \eqref{Znm-Ito}, we 
	use It\^o's formula (Lemma \ref{Ito formula}), Remark  \ref{Remark: wp-n estimates}
	and \ref{Hypo-h div} to obtain that for $t\ge0$, $m\ge1$, and $N>t$,
	\begin{align*} 
		&\E\left[\norm{u_{m}\big(t\land \tau_m^N\land \tau^N\big)
			-u\big(t\land \tau_m^N\land \tau^N\big)}^2_{H^\theta}\Big|\F_0\right]\\
		\lesssim\ &   
		\norm{u_{0,m}-u_{0}}^2_{H^\theta}
		+\E\left[\int_0^{t\land \tau_m^N\land \tau^N}\big(1+N+\widetilde K(t',N)\big)
		\norm{u_{m}(t')-u(t')}^2_{H^\theta}\d t'\bigg|\F_0\right].
	\end{align*}
	From the above estimates, we find
	a function $\mathfrak{a}(t,N)$, 
	non-decreasing in $N$ and locally 
	integrable in $t$ such that for all $t>0$, $m\ge1$, and $N> t$,
	\begin{align*}
		&\E \left[
		1\land \norm{u_{m}(t)-u(t)}^2_{H^\theta}\Big|\F_0\right]\\
		=\ &\E \left[
		\left(1\land \norm{u_{m}(t)-u(t)}^2_{H^\theta}\right)
		\left({\bf 1}_{\{\tau_m^N\land \tau^N<t\}}+{\bf 1}_{\{\tau_m^N\land \tau^N\ge t\}}\right)
		\Big|\F_0\right]\\
		\leq \ & \p\left(\tau_m^N\land \tau^N<t\big|\F_0\right)+\E \left[
		\left(1\land \norm{u_{m}\big(t\land \tau_m^N\land \tau^N\big)
			-u\big(t\land \tau_m^N\land \tau^N\big)}^2_{H^\theta}\right)
		\Big|\F_0\right]\\
		\leq\ & 
		\frac{\Big(V(\|u_{0,m}\|^2_{H^s})+V(\|u_0\|^2_{H^s})+2\mathscr{G}_2t\Big){\rm e}^{\mathscr{G}_1t}}{V(N)}
		+\Big[1\land \mathfrak{a}(t,N)\|u_{0,m}-u_0\|^2_{H^\theta}\Big].
	\end{align*}
	
	By \eqref{u-0-m condition}, for any fixed $N > t$, taking the limit supremum as $m \to \infty$ on both sides yields:
	$$ \limsup_{m\to\infty}\E \left[1\land \norm{u_{m}(t)-u(t)}^2_{H^\theta}\Big|\F_0\right] \le \frac{\Big(V(R^2)+V(\|u_0\|^2_{H^s})+2\mathscr{G}_2t\Big){\rm e}^{\mathscr{G}_1t}}{V(N)}. $$
	Since this inequality holds for all sufficiently large $N$ (i.e., $N > t$), we  pass to the limit $N \to \infty$ to obtain \eqref{H-theta stabilty}.
	
	\begin{Remark}
		To accommodate the weak lower semicontinuity, we  define the stopping times using a strict inequality (i.e., ``$>N$'') in \eqref{eq: un tau-n-N} and \eqref{eq: tau-m-N tau-N}. If we still used $\ge N$ in them, then \eqref{eq: strict inclusion} would not hold.
	\end{Remark}

	\subsubsection{Existence of an Invariant Probability Measure: Proof of Theorem  \ref{Thm-(u P)-long-time} \ref{Thm-(u P)-IM}} 
	Recall that $\widetilde h$ is time-homogeneous. We recall the Markov propagator $\{\mathscr{P}_{t}\}_{t\ge0}$ defined by \eqref{Markov semigroup-SES} with $\bfH=\Hdiv^s$. Recall the space $\Hdiv^{s,\theta}$ defined in \eqref{H-s1 s2}.

	\begin{Lemma}\label{Lemma Euler restricted Feller}
		Under the conditions of Theorem  \ref{Thm-(u P)-long-time} \ref{Thm-(u P)-IM}, the Markov semigroup $\{\mathscr{P}_{t}\}_{t\ge0}$ associated with \eqref{SEuler-(u P) damp} is  restricted Feller under mismatch $H^{\theta}$ topology in the sense of Definition \ref{Locally Feller definition} with $E=\Hdiv^{s}$.
	\end{Lemma}
	
	\begin{proof}
		We only need to prove the case  $t>0$. Let $t > 0$, $R>0$, and $\varphi \in C_B(\Hdiv^{s,\theta})$ be arbitrarily given. Since the $H^\theta$-topology is metrizable, it suffices to prove the sequential continuity of the restricted mapping $x \mapsto \mathscr{P}_t\varphi(x)$ on the set $B^s_R = \{v \in \Hdiv^s : \|v\|_{H^s} \le R\}$.  Let $v \in B^s_R$ and let $\{v_m\}_{m\ge 1} \subset B^s_R$ be an arbitrary deterministic sequence such that $\lim_{m\to\infty}\|v_m - v\|_{H^\theta} = 0$. 
		
		Note that because the $H^s$-norm is lower semicontinuous with respect to the $H^\theta$-topology (as $\theta < s$), we have
		\begin{equation*}
			\|v\|_{H^s} \le \liminf_{m\to\infty} \|v_m\|_{H^s} \le R,
		\end{equation*}
		which means that $B^s_R$ is a closed subset in $\Hdiv^{s,\theta}$.
		
		Since $\{v_m\}_{m\ge 1}$ and $v$ are deterministic initial data, we can trivially view them as constant $\mathcal{F}_0$-measurable random variables by setting $u_{0,m} \triangleq v_m$ and $u_0 \triangleq v$. Then \eqref{u-0-m condition} is satisfied deterministically (and thus almost surely).
		
		Let $u_m(t)$ and $u(t)$ be the solutions corresponding to the initial data $v_m$ and $v$, respectively. As in \eqref{Markov semigroup-SES}, we have
		\begin{equation}\label{eq: semigroup_def}
			\mathscr{P}_t\varphi(v_m) = \mathbb{E}\big[\varphi(u_m(t))\big] \quad \text{and} \quad \mathscr{P}_t\varphi(v) = \mathbb{E}\big[\varphi(u(t))\big].
		\end{equation}
		Applying the stability estimate \eqref{H-theta stabilty}, we have
		\begin{equation*}
			\lim_{m\to\infty} \mathbb{E} \left[ 1 \land \|u_m(t) - u(t)\|^2_{H^\theta} \;\bigg|\; \mathcal{F}_0 \right] = 0 \quad \pas
		\end{equation*}
		Taking the expectation on both sides and applying the dominated convergence theorem, we obtain:
		\begin{equation}\label{eq: prob convergence limit}
			\lim_{m\to\infty} \mathbb{E} \left[ 1 \land \|u_m(t) - u(t)\|^2_{H^\theta} \right] = 0,
		\end{equation}
		which is precisely the characterization for $u_m(t) \to u(t)$  in probability with respect to the $H^\theta$-topology as $m \to \infty$. 
		
		Note that \eqref{eq: prob convergence limit} implies that $u_m(t)\to u(t)$ in probability in the $H^\theta$-topology. Hence, for every $\varphi\in C_B(H^{s,\theta}_{\mathrm{div}})$, which is  globally bounded by some constant $M > 0$, the continuity of $\varphi$ (with respect to the $H^\theta$-norm) gives $\varphi(u_m(t))\to \varphi(u(t))$ in probability.   Since  $\varphi(u_m(t))$ are uniformly bounded by $M$,  the Vitali's convergence theorem implies
		\begin{equation*}
			\lim_{m\to\infty} \mathbb{E}\big[\varphi(u_m(t))\big] = \mathbb{E}\big[\varphi(u(t))\big].
		\end{equation*}
		
		In view of \eqref{eq: semigroup_def}, this establishes exactly that $\lim_{m\to\infty} \mathscr{P}_t\varphi(v_m) = \mathscr{P}_t\varphi(v)$. Since the sequence $\{v_m\}_{m\ge 1} \subset B^s_R$ and the limit $v \in B^s_R$ were chosen arbitrarily, we rigorously conclude that the restriction of the mapping $x \mapsto \mathscr{P}_t\varphi(x)$ to the closed ball $B^s_R$ is continuous with respect to the $H^\theta$-topology. This completes the proof.
	\end{proof}
	
	\begin{proof}[Proof of Theorem  \ref{Thm-(u P)-long-time} \ref{Thm-(u P)-IM}]
		
		By Remark \ref{Remark-SES-examples}, \eqref{Target problem (u P) damp} fits into the abstract framework \eqref{SES} with
		\begin{align*} 
			&b_1(u) = -\bigg[\Pi\big((u\cdot\nabla)u\big)+\Upgamma u-\frac{1}{2}\mathcal{Q}_1^2u\bigg]
			+\int_{|l|\le1} \Big\{\wp\big(1,l,u\big)-u-l\cdot\mathcal{Q}_2u\Big\}\,\nu(\mathrm{d}l), \\
			&b_2(u) = \mathcal{Q}_1u, \quad b_3(l,u) = \wp\big(1,l,u\big)-u,\quad b_4(u)=\widetilde h (u).
		\end{align*} 
		and
		\begin{equation*}
			\bfH=\Hdiv^{s},\quad
			\widetilde{\bfH}=\Hdiv^{\theta},\quad \theta\in\left(\frac{d}{2}+p,{s-\max \{3\zeta,1\}}\right).
		\end{equation*}
		If \ref{Hypo-DNI-1} is strengthened to \ref{Hypo-DNI-2}, we can take $\mathscr{G}_1=\mathscr{G}_2=0$ in the proof of   \eqref{u Hs bound V-1} to obtain \eqref{u Hs bound V-2}.

		When working on $\T^d$, the embedding $\Hdiv^s\hookrightarrow\hookrightarrow \Hdiv^{\theta}$ holds. Therefore, combining \eqref{u Hs bound V-2}, Lemma \ref{Lemma Euler restricted Feller} and Theorem \ref{Thm-SES-IM} yields the existence of an invariant probability measure.
	\end{proof}
	
	\subsubsection{Uniqueness of the Invariant Probability Measure: Proof of Theorem \ref{Thm-(u P)-long-time} \ref{Thm-(u P)-decay}}

	Recall the stopping $\tau^N$ in \eqref{eq: tau-m-N tau-N}, i.e.,
	\begin{equation*}
		\tau^{N}\triangleq N \land \inf\{t\ge 0: \|u(t)\|^2_{H^s}> N\}, \quad N \ge 1,
	\end{equation*}
	and consider the function
	\begin{equation*} 
		G(\bullet\land \tau^N) \triangleq \mathbb{E}\left[V(\|u(\bullet\land \tau^N)\|_{H^\theta}^2 )\bigg|\mathcal{F}_0\right],\quad N\ge1,\quad \theta\in(d/2+p,{s-\max \{3\zeta,1\}}).
	\end{equation*}
	Let $N\ge1$ and let $V\in\mathscr{V}$ be as in \ref{Hypo-DNI-3}. Keep in mind that $\Q_2\in\mathbb{B}^\beta$ and $\Q_2+\Q_2^*=0$. 	
	In this case we have $\mathscr{a}_2=0$ (the Marcus flow acts as an exact isometry in $H^\theta$, rendering the jump contribution identically zero, cf. Remarkcf. \eqref{SCR ai}).  For any $t_1 > t_2 \ge 0$, we have
	\begin{align*}
		&G(t_1 \land \tau^N) - G(t_2 \land \tau^N)  \notag \\
		=\ &\E \bigg[
		\int_{t_2}^{t_1} \mathbf{1}_{\{t' \le \tau^N\}}  V'(\|u(t')\|^2_{H^{\theta}})
		\Big(-2\IP{ \Pi[(u(t')\cdot\nn)u(t')],u(t')}_{H^{\theta}} -2\Upgamma \|u(t')\|^2_{H^{\theta}}\Big)\, {\rm d}t'\bigg|\F_0\bigg]\notag\\
		&+\E \bigg[
		\int_{t_2}^{t_1} \mathbf{1}_{\{t' \le \tau^N\}}  V'(\|u(t')\|^2_{H^{\theta}})
		\Big(\IP{\Q_{1}^2u(t') , u(t')}_{H^{\theta}}+\|\Q_{1} u(t')\|^2_{H^{\theta}}
		+\|  \widetilde h(u(t')) \|^2_{H^{\theta}}\Big)\, {\rm d}t'\bigg|\F_0\bigg]\notag\\
		&+\E \bigg[
		\int_{t_2}^{t_1} \mathbf{1}_{\{t' \le \tau^N\}}   2V''(\|u(t')\|^2_{H^{\theta}}) \Big(\IP{\Q_{1}u(t'), u(t')}^2_{H^{\theta}}
		+\IP{ \widetilde h(u(t')),u(t')}_{H^{\theta}}^2 \Big)\, {\rm d}t'\bigg|\F_0\bigg].
	\end{align*}
	Using the properties of $V$ (cf. \eqref{SCRV} and \ref{Hypo-DNI-3}), \ref{Hypo-h div},  \eqref{wp estimate 1 C21}, and Lemma \ref{uux u M}, we find a constant $c_N>0$ such that for any $t_1 > t_2 \ge 0$,
	\begin{align*}
		& \Abs{G(t_1 \land \tau^N) - G(t_2 \land \tau^N)}  \notag \\
		\leq\ &\E \bigg[
		\int_{t_2}^{t_1} \mathbf{1}_{\{t' \le \tau^N\}}  V'(\|u(t')\|^2_{H^{\theta}})
		\Big(\big(\mathscr{a}_1+ 2\Upgamma + 2\mathscr{c}  \|u(t')\|_{\Wlip} \big) \|u(t')\|^2_{H^{\theta}}
		+\|\widetilde h(u(t')) \|^2_{H^{\theta}} \Big)\, {\rm d}t'\Big|\F_0\bigg]\notag\\
		&+\E \bigg[
		\int_{t_2}^{t_1} \mathbf{1}_{\{t' \le \tau^N\}}   2\Abs{V''(\|u(t')\|^2_{H^{\theta}}) }\Big(\IP{\Q_{1}u(t'), u(t')}^2_{H^{\theta}}
		+\IP{\widetilde h(u(t')),u(t')}_{H^{\theta}}^2 \Big)\, {\rm d}t'\bigg|\F_0\bigg]\notag\\
		\leq \ &c_N (t_1-t_2).
	\end{align*}
	This implies that $G(\bullet \land \tau^N)$ is locally absolutely continuous almost surely for all $N\ge1$.
	Again using the properties of $V$ (cf. \eqref{SCRV} and \ref{Hypo-DNI-3}), \ref{Hypo-h div}, estimate \eqref{wp estimate 1 C21}, and Lemma \ref{uux u M}, we derive the differential inequality:
	\begin{align}
		&\frac{{\rm d}}{{\rm d}t}G(t \land \tau^N) \notag\\
		=\ &\E \bigg[
		\mathbf{1}_{\{t \le \tau^N\}}  V'(\|u(t')\|^2_{H^{\theta}})
		\Big(-2\IP{ [(u(t)\cdot\nn)u(t)],u(t)}_{H^{\theta}}
		-2\Upgamma \|u(t)\|^2_{H^{\theta}}\Big)\bigg|\F_0\bigg]\notag\\
		&+\E \bigg[
		\mathbf{1}_{\{t \le \tau^N\}}  V'(\|u(t)\|^2_{H^{\theta}})
		\Big(\IP{\Q_{1}^2u(t) , u(t)}_{H^{\theta}}+\|\Q_{1} u(t)\|^2_{H^{\theta}}
		+\| \widetilde h(u(t)) \|^2_{H^{\theta}}\Big)\bigg|\F_0\bigg]\notag\\
		&+\E \bigg[
		\mathbf{1}_{\{t\le \tau^N\}}   2V''(\|u(t)\|^2_{H^{\theta}}) \Big(\IP{\Q_{1}u(t), u(t)}^2_{H^{\theta}}
		+\IP{ \widetilde h(u(t)),u(t)}_{H^{\theta}}^2 \Big)\bigg|\F_0\bigg]\notag\\
		\leq \ &  \E \bigg[V'(\|u(t\land \tau^N)\|^2_{H^{\theta}})\Big(\big(\mathscr{a}_1+ 2\mathscr{c}  \|u(t\land \tau^N)\|_{\Wlip} \big) \|u(t\land \tau^N)\|^2_{H^{\theta}}
		+\|  \widetilde h(u(t\land \tau^N)) \|^2_{H^{\theta}}\Big)\bigg|\F_0\bigg]\notag\\
		& -2\Upgamma\E \bigg[\mathbf{1}_{\{t \le \tau^N\}} V'(\|u(t)\|^2_{H^{\theta}})\|u(t)\|^2_{H^{\theta}}\bigg|\F_0\bigg]+\E \bigg[
		\mathbf{1}_{\{t\le \tau^N\}}   2V''(\|u(t)\|^2_{H^{\theta}}) \IP{ \widetilde h(u(t)),u(t)}_{H^{\theta}}^2 \Big)\bigg|\F_0\bigg]
		\label{V u t tau-N differential}
	\end{align}
	Since \ref{Hypo-DNI-3} implies \ref{Hypo-DNI-1}, the solution exists globally; that is,
	\begin{equation*}
		\p(\lim_{N\to\infty}\tau^N=\tau^*=\infty)=1.
	\end{equation*}
	We note that $\mathscr{a}_2=0$ (cf. \ref{Hypo-DNI-3} and \eqref{SCR ai}) and then the function $\mathbf{V}_{s}^{\Q_1,\, \Q_2,\, \widetilde h}(f)$ defined in \eqref{V Lyapunov type term} reduces to
	\begin{align*}  
		\mathbf{V}_{s}^{\Q_1,\, \Q_2,\, \widetilde h}(f)
		=  V'(\|f\|^2_{H^s}) \bigg\{  
		\big(\mathscr{a}_1 + 2\mathscr{c}  \|f\|_{\Wlip} \big) \|f\|^2_{H^s}    + \|\widetilde h(f)\|^2_{H^s} \bigg\}  +2V''(\|f\|^2_{H^s}) \langle \widetilde h(f), f \rangle_{H^s}^2.
	\end{align*}  
	Hence, passing to the limit $N\to\infty$ in \eqref{V u t tau-N differential},  we arrive at
	\begin{equation*}
		\frac{{\rm d}}{{\rm d}t}\mathbb{E}\left[V(\|u(t)\|_{H^\theta}^2 )\bigg|\mathcal{F}_0\right]
		\leq \mathbb{E}\left[\mathbf{V}_{\theta}^{\Q_1,\, \Q_2,\, \widetilde h}(u(t)) 
		-2\Upgamma V'(\|u(t)\|^2_{H^{\theta}} )\|u(t)\|^2_{H^{\theta}} \bigg|\mathcal{F}_0\right].
	\end{equation*}
	This, together with \ref{Hypo-DNI-3}, implies
	\begin{equation*}
		\frac{{\rm d}}{{\rm d}t}\mathbb{E}\left[V(\|u(t)\|_{H^\theta}^2 )\bigg|\mathcal{F}_0\right]
		\leq -\mathscr{G}_3  \mathbb{E}\left[V\left(\frac{\|u(t)\|^2_{W^{p,\infty}}}{M^2}\right)\bigg|\mathcal{F}_0\right].
	\end{equation*}
	Applying Lemma \ref{Gronwall g f} below with 
	$$
	f_1=\mathbb{E}\left[V\left(\|u(t)\|^2_{H^\theta}\right)\bigg|\mathcal{F}_0\right],\quad f_2=\mathbb{E}\left[V\Big(\frac{\|u(t)\|^2_{W^{p,\infty}}}{M^2}\Big)\bigg|\mathcal{F}_0\right],\quad 
	q(t)\equiv\mathscr{G}_3
	$$
	yields \eqref{u Wlip -> 0}, which completes the proof.

	\begin{Lemma}\label{Gronwall g f}
		Let $f_1:[0,\infty)\to[0,\infty)$ be locally absolutely continuous,  let $f_2:[0,\infty)\to [0,\infty)$ be locally bounded and let $q:[0,\infty)\to [0,\infty)$ be locally integrable. 
		Suppose that for almost every $t \ge 0$,
		\[
		f_1'(t) \le -q(t)\,f_2(t), \qquad f_2(t) \le f_1(t).
		\]
		Then, for every $t \ge 0$,
		\[
		\int_0^t q(t')\, f_2(t')\, dt' \;\le\; f_1(0)\Big(1 - e^{-\int_0^t q(t')\, dt'}\Big).
		\]
	\end{Lemma}
	
	\begin{proof}
		Since $f_1$ is locally absolutely continuous and $f_1'(t) \le -q(t)f_2(t)$ a.e., integrating from $0$ to $t$ gives, for every $t\ge 0$,
		\[
		f_1(t) \le f_1(0) - \int_0^t q(t') f_2(t')\d t' \;=\; f_1(0) - k(t),\quad k(t) \triangleq \int_0^t q(t') f_2(t')\d t'.
		\]
		Since $qf_2$ is locally integrable and nonnegative, $k$ is locally absolutely continuous with $k(0)=0$ and
		\[
		k'(t) = q(t) f_2(t) \quad \text{for a.e. } t \ge 0.
		\]
		Using $f_2(t) \le f_1(t)$ and the bound on $f_1(t)$ above, we get for a.e. $t \ge 0$,
		\[
		k'(t) = q(t) f_2(t) \le q(t)\, f_1(t) \le q(t)\,\big(f_1(0) - k(t)\big).
		\]
		Set $$K(t) \triangleq k(t) {\rm e}^{ \int_0^t q(t')\d t'}.$$ Since \( k(t) \) and \( t \mapsto \int_0^t q(t')\d t' \) are locally absolutely continuous, so is \( K \), and for almost every \( t \ge 0 \),
		\[
		K'(t) = \big( k'(t) + q(t)k(t) \big) e^{\int_0^t q(t')\d t'} \le q(t) f_1(0) e^{\int_0^t q(t')\d t'}.
		\]
		Integrating from \( 0 \) to \( t \) (note that \( K(0) = 0 \)) gives
		\[
		K(t) \le f_1(0) \int_0^t q(t') e^{\int_0^{t'} q(t'')\d t''}\d t' = f_1(0) \Big( e^{\int_0^t q(t')\d t'} - 1 \Big).
		\]
		Hence,
		\[
		k(t)=
		\int_0^t q(t')\, f_2(t')\, dt' \;\le\; f_1(0)\Big(1 - e^{-\int_0^t q(t')\, dt'}\Big).
		\]
		This holds for every $t \ge 0$ because both sides are continuous in $t$.
	\end{proof}
	\begin{Remark}\label{Remark: global-(u P) viscosity}
		Because the \textbf{DNI} conditions \ref{Hypo-DNI-1}, \ref{Hypo-DNI-2}, and \ref{Hypo-DNI-3} are entirely independent of any viscous smoothing effect (as discussed in Remark \ref{Remark-DNI-Compare}), the global-in-time results established in Theorem \ref{Thm-(u P)-long-time} extend easily to the stochastic incompressible Navier--Stokes equations with fractional or degenerate viscosity (cf. Remark~\ref{Remark: local-(u P) viscosity}). Indeed, by accommodating viscous dissipation purely through its favorable sign condition rather than its regularizing properties, the proofs of Theorem \ref{Thm-(u P)-long-time}  also provide related results on invariant measures for the damped stochastic Navier--Stokes equations on $\K^d$. For more details and further extensions to other physically relevant equations, we refer the reader to Section \ref{Section: Extension-Related Models}.

	\end{Remark}

	\section{Discussion on Extensions and Open Problems}\label{Section : Further Discussions}
	
	\subsection{Homogeneous Mikhlin-type Noise Amplitudes}\label{Section:Mikhlin}
	
	The standard pseudo-differential symbol classes $\mathbf{S}^s$ defined in Section \ref{Section:PDO} require the symbols to be smooth with respect to the frequency variable everywhere. Consequently, many fundamental operators whose symbols are homogeneous and singular (or undefined) at the origin do not fall into these classical categories. Notable examples include the fractional Laplacian $(-\Delta)^{\alpha}$ (for non-integer $\alpha>0$) and the $x_j$-directional Riesz transforms $\mathcal{R}_j$ ($1\le j\le d$). 
	
	To accommodate these important operators, we introduce the homogeneous Mikhlin-type symbol classes. In the continuous setting, since the singularity occurs at the zero frequency, it is natural to define the symbol class strictly away from the origin. For $s\in\R$, we define the continuous Mikhlin-type class $\mathfrak{M}^s(\R^{d}\setminus\{0\};\mathbb{C}^{m\times m})$ as
	\begin{align}\label{def Ms R}
		\mathfrak{M}^s(\R^{d}\setminus\{0\};\mathbb{C}^{m\times m})  \triangleq 
		\left \{\mathscr{p}\in C^\infty(\R^d \setminus \{0\};\mathbb{C}^{m\times m}) \, : \,  
		|\mathscr{p}|^{\alpha ;s}_{\R^{d}} \triangleq 
		\sup_{\xi\in \R^{d}\setminus\{0\}}
		\frac{\left |\partial _{\xi}^{\alpha} \mathscr{p} (\xi )\right |_{m\times m}}{|\xi |^{ s-|\alpha |_{1}}}<\infty,\quad  \alpha \in \mathbb N_{0}^{d}\right \}.
	\end{align}
	
	In the periodic setting, however, defining the discrete counterpart requires more caution. Recall the partial difference operator $\triangle _{k}^{\alpha}$ given by \eqref{difference operator}. Computing $\triangle _{k}^{\alpha}\mathscr{p}(k)$ involves evaluating the symbol at shifted points $k+\beta$. Even for a non-zero frequency $k \neq 0$, the shifted point may hit the origin. Therefore, the symbol must be algebraically defined on the entire lattice $\Z^d$. 
	
	To this end, we directly define the discrete Mikhlin-type class $\mathfrak{M}^s(\Z^{d};\mathbb{C}^{m\times m})$ on $\Z^d$   with a built-in zero condition at the origin:
	\begin{align}\label{def Ms Z}
		\mathfrak{M}^s(\mathbb Z^{d};\mathbb{C}^{m\times m})\triangleq\left\{\mathscr{p}\in\mathscr{M}(\mathbb Z^{d};\mathbb{C}^{m\times m}) \, :\,\mathscr{p}(0) = 0_{m\times m},\quad 
		|\mathscr{p}|^{\alpha ;s}_{\mathbb Z^{d}} \triangleq \sup_{k\in \mathbb Z^{d}\setminus\{0\}}
		\frac{\left|\triangle _{k}^{\alpha} \mathscr{p} (k)\right|_{m\times m}}{|k|^{ s-|\alpha |_{1}}}<\infty\quad \forall \alpha \in \mathbb N_{0}^{d}\right\}.
	\end{align}

	Since we are considering real-valued evolution systems, we assume
	that for the frequency variable $\xi\in\R^d$ (if $x\in \mathbb R^{d}$) or $k\in\Z^d$ (if $x\in \mathbb T^{d}$):
	\begin{align}
		\mathscr{p}(-\cdot)=\overline{ \mathscr{p}(\cdot)},
		\label{Real symbol Mikhlin}
	\end{align}
	Note that $\mathfrak{M}^s(\mathbb R^{d}\setminus\{0\};\mathbb{C}^{m\times m}) $ and $\mathfrak{M}^s(\mathbb Z^{d};\mathbb{C}^{m\times m}) $ are Fr\'{e}chet spaces with respect to the seminorms $\{|\cdot |^{\alpha ;s}_{\mathbb R^{d}}\}_{
		\alpha \in \mathbb N_{0}^{d}}$ and $\{|\cdot|^{\alpha ;s}_{\mathbb Z^{d}}\}_{
		\alpha \in \mathbb N_{0}^{d}}$, respectively.
	If there is no ambiguity in the context, we unify notations and write 
	\begin{equation*} 
		\mathfrak{M}^{s} \triangleq  \left \{\mathscr{p} \in\mathfrak{M}^s(\mathbb R^{d}\setminus\{0\};\mathbb{C}^{m\times m})  \, : \, \text{{\eqref{Real symbol Mikhlin}}}\ \text{holds} \right \}\  \text{or}\
		\left \{\mathscr{p} \in \mathfrak{M}^s(\mathbb Z^{d};\mathbb{C}^{m\times m}) \, : \, \text{{\eqref{Real symbol Mikhlin}}}\ \text{holds} \right \},
	\end{equation*}
	
	As in \eqref{OP define}, $\OP(\mathscr{p})$ with symbol $\mathscr{p}$ is defined by
	\begin{align*}
		\OP(\mathscr{p}) \triangleq \mathcal P,\quad
		[\mathcal P f]  \triangleq  
		\mathscr F^{-1}\Big[\mathscr{p} \big(\mathscr Ff\big)\Big]
	\end{align*}
	and we denote
	\begin{align}
		{\mathrm{OP}}\mathfrak M^{s} \triangleq 
		\big\{{\mathrm{OP}}({\mathscr{p} })\, : \, {\mathscr{p} } \in \mathfrak M^{s}\big\},
		\quad 
		\{|\cdot |^{\alpha ;s}\}_{\alpha \in \mathbb N_{0}^{d}} 
		\triangleq 
		\{|\cdot |^{\alpha ;s}_{\mathbb R^{d}}\}_{\alpha \in \mathbb N_{0}^{d}}
		\ \text{or} \
		\{|\cdot |^{\alpha ;s}_{\mathbb Z^{d}}\}_{\alpha \in \mathbb N_{0}^{d}},
		\quad  
		s\in \mathbb R, \ \alpha \in \mathbb N_{0}^{d}.\label{OPM seminorms}
	\end{align}
	By direct computation,  one we have
	\begin{Lemma}\label{Lem: OP-continuous-M}
		If $s\ge0$, and $\mathscr{p}\in \mathfrak{M}^s$, then  there
		is $\widetilde \alpha \in \mathbb N_{0}^{d}$ and a constant
		$C=C(s,q)>0$ such that
		\begin{align*}
			\|{\mathrm{OP}}(\mathscr{p})\|_{\mathscr L(H^{q+s};H^{q})}\leq C(s,q)|\mathscr{p}|^{\widetilde \alpha ; s}.
		\end{align*}
	\end{Lemma}
	To highlight the \emph{scalar} symbols, we define
	\begin{equation}\label{Ms define} 
		\mathcal{M}^{s} \triangleq  \left \{\mathscr{p} \in\mathfrak{M}^s(\mathbb R^{d}\setminus\{0\};\mathbb{C})  \, : \, \text{{\eqref{Real symbol Mikhlin}}}\ \text{holds} \right \}\  \text{or}\
		\left \{\mathscr{p} \in \mathfrak{M}^s(\mathbb Z^{d};\mathbb{C}) \, : \, \text{{\eqref{Real symbol Mikhlin}}}\ \text{holds} \right \},
	\end{equation}
	where  $\mathscr{p} \in\mathfrak{M}^s(\mathbb R^{d}\setminus\{0\};\mathbb{C})$ and $\mathscr{p} \in \mathfrak{M}^s(\mathbb Z^{d};\mathbb{C})$ are obtained formally by taking $m=1$ in \eqref{def Ms R} and \eqref{def Ms Z}, with $|\cdot|_{1\times 1}$ identified with the modulus on $\mathbb C$.
	Similarly, we define  $\OP\mathcal{M}^s$ as the operators with symbols in $\mathcal{M}^s$.

	\begin{Remark} 
		The convention $\mathscr{p}(0) = 0_{m\times m}$ in \eqref{def Ms Z} is not merely an algebraic convenience. Indeed, besides ensuring that the nonlocal difference operator $\triangle _{k}^{\alpha}\mathscr{p}(k)$ is well-defined for all $k \neq 0$, there are some additional reasons: 
		\begin{itemize}[leftmargin=0.79cm]\setlength\itemsep{0.2em}
			\item If $s>0$, homogeneity forces $\mathscr{p}(0)$ to be zero.  If $s=0$, we take  the Riesz transforms $\mathcal{R}_j$ ($2\le j\le d$) with symbol $p(\xi) = -i \xi_j / |\xi|$ as an example. This symbol has no classical limit at the origin. However, in Calder\'on-Zygmund theory, such multipliers correspond to convolution kernels $K(x)$ (after periodic extension) that are odd functions. Therefore, the value of the Fourier multiplier at the origin corresponds to the principal value integral of the kernel over the spatial domain:
			$p(0) = \mathrm{p.v.} \int_{\R^d} K(x) \d x=0.$
			\item Without the condition $\mathscr{p}(0)=0_{m\times m}$, the associated operator $\OP(\mathscr{p})$ is not even defined on $L^2$.
		\end{itemize}
	\end{Remark}
	
	\begin{Definition} 
		Let $d, m \ge 1$ and $s\in\R$. A subset $\mathscr O\subset {\mathrm{OP}}\mathfrak M^{s}$ is said to be bounded if
		$
		\big\{\mathscr{p}: {\mathrm{OP}}(\mathscr{p})\in \mathscr O\big\}\subset \mathfrak M^{s}
		$
		is bounded with respect to the seminorms in \eqref{OPM seminorms}. Let $\varsigma \ge 0$.
		We define
		\[
		\mathbb{M}^{\varsigma} \triangleq \left\{\mathcal K:\
		\text{there exist }\mathcal H={\mathrm{diag}}(\mathcal H_{1},\cdots ,\mathcal H_{m})\in\mathrm{OP}\mathfrak M^{\varsigma}
		\text{ and }\mathcal E\in\mathrm{OP}\mathfrak M^0\text{ such that }
		\mathcal K=\mathcal H+\mathcal E
		\text{ and  \ref{Kk:(K1 K2)}  holds} \right\},
		\]
		where condition \ref{Kk:(K1 K2)} is specified as follows:
		\begin{enumerate}[label={$(\mathbf{K}^{\varsigma})$},leftmargin=0.79cm]
			\item\label{Kk:(K1 K2)}  
			$\mathcal H_{i} \in {\mathrm{OP}}\mathcal M^{\varsigma}$ and $\mathcal H_{i} + \mathcal H_{i}^{*} \in {\mathrm{OP}}\mathcal M^{0}$ for $1 \le i \le m$. 
			Moreover, the symbols of $\mathcal H$ and $\mathcal E$ are commuting matrices.
		\end{enumerate}
		A subset $\mathscr{N} \subset \mathbb{M}^{\varsigma}$ is called bounded if for every
		$\mathcal K\in\mathscr N$,  one can find a decomposition
		$\mathcal K=\mathcal H_{\mathcal K}+\mathcal E_{\mathcal K}$ such that
		$
		\{\mathcal H_{i,\mathcal K}:\mathcal K\in\mathscr N,1 \le i \le m\}$
		is bounded in $\mathrm{OP}\mathcal M^{\varsigma}$, 
		$	\{\mathcal H_{i,\mathcal K}+\mathcal H^*_{i,\mathcal K}:\mathcal K\in\mathscr N,1 \le i \le m\}$  is bounded in $\mathrm{OP}\mathcal M^0$,  and $\{\mathcal G_{\mathcal K}:\mathcal K\in\mathscr N\}$
		are bounded in $\mathrm{OP}\mathfrak M^0$.
	\end{Definition}

	Repeating the proof of \ref{x-independent case} in Theorem \ref{Thm-cancel} shows that 
	\eqref{cancel-PB1} and \eqref{cancel-PB2} hold when $\mathbb{B}^{\beta}$ is replaced by $\mathbb{M}^{\varsigma}$.  
	Moreover, note that, as long as the following commutators make sense (for instance, when the operators act on zero-average functions), we have
	$$[\mathcal K,J_n]=[\mathcal K,\D^s]=[\mathcal K, \Pi_d]=[\mathcal K, \Pi_0]=0,\quad \mathcal K\in {\mathrm{OP}}\mathcal M^{\varsigma},$$
	where $J_n$ is the mollifier defined in \eqref{Define Jn}, $\D^s$ is the Bessel potential 
	$\mathcal D^{s}=({\mathbf I}-\Delta )^{s/2}$ defined in \eqref{Define Ds}, $\Pi_d$ is the Leray projection defined in \eqref{Pi-d define}, and $\Pi_0$ is the zero-average projection defined in \eqref{Pi-0 define}. With these facts and Lemma \ref{Lem: OP-continuous-M}, one can analogously carry out all the results in Section \ref{Section : Appl-cancel} for the case $\Q_1,\Q_2\in\mathbb{M}^{\varsigma}$.
	
	In summary, we can replace \ref{Hypo-Qi} by 
	\begin{Hypothesis}\label{Hypo-Qi-M}
		Assume that $\mathcal Q_1,\mathcal Q_2\in \mathbb{A}^{\alpha}\cup \mathbb{B}^{\beta}\cup \mathbb{M}^{\varsigma}$, with $d=m\ge2$  in \eqref{Ss define} and \eqref{Ms define}, respectively.
	\end{Hypothesis}
	
	As before, once $\Q_1$ and $\Q_2$ are given, their decomposition and exponent are chosen once and for all, even if even if $\mathcal Q_i$ ($i=1,2$) has another admissible decomposition.

	Then, after adjusting the value of $\zeta=\zeta_{\Q_1,\Q_2}$  in \eqref{zeta-Q12} to accommodate the case $\Q_i\in \mathbb{M}^{\varsigma}$, Theorem \ref{Thm-(rho u)} still holds true. Indeed, the same proof still works. Similarly, Theorem \ref{Thm-(u P)-local}
	also holds true if we require that $\Q_1$ and $\Q_2$ satisfy \ref{Hypo-Qi-M} and  
	$$
	\Pi\Q_i\big|_{\Hdiv^s}=\Q_i\big|_{\Hdiv^s}.
	$$
	Furthermore, since \eqref{cancel-PB1} and \eqref{cancel-PB2} also hold true with $\mathbb{M}^{\varsigma}$ replacing $\mathbb{B}^{\beta}$, the 
	Lyapunov-type function \eqref{V Lyapunov type term} can be defined in the same form but with constants $\mathscr{a}_1$ and $\mathscr{a}_2$ taking possibly different values (see Remark \ref{Remark singualrity IC}). Hence Theorem \ref{Thm-(u P)-long-time} also holds true if $\mathcal Q_1,\mathcal Q_2\in \mathbb{A}^{\alpha}\cup \mathbb{B}^{\beta}\cup \mathbb{M}^{\varsigma}$ and $
	\Pi\Q_i\big|_{\Hdiv^s}=\Q_i\big|_{\Hdiv^s}.
	$
	
	As a notable application, Theorems \ref{Thm-(rho u)}, \ref{Thm-(u P)-local}, and \ref{Thm-(u P)-long-time} remain valid for the class of fractional homogeneous operators introduced in Section \ref{Section : PD Noise Amplitudes}:
	$$
	\mathcal{Q}_j u
	= \sum_{i=1}^d c_i 
	\mathcal{R}_i(-\Delta)^{\varsigma/2} u,
	\qquad \varsigma \ge 0,
	$$
	where $(-\Delta)^{\varsigma/2}$ is the fractional Laplacian, $\mathcal{R}_i$ denotes the $i$-th Riesz transform, and $c_i \in \mathbb{R}$.
	
	\subsection{Related Models and Sequences of Noise}
	\subsubsection{Related Models}\label{Section: Extension-Related Models}
	From the proof of Theorem \ref{Thm-(rho u)}, we observe that the techniques utilized to control the delicate interaction between noise and nonlocal operators can also be applied to a broader class of fluid equations. Indeed, with a slight abuse of notation, we consider the following SPDE:
	\begin{align}\label{SPDE-general}
		{\rm d} u + \left[\epsilon(-\Delta)^{\ell/2}u+\Upgamma u+(u \cdot \nabla)u
		+F(u)\right]\d  t
		=\mathcal{Q}_1 u\, \circ\, {\rm d} W
		+\mathcal{Q}_2 u \diamond {\rm d} L 
		+ h(u)\d \widetilde{W},\quad \epsilon,\ \, \Upgamma\ge0,\quad \ell\in(0,2],
	\end{align}
	where $F(\cdot)$ is a new model-dependent ingredient, $\Upgamma u$ represents the Ekman damping term, and $(-\Delta)^{\ell/2}u$ denotes the fractional viscous term.
	
	\begin{Hypothesis}
		\label{Hypo-F-h}
		Let $d \ge 1$, $p \ge 1$, and $\sigma > \frac{d}{2} + p$. Assume that
		there exists a positive, increasing function $K:[0,\infty)\to (0,\infty)$ such that for all $u, u_1,u_2\in H^{\sigma}$,
		\begin{equation*}
			\norm{F(u)}^2_{H^{\sigma}},\, \norm{h(u)}^2_{H^{\sigma}} \leq  K(\|u\|_{W^{p,\infty}}
			)(1+\|u\|^{2}_{H^{\sigma}}),
		\end{equation*}
		\begin{align*}
			\norm{F(u_1)- F(u_2)}^{2}_{H^{\sigma}},\,\norm{h(u_1)- h(u_2)}^{2}_{H^{\sigma}} \le K(\|u_1\|_{H^{\sigma}}+\|u_2\|_{H^{\sigma}})\|u_1-u_2\|^{2}_{H^{\sigma}}.
		\end{align*}
	\end{Hypothesis}
	
	Under \ref{Hypo-Qi-M} and \ref{Hypo-F-h}, one can apply the same analysis  to the velocity equation in \eqref{SEuler-(rho u)}$_2$ (one may compare \ref{Hypo-F-h} with Lemmas \ref{Lemma-F} and \ref{Lemma-F-difference}). By treating the viscous term $\epsilon(-\Delta)^{\ell/2}u$ as a singular term (see, for instance, \cite{Majda-Bertozzi-2002-book,Li-Liu-Tang-2021-SPA,Alonso-Pang-Tang-2026-JLMS}), we obtain the following local-in-time theory for \eqref{SPDE-general} in a sense similar to Definition \ref{solution definition (q u)}. 
	
	\begin{Theorem}[\textbf{Local-in-time theory for pathwise classical solutions to \eqref{SPDE-general}}] \label{Thm-local-Appendix}
		Let $d \geq 1$ and $p\ge 1$.  
		Assume \ref{Hypo-W L}, \ref{Hypo-Qi-M}, and \ref{Hypo-F-h} hold with $\sigma> \frac{d}{2} + p$. Let $s > \frac{d}{2} + p + \max \{3(\alpha\lor \beta\lor \varsigma),\ell\mathbf{1}_{\{\epsilon>0\}},1\}$. Let $u_0\in H^s$ be an $\mathcal{F}_{0}$-measurable random variable. Then \eqref{SPDE-general} with initial data $u(0)=u_0$ admits a unique maximal $H^s$ classical solution $(u,\tau^{*})$. Moreover, 
		\begin{equation*}
			u\in D([0,\tau^*);H^{s'})\quad \forall s'<s\quad \pas,
		\end{equation*}
		the lifetime $\tau^*$ is independent of $s$, and the following blow-up criterion holds:
		\begin{equation*} 
			\limsup_{t \rightarrow \tau^*} \|u(t)\|_{W^{p,\infty}} = \infty \quad \text{a.s. on } \{\tau^* < \infty\}.
		\end{equation*}
	\end{Theorem}
	
	Analogous to \eqref{V Lyapunov type term}, to study the long-time dynamics of \eqref{SPDE-general} for a given $V\in\mathscr{V}$ and $F,h$ satisfying \ref{Hypo-F-h} with $p\ge1$, we define $\widetilde \mathbf{V}_{\sigma}^{\Q_1,\,\Q_2,\, h}$ by
	\begin{align}  
		\widetilde \mathbf{V}_{\sigma}^{\Q_1,\, \Q_2,\,  h}(f)
		\triangleq\ & V'(\|f\|^2_{H^{\sigma}}) \bigg\{  
		\big(\mathscr{a}_1 + 2\mathscr{c}_{\sigma,d} \|f\|_{\Wlip} \big) \|f\|^2_{H^{\sigma}}   +\bIP{F(f), f}_{H^{\sigma}}+ \| h(f)\|^2_{H^{\sigma}} \bigg\}  \notag\\
		& + 2V''(\|f\|^2_{H^{\sigma}})\bIP{h(f), f}_{H^{\sigma}}^2+\mathscr{a}_2 V(\|f\|^2_{H^{\sigma}}), \qquad f\in  H ^{\sigma},\quad \sigma > \tfrac{d}{2} + p, \label{V Lyapunov type term F-h}
	\end{align}  
	where $\mathscr{a}_1$ and $\mathscr{a}_2$ are given in \eqref{SCR ai}, and $\mathscr{c}=\mathscr{c}_{\sigma,d}>0$ is a constant such that $\IP{(u\cdot\nn)u,u}_{H^{\sigma}}\leq \mathscr{c}\|u\|_{\Wlip}\|u\|^2_{H^{\sigma}}$ for smooth $u$. With a slight abuse of notation, we formulate the following assumptions in the same spirit as \ref{Hypo-DNI-1}, \ref{Hypo-DNI-2}, and \ref{Hypo-DNI-3}:
	
	\begin{Hypothesis}
		\label{Hypo-DNI-1 F-h}
		There exist $V\in \mathscr{V}$ and constants $\mathscr{G}_1,\mathscr{G}_2>0$ such that
		\begin{align*}
			\widetilde \mathbf{V}_{\sigma}^{\Q_1,\, \Q_2,\, h}(f)
			\leq \mathscr{G}_1\,V(\|f\|^2_{H^{\sigma}})+\mathscr{G}_2,\qquad f\in  H^{\sigma},\quad \sigma>d/2+p.
		\end{align*}
	\end{Hypothesis}

	\begin{Hypothesis}
		\label{Hypo-DNI-2 F-h}
		$\Q_2\in\mathbb{B}^{\beta}\cup \mathbb{M}^{\varsigma}$ and $Q_2+\Q_2^*=0$. Moreover, there exists $V\in \mathscr{V}$ such that
		\begin{align*}
			\widetilde \mathbf{V}_{\sigma}^{\Q_1,\, \Q_2,\,  h}(f)-2\Upgamma V'(\|f\|^2_{H^{\sigma}} )\|f\|^2_{H^{\sigma}} 
			\leq 0,\qquad f\in  H^{\sigma},\quad \sigma>d/2+p.
		\end{align*}
	\end{Hypothesis}
	
	\begin{Hypothesis}
		\label{Hypo-DNI-3 F-h}
		$\Q_2\in\mathbb{B}^{\beta}\cup \mathbb{M}^{\varsigma}$ and $Q_2+\Q_2^*=0$. Moreover, there exists $V\in \mathscr{V}$ with $\sup_{x\in\R}|V''(x)|<\infty$ and a constant $\mathscr{G}_3>0$ such that 
		\begin{align*} 
			\widetilde  \mathbf{V}_{\sigma}^{\Q_1,\, \Q_2,\, h}(f) -2\Upgamma V'(\|f\|^2_{H^{\sigma}} )\|f\|^2_{H^{\sigma}} 
			\leq -\mathscr{G}_3  V\left(\frac{\|f\|^2_{W^{p,\infty}}}{M^2}\right),\qquad f\in  H^{\sigma},\quad \sigma>d/2+p.
		\end{align*}
		Here $M=M_{\sigma,d}>0$ is the embedding constant for $\|\cdot\|_{W^{p,\infty}}\leq M_{\sigma,d}\|\cdot\|_{H^\sigma}$ when $\sigma>d/2+p$.
	\end{Hypothesis}
	
	\begin{Theorem}[\textbf{Long-time dynamics of \eqref{SPDE-general}}]\label{Thm-long-time-Appendix}
		Let the conditions of Theorem \ref{Thm-local-Appendix} hold. Let  
		$\frac{d}{2}+p<\theta<s-\max \{3(\alpha\lor \beta\lor \varsigma),\ell\mathbf{1}_{\{\epsilon>0\}},1\}$.
		\begin{enumerate}[label={\bf{(\arabic*)}},leftmargin=0.79cm]\setlength\itemsep{0.2em}
			\item  Under \ref{Hypo-DNI-1 F-h}, the system \eqref{SPDE-general} with initial data $u(0)=u_0$ admits a global $H^s$ solution satisfying
			\begin{equation*}
				\mathbb{E} \left[V(\|u(t)\|_{H^s}^2)\Big|\mathcal{F}_0\right]  
				\leq \Big(V(\|u_0\|^2_{H^s})+\mathscr{G}_2t\Big) {\rm e}^{\mathscr{G}_1t},
				\quad t>0.
			\end{equation*}
			Moreover, the solution map $H^{s,\theta}\ni u_0\mapsto u(t)\in H^{s,\theta}$ is continuous, where $H^{s,\theta}$ denotes the space $H^s$ equipped with the $H^{\theta}$-topology.
			
			\item   If \ref{Hypo-DNI-1 F-h} is strengthened to \ref{Hypo-DNI-2 F-h}, then $u$ satisfies 
			\begin{equation*} 
				\mathbb{E}\left[V(\|u(t)\|_{H^s}^2) \Big| \mathcal{F}_0\right]  
				\leq V(\|u_0\|^2_{H^s}), \qquad t\ge0.
			\end{equation*}
			Moreover, the problem \eqref{SPDE-general} on $\T^d$ admits an invariant measure $\mu$ in the sense of Definition \ref{IM definition} with $E=H^s$. 
			
			\item  If \ref{Hypo-DNI-3 F-h} holds, then  
			\begin{equation*} 
				\int_0^t\mathbb{E}  \left[V\left(\frac{\|u(t')\|^2_{W^{p,\infty}}}{M^2}\right)\Bigg|\mathcal{F}_0\right]\d t'
				\leq  \frac{V(\| u_0\|^2_{H^\theta})}{\mathscr{G}_3}\left(1-{\rm e}^{-\mathscr{G}_3t}\right),
				\quad t>0.
			\end{equation*}
			As a result, the Dirac measure centered at $0$ is the \textbf{unique invariant probability measure} on both $\T^d$ and $\R^d$ when $h(0)=0$.
		\end{enumerate}
	\end{Theorem}
	
	We note that many (nonlocal) fluid equations can be formulated in the form of \eqref{SPDE-general}
	with $F$ satisfying \ref{Hypo-F-h}. For instance, the well-known Camassa-Holm equation (cf. \cite{Tang-Zhao-Liu-2014-AA,Camassa-Holm-1993-PRL}) can be rewritten in the form of \eqref{SPDE-general} with $d=1$, $\epsilon=\Upgamma=0$ and 
	$$F(u)\triangleq \partial_{x}(\I-\partial^2_x)^{-1}(u^2+\frac{1}{2}u^2_x).$$ A direct computation (cf. \cite{Tang-Zhao-Liu-2014-AA}) shows that such an $F$ satisfies \ref{Hypo-F-h} with $p=1$. 
	For the case of stochastic Burgers' equation ($F=0$ and $\Upgamma=0$), we refer the reader to \cite{Boritchev-Kuksin-2021-Book}.

	\subsubsection{Sequences of Noise}
	In this work, in \eqref{SEuler-(rho u)} and \eqref{SEuler-(u P)}, we initially consider noise terms of the form
	$
	\mathcal{Q}_1 u\, \circ\, {\rm d}W
	+
	\mathcal{Q}_2 u \diamond {\rm d}L,
	$
	accompanied by additional It\^{o}-type stochastic forcing given, respectively, by
	$$
	h(t,\rho,u)\,{\rm d}\widetilde{W} \quad\text{and}\quad \widetilde{h}(u)\,{\rm d}\widetilde{W},
	$$
	driven by another standard 1-D Brownian motion $\widetilde{W}=\widetilde{W}(t)$.
	
	Let $\big\{W_k(t),\widetilde{W}_k(t)\big\}_{k\ge1}$ be a sequence of 1-D standard Brownian motions, and let $\{L_k(t)\}_{k\ge1}$ be a sequence of pure-jump L\'evy processes with Poisson random measures $\eta_k$ counting jumps of size $l$, intensity measures $\nu_k$, and compensators $\widetilde{\eta}_k = \eta_k - \nu_k \otimes {\rm d}t$. Assume that 
	$\nu_k([-1,1]^C)=0$, $\eta_k([-1,1]^C)=0$, and that the family 
	$$\big\{W_k(t),\widetilde{W}_k(t),L_k(t)\big\}_{k\ge1}$$ 
	is mutually independent.
	
	Assume that $\{a_{1,k}\}_{k\ge1},\{a_{2,k}\}_{k\ge1}\in l^2$, and that the family $\{\mathcal Q_{1,k},\mathcal Q_{2,k}\}_{k\ge1}\subset \mathbb{A}^{\alpha}\cup \mathbb{B}^{\beta}\cup \mathbb{M}^{\varsigma}$ is bounded (with $d=m\ge2$ in \eqref{Ss define}). Similarly, for the sequences $\{h_k\}_{k\ge1}$, one can propose a sequence version of  \ref{Hypo-h} with suitable summability conditions. Then, with minor adaptations, the analysis in Section \ref{Section : Compressible case} can be applied to establish local-in-time theory for the case where the noise structure in \eqref{SEuler-(rho u)} is extended to 
	$$\sum_{k=1}^{\infty}\Big(a_{1,k}\mathcal{Q}_{1,k} u\, \circ\, {\rm d}W_{k}+
	a_{2,k}\mathcal{Q}_{2,k} u \diamond {\rm d}L_k+h_k(t,\rho,u)\,{\rm d}\widetilde{W}_k\Big).$$
	
	Similarly, when  \ref{Hypo-h div}--\ref{Hypo-DNI-3} are upgraded to their sequence counterparts under appropriate summability conditions, one can also extend the incompressible case to accommodate noise of the form
	$$\sum_{k=1}^{\infty}\Big(a_{1,k}\mathcal{Q}_{1,k} u\, \circ\, {\rm d}W_{k}+
	a_{2,k}\mathcal{Q}_{2,k} u \diamond {\rm d}L_k+\widetilde{h}_k(u)\,{\rm d}\widetilde{W}_k\Big).
	$$
	
	Indeed, we note that the Stratonovich integral can be converted as follows:
	\begin{equation*}
		\sum_{k=1}^{\infty}a_{1,k}\mathcal{Q}_{1,k} u\, \circ\, {\rm d}W_k
		=
		\frac{1}{2}\sum_{k=1}^{\infty}a^2_{1,k}\mathcal{Q}_{1,k}^2 u\,{\rm d}t+\sum_{k=1}^{\infty}a_{1,k}\mathcal{Q}_{1,k} u\,{\rm d}W_k.
	\end{equation*}
	For the jump term $a_{2,k}\Q_{2,k} u \diamond {\rm d}L_k$, scaling the linear operator $\Q_{2,k}$ by the coefficient $a_{2,k}$ is mathematically equivalent to scaling the jump magnitude $l$ within the original Marcus flow. Therefore, we have
	\begin{align*}
		&\int_0^t a_{2,k}\Q_{2,k} u(t') \diamond {\rm d}L_k(t')\\
		\triangleq\  & \int_0^t\int_{|l|\le1} \Big\{\wp_k\big(1, a_{2,k} l, u(t'-)\big)-u(t'-)\Big\}\, \widetilde{\eta}_k({\rm d}l,{\rm d}t')\\
		&+\int_0^t\int_{|l|\le1} \Big\{\wp_k\big(1, a_{2,k} l, u(t')\big)-u(t')-a_{2,k}l\cdot \Q_{2,k} u(t')\Big\}\, \nu_k({\rm d}l){\rm d}t',\quad t>0,
	\end{align*}
	where the unscaled Marcus flow $\wp_k$ is the solution to the following equation depending on the jump size $l$ and initial data $f$:
	\begin{equation*} 
		\wp_k(r)\triangleq\wp_k(r,l,f),\quad 
		\frac{{\rm d}}{{\rm d}r}\wp_k(r)= l\cdot \Q_{2,k}\wp_k(r),\quad r\in[0,1],  \quad \wp_k(0) = f.
	\end{equation*}
	
	Since $\{a_{1,k}\}_{k\ge1},\{a_{2,k}\}_{k\ge1}\in l^2$ and the family $\{\mathcal Q_{1,k},\mathcal Q_{2,k}\}_{k\ge1}\subset \mathbb{A}^{\alpha}\cup \mathbb{B}^{\beta}\cup \mathbb{M}^{\varsigma}$ is bounded, the cancellation properties within the summation can be established for the sequences. This facilitates closing the energy estimates for the aforementioned cases.
	
	\subsection{The 2-D Incompressible Case: An Interesting Open Problem}
	\label{Section : Further-2D-problem}
	In Section \ref{Section : SEuler-global-proof}, it is proved that when the growth of $\widetilde h(\cdot)$ exceeds a certain threshold (cf. \ref{Hypo-DNI-1}), the stochastic incompressible Euler system \eqref{SEuler-(u P)} admits a global classical solution. For convenience, we restate the system \eqref{SEuler-(u P)}:
	\begin{equation*} 
		\left\{
		\begin{aligned}
			&\mathrm{d} u +  (u \cdot \nabla)u \d t + \nabla \mathrm{d}P  = \mathcal{Q}_1 u \circ \d W
			+\mathcal{Q}_2 u \diamond \d L 
			+ \widetilde h(t, u)\d \widetilde{W},\\
			&\Div u = 0,\\
			&u(0) = u_0.
		\end{aligned}
		\right.
	\end{equation*}
	Indeed, \ref{Hypo-DNI-1} characterizes a class of noise coefficients $\widetilde h(\cdot)$ that counteract the growth contributions from other terms, leading to linear growth in $\mathbb{E}\big[V(\|u(t)\|^2_{H^\sigma})\,|\,\mathcal{F}_0\big]$ for the solution $u$, which implies global regularity. 
	
	Dropping the fast-growing noise $\widetilde h(\cdot)\d\widetilde{W}$, we consider the system \eqref{SEuler-(u P)} on $\K^2$:
	\begin{equation}\label{SEuler-2D} 
		\left\{
		\begin{aligned}
			&\mathrm{d} u +  (u \cdot \nabla)u \d t + \nabla \mathrm{d}P  = \mathcal{Q}_1 u \circ \d W
			+\mathcal{Q}_2 u \diamond \d L,\\
			&\Div u = 0,\\
			&u(0) = u_0.
		\end{aligned}
		\right.
	\end{equation}
	Recall that the 2-D deterministic incompressible Euler system admits global regularity. Since the linear noise $\mathcal{Q}_1 u \circ \d W +\mathcal{Q}_2 u \diamond \d L$ typically induces at most exponential growth,  and is usually harmless for global existence. Therefore, at first glance, one might intuitively think that adding linear noise should not induce finite-time blow-up, especially if the noise is conservative. 
	
	However, when $\Q_1$ and $\Q_2$ are \emph{nonlocal}, the problem remains \textbf{open}: neither a rigorous proof of global existence nor a counterexample demonstrating finite-time blow-up is known. 
	The core difficulty lies in the absence of stochastic characteristics, as we explain below. For comparison, we recall the deterministic Cauchy problem for the incompressible  Euler system \eqref{Euler (u P)} in $\K^2$:
	\begin{equation*}
		\left\{\begin{aligned}
			&\partial_t u + (u \cdot\nabla)u + \nabla P= 0,\\ 
			& \Div u= 0,\\
			&u(0) = u_0.
		\end{aligned}
		\right.
	\end{equation*}
	When $d=2$, it is a classical result that strong solutions exist globally for sufficiently regular initial data. The cornerstone of this global well-posedness is the fact that the scalar vorticity is purely transported along the fluid particle trajectories:
	\begin{equation*}
		\partial_t v + (u\cdot\nabla)v = 0, \qquad v \triangleq \partial_{x_1}u_2-\partial_{x_2}u_1.
	\end{equation*}
	This transport structure guarantees the conservation of all $L^p$ norms of the vorticity, particularly the $L^\infty$ norm. By the classical Beale-Kato-Majda (BKM) criterion, the condition $\int_0^T \|v(t)\|_{L^\infty} \d t < \infty$ for any $T>0$ is sufficient to guarantee global regularity \cite[Section 7.2]{Bahouri-Chemin-Danchin-2011-book}.
	
	\subsubsection{Local Transport Noise: Preservation of Characteristics}

	If $\Q_1$ and $\Q_2$ are \emph{classical transport operators} (i.e., local, first-order differential operators) of the form
	\[
	\Q_1 u = (\sigma_1\cdot\nabla)u,\qquad \Q_2 u = (\sigma_2\cdot\nabla)u,
	\]
	with \textbf{constant} vector fields $\sigma_1,\sigma_2\in\R^2$, the stochastic system behaves remarkably well. Indeed, since $\Q_i$ commutes with $\nabla^\perp\cdot$ in this case, \eqref{SEuler-2D} becomes
	\begin{equation}\label{eq:vorticity_Spde}
		\d v + (u\cdot\nabla)v\d t =  \mathcal{Q}_1 v \circ \d W
		+\mathcal{Q}_2 v \diamond \d L, \qquad v \triangleq \nabla^\perp\cdot u,\qquad \Div u=0,
		\qquad v(0)=v_0\triangleq \nabla^\perp\cdot u_0.
	\end{equation}
	Similar to the deterministic case, one can also establish the BKM criterion for a classical solution $u$ to \eqref{SEuler-2D} (cf. \cite{Crisan-Flandoli-Holm-2019-JNLS}):
	$u$ blows up at a stopping time $\tau^*<\infty$ $\pas$ if and only if $\int_0^{\tau^*}\|v(t)\|_{L^\infty} \d t =\infty$ $\pas$ 
	
	Then, we define the \emph{stochastic characteristics} $X(t,x)$ for \eqref{eq:vorticity_Spde} via the Stratonovich--Marcus SDE:
	\begin{equation*} 
		\d X = u(t,X)\d t - \sigma_1 \circ \d W - \sigma_2 \diamond \d L,\qquad X(0,x)=x.
	\end{equation*}
	Along these stochastic trajectories, the vorticity is conserved almost surely:
	\[
	v(t,X(t,x)) = v_0(x)\quad\pas\quad \implies\quad \|v(t)\|_{L^\infty} = \|v_0\|_{L^\infty}\quad\pas.
	\]
	This $L^\infty$ bound seamlessly satisfies the BKM criterion, allowing the deterministic 2-D global well-posedness results to be extended to the stochastic setting. Moreover, $X(t,\cdot)$ remains a $C^1$-diffeomorphism, yielding robust pathwise regularity.
	
	\subsubsection{Nonlocal Noise: The Breakdown of the Lagrangian Flow}
	The situation changes drastically if $\Q_1,\Q_2$ are general pseudo-differential operators or 
	homogeneous Mikhlin-type operators (e.g., Riesz transforms, fractional derivatives like $\nabla(-\Delta)^{s-1/2}$, or other skew-symmetric nonlocal multipliers). For nonlocal operators, the noise term $\Q_i u$ depends on the values of $u$ across the entire spatial domain $\K^2$, not just in an infinitesimal neighborhood of $x$.
	
	Consequently, it is almost impossible to construct a finite-dimensional characteristic flow $X(t)$ that absorbs the noise. The Lagrangian description is entirely lost. This leads to severe analytical consequences:
	\begin{itemize}[leftmargin=0.79cm]\setlength\itemsep{0.2em}
		\item \textbf{Loss of the Maximum Principle:} The vorticity equation is no longer a pure transport equation, as it features a nonlocal stochastic term. It is highly non-trivial whether the $L^\infty$ norm of the vorticity $v$ remains bounded or blows up in finite time.
		\item \textbf{Failure of SDE Reduction:} The SPDE cannot be reduced to a system of ordinary SDEs along characteristics. The delicate interplay between the nonlinear convection $(u\cdot\nabla)u$ and the nonlocal noise must be treated directly in Eulerian coordinates.
	\end{itemize}
	
	\subsubsection{The Open Problem on Global Regularity}
	Intuitively, one might expect that in the 2-D case, the addition of a \emph{linear} noise should not trigger finite-time blow-up, particularly when the noise is energy-conserving. Linear perturbations typically induce \textit{at most} exponential growth, which poses \textit{no threat} to global existence. 
	To isolate the core difficulty, let us impose the most favorable structural assumptions that isolate the nonlocal property while preserving the kinetic energy:
	\begin{enumerate}[leftmargin=0.79cm]\setlength\itemsep{0.2em}
		\item $\Q_1,\Q_2\in\OP\mathcal S^{r}_0$ or $\Q_1,\Q_2\in\OP\mathcal M^{r}$ (see Sections \ref{Section:PDO} and \ref{Section:Mikhlin}), so that they commute with both the curl and Leray projectors. Furthermore, we assume that at least one of $\Q_1$ and $\Q_2$ is non-trivial and strictly nonlocal (i.e., \textbf{not} a classical first-order transport operator).
		\item $\Q_1,\Q_2$ are \emph{skew-adjoint} in $L^2$ (i.e., $\Q_i^* = -\Q_i$). By It\^o's formula, this ensures the almost sure conservation of kinetic energy: $\|u(t)\|_{L^2} = \|u_0\|_{L^2}$ (see \ref{Q-n cancel 0} in Lemma \ref{Lemma:Qn}).
	\end{enumerate}
	These two ideal assumptions are much stronger than \ref{Hypo-Q-Pi}. However, whether a strong solution to \eqref{SEuler-2D} in $\K^2$ exists globally \textbf{remains an open problem}.
	
	\appendix
	
	\section{Technical Background and Auxiliary Results}
	\label{Section : Appendix}

	This appendix collects technical background and auxiliary results required for the proofs in the main text.

	\subsection{Stochastic Background}

	\subsubsection{It\^o's formula}
	
	It\^o's formula for stochastic processes driven by compensated Poisson random measures associated with L\'evy processes is well-established in the literature. For our purposes, we state it in the following convenient form and refer to \cite[Theorem 3.5.3]{Zhu-2010-Thesis} for a more general formulation. 
	
	Let $\eta$ be a Poisson random measure counting jumps of size $|l| \in (0,1]$, with intensity measure $\nu$ and compensator $\widetilde{\eta} = \eta -  \nu\otimes\mathrm{d}t $.
	
	\begin{Lemma}\label{Ito formula}
		Let \(\H\) be a Hilbert space and \(X\) a process given by
		\[
		X(t) = X_0 + \int_0^t a(t') \, \mathrm{d}t' + \int_0^t \int_{|l| \le 1} f(l,t') \, \widetilde{\eta}(\mathrm{d}l,\mathrm{d}t'), \quad t \geq 0,
		\]
		where \(a \in \mathscr{M}([0,\infty); \H) \cap L^1([0,\infty); \H)\) and \(f\) is a predictable \(\H\)-valued process satisfying
		$$
		\int_0^t \int_{|l| \le 1} \|f(l, t')\|_{\H}^2 \, \nu(\mathrm{d}l) \mathrm{d}t' < \infty\quad \text{for each } t > 0.
		$$
		Let \(\mathbb{S}\) be a separable Hilbert space and \(\phi: \H \rightarrow \mathbb{S}\) be a function of class \(C^1\) such that its Fr\'{e}chet derivative \(\phi': \H \rightarrow \mathcal{L}(\H; \mathbb{S})\) is locally Lipschitz continuous. Then, for every \(t > 0\), we have $\pas$,
		\begin{align*}
			\phi(X(t)) - \phi(X_0) =\ & \int_0^t \Big[\phi'(X(t'-))\Big](a(t')) \, \mathrm{d}t'  + \int_0^t \int_{|l| \le 1} \bigg\{ \phi\Big(X(t'-) + f(l,t')\Big) - \phi(X(t'-))  \bigg\}\, \widetilde{\eta}(\mathrm{d}l,\mathrm{d}t')\\
			& + \int_0^t \int_{|l| \le 1} \bigg\{ \phi\Big(X(t'-) + f(l,t')\Big) - \phi(X(t'-)) - \Big[\phi'(X(t'-))\Big](f(l,t')) \bigg\}\, \nu(\mathrm{d}l)\mathrm{d}t'.
		\end{align*}
	\end{Lemma}
	
	\subsubsection{Skorokhod Topology} 
	
	We recall some basic facts on the Skorokhod topology in this section, and we refer the readers to \cite{Ethier-Kurtz-1986-Book} for more details.  Let $E$ be a metric space equipped with a metric $d_{E}(\cdot,\cdot)$, and let 
	$$d_{1,E}(\cdot,\cdot)\triangleq 1\land d_{E}(\cdot,\cdot).$$
	
	Let $T>0$. 	Define $\mathfrak{F}$ as the collection of strictly increasing continuous functions $f: [0,T] \to [0,T]$ satisfying $f(0) = 0$, $f(T) = T$. Let $\mathfrak{F}_L$ denote the subset of Lipschitz continuous $f \in \mathfrak{F}$ such that
	\begin{equation*}
		\Upupsilon(f) \triangleq  \sup_{[t_2,t_1]\subset [0,T]} \left| \log \frac{f(t_1) - f(t_2)}{t_1-t_2} \right| < \infty.
	\end{equation*}
	The Skorokhod metric on $D([0,T];E)$ is defined by
	$$d_S(g_1, g_2) \triangleq \inf_{f \in \mathfrak{F}_L} \left[ \Upupsilon(f) \lor \sup_{t \in [0,T]} d_{1,E}\big(g_1(t), g_2(f(t))\big) \right],\quad g_1,g_2\in D([0,T];E),$$
	and 
	the topology (on \( D([0,T];E) \)) induced by the metric \( d_S \) is called  Skorokhod topology. The following proposition gives a characterization of convergence in $D([0,T];E)$ equipped with the Skorokhod topology, as shown in \cite[Propositions 5.2 \& 5.3]{Ethier-Kurtz-1986-Book}.

	\begin{Lemma}\label{Lemma: D-convergence}
		Let \( \{g_n\}_{n\ge1} \) and \( g \) be in \( D([0,T];E) \). Then we have the following equivalent conditions for the convergence \( \lim_{n \to \infty} d_S(g_n, g) = 0 \):
		\begin{enumerate}[label={{\bf (\arabic*)}},leftmargin=0.79cm]\setlength\itemsep{0.2em}
			\item \( \lim_{n \to \infty} d_S(g_n, g) = 0 \).
			\item There exists a sequence of time changes \( \{f_n\}_{n\ge1} \subset \mathfrak{F}_L \) such that \( \lim_{n \to \infty} \Upupsilon(f_n) = 0 \) holds and
			\begin{equation}\label{r-converge}
				\lim_{n \to \infty} \sup_{t \in [0,T]} d_{E}\big(g_n(t), g(f_n(t))\big) = 0.
			\end{equation}
			\item There exists a sequence of time changes \( \{f_n\}_{n\ge1} \subset \mathfrak{F} \) such that \( \lim_{n \to \infty} \sup_{t \in [0,T]} |f_n(t) - t| = 0 \) and \eqref{r-converge} holds.
		\end{enumerate}
		
		Particularly, if \( \lim_{n \to \infty} d_S(g_n, g) = 0 \) holds, then for all continuity points \( r \in [0,T] \) of \( g \), we have the pointwise convergence:
		\[
		\lim_{n \to \infty} g_n(r) = \lim_{n \to \infty} g_n(r-) = g(r).
		\]
	\end{Lemma}

	\subsubsection{Weak Convergence of Probability Measures}\label{Section: Probability Measures}
	We recall some facts regarding the weak convergence of probability measures on metric space. Let $(E,d_E)$ be an \textbf{arbitrary} metric space.  
	A family of probability measures $\mathscr{Q}\subset \mathbf{P}(E)$ is tight if for each $\epsilon>0$, there is  a compact set $K_\epsilon\subset E$ such that 
	$$\inf_{\mu\in\mathscr{Q}}\mu(K_\epsilon)>1-\epsilon.$$
	The following lemma can be found at \cite[Page 105, Corolary 2.3]{Ethier-Kurtz-1986-Book}:
	\begin{Lemma}\label{Lemma-tightness}
		Let $(E,d_E)$ be an arbitrary metric space and let $(\mathbf{P}(E),\widetilde d_E)$  be the corresponding space of probability measures equipped with the Prohorov metric $\widetilde d_E$. If $\mathscr{Q}\subset \mathbf{P}(E)$ is tight, then there is a sequence $\{\mu_n\}_{n\ge1}\subset \mathscr{Q}$ and a probability measure $\mu\in \mathbf{P}(E)$ such that $\widetilde d_E(\mu_n,\mu)\to0$ as $n\to\infty$.
	\end{Lemma}
	Let $\{\mu_n,\mu\}_{n\ge1}\subset \mathbf{P}(E)$. 
	$\{\mu_n\}_{n\ge1}$ is called to converge weakly to $\mu$ in $\mathbf{P}(E)$ if 
	$$\lim_{n\to\infty}\int_E f(x)\mu_{n}({\rm d}x)=\int_Xf(v)\mu({\rm d}v),\quad f\in C_B(E).$$
	Then we refer to \cite[Page 108, Theorem 3.1]{Ethier-Kurtz-1986-Book} for the following lemma:
	\begin{Lemma}\label{Lemma-convergence of measures}
		Let $(E,d_E)$ be an arbitrary metric space.
		Let $(\mathbf{P}(E),\widetilde d_E)$  be the corresponding space of probability measures equipped with the Prohorov metric $\widetilde d_E$. Then,  for $\{\mu_n,\mu\}_{n\ge1}\subset \mathbf{P}(E)$, the following conditions \ref{measure conver 2.2}--\ref{measure conver 2.4} are equivalent and are implied by \ref{measure conver 2.1}:
		\begin{enumerate}[leftmargin=0.79cm,label={${\bf (\alph*)}$}]\setlength\itemsep{0.2em}
			\item\label{measure conver 2.1} $\widetilde d_E(\mu_n,\mu)\to0$ as $n\to\infty$;
			\item\label{measure conver 2.2} $\mu_n$ converges weakly to $\mu$ as $n\to\infty$;
			\item\label{measure conver 2.4} For all open subset $G\subset E,\, \mu(G) \le \liminf_{n\to\infty}\mu_n(G)$.
		\end{enumerate}
		
		Moreover, when $(E,d_E)$ is separable, then all the conditions 
		\ref{measure conver 2.1}--\ref{measure conver 2.4} are equivalent.
	\end{Lemma}

	\subsection{Deterministic Estimates and Analysis}
	\subsubsection{Mollifiers and Related Estimates}
	
	We first introduce the Friedrichs mollifier.
	Recall the operator $\OP$ defined in \eqref{OP define}. For $n\ge 1$, we define the Friedrichs mollifier $J_n$ by
	\begin{align}
		J_{n}  \triangleq {\mathrm{OP}}\big(j(\cdot /n)\big),\qquad n\ge 1,
		\label{Define Jn}
	\end{align}
	where $j\in \mathscr{S}(\mathbb R^{d};\mathbb R)$ (the Schwartz space of rapidly decreasing $C^{\infty}$ functions on $\mathbb R^{d}$) is even function in each components and satisfies
	$0\leq j(y)\leq 1\ \text{for all}\ y\in \mathbb R^{d}$, 
	$j(y)=1\ \text{whenever}\ |y|\leq 1.$
	
	\begin{Lemma}[\cite{Li-Liu-Tang-2021-SPA,Taylor-2011-PDEbook3}]\label{Lemma-Jn}
		Let $\K=\R$ or $\T$. Let $J_n$ be defined by \eqref{Define Jn}. Then the following properties hold:
		\begin{enumerate}[label={${\bf (\arabic*)}$},leftmargin=0.79cm]\setlength\itemsep{0.2em}
			\item $J_n$ is self-adjoint, that is,
			$$\left \langle J_{n}f, g\right \rangle _{L^{2}}= \left \langle f, J_{n}g
			\right \rangle _{L^{2}},\quad f,g\in L^2.$$

			\item For all $\sigma\ge0$, $n\ge1$, and $f\in H^\sigma$, we have
			\[
			\sup_n\|J_{n}\|_{\mathscr L(L^{\infty};L^{\infty})}<\infty,\qquad
			\sup_n\|J_{n}\|_{\mathscr L(H^{\sigma};H^{\sigma})}\leq 1,
			\qquad \text{and} \qquad
			\lim_{n\to\infty}  \|J_n f-f\|_{H^\sigma}=0.
			\]
			Moreover, for all $s>\sigma\ge0$,
			\begin{align*}
				\|{\mathbf I}-J_{n}\|_{\mathscr L(H^{s};H^{\sigma})}\lesssim  
				\frac{1}{n^{s-\sigma}},\qquad
				\|J_{n}\|_{\mathscr L(H^{\sigma};H^{s})}\sim  O(n^{s-\sigma}).
			\end{align*}
			
			\item Let $\D^s$, $\Pi_d$, and $\Pi_0$ be given in \eqref{Define Ds}, \eqref{Pi-d define}, and \eqref{Pi-0 define}, respectively. For all $d\ge2$, $n\ge 1$, and $s\ge 0$, whenever the following commutators are well defined (for instance, $[\Pi_d,J_n]$ acts on zero-average functions if $x\in\T^d$), we have
			\begin{align*}
				[\mathcal D^{s},J_{n}]=[\Pi_d,J_{n}]=[\Pi_0,J_n]=0.
			\end{align*}
			
		\end{enumerate}
	\end{Lemma}

	Then we recall some estimates in the Sobolev spaces $H^{s}$.
	\begin{Lemma}[\cite{Kenig-Ponce-Vega-1991-JAMS}]
		\label{KP-commutator}
		Let $d\ge1$.
		If $f,g\in H^{s}\bigcap W^{1,\infty}$ with $s>0$, then for
		$p,p_{i}\in (1,\infty )$ with $i=2,3$ and
		$\frac{1}{p}=\frac{1}{p_{1}}+\frac{1}{p_{2}}=\frac{1}{p_{3}}+
		\frac{1}{p_{4}}$, we have
		\begin{equation*}
			\|\left [\mathcal D^{s},f{\mathbf I}\right ]g\|_{L^{p}}\leq C_{s}(\|
			\nabla f\|_{L^{p_{1}}}\|\mathcal D^{s-1}g\|_{L^{p_{2}}}+\|\mathcal D^{s}f
			\|_{L^{p_{3}}}\|g\|_{L^{p_{4}}}),
		\end{equation*}
		and
		\begin{equation*}
			\|\mathcal D^{s}(fg)\|_{L^{p}}\leq C_{s}(\|f\|_{L^{p_{1}}}\|
			\mathcal D^{s} g\|_{L^{p_{2}}}+\|\mathcal D^{s} f\|_{L^{p_{3}}}\|g\|_{L^{p_{4}}}).
		\end{equation*}

	\end{Lemma}

	\begin{Lemma}\label{uux u M}
		Let $d\ge1$, $s>\frac{d}{2}+1$. Then there exists a constant $\mathscr{c}=\mathscr{c}_{s,d}>0$ such that
		\begin{equation*}
			\IP{[(v\cdot\nn)v],v}_{H^s}\leq \mathscr{c}\|v\|_{\Wlip}\|v\|^2_{H^s},\qquad v\in H^{s+1},\ s>\frac{d}{2}+1.
		\end{equation*}
	\end{Lemma}
	\begin{proof}
		Applying Lemmas \ref{Lemma-Jn} and \ref{KP-commutator}, we arrive at
		\begin{align*}
			\IP{[(v\cdot\nn)v],v}_{H^s} 
			= \IP{[\D^s,(v\cdot\nn)]\,  v,\D^s  v}_{L^2}
			+\IP{(v\cdot\nn) \D^s v,\D^s  v}_{L^2} 
			\leq \mathscr{c}\|v\|_{\Wlip}\|v\|^2_{H^s},
		\end{align*}
		which is the desired estimate.
	\end{proof}

	\begin{Lemma}\label{Lemma-F}
		Let $s>\frac{d}{2}+1$.
		Let  $F$ be defined in \eqref{F-Euler}.  
		\begin{itemize}[leftmargin=0.79cm]\setlength\itemsep{0.2em}
			\item If \ref{Constant Sound} holds, then there exists a constant $c_{1,s}>0$ such that   for all $X=(\vrho,u)^T\in \H^{s+1}$,
			\begin{equation}
				\Abs{\IP{F(X), X}_{\H^s}}\leq c_{1,s}  \big(\|\nabla u\|_{L^\infty}+\|\nabla \vrho\|_{L^\infty}\big)\|X\|_{\H^s}^2.\label{F(X) X-1}
			\end{equation}

			\item If \ref{General Sound} holds, then there exists a constant $c_{2,s}>0$ such that  for all $X=(\vrho,u)^T\in \H^{s+1}$,
			\begin{equation}
				\Abs{\IP{F(X), X}_{\H^s}}\leq c_{2,s} \big(1+\Phi_\Lambda(\|\vrho\|_{L^\infty})\big)\big(\|\nabla u\|_{L^\infty}+\|\nabla \vrho\|_{L^\infty}\big)\|X\|_{\H^s}^2.\label{F(X) X-2}
			\end{equation}
		\end{itemize}

	\end{Lemma}
	
	\begin{proof} 
		We only prove the case where \ref{General Sound} holds, since it is more involved than the first case where \eqref{F(X) X-1} holds.
		When $X=(\vrho,u)\in\H^{s+1}$, we have
		\begin{align*}
			\IP{F(X), X}_{\H^s}=\ &\IP{\D^s(\Lambda(\vrho) \, \Div u),\D^s\vrho}_{L^2}
			+\IP{\D^s(u\cdot\nabla \vrho),\D^s\vrho}_{L^2}\\
			&+\IP{\D^s((u\cdot\nn)u),\D^su}_{L^2}+\IP{\D^s(\Lambda(\vrho)\,\nabla \vrho),\D^su}_{L^2}\\ \triangleq\ & \sum_{i=1}^{4}I_i.
		\end{align*}
		By integration by parts, we have
		\begin{align*}
			I_1=\ &\bIP{[\D^s,\Lambda(\vrho)]\, \Div u,\D^s\vrho}_{L^2}
			+\bIP{\Lambda(\vrho)\,\D^s\, \Div u,\D^s\vrho}_{L^2},
		\end{align*}
		and
		\begin{align*}
			I_4=\ &\bIP{[\D^s,\Lambda(\vrho)]\nabla \vrho,\D^su}_{L^2}
			+\bIP{\Lambda(\vrho)\,\D^s\nabla \vrho,\D^su}_{L^2}\\
			=\ &\IP{[\D^s,\Lambda(\vrho)]\nabla \vrho,\D^su}_{L^2}
			-\IP{(\nabla \Lambda(\vrho))\D^s \vrho,\D^su}_{L^2}-\IP{\Lambda(\vrho)\D^s \vrho,\D^s\, \Div u}_{L^2}.
		\end{align*}
		Therefore, the (singular) term $\bIP{\Lambda(\vrho)\D^s \vrho,\D^s\, \Div u}_{L^2}$ cancels out in $I_1+I_4$. 
		
		On account of $\norm{\nabla \Lambda(\vrho)}_{L^\infty}\lesssim \norm{\nabla \vrho}_{L^\infty}$ (by  \ref{Hypo-Pressure}) and Lemma \ref{KP-commutator}, we derive
		\begin{align}
			|I_1+I_4|
			\lesssim\ & \Big(\Phi_\Lambda(\|\vrho\|_{L^\infty})\big(\|\nabla u\|_{L^\infty}+\|\nabla \vrho\|_{L^\infty}\big)+\|\nabla \vrho\|_{L^\infty}\Big)\|X\|_{\H^s}^2.\label{q=r(rho) no singularity}
		\end{align}
		Similarly, we obtain
		\begin{align*}
			|I_2|=\ &\Abs{\IP{[\D^s,u]\cdot\nabla \vrho,\D^s\vrho}_{L^2}
				+\IP{u\cdot\D^s\nabla \vrho,\D^s\vrho}_{L^2}}\\
			\leq\ & (\|u\|_{H^s}\|\nabla \vrho\|_{L^\infty}+\|\nabla u\|_{L^\infty}\|\vrho\|_{H^s})\|\vrho\|_{H^s}
			+\Abs{\frac12\int_{\K^d}(\Div u)(\D^s \vrho)^2\d x}\\
			\lesssim\ & \big(\|\nabla u\|_{L^\infty}+\|\nabla \vrho\|_{L^\infty}\big)\|X\|_{\H^s}^2,
		\end{align*}
		and
		\begin{align*}
			|I_3|=\Abs{\IP{[\D^s,(u\cdot\nn)]u,\D^su}_{L^2}
				+\IP{(u\cdot\nn)\D^su,\D^su}_{L^2} }
			\lesssim \|\nabla u\|_{L^\infty}\|u\|_{H^s}^2.
		\end{align*}
		Collecting the above estimates, we obtain \eqref{F(X) X-2}. 
	\end{proof}
	
	\begin{Lemma}\label{Lemma-F-difference}
		Let $s>\frac{d}{2}+2$ and $\theta\in(\frac{d}{2}+1,s-1)$.
		Let $J_n$ and $F$ be defined in \eqref{Define Jn} and \eqref{F-Euler}, respectively. 
		\begin{itemize}[leftmargin=0.79cm]\setlength\itemsep{0.2em}
			\item If \ref{Constant Sound} holds and $X,Y\in\H^s$, then 
			\begin{align}
				\IP{J_n F(J_nX)-J_m F(J_mY),X-Y}_{\H^\theta} 
				\lesssim  \Big(1+\|X\|^4_{\H^s}+\|Y\|^4_{\H^s}\Big)
				\left((n\wedge m)^{-2(s-1-\theta)}+\|X-Y\|^2_{\H^\theta}\right).\label{F-asym-monoton-1}
			\end{align}

			\item If \ref{General Sound} holds and $X,Y\in\H^s$, then 
			\begin{align*}
				&\IP{J_n F(J_nX)-J_m F(J_mY),X-Y}_{\H^\theta}\notag\\
				\lesssim\ &  \Big(1+ \Phi^2_{\Lambda}\big(\|X\|_{\H^{s}}+\|Y\|_{\H^{s}}\big)+\|X\|^4_{\H^s}+\|Y\|^4_{\H^s}\Big)
				\left((n\wedge m)^{-2(s-1-\theta)}+\|X-Y\|^2_{\H^\theta}\right).\label{F-asym-monoton-2}
			\end{align*}
		\end{itemize}

	\end{Lemma}
	\begin{proof} 
		We first prove \eqref{F-asym-monoton-2}. 
		Let $X=(\vrho,u)^T$ and $Y=(\varpi,v)^T$. We have
		\begin{align*}
			J_n F(J_nX)- J_m  F( J_m Y) = \left(
			\begin{array}{c}
				\Psi_1+\Psi_2\\
				\Psi_3+\Psi_4
			\end{array}
			\right),
		\end{align*}
		where 
		\begin{align*}
			\Psi_1 &= J_n\Big(\Lambda(J_n \vrho)\,\Div (J_nu)\Big) -  J_m \Big(\Lambda( J_m  \varpi)\,\Div ( J_m v)\Big),\\ 
			\Psi_2 &=  J_n\Big(J_nu\cdot\nabla J_n \vrho\Big) -  J_m \Big( J_m v\cdot\nabla  J_m  \varpi\Big),\\
			\Psi_3 &= J_n\Big((J_nu\cdot\nn)J_nu\Big) -  J_m \Big(( J_m v\cdot\nn) J_m v\Big),\\ 
			\Psi_4 &= J_n\Big(\Lambda(J_n \vrho)\,\nabla  J_n \vrho\Big) -  J_m \Big(\Lambda( J_m  \varpi)\,\nabla   J_m  \varpi\Big).
		\end{align*}
		For $\Psi_1$ and $\Psi_4$, we have
		\begin{align*}
			\Psi_1 =\ &(J_n- J_m )\Big(\Lambda(J_n \vrho)\,\Div (J_nu)\Big)
			+ J_m \Big(\big(\Lambda(J_n \vrho)-\Lambda( J_m  \vrho)\big)\,\Div (J_nu)\Big)\\
			&+ J_m \Big(\big(\Lambda( J_m  \vrho)-\Lambda( J_m  \varpi)\big)\,\Div (J_nu)\Big)
			+ J_m \Big(\Lambda( J_m  \varpi)\,\Div (J_nu- J_m u)\Big)\\
			&+ J_m \Big(\Lambda( J_m  \varpi)\,\Div ( J_m u- J_m v)\Big)\\
			\triangleq\ &\sum_{j=1}^{5}\Psi_{1,j},  
		\end{align*}
		and 
		\begin{align*}
			\Psi_4 =\ &(J_n- J_m )\Big(\Lambda(J_n \vrho)\,\nabla  J_n \vrho\Big) + 
			J_m \Big(\big(\Lambda(J_n \vrho)-\Lambda( J_m  \vrho)\big)\,\nabla  J_n \vrho\Big)\\
			&+ J_m \Big(\big(\Lambda( J_m  \vrho)-\Lambda( J_m  \varpi)\big)\,\nabla  J_n \vrho\Big)
			+ J_m \Big(\Lambda( J_m  \varpi)\,\nabla  (J_n \vrho- J_m  \vrho)\Big)\\
			&+ J_m \Big(\Lambda( J_m  \varpi)\,\nabla  ( J_m  \vrho- J_m  \varpi)\Big)\\
			\triangleq\ &\sum_{j=1}^{5}\Psi_{4,j}.
		\end{align*}
		Since $\theta\in\left(\frac{d}{2}+p,{s-1}\right)$, we use Lemma \ref{Lemma-Jn} and \ref{Hypo-Pressure} to  obtain that
		\begin{align*}
			&\bIP{\Psi_{1,1},\vrho-\varpi}_{H^{\theta}} \lesssim \Phi_{\Lambda}(\|\vrho\|_{L^{\infty}})
			(n\wedge m)^{-(s-1-\theta)}\|\vrho-\varpi\|_{H^\theta}\|u\|_{H^s}\|\vrho\|_{H^s},\\
			&\bIP{\Psi_{1,2},\vrho-\varpi}_{H^{\theta}} \lesssim \Phi_{\Lambda}(\|\vrho\|_{H^{\theta}})(n\wedge m)^{-(s-\theta)}\|\vrho-\varpi\|_{H^\theta}\|\vrho\|_{H^s}\|u\|_{H^s},\\
			&\bIP{\Psi_{1,3},\vrho-\varpi}_{H^{\theta}} \lesssim \Phi_{\Lambda}(\|\vrho\|_{H^{\theta}}+\|\varpi\|_{H^{\theta}})\|\vrho-\varpi\|_{H^\theta}^2\|u\|_{H^s},\\
			&\bIP{\Psi_{1,4},\vrho-\varpi}_{H^{\theta}} \lesssim \Phi_{\Lambda}(\|\varpi\|_{L^{\infty}})(n\wedge m)^{-(s-\theta-1)}\|\vrho-\varpi\|_{H^\theta}\|\varpi\|_{H^s}\|u\|_{H^s}\\
			&\bIP{\Psi_{1,5},\vrho-\varpi}_{H^{\theta}} \lesssim \Phi_{\Lambda}(\|\varpi\|_{L^{\infty}})\|\vrho-\varpi\|_{H^\theta}\|\varpi\|_{H^s}\big(\|u\|_{H^s}+\|v\|_{H^s}\big).
		\end{align*}
		Similarly, we have
		\begin{align*}
			&\bIP{\Psi_{4,1},u-v}_{H^{\theta}}\lesssim \Phi_{\Lambda}(\|\vrho\|_{L^{\infty}})(n\wedge m)^{-(s-\theta-1)}\|u-v\|_{H^\theta}\|\vrho\|_{H^s}^2,\\
			&\bIP{\Psi_{4,2},u-v}_{H^{\theta}}\lesssim \Phi_{\Lambda}(\|\vrho\|_{H^{\theta}})(n\wedge m)^{-(s-\theta)}\|u-v\|_{H^\theta}\|\vrho\|_{H^s}^2,\\
			&\bIP{\Psi_{4,3},u-v}_{H^{\theta}}\lesssim \Phi_{\Lambda}(\|\vrho\|_{H^{\theta}}+\|\varpi\|_{H^{\theta}})\|u-v\|_{H^\theta}\|\vrho-\varpi\|_{H^\theta}\|\vrho\|_{H^s},\\
			&\bIP{\Psi_{4,4},u-v}_{H^{\theta}}\lesssim \Phi_{\Lambda}(\|\varpi\|_{L^{\infty}})(n\wedge m)^{-(s-\theta-1)}\|u-v\|_{H^\theta}\|\vrho\|_{H^s}\|\varpi\|_{H^s},\\
			&\bIP{\Psi_{4,5},u-v}_{H^{\theta}}\lesssim \Phi_{\Lambda}(\|\varpi\|_{L^{\infty}})\|u-v\|_{H^\theta}\|\varpi\|_{H^s}\big(\|\vrho\|_{H^s}+\|\varpi\|_{H^s}\big). 
		\end{align*}
		As a result, we have
		\begin{align*}
			&\bIP{\Psi_{1},\vrho-\varpi}_{H^{\theta}}+\bIP{\Psi_{4},u-v}_{H^{\theta}}\\
			\lesssim\ & \Big(1+ \Phi^2_{\Lambda}(\|X\|_{\H^{s}}+\|Y\|_{\H^{s}})+\|X\|^4_{\H^s}+\|Y\|^4_{\H^s}\Big)
			\left((n\wedge m)^{-2(s-1-\theta)}+\|X-Y\|^2_{\H^\theta}\right).
		\end{align*}
		For $\Psi_2$, we have
		\begin{align*}
			\Psi_2 =\ &(J_n- J_m )\Big(J_nu\cdot\nabla J_n \vrho\Big)
			+ J_m \Big((J_n- J_m )u\cdot\nn(J_n \vrho)\Big)\\
			&+ J_m \Big( J_m (u-v)\cdot\nn(J_n \vrho)\Big)
			+ J_m \Big( J_m v\cdot\nn(J_n \vrho- J_m  \vrho)\Big)\\
			&+ J_m \Big( J_m v\cdot\nabla  J_m (\vrho-\varpi)\Big)\\
			\triangleq\ &\sum_{j=1}^{5}\Psi_{2,j}.
		\end{align*}
		Then we obtain
		\begin{align*}
			&\bIP{\Psi_{2,1},\vrho-\varpi}_{H^{\theta}} \lesssim (n\wedge m)^{-(s-1-\theta)}\|\vrho-\varpi\|_{H^\theta}\|\vrho\|_{H^s}\|u\|_{H^s},\\
			&\bIP{\Psi_{2,2},\vrho-\varpi}_{H^{\theta}} \lesssim(n\wedge m)^{-(s-\theta)}\|\vrho-\varpi\|_{H^\theta}\|\vrho\|_{H^s}\|u\|_{H^s},\\
			&\bIP{\Psi_{2,3},\vrho-\varpi}_{H^{\theta}} \lesssim \left(\|u-v\|_{H^\theta}^2+\|\vrho-\varpi\|_{H^\theta}^2\right)\|\vrho\|_{H^s},\\
			&\bIP{\Psi_{2,4},\vrho-\varpi}_{H^{\theta}} \lesssim (n\wedge m)^{-(s-\theta-1)}\|\vrho-\varpi\|_{H^\theta}\left(\|\vrho\|_{H^s}^2+\|v\|_{H^s}^2\right),\\
			&\bIP{\Psi_{2,5},\vrho-\varpi}_{H^{\theta}} \lesssim \|\vrho-\varpi\|_{H^\theta}^2\|v\|_{H^s}.
		\end{align*}
		Therefore, we arrive at
		\begin{align}
			\bIP{\Psi_2,\vrho-\varpi}_{H^{\theta}}\lesssim 
			\big(1+\|X\|_{\H^s}^2+\|Y\|_{\H^s}^2\big)\left[(n\wedge m)^{-2(s-1-\theta)}+\|X-Y\|_{\H^\theta}^2\right].
		\end{align}
		For $\Psi_3$, we have
		\begin{align*}
			\Psi_3 = \ &(J_n- J_m )\Big((J_nu\cdot\nn)J_nu\Big) + 
			J_m \Big(\big((J_n- J_m )u\cdot\nn\big)J_nu\Big)\\
			&+ 
			J_m \Big(\big( J_m (u-v)\cdot\nn\big)J_nu\Big) + 
			J_m \Big(( J_m v\cdot\nn)(J_n- J_m )u\Big)\\
			&+
			J_m \Big(( J_m v\cdot\nn)( J_m (u-v))\Big)\\
			\triangleq\ &\sum_{j=1}^{5}\Psi_{3,j}.
		\end{align*}
		Then, using Lemma \ref{Lemma-Jn} yields
		\begin{align*}
			&\bIP{\Psi_{3,1},u-v}_{H^{\theta}} \lesssim 
			(n\wedge m)^{-(s-\theta-1)}\|u-v\|_{H^\theta}\|u\|_{H^s}^2,\\
			&\bIP{\Psi_{3,2},u-v}_{H^{\theta}} \lesssim 
			(n\wedge m)^{-(s-\theta)}\|u-v\|_{H^\theta}\|u\|_{H^s}^2,\\
			&\bIP{\Psi_{3,3},u-v}_{H^{\theta}} \lesssim 
			\|u-v\|_{H^\theta}^2\|u\|_{H^s},\\
			&\bIP{\Psi_{3,4},u-v}_{H^{\theta}} \lesssim 
			(n\wedge m)^{-(s-\theta-1)}\|u-v\|_{H^\theta}\left(\|u\|_{H^s}^2+\|v\|_{H^s}^2\right).
		\end{align*}
		For $\bIP{\Psi_{3,5},u-v}_{H^{\theta}}$, we can infer  from integration by parts, $H^s\hookrightarrow H^{\theta}\hookrightarrow W^{p,\infty}$,  Lemmas \ref{Lemma-Jn} and \ref{KP-commutator} that
		\begin{align*}
			&\bIP{\Psi_{3,5},u-v}_{H^{\theta}} \\
			=\ &\sum_{i=1}^{d}
			\Big(\IP{[\D^{\theta},( J_m v)_i]\partial_{x_i} J_m (u-v),\, \D^{\theta} J_m (u-v)}_{L^2}+\IP{( J_m v)_i\partial_{x_i}\D^{\theta} J_m (u-v),\, \D^{\theta} J_m (u-v)}_{L^2}\Big)\\
			\lesssim\ &\sum_{i=1}^{d} \Big(\|\nabla v_i\|_{L^\infty}\|u-v\|^2_{H^{\theta}}
			+\| v_i\|_{H^{\theta}}\|u-v\|_{L^\infty}\|u-v\|_{H^{\theta}}\Big)
			+\|\Div  v\|_{L^\infty}\|u-v\|^2_{H^{\theta}}\\
			\lesssim\ & \|v\|_{H^s}\|u-v\|^2_{H^{\theta}}.
		\end{align*}
		Hence we arrive at
		\begin{align*}
			\bIP{\Psi_3,u-v}_{H^{\theta}} \lesssim \big(1+\|X\|_{\H^s}^2+\|Y\|_{\H^s}^2\big)\left[(n\wedge m)^{-2(s-1-\theta)}+\|X-Y\|_{\H^\theta}^2\right].
		\end{align*}
		Collecting the above estimates yields the desired estimate \eqref{F-asym-monoton-2}. When \ref{Constant Sound} holds, $\Lambda(\cdot)\equiv \Theta$. One can repeat the above analysis to obtain \eqref{F-asym-monoton-1}.
	\end{proof}

	\subsubsection{Some Properties of Pseudo-differential Operators}
	\begin{Lemma}
		\label{LOP}
		Let $r,r_{1},r_{2}\in \mathbb R$, $\mathscr{p}\in \mathbf S^{r}$,
		$\mathscr{p}_{1}\in \mathbf S^{r_{1}},\mathscr{p}_{2}\in \mathbf S^{r_{2}}$,
		and let ${\mathrm{OP}}$ be given in \eqref{OP define}. The following results hold:
		
		\begin{enumerate}[label=\textup{\textbf {(\arabic{enumi})}},leftmargin=0.79cm]\setlength\itemsep{0.2em} 
			\item
			\label{OP-continuous}
			\textup{\textbf{(Continuity of ${\normalfont{\mathrm{OP}}}$)}} For any $q,s\in \mathbb R$,
			${\mathrm{OP}}: \mathbf S^{s}\to \OP\SS^s$ is one-to-one  and ${\mathrm{OP}}: \mathbf S^{s}\to \mathscr L(H^{q+s};H^{q})$ is bounded. More precisely, there
			are $\widetilde \beta ,\widetilde \alpha \in \mathbb N_{0}^{d}$ and a constant
			$C=C(s,q)>0$ such that
			\begin{align*}
				\|{\mathrm{OP}}(\mathscr{p})\|_{\mathscr L(H^{q+s};H^{q})}\leq C(s,q)|\mathscr{p}|^{\widetilde \beta ,\widetilde \alpha ; s},
			\end{align*}
			where $|\cdot|^{\widetilde \beta ,\widetilde \alpha ; s}$ is the seminorm given in \eqref{OPS+seminorms}.
			
			\item
			\textup{\textbf{(Adjoint)}} There is a symbol $\widetilde{\mathscr{p}}\in \SS^{r}$ such that
			$$\big({\mathrm{OP}}(\mathscr{p})\big)^{*}={\mathrm{OP}}(\widetilde{\mathscr{p}}),\quad 
			\widetilde{\mathscr{p}}\equiv \mathscr{p}\ \text{mod}\ \SS^{r-1},$$
			and 
			\begin{equation*}
				\mathbf S^{r}\ni \mathscr{p}\mapsto \widetilde{\mathscr{p}}\in
				\mathbf S^{r}  \ \text{is continuous}.
			\end{equation*}

			\item 
			\textup{\textbf{(Composition)}} 
			There is a symbol denoted by $\mathscr{p}_{1}\#\mathscr{p}_{2}\in \SS^{r_1+r_2}$ such that
			$${\mathrm{OP}}(\mathscr{p}_1){\mathrm{OP}}(\mathscr{p}_2)={\mathrm{OP}}(
			\mathscr{p}_{1}\#\mathscr{p}_{2}),\quad \mathscr{p}_{1}\#\mathscr{p}_{2}\equiv\mathscr{p}_{1}\mathscr{p}_{2} \ \text{mod}\ \SS^{r_1+r_2-1},$$
			and  
			\begin{equation*}
				\mathbf S^{r_{1}}\times \mathbf S^{r_{2}}\ni (\mathscr{p}_{1},
				\mathscr{p}_{2})\mapsto \mathscr{p}_{1}\# \mathscr{p}_{2}\in
				\mathbf S^{r_{1}+r_{2}}\ \text{is bi-linear continuous},
			\end{equation*}

			\item 
			\textup{\textbf{(Commutator)}} If
			$\mathscr{p}_{1}\in \mathbf S^{r_{1}}$ and
			$\mathscr{p}_{2}\in \mathbf S^{r_{2}}$ are commuting matrices, then
			\begin{equation*}
				\big[{\mathrm{OP}}(\mathscr{p}_{1}),{\mathrm{OP}}(\mathscr{p}_{2})\big]\in {
					\mathrm{OP}}\mathbf S^{r_{1}+r_{2}-1}.
			\end{equation*}
		\end{enumerate}
	\end{Lemma}
	
	\begin{proof}
		For readers' convenience, we provide key references for the following  well-established properties. The first property is treated in \cite[Theorem 2.7, Page 124 \& Proposition 1.2, Page 56]{Kumano-go-1981-Book} and \cite[Theorem 3.41]{Abels-2012-Book}. The second result is detailed in \cite[Theorem 1.1.21 \& Corollary 1.1.22]{Lerner-2010-Book}. Regarding the third property, comprehensive discussions appear in \cite[Theorem 1.2.16]{Nicola-Rodino-2010-book} and \cite[Page 72]{Abels-2012-Book}. For the fourth aspect, we refer to \cite[Page 32]{Taylor-1974-note} and \cite[Theorem C.3]{Benzoni-Gavage-Serre-2007-Book}.
	\end{proof}
	
	\begin{Lemma}
		\label{Lemma-[pn qm]}
		Let $s,\, r_{1},\, r_{2}\in \mathbb R$. Suppose that
		$\mathscr{N}\times \mathscr{O}\subset \mathbf S^{r_{1}}\times
		\mathbf S^{r_{2}}$ is a bounded subset such that for any
		$(\mathscr{p}_{1},\mathscr{p}_{2})\in \mathscr{N}\times \mathscr{O}$,
		$\mathscr{p}_{1}$ and $\mathscr{p}_{2}$ are commuting matrices. Then
		we have:
		\begin{equation*}
			\sup _{(\mathscr{p}_{1},\mathscr{p}_{2})\in \mathscr{N}\times
				\mathscr{O}}\big\|\big[{\mathrm{OP}}(\mathscr{p}_{1}),{\mathrm{OP}}(
			\mathscr{p}_{2})\big]\big\|_{\mathscr L(H^{s+r_{1}+r_{2}-1};H^{s})}<
			\infty .
		\end{equation*}
	\end{Lemma}
	\begin{proof}
		When $\mathscr{p}_1$ and $\mathscr{p}_2$ are commuting matrices, some direct
		computations (cf. \cite[Corollary 1.1.22]{Lerner-2010-Book} or
		\cite[Theorem C.3]{Benzoni-Gavage-Serre-2007-Book}) yield
		$\mathscr{p}_{1}\# \mathscr{p}_{2}-\mathscr{p}_{2}\# \mathscr{p}_{1}
		\in \mathbf S^{r_{1}+r_{2}-1}$. From this and {Lemma~\ref{LOP}}, we see that
		\begin{equation*}
			(\mathscr{p}_{1},\mathscr{p}_{2})\mapsto \big[{\mathrm{OP}}(\mathscr{p}_{1}),{
				\mathrm{OP}}(\mathscr{p}_{2})\big] ={\mathrm{OP}}\big(\mathscr{p}_{1}\#
			\mathscr{p}_{2}- \mathscr{p}_{2}\# \mathscr{p}_{1}\big)
		\end{equation*}
		is continuous from $\mathbf S^{r_{1}}\times \mathbf S^{r_{2}}$ to
		$\mathscr L(H^{s+r_{1}+r_{2}-1};H^{s})$, which implies the desired result.
	\end{proof}

	\setlength{\bibsep}{1.5ex}

\end{document}